\documentclass[a4paper,12pt]{article}

\usepackage{amscd}
\usepackage{amsfonts}
\usepackage{amsmath}
\usepackage{amssymb}
\usepackage{amsthm}
\usepackage{ascmac}
\usepackage{mathrsfs}
\usepackage{txfonts}
\usepackage[all]{xy}

\allowdisplaybreaks

\newcommand{\C}{\mathbb{C}}
\newcommand{\F}{\mathbb{F}}
\newcommand{\N}{\mathbb{N}}
\newcommand{\Q}{\mathbb{Q}}
\newcommand{\R}{\mathbb{R}}
\newcommand{\Z}{\mathbb{Z}}

\newcommand{\Aff}{\mathbb{A}}
\newcommand{\G}{\mathbb{G}}
\newcommand{\Hlf}{\mathbb{H}}
\newcommand{\Mod}{\mathbb{M}}

\newcommand{\T}{\mathbb{T}}

\newcommand{\A}{\mathscr{A}}

\newcommand{\Cat}{\mathscr{C}}

\newcommand{\Fil}{\mathscr{F}}

\newcommand{\Hil}{\mathscr{H}}
\newcommand{\I}{\mathscr{I}}

\newcommand{\Leb}{\mathscr{L}}
\newcommand{\M}{\mathscr{M}}
\newcommand{\Nor}{\mathscr{N}}
\newcommand{\Opn}{\mathscr{O}}

\newcommand{\U}{\mathscr{U}}

\newcommand{\W}{\mathscr{W}}

\theoremstyle{plain}
\newtheorem{thm}{Theorem}[section]
\newtheorem{lmm}[thm]{Lemma}
\newtheorem{prp}[thm]{Proposition}
\newtheorem{crl}[thm]{Corollary}
\theoremstyle{definition}
\newtheorem{dfn}[thm]{Definition}
\newtheorem{rmk}[thm]{Remark}
\newtheorem{exm}[thm]{Example}

\def\ens#1{\{ #1 \}}
\def\Ens#1{\left\{ #1 \right\}}
\def\set#1#2{\{ #1 \mid #2 \}}
\def\Set#1#2{\left\{ #1 \ \middle| \ #2 \right\}}
\def\t#1{\text{\rm #1}}
\def\m#1{\text{#1}}
\def\v#1{| #1 |}
\def\vv#1{\left| #1 \right|}
\def\n#1{\| #1 \|}

\newcommand{\im}{\t{im}}

\newcommand{\coim}{\t{coim}}

\title{Galois Representations Associated to $p$-adic Families of Modular Forms of Finite Slope}
\author{Tomoki Mihara}
\date{}

\begin{document}

\maketitle
\begin{abstract}
We define a pro-$p$ Abelian sheaf on a modular curve of a fixed level $N \geq 5$ divisible by a prime number $p \neq 2$. Every $p$-adic representation of $\t{Gal}(\overline{\Q}/\Q)$ associated to an eigenform is obtained as a quotient of its \'etale cohomology. For any compact $\Z_p[[1 + N \Z_p]]$-algebra $\Lambda_1$ satisfying certain suitable conditions, we construct a representation of $\t{Gal}(\overline{\Q}/\Q)$ over $\Lambda_1$ associated to a $\Lambda_1$-adic cuspidal eigenform of finite slope as a scalar extension of a quotient of the \'etale cohomology.
\end{abstract}

\tableofcontents

\section{Introduction}
\label{Introduction}

Let $p$ be a prime number with $p \neq 2$. We give a new explicit geometric construction of a $p$-adic family of Galois representations associated to modular forms of a fixed level $N \geq 5$ with $p \mid N$ of finite bounded slope. The result is deeply related to one of@'Ä'¥ open questions in \cite{CM98}. R.\ Coleman and B.\ Mazur originally defined the eigencurve of tame level $1$ in \cite{CM98} 6.1 Definition 1. Excluding a discrete subspace from the reduced eigencurve, they constructed a continuous representation of $\t{Gal}(\overline{\Q}/\Q)$ of rank $2$ over the ring of rigid analytic functions from the pseudo-representation obtained as the pull-back of the universal deformation of a pseudo-representation over a finite field. Then they asked a question whether this Galois representation is obtained as the Pontryagin dual of the direct limit of \'etale cohomologies of a tower of modular curves. Instead of the original reduced eigencurve, we consider the reduced eigencurve introduced in \cite{Eme} Theorem 2.23 obtained as the closed subspace of $\t{Spf}(\T_N^{(p)}) \times \Aff^1_{\Q_p}$ interpolating classical Hecke eigenforms, where $\T_N^{(p)}$ is the universal Hecke algebra of level $N$ generated by Hecke operators $T_{\ell}$ for each prime number $\ell \neq p$ and $S_{\ell}$ for each prime number $\ell$ coprime to $N$. Two geometric constructions of a family of Galois representations associated to modular forms are known for the case where the modular forms are ordinary. One is given by H.\ Hida as the inverse limit of Tate modules of Jacobian varieties of towers of modular curves (\cite{Hid86} Theorem 2.1) $(Y_1(p^r N))_{r \in \N}$. The other one is given by A.\ Wiles as gluing of pseudo-representations along Hida family (\cite{Wil88} Theorem 2.2.1). Even if we restrict it to the case where modular forms are ordinary, our construction completely differs from the two. Indeed, we construct a family as a quotient of the cohomology of a compact sheaf on a single modular curve of level $N$.

\vspace{0.2in}
In \S \ref{Topological Modules over Topological Rings}, we recall topological modules over topological rings. We define the notion of profiniteness of a topological module. Roughly speaking, a profinite module is a topological module which is isomorphic to the inverse limit of its quotients by open submodules of finite indices. Every profinite module is a compact Hausdorff topological module such that the set of open subgroups forms a fundamental system of neighbourhoods of $0$, and the converse holds in the case where the topological ring is itself a profinite as is verified in Proposition \ref{zero dimensional}. In \S \ref{Topological Modules over Topological Monoids}, we recall topological modules over monoid algebras of topological monoids over topological rings. We are interested not only in a topological group but also in a topological monoid, because we need the latter in order to give an action of Hecke operators on cohomologies in \S \ref{Actions of the Absolute Galois Group and Hecke Operators}. In \S \ref{p-adic Modular Forms and Hecke Algebras}, we recall modular forms and several variants of Hecke algebras over a $p$-adic field. We introduce the universal Hecke algebras $\T_N^{[< s]}$ and $\T_N^{< s}$ of slope $< s$. We will use them in order to formulate a cuspidal family of systems of Hecke eigenvalues of finite slope in \S \ref{p-adic Family of Galois Representations of Finite Slope}.

\vspace{0.2in}
In \S \ref{Prodiscrete Cohomologies of Complete Topological Modules}, we introduce the notion of a prodiscrete cohomology of a complete topological module with a continuous action of a topological monoid. In \S \ref{Profinite Zp Sheaves on Modular Curves}, we introduce a notion of a profinite $\Z_p$-sheaf. It is an inverse system of \'etale sheaves of finite Abelian $p$-groups, and unlike a smooth $\Z_p$-sheaf, we assume no finiteness condition. Similar with a prodiscrete cohomology of a complete topological module with a continuous action of a topological monoid, we define a prodiscrete cohomology of a profinite $\Z_p$-sheaf as the inverse limit of the cohomologies. It is not a derived functor, but works well due to the profiniteness of sheaves. A prodiscrete cohomology is a cohomology of a single profinite $\Z_p$-sheaves on a single scheme by definition, while a completed cohomology introduced by Emerton in \cite{Eme} is the inverse limit of the direct limit of torsion cohomologies of a compatible system of $p$-adic sheaves on a tower of schemes. Of course, Shapiro's lemma gives an interpretation of a compatible system of sheaves on a tower of schemes as a compatible system of sheaves on a single scheme. However, such an interpretation yields a direct limit of $p$-adic sheaves, and hence completely differs from a profinite $\Z_p$-sheaf. A prodiscrete cohomology of profinite $\Z_p$-sheaves is compact, while a completed cohomology of a compatible system of sheaves on a tower of schemes is Banach, which is far from compact. The compactness is important for interpolation. The Iwasawa algebra, which is compact, is identified with the algebra of rigid analytic functions, and interpolation by rigid analytic functions has good congruence property. On the other hand, the algebra of bounded continuous functions, which is Banach, has infinitely many idempotents, and hence interpolation by continuous functions does not perform so well. We guess that in order to compare with the two cohomologies in a direct way, one needs some duality theory of sheaves extending Schneider--Teitelbaum theory (\cite{ST02}). In \S \ref{Actions of the Absolute Galois Group and Hecke Operators}, we define actions of $\t{Gal}(\overline{\Q}/\Q)$ and Hecke operators on prodiscrete cohomologies. We verify the action of Hecke operators is $\t{Gal}(\overline{\Q}/\Q)$-equivariant in Proposition \ref{Galois - Hecke}. We will use the actions in a geometric construction of a $p$-adic family of Galois representations associated to modular forms.

\vspace{0.2in}
In \S \ref{Interpolation along the Weight Spaces}, we give an explicit way to interpolate $\t{Sym}^{k-2}(\Q_p^2)$ along weights $k \in \N \cap [2,\infty)$. Although the dimension of $\t{Sym}^{k-2}(\Q_p^2)$ for each $k \in \N \cap [2,\infty)$ are pairwise distinct, there are infinite dimensional extensions of them as is shown in Theorem \ref{interpolation} and Remark \ref{extension}. They share the underlying topological modules, and hence can be easily interpolated. In \S \ref{Restriction to Families of Finite Slope}, we construct a profinite module over the Iwasawa algebra with a continuous actions of $\t{Gal}(\overline{\Q}/\Q)$ and $\T_N^{< s}$. We verify the finiteness of it as a module over the topological ring generated by the Iwasawa algebra and Hecke operators in Theorem \ref{finiteness}. In \S \ref{p-adic Family of Modular Forms of Finite Slope}, we introduce a notion of a $\Lambda$-adic domain in Definition \ref{Lambda-adic domain}. Roughly speaking, it is a $1$-dimensional topological algebra over the Iwasawa algebra with ``enough arithmetic points'' and ``the identity theorem'. We define a notion of a modular form over a $\Lambda$-adic domain. As is shown in Remark \ref{eigencurve}, the reduced eigencurve admits a smooth alteration with an open covering of the complement of a discrete subspace by $\Q_p$-analytic spaces associated to $\Lambda$-adic domains. In particular, there are plentiful modular forms over $\Lambda$-adic domains. We verify a certain finiteness of the space of modular forms over a $\Lambda$-adic domain in Theorem \ref{finiteness 3}. In \S \ref{p-adic Family of Galois Representations of Finite Slope}, we generalise a result of \cite{Gro90} in Theorem \ref{Galois representation}. B.\ H.\ Gross proved that for any normalised cuspidal eigenform over $\overline{\Q}_p$ of weight $2$ and level $N$, the quotient of the rational Tate module of the Jacobian of $Y_1(N)$ by the corresponding system of Hecke eigenvalues is naturally isomorphic to the Galois representation associated to the cusp form twisted by a character in the proof of \cite{Gro90} Theorem 11.4, and we loosen the restriction of weight $2$ to weight $\geq 2$. Finally, we construct a $p$-adic family of Galois representations associated to modular forms in Theorem \ref{geometric construction 2} of finite slope.

\section{Preliminaries}
\label{Preliminaries}

In this section, let $p$ denote a prime number. We recall several notions of algebraic objects with topologies. We also recall modular forms and Hecke algebras.

\subsection{Topological Modules over Topological Rings}
\label{Topological Modules over Topological Rings}

A {\it topological monoid} is a monoid $G$ endowed with a topology such that the multiplication $G \times G \to G \colon (g,g') \mapsto gg'$ is continuous. A monoid is always equipped with the discrete topology unless specified so that it is regarded as a topological monoid. A {\it topological group} is a topological monoid $G$ such that its underlying monoid is a group and the inverse $G \to G \colon g \mapsto g^{-1}$ is continuous. A topological group is said to be {\it Abelian} if its underlying group is Abelian. A topological group admits two canonical uniform structures compatible with its topology, and for an Abelian topological group, the two canonical uniform structures coincide with each other. Therefore we always equip a topological Abelian group with the canonical uniform structure.

\begin{exm}
\label{opposite monoid}
Let $G$ be a topological monoid. We denote by $G^{\t{op}}$ the opposite monoid of $G$ endowed with the topology induced by the identity map $(\cdot)^{\t{op}} \colon G \to G^{\t{op}} \colon g \mapsto g^{\t{op}}$ of the underlying sets. Then $G^{\t{op}}$ is a topological monoid. If $G$ is a topological group, then so is $G^{\t{op}}$, and the map $G \to G^{\t{op}} \colon g \mapsto (g^{-1})^{\t{op}}$ is a homeomorphic group isomorphism.
\end{exm}

\begin{exm}
Let $I$ be a set, and $G = (G_i)_{i \in I}$ a family of topological monoids (resp.\ topological groups). Then the direct product $\prod_{i \in I} G_i$ is a topological monoid (resp.\ a topological group) with respect to the direct product topology.
\end{exm}

Let $M$ be an Abelian group. The set $\set{m + p^n M \subset M}{(m,n) \in M \times \N}$ forms a basis of a topology $\Opn_{M,p}$ on $M$, and we call $\Opn_{M,p}$ {\it the $p$-adic topology on $M$}. Then $M$ is an Abelian topological group with respect to the $p$-adic topology on $M$. We say that $M$ is {\it $p$-adically separated} (resp.\ {\it $p$-adically complete}) if the group homomorphism
\begin{eqnarray*}
  \iota_{M,p} \colon M & \to & \varprojlim_{r \in \N} M/p^r M \\
  m & \mapsto & (m + p^r M)_{r = 0}^{\infty}
\end{eqnarray*}
is injective (resp.\ an isomorphism). By definition, the $p$-adic topology on $M$ is the weakest topology for which $\iota_{M,p}$ is continuous with respect to the inverse limit topology of the discrete topology on the target. In particular, $M$ is $p$-adically separated (resp.\ $p$-adically complete) if and only if $M$ is Hausdorff (resp.\ complete) with respect to the $p$-adic topology (resp.\ the canonical uniform structure associated to the $p$-adic topology).

\begin{exm}
Let $I$ be a set, and $M = (M_i)_{i \in I}$ a family of $p$-adically separated (resp.\ $p$-adically complete) Abelian groups. Then the direct product $\prod_{i \in I} M_i$ is a $p$-adically separated (resp.\ $p$-adically complete) Abelian group. The direct product topology of the $p$-adic topologies does not necessarily coincide with the $p$-adic topology on the direct product.
\end{exm}

\begin{prp}
\label{p-adic continuity}
Let $M$ and $N$ be Abelian groups. Every group homomorphism $M \to N$ is continuous with respect to the $p$-adic topologies.
\end{prp}

\begin{proof}
Let $\varphi \colon M \to N$ be a group homomorphism. We have
\begin{eqnarray*}
  \bigcup_{m' \in \ker(\varphi)} (m+m') + p^r M = m + \ker(\varphi) + p^r M \subset \varphi^{-1} \left( \varphi(m) + p^r N \right)
\end{eqnarray*}
for any $(m,r) \in M \times \N$. It ensures the continuity of $\varphi$.
\end{proof}

A $\Z_p$-module is said to be {\it $p$-adically separated} (resp.\ {\it $p$-adically complete}) if its underlying Abelian group is $p$-adically separated (resp.\ $p$-adically complete). For any $p$-adically complete Abelian group $M$, the natural action of $\Z_p$ on the target of $\iota_{M,p}$ makes $M$ a $\Z_p$-module. Therefore the notion of a $p$-adically complete Abelian group is equivalent to that of a $p$-adically complete $\Z_p$-module.

\begin{exm}
Every finitely generated $\Z_p$-module is $p$-adically complete, while there is no non-trivial $\Q_p$-vector space which is $p$-adically separated.
\end{exm}

\begin{prp}
\label{quotient map}
Let $M$ be a $p$-adically complete $\Z_p$-module. For any $\Z_p$-submodule $L \subset M$ closed with respect to the $p$-adic topology, the canonical projection $M \twoheadrightarrow M/L$ is a quotient map with respect to the $p$-adic topologies.
\end{prp}

\begin{proof}
Let $\varphi \colon M \twoheadrightarrow M/L$ denote the canonical projection. The continuity of $\varphi$ follows from Proposition \ref{p-adic continuity}. We have $\varphi(m + p^r M) = (m + L) + p^r (M/L)$ for any $(m,r) \in M \times \N$, and hence $\varphi$ is an open map. Thus the surjective map $\varphi$ is a quotient map.
\end{proof}

In this paper, a ring is always assumed to be unital and associative, but not necessarily commutative. A {\it topological ring} is a ring $R$ endowed with a topology such that $R$ is a topological Abelian group with respect to the addition $R \times R \to R \colon (r,r') \mapsto r+r'$, and is a topological monoid with respect to the multiplication $R \times R \to R \colon (r,r') \mapsto rr'$. We always equip a topological ring with the canonical uniform structure associated to the topological Abelian group structure given by the addition. A topological ring is said to be {\it commutative} if its underlying ring is commutative.

\begin{exm}
For a ring $R$, {\it the $p$-adic topology on $R$} is the $p$-adic topology on the additive group of $R$. Since $p^n R \subset R$ is a two-sided ideal for any $n \in \N$, $\iota_{R,p}$ is a ring homomorphism, and hence $R$ is a topological ring with respect to the $p$-adic topology.
\end{exm}

\begin{rmk}
\label{character}
Let $R$ be a topological ring. The unit group $R^{\times} \subset R$ is a submonoid of $R$ with respect to the multiplication, and we regard it as a topological monoid with respect to the relative topology. It is not necessarily a topological group. For a topological group $G$, we call a continuous monoid homomorphism $G \to R^{\times}$ a {\it continuous character}.
\end{rmk}

Let $R$ be a topological ring. A {\it topological $R$-module} is a topological Abelian group $M$ endowed with a structure of a left module over the underlying ring of $R$ on the underlying Abelian group of $M$ such that the scalar multiplication $R \times M \to M \colon (r,m) \mapsto rm$ is continuous. A topological $R$-module is said to be a {\it discrete $R$-module} if its underlying topology is the discrete topology, is said to be a {\it finite $R$-module} if it is a discrete $R$-module whose underlying set is a finite set, and is said to be a {\it profinite $R$-module} if it is homeomorphically isomorphic to the inverse limit of finite $R$-modules. For a topological $R$-module $M$, we denote by $\Opn_M$ the set of open $R$-submodules of $M$. A topological $R$-module is said to be {\it linearly complete} if the continuous $R$-linear homomorphism
\begin{eqnarray*}
  M & \to & \varprojlim_{L \in \Opn_M} M/L \\
  m & \mapsto & (m + L)_{L \in \Opn_M}
\end{eqnarray*}
is a homeomorphic isomorphism. For any linearly complete topological $R$-module, $\Opn_M$ forms a fundamental system of neighbourhoods of $0$. Every discrete $R$-module is linearly complete. In particular, every finite $R$-module is linearly complete. Every inverse limit of linearly complete topological $R$-modules is linearly complete. Therefore every profinite $R$-module is linearly complete.

\begin{exm}
For any ring $R$ and left $R$-module $M$, $M$ is a topological $R$-module with respect to the $p$-adic topologies on $R$ and $M$.
\end{exm}

\begin{exm}
Let $R$ be a commutative topological ring. For linearly complete (resp.\ profinite) $R$-modules $M_0$ and $M_1$, we set
\begin{eqnarray*}
  M_0 \hat{\otimes}_R M_1 \coloneqq \varprojlim_{(L_0,L_1) \in \Opn_{M_0} \times \Opn_{M_1}} (M_0/L_0) \otimes_R (M_1/L_1),
\end{eqnarray*}
and endow it with the inverse limit topology of the discrete topologies. Then $M_0 \hat{\otimes}_R M_1$ is a linearly complete (resp.\ profinite) $R$-module.
\end{exm}

\begin{prp}
\label{right exact}
Let $R$ be a commutative topological ring with underlying ring $\v{R}$, and $M_0$, $M_1$, and $M_2$ linearly complete $R$-modules with underlying $\v{R}$-modules $\v{M_0}$, $\v{M_1}$, and $\v{M_2}$ respectively. For any continuous $R$-linear homomorphism $f \colon M_1 \to M_2$, there uniquely exists a continuous $R$-linear homomorphism
\begin{eqnarray*}
  \t{id}_{M_0} \hat{\otimes} f \colon M_0 \hat{\otimes}_R M_1 \to M_0 \hat{\otimes}_R M_2
\end{eqnarray*}
extending the $\v{R}$-linear homomorphism
\begin{eqnarray*}
  \t{id}_{\vv{M_0}} \otimes f \colon \v{M_0} \otimes_{\v{R}} \v{M_1} & \to & \v{M_0} \otimes_{\v{R}} \v{M_2} \\
  m_0 \otimes m_1 & \mapsto & m_0 \otimes f(m_1).
\end{eqnarray*}
Moreover, if $M_0$ and $M_1$ are profinite $R$-modules and $f$ is surjective, then so is $\t{id}_{M_0} \hat{\otimes} f$. 
\end{prp}

\begin{proof}
The uniqueness is obvious because the image of $\v{M_0} \otimes_{\v{R}} \v{M_1}$ is dense in $M_0 \hat{\otimes}_R M_1$ and $M_0 \hat{\otimes}_R M_2$ is Hausdorff. The $\v{R}$-linear homomorphism $\t{id}_{\v{M_0}} \otimes f$ induces a well-defined $R$-linear homomorphism
\begin{eqnarray*}
  M_0/L_0 \otimes_{\v{R}} M_1/L_1 & \to & M_0/L_0 \otimes_{\v{R}} M_2/L_2 \\
  (m_0 + L_0) \otimes (m_1 + L_1) & \mapsto & (m_0 + L_0) \otimes (f(m_1) + L_2)
\end{eqnarray*}
continuous with respect to the discrete topologies for any $(L_0,L_1,L_2) \in \Opn_{M_0} \times \Opn_{M_1} \times \Opn_{M_2}$ with $f(L_1) \subset L_2$. Since $f$ is continuous, $f^{-1}(L_2) \in \Opn_{M_1}$ for any $L_2 \in \Opn_{M_2}$, and hence we have $\set{L_2 \in \Opn_{M_2}}{{}^{\exists} L_1 \in \Opn_{M_1}, \t{s.t.\ } f(L_1) \subset L_2} = \Opn_{M_2}$. Therefore taking the inverse limit, we obtain a continuous $R$-linear homomorphism $\t{id}_{M_0} \hat{\otimes} f \colon M_0 \hat{\otimes}_R M_1 \to M_0 \hat{\otimes}_R M_2$ extending $\t{id}_{\v{M_0}} \otimes f$ by the continuity of $f$. Suppose that $M_0$ and $M_1$ are profinite $R$-modules and $f$ is surjective. The image of $\v{M_0} \otimes_{\v{R}} \v{M_2}$ is dense in $M_0 \hat{\otimes}_R M_2$, and coincides with the image of $\v{M_0} \otimes_{\v{R}} \v{M_1}$ by $\t{id}_{M_0} \hat{\otimes} f$. Since $M_0$ and $M_1$ are profinite $R$-modules, so is $M_0 \hat{\otimes}_R M_1$. It ensures that $\t{id}_{M_0} \hat{\otimes} f$ is a continuous homomorphism from a compact module to a Hausdorff module, and hence is a closed map. Thus $\t{id}_{M_0} \hat{\otimes} f$ is surjective.
\end{proof}

\begin{prp}
Every profinite $\Z_p$-module is $p$-adically separated, and its $p$-adic topology is finer than or equal to its original topology.
\end{prp}

\begin{proof}
Let $M$ be a profinite $\Z_p$-module. Since $M$ is Hausdorff, the second assertion implies the first assertion. Since the topology of a linearly complete $\Z_p$-module is generated by open $\Z_p$-submodules, it suffices to verify that for any open $\Z_p$-submodule $L \subset M$, there is an $r \in \N$ such that $p^r M \subset L$. Since $M$ is compact, $L$ is of finite index as an additive subgroup of $M$. Therefore $M/L$ is a $\Z_p$-module whose underlying group is a finite Abelian group, and hence there is an $r \in \N$ such that $p^r (M/L) = 0$. It implies $p^r M \subset L$.
\end{proof}

\begin{crl}
\label{p-adically open}
For a $p$-adically separated $\Z_p$-module $M$ and a profinite $\Z_p$-module $N$, every $\Z_p$-linear homomorphism $M \to N$ is continuous with respect to the $p$-adic topology on $M$.
\end{crl}

\begin{prp}
\label{zero dimensional}
Let $R$ be a compact topological ring, and $M$ a topological $R$-module. Then $M$ is a profinite $R$-module if and only if $M$ is a compact Hausdorff $R$-module such that the set of open $\Z$-submodules forms a fundamental system of neighbourhoods of $0$.
\end{prp}

\begin{proof}
The necessary implication follows from Tychonoff's theorem. Suppose that $M$ is a compact Hausdorff $R$-module such that the set of open $\Z$-submodules forms a fundamental system of neighbourhoods of $0$. In order to verify that $M$ is a profinite $R$-module, it suffices to show that $\Opn_M$ forms a fundamental system of neighbourhoods of $0$. Let $L \subset M$ be an open $\Z$-submodule of $M$. By the continuity of the scalar multiplication $R \times M \to M$, there are an open neighbourhood $I \subset R$ of $0 \in R$ and an open $\Z$-submodule $L'_0 \subset M$ such that $\set{il}{(i,l) \in I \times L'_0} \subset L$. For each $r \in R$, we denote by $r + I$ the subset $\set{r + i}{i \in I} \subset R$, which is open by the continuity of the addition $R \times R \to R$. Since $R$ is compact, the open covering $\set{r + I}{r \in R}$ admits a finite subcovering. Suppose that $\set{r_h + I}{h \in \N \cap [1,d]}$ covers $R$ for a $d \in \N$ and an $(r_h)_{h=1}^{d} \in R^d$. By the continuity of the scalar multiplication $R \times M \to M$ again, there is an open $\Z$-submodule $L'_h \subset M$ such that $\set{r_hl}{l \in L'_h} \subset L$ for any $h \in \N \cap [1,d]$. Put $L' \coloneqq \bigcap_{h = 0}^{d} L'_h$. Then $L'$ generates an open $R$-submodule contained in $L$.
\end{proof}

\begin{prp}
\label{canonical topology}
Let $R$ be a commutative compact Hausdorff ring with the underlying ring $\v{R}$, and $M$ a finitely generated $\v{R}$-module. For any finite subset $S \subset M$ of generators, the quotient topology on $M$ with respect to the surjective $\v{R}$-linear homomorphism
\begin{eqnarray*}
  R^S & \twoheadrightarrow & M \\
  (r_s)_{s \in S} & \mapsto & \sum_{s \in S} r_ss
\end{eqnarray*}
makes $M$ a compact Hausdorff topological $R$-module, and is independent of $S$.
\end{prp}

We call this quotient topology {\it the canonical topology on $M$}.

\begin{proof}
The first assertion follows from Tychonoff's theorem. Let $S_0,S_1 \subset M$ be finite subsets of generators with $S_0 \subset S_1$. Since $S_0$ generates $M$, for each $s \in S_1 \backslash S_0$, there is an $(r_{s,s'})_{s' \in S_0} \in R^{S_0}$ with $s = \sum_{s' \in S_0} r_{s,s'} s'$. The surjective $R$-linear homomorphism
\begin{eqnarray*}
  R^{S_1} & \twoheadrightarrow & R^{S_0} \\
  (r_s)_{s \in S_1} & \mapsto & \sum_{s \in S_0} \left( r_s + \sum_{s' \in S_1 \backslash S_0} r_{s,s'} \right) s
\end{eqnarray*}
is a continuous map between compact Hausdorff topological spaces by the continuity of the addition and the multiplication, and hence is a quotient map. The second assertion follows from the fact that the set of finite subsets of generators of $M$ is directed by inclusions.
\end{proof}

Let $R$ be a commutative topological ring. A {\it topological $R$-algebra} is a topological ring $\A$ endowed with a continuous ring homomorphism $R \to \A$ whose image lies in the centre of $\A$. Every topological $R$-algebra $\A$ is a topological left $R$-module by the continuity of the multiplication $\A \times \A \to \A$ and the structure map $R \to \A$. A topological $R$-algebra is said to be a {\it profinite $R$-algebra} if it is homeomorphically isomorphic to the inverse limit of topological $R$-algebras whose underlying topological $R$-modules are finite $R$-modules. Every profinite $R$-algebra is a profinite $R$-module by definition. We do not use the term ``a finite $R$-algebra'' because it is ambiguous here.

\begin{exm}
Let $K$ be an algebraic extension of $\Q_p$, and $O_K$ the integral closure of $\Z_p$ in $K$. Then $K$ and $O_K$ are commutative topological $\Z_p$-algebras with respect to a unique extension $\v{\cdot} \colon K \to [0,\infty)$ of a $p$-adic norm on $\Q_p$, and the relative topology of $O_K \subset K$ coincides with the $p$-adic topology.
\end{exm}

\begin{exm}
\label{continuous functions}
Let $X$ be a topological space, and $R$ a commutative topological ring. Then the $R$-algebra $\t{C}(X,R)$ of continuous maps $X \to R$ is a commutative topological $R$-algebra with respect to the topology of uniform convergence.
\end{exm}

\begin{exm}
\label{tensor}
Let $R$ be a commutative topological ring, and $\A$ a profinite $R$-algebra. For a linearly complete (resp.\ profinite) $R$-module $M$, we regard $\A \hat{\otimes}_R M$ as a linearly complete (resp.\ profinite) $\A$-module with respect to the natural action of $\A$. When we emphasis the structure morphism $\varphi \colon R \to \A$, then we write $(\A,\varphi) \hat{\otimes}_R M$ instead of $\A \hat{\otimes}_R M$.
\end{exm}

\begin{exm}
\label{tensor 2}
Let $R$ be a commutative topological ring with underlying ring $\v{R}$, and $\A_0$ and $\A_1$ commutative profinite $R$-algebras with underlying $\v{R}$-algebras $\v{\A_0}$ and $\v{\A_1}$ respectively. Then the structure of $\v{\A_0} \otimes_{\v{R}} \v{\A_1}$ as a commutative $\v{R}$-algebra uniquely extends a structure of $\A_0 \hat{\otimes}_R \A_1$ as a topological $R$-algebra, for which $\A_0 \hat{\otimes}_R \A_1$ is a commutative profinite $R$-algebra. For any profinite $R$-module $M$, if $M$ is endowed with structures of a profinite $\A_0$-module and a profinite $\A_1$-module extending the structure of a profinite $R$-module, then there uniquely exists a structure of a profinite $(\A_0 \hat{\otimes}_R \A_1)$-module on $M$ extending the structures of a profinite $\A_0$-module and a profinite $\A_1$-module.
\end{exm}

\begin{prp}
Let $R$ be a commutative topological ring. A topological $R$-algebra $\A$ is a profinite $R$-algebra if and only if the underlying topological ring of $\A$ is a compact Hausdorff topological ring such that the set of open two-sided ideals of $\A$ forms a fundamental system of neighbourhoods of $0$.
\end{prp}

\begin{proof}
Let $\I$ denote the set of open two-sided ideals of $\A$. If $\A$ is a profinite $R$-algebra, then $\A$ is homeomorphically isomorphic to the inverse limit of its quotients by open two-sided ideals of finite index, and hence the underlying topological ring of $\A$ is a compact Hausdorff topological ring such that $\I$ forms a fundamental system of neighbourhoods of $0$. Conversely, suppose that the underlying topological ring of $\A$ is a compact Hausdorff topological ring such that $\I$ forms a fundamental system of neighbourhoods of $0$. For any $I \in \I$, the compactness of $\A$ ensures that $\A/I$ is a topological $R$-algebra whose underlying $R$-module is a finite $R$-module. The $R$-algebra homomorphism
\begin{eqnarray*}
  \iota \colon \A & \to & \varprojlim_{I \in \I} \A/I \\
  a & \mapsto & (a + I)_{I \in \I}
\end{eqnarray*}
is continuous, and the image is dense with respect to the inverse limit topology on the target. Since $\A$ is compact, $\iota$ is a closed map and hence is surjective. Since $\A$ is Hausdorff, $\iota$ is injective because $\I$ forms a fundamental system of neighbourhoods of $0$. Thus $\iota$ is a homeomorphic $R$-algebra isomorphism, and hence $\A$ is a profinite $R$-algebra.
\end{proof}

\begin{crl}
\label{profinite ring}
Let $R_0$ be a commutative topological ring, $R_1$ a commutative topological $R$-algebra, and $\A$ a topological $R_1$-algebra. Then $\A$ is a profinite $R_0$-algebra if and only if $\A$ is a profinite $R_1$-algebra. 
\end{crl}

\begin{crl}
Let $R$ be a commutative topological ring, and $\A$ a $\Z_p$-algebra finitely generated as a $\Z_p$-module. For any continuous ring homomorphism $R \to \A$, $\A$ is a profinite $R$-algebra with respect to the $p$-adic topology on $\A$.
\end{crl}

\begin{prp}
\label{finitely generated - p-adic}
Let $\A$ be a topological $\Z_p$-algebra. For any Hausdorff topological $\A$-module $M$ whose underlying $\Z_p$-module is finitely generated, the topology of $M$ coincides with the $p$-adic topology, and $M$ is a profinite $\A$-module.
\end{prp}

\begin{proof}
Take a finite subset $S \subset M$ of generators of the underlying $\Z_p$-module of $M$. By the continuity of the addition $M \times M \to M$ and the scalar multiplication $\Z_p \times M \to M$, the surjective $\Z_p$-linear homomorphism
\begin{eqnarray*}
  \varphi \colon \Z_p^S & \to & M \\
  (a_m)_{m \in S} & \mapsto & \sum_{m \in S} a_m m
\end{eqnarray*}
is continuous. Since $\Z_p^S$ is compact and $M$ is Hausdorff, $\varphi$ is a quotient map. It follows from Proposition \ref{quotient map} that the topology of $M$ coincides with the $p$-adic topology. Since $M$ is finitely generated as a $\Z_p$-module, it is $p$-adically complete. Since $p^r$ lies in the centre of $\A$, the action of $\A$ on $M$ extends to a unique action on $M/p^r M$ for any $r \in \N$. We have a homeomorphic $R$-linear isomorphism
\begin{eqnarray*}
  M \stackrel{\sim}{\to} \varprojlim_{r \in \N} M/p^r M,
\end{eqnarray*}
and hence $M$ is a profinite $\A$-module.
\end{proof}

\begin{crl}
\label{finitely generated - adic}
Let $\A$ be a commutative topological $\Z_p$-algebra with a topologically nilpotent element $\epsilon \in \A$. For any Hausdorff topological $\A$-module $M$ whose underlying $\Z_p$-module is finitely generated, the continuous $\A$-linear homomorphism
\begin{eqnarray*}
  M & \to & \varprojlim_{r \in \N} M/\epsilon^r M \\
  m & \mapsto & (m + \epsilon^r M)_{r = 0}^{\infty}
\end{eqnarray*}
is a homeomorphic isomorphism.
\end{crl}

\begin{proof}
By Proposition \ref{finitely generated - p-adic}, the topology of $M$ is given by the $p$-adic topology. Let $r' \in \N \backslash \ens{0}$. Then $M/p^{r'} M$ is a finite $\A$-module. For any $\overline{m} \in M/p^{r'} M$, since $\epsilon$ is a topologically nilpotent, there is an $r \in \N$ such that $\epsilon^r \overline{m} = 0$. Since the underlying set of $\A$ is a finite set, there is an $r \in \N$ such that $\epsilon^r (M/p^{r'} M)$. It implies the canonical projection $M/p^{r'} M \twoheadrightarrow (M/p^{r'} M)/\epsilon^r (M/p^{r'} M) \cong M/(p^{r'} M + \epsilon^r M)$ is a homeomorphic $\A$-linear isomorphism. Since $M$ is finitely generated as a $\Z_p$-module, so is $M/\epsilon^r M$ for any $r \in \N$. Therefore $M/\epsilon^r M$ is $p$-adically complete, and the topology of $M/\epsilon^r M$ coincides with the $p$-adic topology by Proposition \ref{finitely generated - p-adic} for any $r \in \N$. Therefore we obtain a homeomorphic $\A$-linear isomorphism
\begin{eqnarray*}
  & & M \cong \varprojlim_{r' \in \N} M/p^{r'} M \cong \varprojlim_{r' \in \N} \varprojlim_{r \in \N} M/(p^{r'} M + \epsilon^r M) \cong \varprojlim_{r \in \N} \varprojlim_{r' \in \N} M/(p^{r'} M + \epsilon^r M)  \\
  & \cong & \varprojlim_{r \in \N} M/\epsilon^r M.
\end{eqnarray*}
\end{proof}

\begin{prp}
\label{canonical topology 2}
Let $R$ be a commutative compact Hausdorff ring with the underlying ring $\v{R}$, and $\A$ a commutative $\v{R}$-algebra finitely generated as an $\v{R}$-module. Then the canonical topology on $\A$ as an $\v{R}$-module makes $\A$ a compact Hausdorff topological $R$-algebra.
\end{prp}

\begin{proof}
Let $S$ be a finite subset $S \subset \A$ of generators as an $\v{R}$-module. The surjective $R$-linear homomorphism
\begin{eqnarray*}
  R^S & \twoheadrightarrow & \A \\
  (r_s)_{s \in S} & \mapsto & \sum_{s \in S} r_ss.
\end{eqnarray*}
is a quotient map by the definition of the canonical topology. Since $S$ generates $\A$ as an $R$-module, for each $(s',s'') \in S \times S$, there is an $(r_{s,s',s''})_{s \in S}$ such that $s's'' = \sum_{s \in S} r_{s,s',s''} s$. We consider the $R$-bilinear homomorphism
\begin{eqnarray*}
  \nabla \colon R^S \times R^S & \to & R^S \\
  ((r'_s)_{s \in S},(r''_s)_{s \in S}) & \mapsto & \left( \sum_{s',s'' \in S} r_{s,s',s''} r'_{s'} r''_{s''} \right)_{s \in S},
\end{eqnarray*}
which is continuous by the continuity of the addition and the multiplication. By the compatibility of the quotient topology and the direct product, the continuity of $\nabla$ ensures that of the multiplication $\A \times \A \to \A$.
\end{proof}

\begin{dfn}
\label{free}
Let $\A$ be a topological $\Z_p$-algebra. For a topological $\A$-module $M$, we denote by $\t{tor}_p(M) \subset M$ the $\A$-submodule consisting of elements $m \in M$ with $p^rm = 0$ for some $r \in \N$, and by $M_{\t{free}}$ the topological $\A$-module flat over $\Z_p$ obtained as the quotient $M/\t{tor}_p(M)$.
\end{dfn}

For any discrete $\A$-module $M$, its $\A$-submodule $\t{tor}_p(M)$ is closed, and hence its quotient $M_{\t{free}}$ is also a discrete $\A$-module. On the other hand, for a topological $\A$-module $M$ which is not a discrete $\A$-module, $\t{tor}_p(M)$ is not necessarily closed, and hence $M_{\t{free}}$ is not necessarily Hausdorff even if $M$ is Hausdorff. Moreover, the correspondence $M \rightsquigarrow M_{\t{free}}$ does not necessarily commute with direct products.

\begin{exm}
\begin{eqnarray*}
  \Z_p = \varprojlim_{r \in \N} \Z_p/p^r \Z_p \hookrightarrow \left( \prod_{r = 0}^{\infty} \Z/p^r \Z \right)_{\t{free}} \twoheadrightarrow \prod_{r = 0}^{\infty} (\Z/p^r \Z)_{\t{free}} = 0.
\end{eqnarray*}
\end{exm}

\subsection{Topological Modules over Topological Monoids}
\label{Topological Modules over Topological Monoids}

Let $S$ be a topological space. A {\it topological space with an action of $S$} is a pair $(X,\rho)$ of a topological space $X$ and a continuous map $\rho \colon S \times X \to X$. For topological spaces $(X_0,\rho_0)$ and $(X_1,\rho_1)$ with actions of $S$, an {\it $S$-equivariant map} $\varphi \colon (X_0,\rho_0) \to (X_1,\rho_1)$ is a continuous map $\varphi \colon X_0 \to X_1$ satisfying $\varphi(\rho_0(s,x)) = \rho_1(s,\varphi(x))$ for any $(s,x) \in S \times X_0$. Let $G$ be a topological monoid. A {\it $G$-space} is a topological space $(X,\rho)$ with an action of the underlying topological space of $G$ satisfying the following:
\begin{itemize}
\item[(i)] The equality $\rho(g, \rho(g',x)) = \rho(gg',x)$ holds for any $(g,g',x) \in G \times G \times X$.
\item[(ii)] The equality $\rho(1,x) = x$ holds for any $x \in X$.
\end{itemize}

\begin{exm}
\label{linear fractional transformation}
Let $\Hlf \subset \C$ denote the upper half plain $\set{a + b \ \sqrt[]{\mathstrut -1}}{(a,b) \in \R \times (0,\infty)}$. Then $\Hlf$ is an $\t{SL}_2(\Z)$-space with respect to the action
\begin{eqnarray*}
  \t{SL}_2(\Z) \times \Hlf & \to & \Hlf \\
  \left(
  \left(
    \begin{array}{cc}
      a & b \\
      c & d
    \end{array}
  \right),
  z \right) & \mapsto & \frac{az+b}{cz+d}
\end{eqnarray*}
given by linear fractional transformations.
\end{exm}

\begin{prp}
\label{p-adic linear fractional transformation}
Let $p$ be a prime number. The subset
\begin{eqnarray*}
  \Pi_0(p) \coloneqq
  \left(
    \begin{array}{cc}
      \Z_p & \Z_p \\
      p \Z_p & \Z_p^{\times}
    \end{array}
  \right)
  \subset \t{M}_2(\Z_p),
\end{eqnarray*}
is a closed submonoid with respect to the multipication and the $p$-adic topology of the additive group of $\t{M}_2(\Z_p)$, and $\Z_p$ is a $\Pi_0(p)$-space with respect to the action
\begin{eqnarray*}
  m_p \colon \Pi_0(p) \times \Z_p & \to & \Z_p \\
  \left(
    \left(
      \begin{array}{cc}
        a & b \\
        c & d
      \end{array}
    \right),
    z
  \right)
  & \mapsto &
  m_p
  \left(
    \left(
      \begin{array}{cc}
        a & b \\
        c & d
      \end{array}
    \right),
    z
  \right)
  \coloneqq \frac{az+b}{cz+d}
\end{eqnarray*}
given by $p$-adic linear fractional transformations.
\end{prp}

\begin{proof}
To begin with, we verify that the image of $m_p$ actually lies in $\Z_p$. Let $(A,z) \in \Pi_0(p) \times \Z_p$ with 
$
  A = 
  \left(
    \begin{array}{cc}
      a & b \\
      c & d
    \end{array}
  \right)
$
. We have $cz \in p \Z_p$, $d \in \Z_p^{\times}$, and hence $cz + d \in \Z_p^{\times}$. Since $az + b \in \Z_p$, we obtain $m_p(A,z) \in \Z_p$. The function
\begin{eqnarray*}
  \Z_p \times \Z_p \times p \Z_p \times (1 + p \Z_p) \times \Z_p & \to & \Z_p \\
  (a,b,c,d,z) & \mapsto & m_p \left(
  \left(
    \begin{array}{cc}
      a & b \\
      c & d
    \end{array}
  \right),
  z \right)
\end{eqnarray*}
is locally analytic, and hence $m_p$ is continuous. For any $(A,B,z) \in \Pi_0(p) \times \Pi_0(p) \times \Z_p$ with 
$
  A = 
  \left(
    \begin{array}{cc}
      a & b \\
      c & d
    \end{array}
  \right)
$
 and 
$
  B = 
  \left(
    \begin{array}{cc}
      e & f \\
      g & h
    \end{array}
  \right)
$
, we have
\begin{eqnarray*}
  & & m_p(A,m_p(B,z))
  = m_p
  \left(
    \left(
      \begin{array}{cc}
        a & b \\
        c & d
      \end{array}
    \right),
    \frac{ez+f}{gz+h}
  \right)
  = \frac{a(ez+f)+b(gz+h)}{c(ez+f)+d(gz+h)} \\
  & = & \frac{(ae+bg)z+(af+bh)}{(ce+dg)z+(cf+dh)}
  = m_p
  \left(
    \left(
      \begin{array}{cc}
        ae+bg & af+bh \\
        ce+dg & cf+dh
      \end{array}
    \right),
    z
  \right)
  = m_p(AB,z).
\end{eqnarray*}
For any $z \in \Z_p$, $m_p(1,z) = z$ by definition. Thus $(\Z_p,m_p)$ is a $\Pi_0(p)$-space.
\end{proof}

A topological group is said to be a {\it profinite group} if it is homeomorphically isomorphic to the inverse limit of finite groups. The notion of a profinite group is equivalent to the notion of a compact Hausdorff totally disconnected group, and to the notion of a compact topological group such that the set of open normal subgroups forms a system of fundamental neighbourhoods of the unit.

\begin{exm}
Let $K$ be a field, $L$ a Galois algebraic extension of $K$, and $S$ an algebraic variety over $k$. Then $\t{Gal}(L/K)$ is a profinite group, and for any locally constant sheaf $\Fil$ of finite Abelian groups on $S_{\t{\'et}}$, the \'etale cohomology $\t{H}^*(S \times_K L, \Fil)$ of the inverse image of $\Fil$ by the base change $S \times_K L \to S$ is a $\t{Gal}(L/K)$-space whose underlying set is a finite set.
\end{exm}

In the following in this subsection, let $R$ denote a commutative topological ring, and $G$ a topological monoid. A {\it topological $R[G]$-module} is a pair $(M,\rho)$ of a topological $R$-module $M$ and a continuous map $\rho \colon G \times M \to M$ such that $\rho$ makes the underlying topological space of $M$ a $G$-space and satisfies the following:
\begin{itemize}
\item[(iii)] The equality $\rho(g,m+m') = \rho(g,m) + \rho(g,m')$ holds for any $(g,m,m') \in G \times M \times M$.
\item[(iv)] The equality $\rho(g,rm) = r \rho(g,m)$ holds for any $(g,r,m) \in G \times R \times M$.
\end{itemize}
For a topological $R[G]$-module $(M,\rho)$, we denote by $\Gamma(G,(M,\rho)) \subset M$ the closed $R$-submodule consisting of elements $m \in M$ with $\rho(g,m) = m$ for any $g \in G$. A topological $R[G]$-module is said to be {\it a discrete $R[G]$-module} if its underlying topological $R$-module is a discrete $R$-module. A topological $R[G]$-module is said to be {\it a finite $R[G]$-module} if its underlying topological $R$-module is a finite $R$-module, and is said to be a {\it profinite $R[G]$-module} if it is homeomorphically isomorphic to the inverse limit of finite $R[G]$-modules. The underlying topological $R$-module of a finite $R[G]$-module is a finite $R$-module by definition, and hence the underlying topological $R$-module of a profinite $R[G]$-module is a profinite $R$-module. Every profinite $R[G]$-module is a compact Hausdorff totally disconnected topological $R[G]$-module.

\begin{exm}
\label{symmetric product}
We endow $\t{M}_2(R)$ with the direct product topology through the $R$-linear isomorphism $\t{M}_2(R) \cong R^4$ given by the canonical basis of $R^2$. Then $\t{M}_2(R)$ is a topological monoid with respect to the multiplication, and $R^2$ is a topological $R[\t{M}_2(R)]$-module with respect to the natural representation
\begin{eqnarray*}
  \rho_{R^2} \colon \t{M}_2(R) \times R^2 \to R^2 \\
  \left(
    \left(
      \begin{array}{cc}
        a & b \\
        c & d
      \end{array}
    \right),
    \left(
      \begin{array}{c}
        \alpha_0 \\
        \alpha_1
      \end{array}
    \right)
  \right) & \mapsto &
    \left(
      \begin{array}{c}
        a \alpha_0 + b \alpha_1 \\
        c \alpha_0 + d \alpha_1
      \end{array}
    \right).
\end{eqnarray*}
For each $n \in \N$, we denote by $\t{Sym}^n(R^2,\rho_{R^2}) = (\t{Sym}^n(R^2),\t{Sym}^n(\rho_{R^2}))$ the topological $R[\t{M}_2(R)]$-module obtained as the $n$-th symmetric tensor product of $(R^2,\rho_{R^2})$ over $R$. Identifying $\t{Sym}^n(R^2,\rho_{R^2})$ with the $R$-module $\bigoplus_{i = 0}^{n} R T_1^i T_2^{n-i} \subset R[T_1,T_2]$ of homogeneous polynomials of degree $n$, we put
\begin{eqnarray*}
  T_1^i T_2^{n-i} \coloneqq
  \left(
    \begin{array}{c}
      1 \\
      0
    \end{array}
  \right)^i
  \otimes
  \left(
    \begin{array}{c}
      0 \\
      1
    \end{array}
  \right)^{n-i} \in \t{Sym}^n(R^2)
\end{eqnarray*}
for each $(n,i) \in \N \times \N$ with $i \leq n$.
\end{exm}

Let $(M,\rho)$ be a topological $R[G]$-module. An $R$-submodule $L \subset M$ is said to be an {\it $R[G]$-submodule of $(M,\rho)$} if $\rho(g,l) \in L$ for any $(g,l) \in G \times L$. For instance, $rM \subset M$ is an $R[G]$-submodule of $(M,\rho)$ for any $r \in R$. When $R$ is a commutative topological $\Z_p$-algebra, then the kernel $\t{tor}_p(M)$ of the canonical projection $M \twoheadrightarrow M_{\t{free}}$ is an $R[G]$-submodule of $M$.

\begin{exm}
\label{induced}
Let $(M,\rho)$ be a topological $R[G]$-module. There are several examples of topological $R[G]$-modules induced by $(M,\rho)$.
\begin{itemize}
\item[(i)] Every $R[G]$-submodule $L \subset M$ is a topological $R[G]$-module with respect to the relative topology and a well-defined action
\begin{eqnarray*}
  \rho | L \colon G \times L & \to & L \\
  (g,l) & \mapsto & \rho(g,l),
\end{eqnarray*}
and we put $(M,\rho) | L \coloneqq (L,\rho | L)$.
\item[(ii)] For any $R[G]$-submodule $L$ of $(M,\rho)$, $M/L$ is a topological $R[G]$-module with respect to the quotient topology and the well-defined action
\begin{eqnarray*}
  \rho/L \colon G \times M/L & \to & M/L \\
  (g, m + L) & \mapsto & \rho(g,m) + L = \Set{\rho(g,m')}{m' \in m + L},
\end{eqnarray*}
and we put $(M,\rho)/L \coloneqq (M/L,\rho/L)$. In particular, we abbreviate $(M,\rho)/rM$ to $(M,\rho)/r$ for each $r \in R$. When $R$ is a commutative topological $\Z_p$-algebra, then we put $(M,\rho)_{\t{free}} \coloneqq (M,\rho)/\t{tor}_p(M)$.
\item[(iii)] For any topological monoid $H$ and continuous monoid homomorphism $\iota \colon H \to G$, $M$ is a topological $R[H]$-module with respect to the action
\begin{eqnarray*}
  \t{Res}_H^G(\rho) \colon H \times M & \to & M \\
  (h,m) & \mapsto & \rho(\iota(h),m),
\end{eqnarray*}
and we also denote $\t{Res}_H^G(M,\rho) \coloneqq (M, \t{Res}_H^G(\rho))$ simply by $(M,\rho)$ as long as this abbreviation yields no confusion. In other words, we usually regard a topological $R[G]$-module as a topological $R[H]$-module.
\end{itemize}
\end{exm}

For a topological $R[G]$-module $(M,\rho)$, we denote by $\Opn_{(M,\rho)}$ the set of open $R[G]$-submodules of $(M,\rho)$. A topological $R[G]$-module $(M,\rho)$ is said to be {\it linearly complete} if the natural continuous $R$-linear $G$-equivariant homomorphism
\begin{eqnarray*}
  (M,\rho) & \to & \varprojlim_{L \in \Opn_{(M,\rho)}} (M,\rho)/L \\
  m & \mapsto & (m + L)_{L \in \Opn_{(M,\rho)}}
\end{eqnarray*}
is a homeomorphic isomorphism. For any linearly complete $R[G]$-module $(M,\rho)$, $\Opn_{(M,\rho)}$ forms a fundamental system of neighbourhoods of $0 \in M$, and hence is cofinal in $\Opn_M$. Therefore the underlying topological $R$-module of a linearly complete $R[G]$-module is linearly complete. Every discrete $R[G]$-module is linearly complete. In particular, every finite $R[G]$-module is linearly complete. Every inverse limit of linearly complete topological $R[G]$-modules is linearly complete. Therefore every profinite $R[G]$-module is linearly complete.

\begin{exm}
Let $(M_0,\rho_0)$ and $(M_1,\rho_1)$ be linearly complete (resp.\ profinite) $R[G]$-modules. Then the continuous actions $\rho_0$ and $\rho_1$ induce a continuous action
\begin{eqnarray*}
  \rho_0 \hat{\otimes} \rho_1 \colon G \times (M_0 \hat{\otimes}_R M_1) \to M_0 \hat{\otimes}_R M_1,
\end{eqnarray*}
for which $(M_0,\rho) \hat{\otimes}_R (M_1,\rho_1) \coloneqq (M_0 \hat{\otimes}_R M_1, \rho_0 \hat{\otimes} \rho_1)$ is a linearly complete (resp.\ profinite) $R[G]$-module. When $(M_0,\rho_0)$ is the underlying topological $R[G]$-module of a commutative profinite $R$-algebra $\A$ endowed with the trivial action of $G$, then we regard $\A \hat{\otimes}_R (M_1,\rho_1)$ as a linearly complete (resp.\ profinite) $\A[G]$-module with respect to the natural action of $\A$.
\end{exm}

\begin{prp}
For any commutative topological $\Z_p$-algebra $\A$, every Hausdorff topological $\A[G]$-module $(M,\rho)$ whose underlying $\Z_p$-module is finitely generated is a profinite $\A[G]$-module.
\end{prp}

\begin{proof}
Since the underlying $\Z_p$-module of $M$ is finitely generated, it is $p$-adically complete. By Proposition \ref{finitely generated - p-adic}, the topology of $M$ coincides with the $p$-adic topology. Since $p^r M$ is stable under the action of $G$, we have a homeomorphic $\A$-linear $G$-equivariant isomorphism
\begin{eqnarray*}
  (M,\rho) & \to & \varprojlim_{r \in \N} (M,\rho)/p^r \\
  m & \mapsto & (m + p^r M)_{r = 0}^{\infty}.
\end{eqnarray*}
Thus $(M,\rho)$ is a profinite $\A[G]$-module.
\end{proof}

A {\it topological $R[G]$-algebra} is a pair $(\A,\rho)$ of a topological $R$-algebra $\A$ and a continuous map $\rho \colon G \times \A \to \A$ such that $\rho$ makes the underlying topological $R$-module of $\A$ a topological $R[G]$-module and satisfies the following:
\begin{itemize}
\item[(iv)] The equality $\rho(g,ff') = \rho(g,f) \rho(g,f')$ holds for any $(g,f,f') \in G \times \A \times \A$.
\item[(v)] The equality $\rho(g,1) = 1$ holds for any $g \in G$.
\end{itemize}
We remark that when $G$ is a topological group, then the condition (v) follows from other conditions. For any topological $R[G]$-algebra $(\A,\rho)$, we have $\rho(g,r1) = r \rho(g,1) = r1$ by the condition (v) for any $(g,r) \in G \times R$, and hence the image of $R$ is contained in $\Gamma(G,(\A,\rho))$. A topological $R[G]$-algebra is said to be {\it commutative} if its underlying topological $R$-algebra is commutative, and is said to be a {\it profinite $R[G]$-algebra} if it is homeomorphically isomorphic to the inverse limit of topological $R[G]$-algebras whose underlying topological $R$-modules are finite $R$-modules.

\begin{exm}
Suppose that $G$ is a profinite group, and let $\Opn_G$ denote the set of open normal subgroups of $G$. Then
\begin{eqnarray*}
  \Z_p[[G]] \coloneqq \varprojlim_{K \in \Opn_G} \Z_p[G/K]
\end{eqnarray*}
is a profinite $\Z_p[G]$-algebra with respect to the inverse limit topology of the $p$-adic topologies, and we call it {\it the Iwasawa algebra associated to $G$}.  It admits a natural embedding $G \hookrightarrow \Z_p[[G]]^{\times}$, for which for any profinite $\Z_p[G]$-module $(M,\rho)$, there uniquely exists a structure on $M$ as a profinite $\Z_p[[G]]$-module extending the structure as a profinite $\Z_p[G]$-module. In particular, for any profinite $\Z_p$-algebra $\A$ with a continuous character $G \to \A^{\times}$ (Remark \ref{character}), the natural structure on $\A$ as a profinite $\Z_p[G]$-algebra uniquely extends to a structure as a profinite $\Z_p[[G]]$-algebra. We call this property {\it the universality of the Iwasawa algebra}.
\end{exm}

\begin{prp}
\label{geometric action}
Let $(X,\rho)$ be a compact $G$-space. Then $\t{C}(X,R)$ is a commutative topological $R[G^{\t{op}}]$-algebra with respect to the action
\begin{eqnarray*}
  \rho^{\vee} \colon G^{\t{op}} \times \t{C}(X,R) & \to & \t{C}(X,R) \\
  (g^{\t{op}},f) & \mapsto & \left( \rho^{\vee}(g^{\t{op}},f) \colon x \mapsto f(\rho(g,x)) \right).
\end{eqnarray*}
\end{prp}

For the convention of $G^{\t{op}}$, see Example \ref{opposite monoid}.

\begin{proof}
We verify the continuity of $\rho^{\vee}$. For each $f \in \t{C}(X,R)$ and open subset $J \subset R$, put $f + \t{C}(X,J) \coloneqq \set{f' \in \t{C}(X,R)}{(f' - f)(x) \in J, {}^{\forall} x \in X}$, which is an open neighbourhood of $f$, and the set of such subsets forms an open basis of $\t{C}(X,R)$ by the definition of the topology of uniform convergence. For any open subset $J \subset R$, we have $\rho^{\vee}(g^{\t{op}},f) \in \t{C}(X,J)$ for any $(g^{\t{op}},f) \in G^{\t{op}} \times \t{C}(X,J)$ by the definition of $\rho^{\vee}$. Let $(g_0^{\t{op}},f_0) \in G^{\t{op}} \times \t{C}(X,R)$, and $I \subset \t{C}(X,R)$ be an open neighbourhood of $\rho^{\vee}(g_0^{\t{op}},f_0)$. Take an open neighbourhood $J \subset R$ of $0$ such that $\rho^{\vee}(g_0^{\t{op}},f_0) + \t{C}(X,J)$ is contained in $I$. By the continuity of the addition $R \times R \to R$ and the additive inverse $R \to R \colon r \mapsto -r$, there is an open neighbourhood $J_0 \subset R$ of $0$ such that $r-r'+r'' \in J$ for any $(r,r',r'') \in J_0 \times J_0 \times J_0$. For any $x \in X$, the set $U_x \coloneqq \set{x' \in X}{f_0(x') - f_0(x) \in J_0}$ is an open neighbourhood of $x$ by the continuity of $f$, the addition $R \times R \to R$, and the additive inverse $R \to R \colon r \mapsto -r$. For any $x \in X$, the preimage $\rho^{-1}(U_{\rho(g_0,x)}) \subset G \times X$ is an open neighbourhood of $(g_0,x)$ by the continuity of $\rho$, and hence there are open neighbourhoods $U^1_x \subset G$ and $U^2_x \subset X$ of $g_0$ and $x$ respectively such that $U^1_x \times U^2_x \subset \rho^{-1}(U_{\rho(g_0,x)})$, or equivalently, $f_0(\rho(g',x')) - f_0(\rho(g_0,x)) \in J_0$ for any $(g',x') \in U^1_x \times U^2_x$. We denote by $\Cat$ the set of triad $(x,U^1_x,U^2_x)$ of $x \in X$, an open neighbourhood $U^1_x \subset G$ of $g_0$, and an open neighbourhood $U^2_x \subset X$ of $x$ such that $f_0(\rho(g',x')) - f_0(\rho(g_0,x)) \in J_0$ for any $(g',x') \in U^1_x \times U^2_x$. Since $X$ is compact, the open covering $\set{U^2}{{}^{\exists} x \in X, {}^{\exists} U^1 \subset G, \t{s.t.\ } (x,U^1,U^2) \in \Cat}$ admits a finite subcovering $\U = \set{U^2_i}{i \in \N \cap [1,d]}$. For each $i \in \N \cap [1,d]$, take an $x_i \in X$ and a $U^1_{x_i} \subset G$ with $(x_i,U^1_{x_i},U^2_i) \in \Cat$, and formally put $U^2_{x_i} \coloneqq U^2_i$. We do not mean $U^2_i = U^2_j$ even though $x_i = x_j$.  Set $U^1 \coloneqq \bigcap_{i = 1}^{d} U^1_{x_i} \subset G$. We denote by $(U^1)^{\t{op}} \subset G^{\t{op}}$ the image of $U^1$. It is an open neighbourhood of $g_0^{\t{op}}$. Let $((g')^{\t{op}},f') \in U^1 \times (f_0 + \t{C}(X,J_0))$. For any $x \in X$, there is an $i \in \N \cap [1,d]$ such that $x \in U^2_{x_i}$, and we have
\begin{eqnarray*}
  & & \rho^{\vee}((g')^{\t{op}},f')(x) = f'(\rho(g',x)) \\
  & \in & f'(\rho(U^1_{x_i} \times U^2_{x_i})) \subset f'(U_{\rho(g_0,x_i)}) = f' \left( \Set{x' \in X}{f_0(x') - f_0(\rho(g_0,x_i)) \in J_0} \right) \\
  & \subset & ((f' - f_0) + f_0) \left( \Set{x' \in X}{f_0(x') - f_0(\rho(g_0,x_i)) \in J_0} \right) \\
  & \subset & J_0 + (f_0(\rho(g_0,x_i)) + J_0) \subset J_0 + \left( (f_0(\rho(g_0,x)) - J_0) + J_0 \right) \subset \rho^{\vee}(g_0^{\t{op}},f_0)(x) + J.
\end{eqnarray*}
It implies $\rho^{\vee}(g',f') \in \rho^{\vee}(g_0,f_0) + \t{C}(X,J)$. Thus $\rho^{\vee}$ is continuous.

\vspace{0.2in}
We verify the other conditions. For any $(g^{\t{op}},h^{\t{op}},f) \in G^{\t{op}} \times G^{\t{op}} \times \t{C}(X,R)$, we have
\begin{eqnarray*}
  & & \rho^{\vee}(g^{\t{op}}h^{\t{op}},f) = \rho^{\vee}((hg)^{\t{op}},f) = f(\rho(hg,x)) = f(\rho(h,\rho(g,x))) = \rho^{\vee}(h^{\t{op}},f)(\rho(g,x)) \\
  & = & \rho^{\vee}(g^{\t{op}}, \rho^{\vee}(h^{\t{op}},f))(x),
\end{eqnarray*}
for any $x \in X$, and hence $\rho^{\vee}(g^{\t{op}}, \rho^{\vee}(h^{\t{op}},f)) = \rho^{\vee}(g^{\t{op}}h^{\t{op}},f)$. For any $f \in \t{C}(X,R)$, we have
\begin{eqnarray*}
  \rho^{\vee}(1^{\t{op}},f)(x) = f(\rho(1^{\t{op}},x)) = f(x)
\end{eqnarray*}
for any $x \in X$, and hence $\rho^{\vee}(1^{\t{op}},f) = f$. For any $(g^{\t{op}},f,f') \in G^{\t{op}} \times \t{C}(X,R) \times \t{C}(X,R)$, we have
\begin{eqnarray*}
  & & \rho^{\vee}(g^{\t{op}},f+f')(x) = (f+f')(\rho(g,x)) = f(\rho(g,x)) + f'(\rho(g,x)) \\
  & = & \rho^{\vee}(g^{\t{op}},f)(x) + \rho^{\vee}(g^{\t{op}},f')(x),
\end{eqnarray*}
for any $x \in X$, and hence $\rho^{\vee}(g^{\t{op}},f+f') = \rho^{\vee}(g^{\t{op}},f) + \rho^{\vee}(g^{\t{op}},f')$. For any $(g^{\t{op}},r,f) \in G^{\t{op}} \times R \times \t{C}(X,R)$, we have
\begin{eqnarray*}
  \rho^{\vee}(g^{\t{op}},rf)(x) = (rf)(\rho(g,x)) = r(f(\rho(g,x))) = r \rho^{\vee}(g^{\t{op}},f)(x),
\end{eqnarray*}
for any $x \in X$, and hence $\rho^{\vee}(g^{\t{op}},rf) = r \rho^{\vee}(g^{\t{op}},f)$. For any $(g^{\t{op}},f,f') \in G^{\t{op}} \times \t{C}(X,R) \times \t{C}(X,R)$, we have
\begin{eqnarray*}
  \rho^{\vee}(g^{\t{op}},ff')(x) = (ff')(\rho(g,x)) = f(\rho(g,x)) f'(\rho(g,x)) = \rho^{\vee}(g^{\t{op}},f)(x) \rho^{\vee}(g^{\t{op}},f')(x),
\end{eqnarray*}
for any $x \in X$, and hence $\rho^{\vee}(g^{\t{op}},ff') = \rho^{\vee}(g^{\t{op}},f) \rho^{\vee}(g^{\t{op}},f')$. For any $g^{\t{op}} \in G^{\t{op}}$, we have
\begin{eqnarray*}
  \rho^{\vee}(g^{\t{op}},1)(x) = 1(\rho(g,x)) = 1 = 1(x),
\end{eqnarray*}
for any $x \in X$, and hence $\rho^{\vee}(g^{\t{op}},1) = 1$. Thus $(\t{C}(X,R),\rho^{\vee})$ is a commutative topological $R[G^{\t{op}}]$-algebra.
\end{proof}

\begin{dfn}
Let $\A$ be a topological $R[G]$-algebra. A continuous map $\kappa \colon G \to \A$ is said to be a {\it crossed homomorphism $\kappa \colon G \to (\A,\rho)$} if it satisfies $\kappa(1) = 1$ and $\kappa(gg') = \kappa(g) \rho(g, \kappa(g'))$ for any $(g,g') \in G \times G$. For a crossed homomorphism $\kappa \colon G \to (\A,\rho)$, we define a map
\begin{eqnarray*}
  \rho_{\kappa} \colon G \times \A & \to & \A \\
  (g,f) & \mapsto & \kappa(g) \rho(g,f),
\end{eqnarray*}
and call it {\it the action of $G$ on $(\A,\rho)$ of weight $\kappa$}.
\end{dfn}

\begin{prp}
\label{weight}
Let $\A$ be a topological $R[G]$-algebra, and $\kappa \colon G \to (\A,\rho)$ a crossed homomorphism. Then $(\A,\rho_{\kappa})$ is a topological $R[G]$-module.
\end{prp}

\begin{proof}
The continuity of $\rho_{\kappa}$ follows from that of $\rho$, $\kappa$, and the multiplication $\A \times \A \to \A$. We have
\begin{eqnarray*}
  & & \rho_{\kappa}(g, \rho_{\kappa}(g',f)) = \rho_{\kappa}(g, \kappa(g') \rho(g',f)) = \kappa(g) \rho(g, \kappa(g') \rho(g',f)) \\
  & = & \kappa(g) \kappa(g') \rho(g, \rho(g',f) = \kappa(gg') \rho(gg',f) = \rho_{\kappa}(gg',f)
\end{eqnarray*}
for any $(g,g',f) \in G \times G \times \A$. We have
\begin{eqnarray*}
  & & \rho_{\kappa}(1,f) = \kappa(1) \rho(1,f) = f
\end{eqnarray*}
for any $f \in \A$. We have
\begin{eqnarray*}
  & & \rho_{\kappa}(g,f+f') = \kappa(g) \rho(g,f+f') = \kappa(g) \rho(g,f) + \kappa(g) \rho(g,f') = \rho_{\kappa}(g,f) + \rho_{\kappa}(g,f')
\end{eqnarray*}
for any $(g,f,f') \in G \times \A \times \A$. We have
\begin{eqnarray*}
  & & \rho_{\kappa}(g,rf) = \kappa(g) \rho(g,rf) = \kappa(g) r \rho(g,f) = r \kappa(g) \rho(g,f) = r \rho_{\kappa}(g,f)
\end{eqnarray*}
for any $(g,r,f) \in G \times R \times \A$. Thus $(\A,\rho_{\kappa})$ is a topological $R[G]$-module.
\end{proof}

\begin{crl}
\label{geometric action 2}
Let $(X,\rho)$ be a compact $G$-space, and $\kappa \colon G \to \t{C}(X,R) \colon g \mapsto \kappa_g$ a continuous map such that $\kappa_1 = 1$ and $\kappa_{gg'}(x) = \kappa_g(\rho(g',x)) \kappa_{g'}(x)$ for any $(g,g',x) \in G \times G \times X$. Then $\t{C}(X,R)$ is a topological $R[G^{\t{op}}]$-module with respect to the action
\begin{eqnarray*}
  \rho^{\vee}_{\kappa} \colon G^{\t{op}} \times \t{C}(X,R) & \to & \t{C}(X,R) \\
  (g^{\t{op}},f) & \mapsto & \left( \rho^{\vee}_{\kappa}(g^{\t{op}},f) \colon x \mapsto \kappa_g(x) f(\rho(g,x)) \right).
\end{eqnarray*}
\end{crl}

\begin{proof}
We abbreviate $\kappa \circ (\cdot)^{\t{op}}$ to $\kappa$. We have $\rho^{\vee}_{\kappa}(g^{\t{op}},f) = \kappa_{g^{\t{op}}} \rho^{\vee}(g^{\t{op}},f)$ for any $(g^{\t{op}},f) \in G^{\t{op}} \times \t{C}(X,R)$. Therefore by Proposition \ref{geometric action} and Proposition \ref{weight}, it suffices to verify that $\kappa$ is a crossed homomorphism $G^{\t{op}} \to (\t{C}(X,R),\rho^{\vee})$. For any $(g^{\t{op}},(g')^{\t{op}},f) \in G^{\t{op}} \times G^{\t{op}} \times \t{C}(X,R)$, we have
\begin{eqnarray*}
  \kappa_{g^{\t{op}} (g')^{\t{op}}}(x) = \kappa_{(g'g)^{\t{op}}}(x) = \kappa_{(g')^{\t{op}}}(\rho(g,x)) \kappa_{g^{\t{op}}}(x) = \rho^{\vee}(g^{\t{op}}, \kappa_{(g')^{\t{op}}})(x) \kappa_{g^{\t{op}}}(x)
\end{eqnarray*}
for any $x \in X$, and hence $\kappa_{g^{\t{op}} (g')^{\t{op}}} = \kappa_{g^{\t{op}}} \rho^{\vee}(g^{\t{op}}, \kappa_{(g')^{\t{op}}})$. Thus $\kappa \colon G^{\t{op}} \to (\t{C}(X,R),\rho^{\vee})$ is a crossed homomorphism.
\end{proof}

\begin{exm}
\label{modular form}
By Proposition \ref{p-adic linear fractional transformation} and Proposition \ref{geometric action}, $\t{C}(\Z_p,\Z_p)$ admits an action $m_p^{\vee}$ of $\Pi_0(p)^{\t{op}}$ such that $(\t{C}(\Z_p,\Z_p),m_p^{\vee})$ is a commutative topological $\Z_p[\Pi_0(p))^{\t{op}}]$-algebra. For any continuous group homomorphism $\chi \colon \Z_p^{\times} \to \Z_p^{\times}$, the map
\begin{eqnarray*}
  \kappa(\chi) \colon \Pi_0(p) & \to & \t{C}(\Z_p,\Z_p) \\
  \left(
    \begin{array}{cc}
      a & b \\
      c & d
    \end{array}
  \right)
  & \mapsto & \chi(cz+d)
\end{eqnarray*}
satisfies the conditions in Corollary \ref{geometric action 2} with respect to $m_p$, where $z \coloneqq \t{id}_{\Z_p}$. Indeed, we have
\begin{eqnarray*}
  & & \kappa(\chi)
  \left(
    \left(
      \begin{array}{cc}
        a & b \\
        c & d
      \end{array}
    \right)
    \left(
      \begin{array}{cc}
        e & f \\
        g & h
      \end{array}
    \right)
  \right)
  =
  \kappa(\chi)
  \left(
    \left(
      \begin{array}{cc}
        ae+bg & af+bh \\
        ce+dg & cf+dh
      \end{array}
    \right)
  \right) \\
  & = & \chi((ce+dg)z+(cf+dh)) = \chi(c(ez+f)+d(gz+h)) = \chi \left( (gz+h) \left( c \frac{(ez+f}{gz+h} + d \right) \right) \\
  & = & \chi(gz+h) \chi \left( c \frac{(ez+f}{gz+h} + d \right) = \kappa(\chi)
  \left(
    \left(
      \begin{array}{cc}
        e & f \\
        g & h
      \end{array}
    \right)
  \right)
  \rho^{\vee}
  \left(
    \left(
      \begin{array}{cc}
        e & f \\
        g & h
      \end{array}
    \right),
    \kappa(\chi)
    \left(
      \left(
        \begin{array}{cc}
          a & b \\
          c & d
        \end{array}
      \right)
    \right)
  \right)
\end{eqnarray*}
for any 
$
  \left(
    \left(
      \begin{array}{cc}
        a & b \\
        c & d
      \end{array}
    \right),
    \left(
      \begin{array}{cc}
        e & f \\
        g & h
      \end{array}
    \right)
  \right)
  \in \Pi_0(p) \times \Pi_0(p) 
$. Thus we obtain a continuous action $(m_p^{\vee})_{\kappa(\chi)}$ of $\Pi_0(p)^{\t{op}}$ on $\t{C}(\Z_p,\Z_p)$.
\end{exm}
\subsection{$p$-adic Modular Forms and Hecke Algebras}
\label{p-adic Modular Forms and Hecke Algebras}

Let $N$ be a positive integer with $N \geq 5$, and $p$ a prime number dividing $N$. Henceforth, we fix an algebraic closure $\overline{\Q}_p$ of $\Q_p$ and an isomorphism $\iota_{p,\infty} \colon \overline{\Q}_p \stackrel{\sim}{\to} \C$ of fields. We recall $p$-adic modular forms.

\vspace{0.2in}
Let $R$ be a commutative topological ring. For each formal power series $f(q) = \sum_{h = 0}^{\infty} a_h q^h$ over $R$, we put $a_h(f) \coloneqq a_h \in R$ for each $h \in \N$. Suppose that $R$ is a subring of $\overline{\Q}_p$. Let $k \in \N \cap [2,\infty)$. A {\it modular form over $R$ of weight $k$  and level $N$} is an element of the $R$-algebra $R[[q]]$ of formal power series whose image in $\C[[q]]$ is a modular form of level $\Gamma_1(N)$ of weight $k$. The notion of modular forms over $R$ is compatible with extensions of $R$ if $R$ contains $\Z_p$ by \cite{DI95} Theorem 12.3.2 and Theorem 12.3.4/2, and hence then coincides with the usual definition given as an $R$-linear combination of modular forms over $\Z_p$. See also \cite{DR73} VII 3.7. We denote by $\t{M}_k(\Gamma_1(N),R) \subset R[[q]]$ the $R$-submodule of modular forms over $R$ of weight $k$. For a modular form over $R$ of weight $k$ and level $N$, we call its image in $\C[[q]]$ its {\it corresponding modular form}. A modular form over $R$ of weight $k$ and level $N$ is said to be {\it a cusp form over $R$ of weight $k$ and level $N$} if its corresponding modular form is a cusp form.

\vspace{0.2in}
Identifying $\t{M}_k(\Gamma_1(N),\overline{\Q}_p)$ as the $\C$-vector space of modular forms of weight $k$, we have a $\overline{\Q}_p$-linear action of Hecke operators on it. Here a Hecke operator means one of operators $T_{\ell}$ for a prime number $\ell$ and diamond operators $\langle \overline{n} \rangle$ for an $\overline{n} \in (\Z/N \Z)^{\times}$ (\cite{Hid86} \S 1 p.\ 549, \cite{Eme10} \S 1.2 p.\ 4-5). Let $\epsilon \colon (\Z/N \Z)^{\times} \to \overline{\Q}{}_p^{\times}$ be a Dirichlet character. We denote by $\t{M}_k(\Gamma_1(N),\epsilon,R) \subset \t{M}_k(\Gamma_1(N),R)$ the $R$-submodule of elements whose corresponding modular forms are contained in the kernel of $\langle \overline{n} \rangle - \iota_{p,\infty}(\epsilon(\overline{n}))$ for any $d \in (\Z/N \Z)^{\times}$. We denote by $R[\epsilon] \subset \overline{\Q}_p$ the $R$-subalgebra generated by the image of $\epsilon$. Operators $T_{\ell}$ for a prime number $\ell$ and $S_n$ for an $n \in \Z$ coprime to $N$ act on $\t{M}_k(\Gamma_1(N),\epsilon,R[\epsilon])$ through the embedding into $\t{M}_k(\Gamma_1(N),\overline{\Q}_p)$. The action is given explicitly in the following way:
\begin{eqnarray*}
  T_{\ell} \colon \t{M}_k(\Gamma_1(N),\epsilon,R[\epsilon]) & \to & \t{M}_k(\Gamma_1(N),\epsilon,R[\epsilon]) \\
  f(q) & \mapsto &
  \left\{
    \begin{array}{ll}
      \sum_{h = 0}^{\infty} a_{\ell h}(f) q^h + \sum_{h = 0}^{\infty} a_h(f) \epsilon(\ell + N \Z) \ell^{k-2} q^{\ell h} & (\ell \mid \hspace{-.62em}/ N) \\
      \sum_{h = 0}^{\infty} a_{\ell h}(f) q^h & (\ell \mid N)
    \end{array}
  \right. \\
  S_n \colon \t{M}_k(\Gamma_1(N),\epsilon,R[\epsilon]) & \to & \t{M}_k(\Gamma_1(N),\epsilon,R[\epsilon]) \\
  f(q) & \mapsto & \epsilon(n + N \Z) n^{k-2} f(q).
\end{eqnarray*}
A modular form $f(q)$ over $R$ of weight $k$ is said to be an {\it eigenform over $R$ of weight $k$ and level $N$} if its corresponding modular form is an eigenform. It is equivalent to the condition that there is a Dirichlet character $\epsilon_f$ such that the image of $f$ in $\t{M}_k(\Gamma_1(N),R[\epsilon_f])$ lies in $\t{M}_k(\Gamma_1(N),\epsilon_f,R[\epsilon_f])$ and is a simultaneous eigenvector of Hecke operators. An eigenform $f$ over $R$ of weight $k$ is said to be {\it normalised} if $a_1(f) = 1$. A {\it cuspidal eigenform over $R$ of weight $k$ and level $N$} is an eigenform over $R$ of weight $k$  and level $N$ which is a cusp form over $R$ of weight $k$ and level $N$.

\begin{exm}
Let $k \in \N \cap [2,\infty)$ be an even number. Then the formal power series
\begin{eqnarray*}
  E_k(q) \coloneqq -\frac{B_k}{2k} + \sum_{h = 1}^{\infty} \left( \sum_{d \mid h} d^{k-1} \right) q^h \in \Q[[q]]
\end{eqnarray*}
is a normalised eigenform over $\Q$ of weight $k$  and level $N$ which is not cusp. Here $B_k \in \Q$ is the $k$-th Bernoulli number.
\end{exm}

Let $k_0 \in \N \cap [2,\infty)$. We denote by $\t{T}_{k_0,N} \subset \t{End}_{\overline{\Z}_p}(\t{M}_{k_0}(\Gamma_1(N),\overline{\Z}_p))$ the commutative $\Z_p$-subalgebra generated by Hecke operators. Since $\t{M}_{k_0}(\Gamma_1(N),\overline{\Q}_p)$ is a $\overline{\Q}_p$-vector space, $\t{T}_{k_0,N}$ is torsionfree as a $\Z_p$-module. Since $\C \otimes_{\Z_p} \t{T}_{k_0,N}$ is isomorphic to a $\C$-vector subspace of $\t{End}_{\C}(\t{H}^1(\Gamma_1(N),\t{Sym}^{k_0-2}(\C^2,\rho_{\C^2})))$ by the Eichler--Shimura isomorphism (\cite{Shi59} 5 Th\'eor\`eme 1, \cite{Hid93} 6.3 Theorem 4), it is finite dimensional as a $\C$-vector space. By the equality
\begin{eqnarray*}
  \t{T}_{k_0,N} = \t{End}_{\overline{\Z}_p} \left( \t{M}_{k_0} \left( \Gamma_1(N), \overline{\Z}_p \right) \right) \cap (\Q_p \otimes_{\Z_p} \t{T}_{k_0,N})
\end{eqnarray*}
as $\Z_p$-algebras of $\t{End}_{\overline{\Q}_p}(\t{M}_{k_0}(\Gamma_1(N),\overline{\Q}_p)) \cong \Q_p \otimes_{\Z_p} \t{End}_{\overline{\Z}_p}(\t{M}_{k_0}(\Gamma_1(N),\overline{\Z}_p))$, $\t{T}_{k_0,N}$ is a $\Z_p$-algebra finitely generated as a $\Z_p$-module. In fact, the finiteness can be verified with no use of the Eichler--Shimura isomorphism, because Hecke operators act on $\t{M}_{k_0}(\Gamma_1(N),\Z)$ by \cite{DI95} Proposition 12.3.11 and \cite{DI95} Proposition 12.4.1. For any $n \in \N$ with $n \in 1 + N \Z$, we have $S_n = n^{k_0-2} \langle n + N \Z \rangle = n^{k_0-2} \langle 1 + N \Z \rangle = n^{k_0-2} \in \t{T}_{k_0,N}$, and hence $S_n$ is invertible as an element of $\t{T}_{k_0,N}$ because $n$ is coprime to $p$. The map
\begin{eqnarray*}
  \N \cap (1 + N \Z) & \to & \t{T}_{k_0,N}^{\times} \\
  n & \mapsto & n^{k_0-2} = S_n
\end{eqnarray*}
is a monoid homomorphism with respect to the multiplication, and it extends to a continuous character
\begin{eqnarray*}
  S_{\bullet} \colon 1 + N \Z_p & \to & \t{T}_{k_0,N}^{\times} \\
  n & \mapsto & n^{k_0-2}.
\end{eqnarray*}
By the universality of the Iwasawa algebra, $S_{\bullet}$ associates a continuous $\Z_p$-algebra homomorphism $\Z_p[[1 + N \Z_p]] \to \t{T}_{k_0,N}$, and we regard $\t{T}_{k_0,N}$ as a profinite $\Z_p[[1 + N \Z_p]]$-algebra.

\vspace{0.2in}
The $\Z_p$-bilinear pairing
\begin{eqnarray*}
  \langle \cdot, \cdot \rangle \colon \t{T}_{k_0,N} \otimes_{\Z_p} \t{M}_{k_0}(\Gamma_1(N),\overline{\Q}_p) & \to & \overline{\Q}_p \\
  (A,f) & \mapsto & \langle A, f \rangle \coloneqq a_1(Af)
\end{eqnarray*}
is non-degenerate, and induces a $\overline{\Q}_p$-linear isomorphism
\begin{eqnarray*}
  \t{M}_{k_0}(\Gamma_1(N),\overline{\Q}_p) \stackrel{\sim}{\to} \t{Hom}_{\Z_p}(\t{T}_{k_0,N},\overline{\Q}_p)
\end{eqnarray*}
by \cite{Hid93} 5.3 Theorem 1. Therefore the subset of $\t{M}_{k_0}(\Gamma_1(N),\overline{\Q}_p)$ consisting of normalised eigenforms corresponds to the subset of $\t{Hom}_{\Z_p}(\t{T}_{k_0,N},\overline{\Q}_p)$ consisting $\Z_p$-algebra homomorphisms. Let $f$ be a normalised eigenform $f$ over $\overline{\Q}_p$ of weight $k_0$. We denote by $\lambda_f \colon \t{T}_{k_0,N} \to \overline{\Q}_p$ the $\Z_p$-algebra homomorphism corresponding to $f$, and call it {\it the system of Hecke eigenvalues associated to $f$}. We have $a_h(f) = \lambda_f(T_h)$ for any $h \in \N \backslash \ens{0}$. In particular, if $f$ is a cuspidal eigenform over $\overline{\Q}_p$ of weight $k_0$  and level $N$, then we have $f = \sum_{h = 1}^{\infty} \lambda_f(T_h) q^h$. Since $\t{T}_{k_0,N}$ is finitely generated as a $\Z_p$-module, the subextension $\Q_p(f)$ of $\overline{\Q}_p/\Q_p$ generated by $\set{a_h(f)}{h \in \N} = \set{\lambda_f(T_h)}{h \in \N}$ is a finite extension of $\Q_p$. We call $\Q_p(f)$ {\it the $p$-adic Hecke field associated to $f$}.

\vspace{0.2in}
We denote by
\begin{eqnarray*}
  \t{T}_{\leq k_0,N} \subset \t{End}_{\overline{\Q}_p} \left( \bigoplus_{k = 2}^{k_0} \t{M}_k(\Gamma_1(N),\overline{\Q}_p) \right)
\end{eqnarray*}
the commutative $\Z_p$-subalgebra generated by $T_{\ell}$ for each prime number $\ell$ and $S_n$ for each $n \in \N$ coprime to $N$ acting in a diagonal way. By the universality of the Iwasawa algebra, we have a continuous $\Z_p$-algebra homomorphism $\Z_p[[1 + N \Z_p]] \to \t{T}_{\leq k_0,N}$ sending each $n \in \N \cap (1 + N \Z) \subset 1 + \Z_p$ to the Hecke operator $S_n$. There is a natural embedding $\t{T}_{\leq k_0,N} \hookrightarrow \prod_{k = 2}^{k_0} \t{T}_{k,N}$ by definition, and hence $\t{T}_{\leq k_0,N}$ is finitely generated as a $\Z_p$-module. We endow $\t{T}_{\leq k_0,N}$ with the $p$-adic topology, and regard it as a profinite $\Z_p[[1 + N \Z_p]]$-algebra. We put
\begin{eqnarray*}
  \T_N \coloneqq \varprojlim_{k \in \N} \t{T}_{\leq k,N},
\end{eqnarray*}
and regard it as a profinite $\Z_p[[1 + N \Z_p]]$-algebra.

\vspace{0.2in}
Let $k_0 \in \N \cap [2,\infty)$. Henceforth, we fix a $p$-adic norm $\v{\cdot} \colon \Q_p \to [0,\infty)$, and endow $\overline{\Q}_p$ with a unique non-Archimedean norm $\v{\cdot} \colon \overline{\Q}_p \to [0,\infty)$ extending the $p$-adic norm on $\Q_p$. We denote by $\overline{\Z}_p \subset \overline{\Q}_p$ the valuation ring. It is an integral closure of $\Z_p$ in $\overline{\Q}_p$, and hence is independent of the choice of the $p$-adic norm on $\Q_p$. It is well-known that $\t{M}_{k_0}(\Gamma_1(N),\overline{\Q}_p) \subset \overline{\Q}_p[[q]]$ is contained in the image of $\overline{\Z}_p[[q]] \otimes_{\Z_p} \Q_p$, and hence we endow $\t{M}_{k_0}(\Gamma_1(N),\overline{\Q}_p)$ with the non-Archimedean norm $\n{\cdot}$ given by setting
\begin{eqnarray*}
  \n{f} \coloneqq \max_{h \in \N} \v{a_h(f)} < \infty
\end{eqnarray*}
for each $f \in \t{M}_{k_0}(\Gamma_1(N),\overline{\Q}_p)$. An $f \in \t{M}_{k_0}(\Gamma_1(N),\overline{\Q}_p)$ is said to be {\it of finite slope} if $f \in T_p^h(\t{M}_{k_0}(\Gamma_1(N),\overline{\Q}_p))$ for any $h \in \N$. We denote by $\t{M}_{k_0}(\Gamma_1(N),\overline{\Q}_p)^{< \infty}$ the $T_p$-stable $\overline{\Q}_p$-vector subspace of $\t{M}_{k_0}(\Gamma_1(N),\overline{\Q}_p)$ consisting of modular forms of finite slope. Since Hecke operators commute with each other, $\t{M}_{k_0}(\Gamma_1(N),\overline{\Q}_p)^{< \infty} \subset \t{M}_{k_0}(\Gamma_1(N),\overline{\Q}_p)$ is stable under the action of the other Hecke operators. Since $\t{M}_{k_0}(\Gamma_1(N),\overline{\Q}_p)$ is a finite dimensional $\overline{\Q}_p$-vector space, the decreasing sequence $(T_p^h(\t{M}_{k_0}(\Gamma_1(N),\overline{\Q}_p)))_{h = 0}^{\infty}$ is eventually stable. It implies that the restriction of $T_p$ on $\t{M}_{k_0}(\Gamma_1(N),\overline{\Q}_p)^{< \infty}$ is bijective. Moreover, since $\t{M}_{k_0}(\Gamma_1(N),\overline{\Q}_p)^{< \infty}$ is a finite dimensional $\overline{\Q}_p$-vector space, the restriction of $T_p$ on $\t{M}_{k_0}(\Gamma_1(N),\overline{\Q}_p)^{< \infty}$ and its inverse $T_p^{-1}$ are continuous with respect to the norm topology.

\vspace{0.2in}
Let $s \in \N \backslash \ens{0}$. An $f \in \t{M}_{k_0}(\Gamma_1(N),\overline{\Q}_p)$ is said to be {\it of slope $< s$} if $f$ is of finite slope and $\lim_{h \to \infty} \n{(p^s T_p^{-1})^h f} = 0$. If $f$ is an eigenform, then it is equivalent to the condition that the system $\lambda_f \colon \t{T}_{k_0,N} \to \overline{\Q}_p$ of Hecke eigenvalues associated to $f$ satisfies $\v{\lambda_f(T_p)} > \v{p}^s$. Let $R \subset \overline{\Q}_p$ be a subring. An $f \in \t{M}_{k_0}(\Gamma_1(N),R)$ is said to be {\it of slope $< s$} if $f$ is a modular form over $\overline{\Q}_p$ of slope $< s$. We denote by $ \t{M}_{k_0}(\Gamma_1(N),R)^{< s} \subset \t{M}_{k_0}(\Gamma_1(N),R)$ the $R$-submodule of modular forms of slope $< s$. Since Hecke operators commute with each other, $\t{M}_{k_0}(\Gamma_1(N),\overline{\Q}_p)^{< s} \subset \t{M}_{k_0}(\Gamma_1(N),\overline{\Q}_p)$ is stable under the action of Hecke operators. We denote by $\t{T}_{k_0,N}^{[< s]} \subset \t{End}_{\overline{\Q}_p}(\t{M}_{k_0}(\Gamma_1(N),\overline{\Q}_p)^{< s})$ the image of $\t{T}_{k_0,N}$. Since $\t{T}_{k_0,N}$ is finitely generated as a $\Z_p$-module, so is $\t{T}_{k_0,N}^{[< s]}$.

\vspace{0.2in}
The operator $T_p$ is invertible on $\t{M}_{k_0}(\Gamma_1(N),\overline{\Q}_p)^{< s}$ by definition. We denote by
\begin{eqnarray*}
  \t{T}_{k_0,N}^{< s} \subset \t{End}_{\overline{\Q}_p} \left( \t{M}_{k_0}(\Gamma_1(N),\overline{\Q}_p)^{< s} \right)
\end{eqnarray*}
the commutative $\Z_p$-subalgebra generated by $\t{T}_{k_0,N}^{[< s]}$ and $p^sT_p^{-1}$. Let
\begin{eqnarray*}
  F(X) = X^n + a_1 X^{n-1} + \cdots + a_n \in \Z_p[X]
\end{eqnarray*}
be the minimal polynomial of $T_p$ as an element of $\t{T}_{k_0,N}^{[< s]}$. Since every eigenvalue of $T_p$ as an element of $\t{End}_{\overline{\Q}_p}(\t{M}_{k_0}(\Gamma_1(N),\overline{\Q}_p)^{< s})$ is of norm in $(\v{p}^s,1]$, we have $\v{a_i} \leq 1$ for any $i \in \N \cap [1,n]$, $\v{p}^{(n-i)s} a_i < \v{a_n}$ for any $i \in \N \cap [1,n-1]$, and $\v{p}^{ns} < \v{a_n}$. It implies that the polynomial
\begin{eqnarray*}
  Q(X) \coloneqq a_n^{-1}X^n P(p^sX^{-1}) = X^n + \frac{p^s a_{n-1}}{a_n} X^{n-1} + \cdots + \frac{p^{ns}}{a_n}
\end{eqnarray*}
lies in $\Z_p[X]$. Since $\t{End}_{\overline{\Q}_p}(\t{M}_{k_0}(\Gamma_1(N),\overline{\Q}_p)^{< s})$ is a $\overline{\Q}_p$-vector space, its $\Z_p$-subalgebra $\t{T}_{k_0,N}^{< s}$ is torsionfree as a $\Z_p$-module. Therefore the equality $a_n Q(p^sT_p^{-1}) = (p^sT_p^{-1})^nP(T_p) = 0 \in \t{End}_{\overline{\Q}_p}(\t{M}_{k_0}(\Gamma_1(N),\overline{\Q}_p)^{< s})$ ensures that $Q(p^sT_p^{-1}) = 0$ as an element of $\t{T}_{k_0,N}^{< s}$. It implies that $p^sT_p^{-1}$ is integral over $\t{T}_{k_0,N}^{[< s]}$ as an element of $\t{T}_{k_0,N}^{< s}$, and $\t{T}_{k_0,N}^{< s}$ is finitely generated as a $\t{T}_{k_0,N}^{[< s]}$-module.

\begin{prp}
\label{p^s/T_p}
The surjective $\t{T}_{k_0,N}$-algebra homomorphism
\begin{eqnarray*}
  \t{T}_{k_0,N}^{[< s]}[X] & \twoheadrightarrow & \t{T}_{k_0,N}^{< s} \\
  X & \mapsto & p^sT_p^{-1}
\end{eqnarray*}
induces a $\t{T}_{k_0,N}$-algebra isomorphism
\begin{eqnarray*}
  \left( \t{T}_{k_0,N}^{[< s]}[X]/(T_pX-p^s) \right)_{\t{free}} \cong \t{T}_{k_0,N}^{< s}.
\end{eqnarray*}
\end{prp}

\begin{proof}
By the argument above, $p^n(p^sT_p^{-1}) \in \t{T}_{k_0,N}^{< s}$ lies in the image of $\t{T}_{k_0,N}^{[< s]}$ for an $n \in \N$. Therefore the flatness of $\Q_p$ as a $\Z_p$-module ensures the assertion.
\end{proof}

For each $k_0 \in \N$, we denote by
\begin{eqnarray*}
  \t{T}_{\leq k_0,N}^{[< s]} \subset \t{End}_{\overline{\Q}_p} \left( \bigoplus_{k = 2}^{k_0} \t{M}_k(\Gamma_1(N),\overline{\Q}_p)^{< s} \right)
\end{eqnarray*}
the image of $\t{T}_{\leq k_0,N}$, and by
\begin{eqnarray*}
  \t{T}_{\leq k_0,N}^{< s} \subset \t{End}_{\overline{\Q}_p} \left( \bigoplus_{k = 2}^{k_0} \t{M}_k(\Gamma_1(N),\overline{\Q}_p)^{< s} \right)
\end{eqnarray*}
the $\Z_p$-subalgebra generated by $\t{T}_{\leq k_0,N}$ and $p^sT_p^{-1}$. They are finitely generated as $\Z_p$-modules by a similar argument with that in the previous paragraph. We set
\begin{eqnarray*}
  \T_N^{[< s]} & \coloneqq & \varprojlim_{k \in \N} \t{T}_{\leq k,N}^{[< s]} \\
  \T_N^{< s} & \coloneqq & \varprojlim_{k \in \N} \t{T}_{\leq k,N}^{< s},
\end{eqnarray*}
and regard them as profinite $\T_N$-algebras. In particular, they are regarded as profinite $\Z_p[[1 + N \Z_p]]$-algebras.

\begin{prp}
\label{topologically nilpotent}
The continuous $\T_N$-algebra homomorphism
\begin{eqnarray*}
  \T_N^{[< s]}[[X]] & \to & \T_N^{< s} \\
  X & \mapsto & p^sT_p^{-1}
\end{eqnarray*}
is surjective.
\end{prp}

\begin{proof}
To begin with, we verify that the $\T_N$-algebra homomorphism
\begin{eqnarray*}
  \T_N^{[< s]}[X] & \to & \T_N^{< s} \\
  X & \mapsto & p^sT_p^{-1}
\end{eqnarray*}
uniquely extends to a continuous $\T_N$-algebra homomorphism
\begin{eqnarray*}
  \T_N^{[< s]}[[X]] \to \T_N^{< s}.
\end{eqnarray*}
Here $\T_N^{[< s]}[X]/(X^r)$ is regarded as a profinite $\T_N^{[< s]}$-algebra with respect to the topology given by the canonical $\T_N^{[< s]}$-linear basis $(X^h)_{h = 0}^{r-1}$ for each $r \in \N$, and
\begin{eqnarray*}
  \T_N^{[< s]}[[X]] = \varprojlim_{r \in \N} \T_N^{[< s]}[X]/(X^r),
\end{eqnarray*}
is endowed with the inverse limit topology. Let $k_1 \in \N \cap [2,\infty)$. Let $P(X) = X^n + a_1 X^{n-1} + \cdots + a_n \in \Z_p[X]$ be the minimal polynomial of $T_p$ as an element of $\t{T}_{\leq k_1,N}^{[< s]}$. Since every eigenvalue of the action of $T_p$ on $\bigoplus_{k = 2}^{k_1} \t{M}_{k}(\Gamma_1(N),\Z_p)^{< s}$ is of norm in $(\v{p}^s,1]$, we have $\v{a_i} \leq 1$ for any $i \in \N \cap [1,n]$, $\v{p}^{(n-i)s} a_i < \v{a_n}$ for any $i \in \N \cap [1,n-1]$, and $\v{p}^{ns} < \v{a_n}$. Therefore the polynomial
\begin{eqnarray*}
  Q(X) \coloneqq X^n + \frac{p^s a_{n-1}}{a_n} X^{n-1} + \cdots + \frac{p^{ns}}{a_n}.
\end{eqnarray*}
satisfies $Q(X) - X^n \in p \Z_p[X]$. Since $\t{T}_{\leq k_1,N}^{< s}$ is torsionfree as a $\Z_p$-module, the equality
\begin{eqnarray*}
  a_n Q(p^sT_p^{-1}) = (p^sT_p^{-1})^nP(T_p) = 0
\end{eqnarray*}
in $\t{T}_{\leq k_1,N}^{< s}$ ensures $Q(p^sT_p^{-1}) = 0 \in \t{T}_{\leq k_1,N}^{< s}$ and hence $(p^sT_p^{-1})^n \in p \t{T}_{\leq k_1,N}^{< s}$. It implies $p^sT_p^{-1}$ is topologically nilpotent in $\t{T}_{\leq k_1,N}^{< s}$ with respect to the $p$-adic topology, because $\t{T}_{\leq k_1,N}^{< s}$ is $p$-adically complete. Therefore $p^sT_p^{-1}$ is topologically nilpotent in $\t{T}_{\leq k_1,N}^{< s}$, and the $\T_N$-algebra homomorphism
\begin{eqnarray*}
  \T_N^{[< s]}[X] & \to & \t{T}_{\leq k_1,N}^{< s} \\
  X & \mapsto & p^sT_p^{-1}
\end{eqnarray*}
uniquely extends to a continuous $\T_N$-algebra homomorphism
\begin{eqnarray*}
  \T_N^{[< s]}[[X]] \to \t{T}_{\leq k_0,N}^{< s}
\end{eqnarray*}
by the universality of the algebra of formal power series and the $p$-adic completeness of $\t{T}_{\leq k_0,N}^{< s}$. Thus the $\T_N$-algebra homomorphism
\begin{eqnarray*}
  \T_N^{[< s]}[X] & \to & \T_N^{< s} \\
  X & \mapsto & p^sT_p^{-1}
\end{eqnarray*}
uniquely extends to a continuous $\T_N$-algebra homomorphism
\begin{eqnarray*}
  \varphi \colon \T_N^{[< s]}[[X]] \to \T_N^{< s} = \varprojlim_{k \in \N} \t{T}_{\leq k,N}^{< s}
\end{eqnarray*}
by the universality of the inverse limit. The composite
\begin{eqnarray*}
  \T_N^{[< s]}[[X]] \to \T_N^{< s} \twoheadrightarrow \t{T}_{\leq k_0,N}^{< s}
\end{eqnarray*}
is surjective by the definition of $\t{T}_{\leq k_0,N}^{< s}$ for any $k_0 \in \N \cap [2,\infty)$, and hence the image of $\varphi$ is dense by the definition of the inverse limit topology. Since $\T_N^{[< s]}[[X]]$ is compact and $\T_N^{< s}$ is Hausdorff, the continuity of $\varphi$ ensures its surjectivity.
\end{proof}

\section{Actions on Prodiscrete Cohomologies}
\label{Actions on Prodiscrete Cohomologies}

In this section, let $R$ denote a commutative topological ring, and $G$ a monoid endowed with the discrete topology. We mainly consider the case where $R$ is $\Z_p$ or the Iwasawa algebras, and $G$ is a submonoid of $\t{M}_2(\Z_p)$. We introduce the notion of a prodiscrete cohomology of a linearly complete $R[G]$-module. We compare it with the group cohomology of the underlying module, and with the cohomology of the derived functor of $\Gamma(G,\cdot) \circ \varprojlim$. We also introduce the notion of a profinite $R$-sheaf on a modular curve. The prodiscrete cohomology of a profinite $R$-sheaf coincides with that of the corresponding profinite $R[G]$-module for a suitable $G$ under several conditions.

\subsection{Prodiscrete Cohomologies of Complete Topological Modules}
\label{Prodiscrete Cohomologies of Complete Topological Modules}

Suppose that $R$ is discrete. The category $\t{Mod}(R)$ (resp.\ $\t{Mod}(R[G])$) of discrete $R$-modules and $R$-linear homomorphisms (resp.\ discrete $R[G]$-modules and $R$-linear $G$-equivariant homomorphisms) naturally admits a structure of an Abelian category. The correspondence $(M,\rho) \rightsquigarrow \Gamma(G,(M,\rho))$ gives a left exact functor $\Gamma(G,\cdot) \colon \t{Mod}(R[G]) \to \t{Mod}(R)$. We denote by $\t{H}^*(G,\cdot)$ the cohomology of the right derived functor of $\Gamma(G,\cdot)$, and call it {\it the group cohomology}. We remark that the underlying Abelian group of the group cohomology of a discrete $R[G]$-module is naturally isomorphic to the group cohomology of the underlying $\Z[G]$-module, because they can be calculated cocycles and coboundaries in the same way.

\vspace{0.2in}
Now we consider the case where $R$ is not necessarily discrete. We denote by $\v{R}$ the underlying ring of $R$ endowed with the discrete topology. For a topological $R[G]$-module $(M,\rho)$, the pair $\v{(M,\rho)} = (\v{M},\v{\rho})$ of the underlying $\v{R}$-module $\v{M}$ endowed with the discrete topology and the induced action $\v{\rho} \colon G \times \v{M} \to \v{M} \colon (g,m) \mapsto \rho(g,m)$ is a discrete $\v{R}[G]$-module, and we denote by $\t{H}^*(G,(M,\rho))$ the discrete $\v{R}$-module $\t{H}^*(G,\v{(M,\rho)})$. For any finite $R[G]$-module $(M,\rho)$, the annihilator $\t{Ann}_R(M) \subset R$ is an open ideal, and acts trivially on $\t{H}^*(G,(M,\rho))$. Therefore the action of $\v{R}$ makes $\t{H}^*(G,(M,\rho))$ a discrete $R$-module for any finite $R[G]$-module $(M,\rho)$.

\vspace{0.2in}
For a linearly complete topological $R[G]$-module $(M,\rho)$, we set
\begin{eqnarray*}
  \Hil^*(G,(M,\rho)) \coloneqq \varprojlim_{L \in \Opn_{(M,\rho)}} \t{H}^*(G,(M,\rho)/L),
\end{eqnarray*}
and endow it with the inverse limit topology. We call it {\it the prodiscrete cohomology of $(M,\rho)$}, and regard it as a linearly complete $\v{R}$-module. For any profinite $R[G]$-module $(M,\rho)$, the action of $\v{R}$ makes $\Hil^*(G,(M,\rho))$ a linearly complete $R$-module by the argument in the previous paragraph. In this subsection, we show that the prodiscrete cohomology has an aspect of the cohomology of a derived functor reflecting the topology of $R$. We remark that a derived functor usually does not possess information of topologies. For example, the category of topological $R[G]$-modules and continuous $R$-linear $G$-equivariant homomorphisms does not necessarily admit a structure of an Abelian category, and hence one should forget topologies in order to consider a derived functor.

\begin{lmm}
\label{prodiscrete cohomology}
Let $(M,\rho)$ be a first countable linearly complete $R[G]$-module. Then the natural $\v{R}$-linear homomorphism
\begin{eqnarray*}
  \t{H}^i(G,(M,\rho)) \to \Hil^i(G,(M,\rho))
\end{eqnarray*}
is surjective for any $i \in \N$.
\end{lmm}

\begin{proof}
Let $c = (c_L)_{L \in \Opn_{(M,\rho)}} \in \Hil^i(G,(M,\rho))$. Since $M$ is first countable, there is a decreasing sequence $(L_r)_{r = 0}^{\infty}$ in $\Opn_{(M,\rho)}$ such that $\set{L_r}{r \in \N}$ forms a fundamental system of neighbourhoods of $0$. For each $r \in \N$, take an $i$-cocycle $\tilde{c}_{L_r} \in \t{Z}^i(G,(M,\rho)/L_r)$ representing $c_{L_r}$. We construct an inverse system $(\tilde{c}'_{L_r})_{r = 0}^{\infty}$ of $i$-cocycles $\tilde{c}'_{L_r} \in \t{Z}^i(G,(M,\rho)/L_r)$ representing $c_{L_r}$ for any $r \in \N$. If $i = 0$, then we have $\t{Z}^i(G,(M,\rho/L_r)) = \Gamma(G,(M,\rho)/L_r) \cong \t{H}^i(G,(M,\rho)/L_r)$ for any $r \in \N$, and hence $(\tilde{c}'_{L_r})_{r = 0}^{\infty} \coloneqq (\tilde{c}_{L_r})_{r = 0}^{\infty}$ is an inverse system. Suppose $i > 0$. Put $\tilde{c}'_{L_0} \coloneqq \tilde{c}_{L_0}$. Assume that a compatible system $(\tilde{c}'_{L_r})_{r = 0}^{r_0}$ of representatives of $(c_{L_r})_{r = 0}^{r_0}$is taken for an $r_0 \in \N$. Since the image of $\tilde{c}_{L_{r_0+1}}$ in $\t{Z}^i(G,(M,\rho)/L_{r_0})$ represents $c_{L_{r_0}}$, there is a set-theoretical map $b_{r_0} \colon G^{i-1} \to M/L_{r_0}$ which associates the $i$-coboundary $\partial b_{r_0} \in \t{B}^i(G,(M,\rho)_{L_{r_0}})$ given as the difference of $\tilde{c}'_{L_{r_0}}$ and the image of $\tilde{c}_{L_{r_0+1}}$. Take a set-theoretical lift $b'_{r_0+1} \colon G^{i-1} \to M/L_{r_0+1}$, and denote by $\partial b'_{r_0+1} \in \t{B}^i(G,(M,\rho)/L_{r_0+1})$ the $i$-coboundary associated to $b'_{r_0+1}$. We set $\tilde{c}'_{L_{r_0+1}} \coloneqq \tilde{c}_{L_{r_0+1}} + \partial b'_{r_0+1} \in \t{Z}^i(G,(M,\rho)/L_{r_0+1})$. Then the image of $\tilde{c}'_{L_{r_0+1}}$ in $\t{Z}^i(G,(M,\rho)/L_{r_0})$ coincides with $\tilde{c}'_{L_{r_0}}$, and hence $(\tilde{c}'_{L_r})_{r = 0}^{r_0+1}$ is a compatible system of representatives of $(c_{L_r})_{r = 0}^{r_0+1}$. By induction on $r_0$, we obtain an inverse system $(\tilde{c}'_{L_r})_{r = 0}^{\infty}$ of representatives of $(c_{L_r})_{r = 0}^{\infty}$. Since $(M,\rho)$ is linearly complete and $\set{L_r}{r \in \N}$ is cofinal in $\Opn_{(M,\rho)}$, the $R$-linear $G$-equivariant homomorphism
\begin{eqnarray*}
  \iota \colon (M,\rho) & \to & \varprojlim_{r \in \N} (M,\rho)/L_r \\
  m & \mapsto & (m + L_r)_{r = 0}^{\infty}
\end{eqnarray*}
is a homeomorphic isomorphism. Let $\tilde{c} \colon G^i \to M$ denote the composite of the set-theoretical map
\begin{eqnarray*}
  G^i & \to & \varprojlim_{r \in \N} (M,\rho)/L_r \\
  (g_j)_{j = 1}^{i} & \mapsto & \left( \tilde{c}'_{L_r}((g_j)_{j = 1}^{i}) \right)
\end{eqnarray*}
and $\iota^{-1}$. Since $M$ is Hausdorff, we have $\bigcap_{r = 0}^{\infty} L_r = \ens{0}$, and hence the cocycle conditions for $\tilde{c'}_{L_r}$ for each $r \in \N$ ensures the cocycle condition for $\tilde{c}$. Forgetting the topology of the target $M$ of $\tilde{c}$, we regard $\tilde{c}$ as an element of $\t{Z}^i(G,\v{(M,\rho)})$. By the construction of $\tilde{c}$, the image of its cohomology class coincides with $c$.
\end{proof}

\begin{lmm}
\label{prodiscrete cohomology 2}
Suppose that the underlying monoid of $G$ is finitely generated. Let $(M,\rho)$ be a first countable profinite $R[G]$-module. Then the natural $\v{R}$-linear homomorphism
\begin{eqnarray*}
  \t{H}^i(G,(M,\rho)) \to \Hil^i(G,(M,\rho))
\end{eqnarray*}
is an isomorphism for any $i \in \N$.
\end{lmm}

\begin{proof}
By Lemma \ref{prodiscrete cohomology}, it suffices to verify the injectivity of the given homomorphism. Let $c \in \t{H}^i(G,(M,\rho))$ be an element of the kernel of the given homomorphism. If $i = 0$, then $c \in \t{H}^i(G,(M,\rho)) \cong \v{(M,\rho)}^G \subset \v{M}$, and since the image of $c$ in $\t{H}^i(G,(M,\rho)/L) \cong \Gamma(G,(M,\rho)/L) \subset M/L$ is $0$ for any $L \in \Opn_{(M,\rho)}$, we have $c \in \bigcap_{L \in \Opn_{(M,\rho)}} L = \ens{0}$. Suppose $i > 0$. Take a representative $\tilde{c} \in \t{Z}^i(G,\v{(M,\rho)})$. Put
\begin{eqnarray*}
  Z & \coloneqq & \varprojlim_{L \in \Opn_{(M,\rho)}} \t{Z}^i(G,(M,\rho)/L) \\
  B & \coloneqq & \varprojlim_{L \in \Opn_{(M,\rho)}} \t{B}^i(G,(M,\rho)/L).
\end{eqnarray*}
Since the image of $c$ in $\Hil^i(G,(M,\rho))$ is $0$, the image of $\tilde{c}$ in $Z$ lies in the image of $B$. Let $S \subset G$ be a finite set of generators. The evaluation map
\begin{eqnarray*}
  \t{Z}^i(G,\v{(M,\rho)}) & \to & M^{S^i} \\
  c' & \mapsto & (c'(s_1, \ldots, s_i))_{(s_1, \ldots, s_i) \in S^i}
\end{eqnarray*}
is injective by the cocycle condition. We endow $M^{S^i}$ and $M^{G^{i-1}}$ with the direct product topology. They are compact and Hausdorff by Tychonoff's theorem, because $M$ is profinite. Since the cocycle condition is given by equalities, the continuity of $\rho$ and the addition $M \times M \to M$ ensures that the image of $\t{Z}^i(G,\v{(M,\rho)})$ is closed in $M^{S^i}$, and hence $\t{Z}^i(G,\v{(M,\rho)})$ is compact and Hausdorff with respect to the relative topology. The continuity of $\rho$, the addition $M \times M \to M$, and the additive inverse $M \times M \colon m \mapsto -m$ ensures that the map $\partial \colon M^{G^{i-1}} \to \t{Z}^i(G,\v{(M,\rho)})$ associating coboundaries is continuous, and hence its image $\t{B}^i(G,\v{(M,\rho)})$ is closed. Since $M$ is profinite, the $R$-linear homomorphism
\begin{eqnarray*}
  M^{S^i} & \to & N \coloneqq \varprojlim_{L \in \Opn_{(M,\rho)}} (M/L)^{S^i} \cong \left( \varprojlim_{L \in \Opn_{(M,\rho)}} (M/L) \right)^{S^i} \\
  (m_s)_{s \in S^i} & \mapsto & \left( (m_s + L)_{s \in S^i} \right)_{L \in \Opn_{(M,\rho)}}
\end{eqnarray*}
is a homeomorphic isomorphism. By the definition of an $i$-coboundary, the image of $\t{B}^i(G,\v{(M,\rho)})$ in $\t{Z}^i(G,(M,\rho)/L)$ coincides with $\t{B}^i(G,(M,\rho)/L)$ for any $L \in \Opn_{(M,\rho)}$. Regarding $\t{Z}^i(G,(M,\rho)/L)$ as a $R$-submodule of the finite $R$-module $(M/L)^{S^i}$ by a similar evaluation map for each $L \in \Opn_{(M,\rho)}$, we identify $B$ as a closed $R$-submodule of $N$. By the definition of the inverse limit topology, the image of $\t{B}^i(G,\v{(M,\rho)})$ is dense in $B$. Since $\t{B}^i(G,\v{(M,\rho)})$ is compact and $N$ is Hausdorff, the image of $\t{B}^i(G,\v{(M,\rho)})$ in $N$ is closed, and hence coincides with $B$. It implies that $\tilde{c}$ belongs to $\t{B}^i(G,\v{(M,\rho)})$, because the image of $\tilde{c}$ in $N$ lies in $B$. Thus $c = 0$.
\end{proof}

We remark that a group is finitely generated if and only if its underlying monoid is finitely generated. Indeed, for a set $S$ of generators of a group, the underlying monoid of the group is generated by $S \cup S^{-1} \coloneqq \set{g^{\sigma}}{(g,\sigma) \in S \times \ens{-1,1}}$. Therefore Lemma \ref{prodiscrete cohomology 2} is valid also for a finitely generated group. Through the isomorphism in Lemma \ref{prodiscrete cohomology 2}, we equip the source with the pull-back of the topology of the target instead of the discrete topology. The induced topology coincides with the quotient topology of the space of cocycles defined in the proof of Lemma \ref{prodiscrete cohomology 2}. Since $R$ acts continuously on the target, we regard the source as a profinite $R$-module.

\begin{lmm}
\label{cohomological dimension}
Suppose that the underlying monoid of $G$ is a finitely generated free group. Let $(M,\rho)$ be a first countable profinite $R[G]$-module. Then the equality $\Hil^i(G,(M,\rho)) = 0$ holds for any $i \in \N \cap [2,\infty)$.
\end{lmm}

\begin{proof}
By Lemma \ref{prodiscrete cohomology}, it suffices to verify $\t{H}^i(G,(M,\rho)) = 0$ for any $i \in \N \cap [2,\infty)$. By the definition of the group cohomology, it suffices to verify the equality in the case where $R$ is discrete. Since the underlying monoid of $G$ is a finitely generated free group, it is isomorphic to the fundamental group of a based connected $1$-dimensional finite CW-complex $C$ of the form $\t{S}^1 \vee \cdots \vee \t{S}^1$. Therefore there is an equivalence between the category of discrete $R[G]$-modules and $R$-linear $G$-equivariant homomorphisms and the category of sheaves of $R$-modules on $C$ and morphisms of sheaves of $R$-modules. Let $\Fil$ denote the sheaf of $R$-modules on $C$ corresponding to the discrete $R[G]$-module $\v{(M,\rho)}$ by the equivalence. We have a natural $R$-linear isomorphism $\t{H}^*(G,(M,\rho)) \cong \t{H}^*(C,\Fil)$. Since $C$ is a finite CW-complex, it is paracompact, and hence there is a natural $R$-linear isomorphism $\t{H}^*(C,\Fil) \cong \check{\t{H}}{}^*(C,\Fil)$. Since $C$ is $1$-dimensional CW-complex, it is of $\check{\m{C}}$ech-dimension $1$. Therefore we obtain $\check{\t{H}}{}^i(C, \Fil) = 0$ for any $i \in \N \cap [2,\infty)$. Thus the assertion holds.
\end{proof}

A complex of topological $R$-modules (resp.\ topological $R[G]$-modules) with continuous $R$-linear homomorphisms (resp.\ continuous $R$-linear $G$-equivariant homomorphisms) is said to be an {\it exact sequence} if its underlying complex of left $\v{R}$-modules with $\v{R}$-linear homomorphisms is exact.

\begin{prp}
\label{right exact 2}
Suppose that the underlying monoid of $G$ is a finitely generated free group. Let $(M_1,\rho_1)$, $(M_2,\rho_2)$, and $(M_3,\rho_3)$ be first countable profinite $R[G]$-modules with an exact sequence
\begin{eqnarray*}
  0 \to (M_1,\rho_1) \to (M_2,\rho_2) \to (M_3,\rho_3) \to 0
\end{eqnarray*}
of continuous $R$-linear $G$-equivariant homomorphisms. Then it induces an exact sequence
\begin{eqnarray*}
  0 & \to & \Hil^0(G,(M_1,\rho_1)) \to \Hil^0(G,(M_2,\rho_2)) \to \Hil^0(G,(M_3,\rho_3)) \\
  & \to & \Hil^1(G,(M_1,\rho_1)) \to \Hil^1(G,(M_2,\rho_2)) \to \Hil^1(G,(M_3,\rho_3)) \\
  & \to & 0
\end{eqnarray*}
of linearly complete $R$-modules with continuous $R$-linear homomorphisms.
\end{prp}

\begin{proof}
We have an exact sequence
\begin{eqnarray*}
  0 & \to & \t{H}^0(G,(M_1,\rho_1)) \to \t{H}^0(G,(M_2,\rho_2)) \to \t{H}^0(G,(M_3,\rho_3)) \\
  & \to & \t{H}^1(G,(M_1,\rho_1)) \to \t{H}^1(G,(M_2,\rho_2)) \to \t{H}^1(G,(M_3,\rho_3)) \\
  & \to & 0
\end{eqnarray*}
by the cohomology long exact sequence and Lemma \ref{cohomological dimension}. Therefore the assertion follows from Lemma \ref{prodiscrete cohomology 2}. 
\end{proof}

\begin{lmm}
\label{cocycle}
Suppose that $R$ is discrete and the underlying monoid of $G$ is a free group with a basis $E$. For a discrete $R[G]$-module $(M,\rho)$, the evaluation map
\begin{eqnarray*}
  \t{Z}^1(G,(M,\rho)) & \to & M^E \\
  c & \mapsto & (c(e))_{e \in E}
\end{eqnarray*}
is an $R$-linear isomorphism.
\end{lmm}

\begin{proof}
Put
\begin{eqnarray*}
  H & \coloneqq & \Set{\varphi \in \t{Aut}(M \times \Z)}{
  \begin{array}{ll}
    \varphi(m,0) \in M \times \Ens{0}, {}^{\forall} m \in M \\
    \varphi(0,1) \in M \times \ens{1}
  \end{array}
  } \\
  \t{Hom}(G,H)_{\rho} & \coloneqq & \Set{\chi \in \t{Hom}(G,H)}{\chi(g)(m,0) = (\rho(g,m),0)}.
\end{eqnarray*}
For each $\varphi \in H$, we denote by $c_{\varphi} \in M$ the element with $\varphi(0,1) = (c_{\varphi},1)$. The map
\begin{eqnarray*}
  \iota_1 \colon \t{Z}^1(G,(M,\rho)) & \to & \t{Hom}(G,H)_{\rho} \\
  c & \mapsto & \left( g \mapsto \left( (m,n) \mapsto \left( \rho(g,m) + n c(g) , n \right) \right) \right)
\end{eqnarray*}
is bijective because it admits an inverse
\begin{eqnarray*}
  \t{Hom}(G,H)_{\rho} & \to & \t{Z}^1(G,(M,\rho)) \\
  \chi & \mapsto & \left( g \mapsto c_{\chi(g)} \right).
\end{eqnarray*}
The map
\begin{eqnarray*}
  \iota_2 \colon \t{Hom}(G,H)_{\rho} & \to & H^E \\
  \chi & \mapsto & (c_{\chi(e)})_{e \in E}
\end{eqnarray*}
is bijective by the universality of a free group. The map
\begin{eqnarray*}
  \iota_3 \colon H^E & \to & M^E \\
  (\varphi_e)_{e \in E} & \mapsto & (c_{\varphi_e})_{e \in E}
\end{eqnarray*}
is bijective because it admits an inverse
\begin{eqnarray*}
  M^E & \to & H^E \\
  (c_e)_{e \in E} & \mapsto & \left( (m,n) \mapsto (\rho(e,m) + n c_e, n) \right)_{e \in E}.
\end{eqnarray*}
The composite $\iota_3 \circ \iota_2 \circ \iota_1$ of bijective maps coincides with the evaluation map in the assertion.
\end{proof}

For an Abelian category $\Cat$, we denote by $\Cat^{\N}$ the Abelian category of inverse systems of objects of $\Cat$ indexed by $\N$ and compatible systems of morphisms.

\begin{lmm}
\label{derived limit}
Suppose that $R$ is discrete. For any inverse system $((M_r)_{r = 0}^{\infty},\varphi_{\bullet})$ of discrete $R[G]$-modules, the equality $\t{R}^i \varprojlim ((M_r)_{r = 0}^{\infty},\varphi_{\bullet}) = 0$ holds for any $i \in \N \cap [2,\infty)$, where $\varprojlim$ is regarded as a left exact functor $\t{Mod}(R[G])^{\N} \to \t{Mod}(R[G])$.
\end{lmm}

\begin{proof}
Let $((M_r)_{r = 0}^{\infty},\varphi_{\bullet})$ be an inverse system of discrete $R[G]$-modules. We denote by $\varpi_{r_0} \colon \prod_{r = 0}^{r_0+1} M_r \twoheadrightarrow \prod_{r = 0}^{r_0} M_r$ the canonical projection for each $r_0 \in \N$, and by $\varpi_{-1}$ the zero homomorphism $M_0 \twoheadrightarrow 0$. We define an $R$-linear $G$-equivariant homomorphism $\psi_{r_0} \colon \prod_{r = 0}^{r_0} M_r \to \prod_{r = 0}^{r_0-1} M_r$ by setting $\psi_{r_0} \left( (m_r)_{r = 0}^{r_0} \right) \coloneqq \left( m_r - \varphi_r(m_{r+1}) \right)_{r = 0}^{r_0-1}$ for each $r_0 \in \N$ and $(m_r)_{r = 0}^{r_0} \in \prod_{r = 0}^{r_0} M_r$. Then the system $\psi_{\bullet} = (\psi_r)_{r = 0}^{\infty}$ is a morphism
\begin{eqnarray*}
  \left( \left( \prod_{r = 0}^{r_0} M_r \right)_{r_0 = 0}^{\infty}, \varpi_{\bullet} \right) \to \left( \left( \prod_{r = 0}^{r_0-1} M_r \right)_{r_0 = 0}^{\infty}, \varpi_{\bullet-1} \right)
\end{eqnarray*}
in $\t{Mod}(R[G])^{\N}$. Indeed, for any $r_0 \in \N$ and $(m_r)_{r = 0}^{r_0+1} \in \prod_{r = 0}^{r_0+1} M_r$, we have
\begin{eqnarray*}
  & & (\varpi_{r_0} \circ \psi_{r_0+1})((m_r)_{r = 0}^{r_0+1}) = \varpi_{r_0}((m_r - \varphi_r(m_{r+1}))_{r = 0}^{r_0}) = (m_r - \varphi_r(m_{r+1}))_{r = 0}^{r_0-1} \\
  & = & \psi_{r_0}((m_r)_{r = 0}^{r_0}) = (\psi_{r_0} \circ \varpi_{r_0})((m_r)_{r = 0}^{r_0+1}).
\end{eqnarray*}
By the definition of the inverse limit, we obtain an exact sequence
\begin{eqnarray*}
  0 \to \varprojlim \left( (M_r)_{r = 0}^{\infty}, \varphi_{\bullet} \right) \to \prod_{r = 0}^{\infty} M_r \xrightarrow[]{\varprojlim \psi_{\bullet}} \prod_{r = 0}^{\infty} M_r
\end{eqnarray*}
of discrete $R[G]$-modules through natural $R$-linear $G$-equivariant isomorphisms
\begin{eqnarray*}
  \prod_{r = 0}^{\infty} M_r & \cong & \varprojlim \left( \left( \prod_{r = 0}^{r_0} M_i \right)_{r_0 = 0}^{\infty}, \varpi_{\bullet} \right) \\
  \prod_{r = 0}^{\infty} M_r & \cong & \varprojlim \left( \left( \prod_{r = 0}^{r_0-1} M_i \right)_{r_0 = 0}^{\infty}, \varpi_{\bullet} \right).
\end{eqnarray*}
We denote by $\varprojlim^1 ((M_r)_{r = 0}^{\infty},\varphi_{\bullet})$ the cokernel of the right arrow of the exact sequence. This construction gives a functor $\varprojlim^1 \colon \t{Mod}(R[G])^{\N} \to \t{Mod}(R[G])$. The system
\begin{eqnarray*}
  \varprojlim{}^{\bullet} = \left( \varprojlim{}^i \right)_{i = 0}^{\infty} \coloneqq \left( \varprojlim, \varprojlim{}^1, 0, \ldots \right)
\end{eqnarray*}
is a cohomological functor with respect to a natural connecting homomorphism. We verify that $\varprojlim^{\bullet}$ is a right derived functor of $\varprojlim$.

\vspace{0.2in}
We identify $\t{Mod}(\Z)$ with $\t{Mod}(\Z[\ens{1}])$, and also consider $\varprojlim^1 \colon \t{Mod}(\Z)^{\N} \to \t{Mod}(\Z)$. Let $F$ and $F^{\N}$ denote the forgetful functors $\t{Mod}(R[G]) \to \t{Mod}(\Z)$ and $\t{Mod}(R[G])^{\N} \to \t{Mod}(\Z)^{\N}$ respectively. By the exactness of $F$ and by the definitions of $\varprojlim^{\bullet}$, we have natural equivalences $F \circ \varprojlim^i \cong \varprojlim^i \circ F^{\N}$  for each $i \in \N$. In the case where $R = \Z$ and $G = \ens{1}$, then it is well-known that $\varprojlim^{\bullet}$ is a right derived functor of $\varprojlim$. In order to verify that $\varprojlim^{\bullet}$ is a right derived functor of $\varprojlim$ in a general case, it suffices to verify that it is a universal effacable functor. Let $I$ be an injective object of $\t{Mod}(R[G])^{\N}$. By \cite{Jan88} 1.1 Proposition b), $I$ is isomorphic to $((\prod_{r = 0}^{r_0} M_r)_{r_0 = 0}^{\infty}, \varpi_{\bullet})$ for some inverse system $(M_r)_{r = 0}^{\infty}$ of injective objects in $\t{Mod}(R[G])$ whose transition maps are $0$. Therefore we have
\begin{eqnarray*}
  F \left( \varprojlim{}^1 I \right) \cong \varprojlim{}^1(F(I)) \cong \t{R}^1 \varprojlim (F^{\N}(I)) \cong \t{R}^1 \varprojlim \left( \left( \prod_{r = 0}^{r_0} F(M_r) \right)_{r_0 = 0}^{\infty}, \varpi_{\bullet} \right) = 0
\end{eqnarray*}
because $((\prod_{r = 0}^{r_0} F(M_r))_{r_0 = 0}^{\infty}, \varpi_{\bullet})$ is an inverse system of Abelian groups satisfying the Mittag--Leffler condition. We obtain $\varprojlim^i I = 0$ for any $i \in \N \backslash \ens{0}$. Thus $\varprojlim^{\bullet}$ is a universal effacable functor, and hence is a right derived functor of $\varprojlim$. We conclude that $\t{R}^i \varprojlim = 0$ for any $i \in \N \cap [2,\infty)$, and the assertion holds.
\end{proof}

\begin{thm}
\label{prodiscrete cohomology 3}
Let $R$ be a commutative topological ring, $G$ a finitely generated free group endowed with the discrete topology, $(M,\rho)$ a first countable profinite $R[G]$-module, and $(L_i)_{i = 0}^{\infty}$ a countable decreasing sequence of open $R[G]$-submodules of $(M,\rho)$ such that $\set{L_r}{r \in \N}$ forms a fundamental system of neighbourhoods of $0$. Then there exists a natural $\v{R}$-linear isomorphism
\begin{eqnarray*}
  \Hil^*(G,(M,\rho)) \cong \t{R}^* \left( \t{H}^0(G, \cdot) \circ \varprojlim \right) \left( \left( \vv{(M,\rho)/L_r} \right)_{r = 0}^{\infty} \right),
\end{eqnarray*}
where $\varprojlim$ is regarded as the left exact functor $\t{Mod}(\v{R}[G])^{\N} \to \t{Mod}(\v{R}[G])$.
\end{thm}

\begin{proof}
Let $i \in \N$. We construct an $R$-linear isomorphism
\begin{eqnarray*}
  \Hil^i(G,(M,\rho)) \cong \t{R}^i \left( \t{H}^0(G, \cdot) \circ \varprojlim \right) \left( \left( \vv{(M,\rho)/L_r} \right)_{r = 0}^{\infty} \right).
\end{eqnarray*}
When $i = 0$, then the assertion follows from the linear completeness of $(M,\rho)$. Suppose $i > 0$. By \cite{Jan88} 1.1 Proposition b), $\varprojlim$ sends injective objects to a direct product of injective objects, which is acyclic with respect to $\Gamma(G,\cdot)$. Therefore we have a spectral sequence
\begin{eqnarray*}
  \t{E}_2^{s,t} \coloneqq \t{H}^s \left( G, \t{R}^t \varprojlim \left( \vv{(M,\rho)/L_r} \right)_{r = 0}^{\infty} \right) \Longrightarrow \t{R}^{s+t} \left( \t{H}^0(G, \cdot) \circ  \varprojlim \right) \left( \left( \vv{(M,\rho)/L_r} \right)_{r = 0}^{\infty} \right).
\end{eqnarray*}
Since $(\v{(M,\rho)/L_r})_{r = 0}^{\infty}$ is a surjective system, it satisfies the Mittag--Leffler condition, and hence 
\begin{eqnarray*}
  \t{R}^1 \varprojlim \left( \vv{(M,\rho)/L_r} \right)_{r = 0}^{\infty} = 0
\end{eqnarray*}
by a similar argument with that in the proof of Lemma \ref{derived limit} with the forgetful functor $\t{Mod}(\v{R}[G]) \to \t{Mod}(\Z)$. Together with Lemma \ref{cohomological dimension} and Lemma \ref{derived limit}, we obtain $\t{E}_2^{s,t} = 0$ for any $(s,t) \in \N \times \N$ with $s \geq 2$ or $t \geq 1$. In particular, when $i \geq 2$, then we have
\begin{eqnarray*}
  \t{R}^i \left( \t{H}^0(G, \cdot) \circ \varprojlim \right) \left( \left( \vv{(M,\rho)/L_r} \right)_{r = 0}^{\infty} \right) = 0
\end{eqnarray*}
and hence we obtain an $\v{R}$-linear isomorphism by Lemma \ref{cohomological dimension}. Suppose $i = 1$. We have
\begin{eqnarray*}
  & & \t{R}^i \left( \t{H}^0(G, \cdot) \circ \varprojlim \right) \left( \left( \vv{(M,\rho)/L_r} \right)_{r = 0}^{\infty} \right) \cong E_2^{1,0} = \t{H}^1 \left( G, \varprojlim \left( \vv{(M,\rho)/I_r} \right) \right) \\
  & \cong & \t{H}^1(G, \vv{(M,\rho)}) = \t{H}^1(G, (M,\rho)) \cong \Hil^1(G, (M,\rho))
\end{eqnarray*}
by the linear completeness of $(M,\rho)$ and Lemma \ref{prodiscrete cohomology 2}.
\end{proof}

Thus $\Hil^*(G,\cdot)$ is the cohomology of a derived functor together with a topology and a continuous action of $R$.

\subsection{Profinite $\Z_p$-Sheaves on Modular Curves}
\label{Profinite Zp Sheaves on Modular Curves}

In this subsection, we introduce the notion of a profinite $R$-sheaf on a modular curve in order to construct a profinite Galois representation endowed with a compatible action of Hecke operators. Let $R$ be a commutative topological ring $R$ and $G$ a topological monoid.

\begin{dfn}
\label{etale cohomology}
Let $S$ be a Noetherian scheme. A {\it profinite $R$-sheaf on $S$} is an inverse system of sheaves on $S_{\t{\'et}}$ of finite $R$-modules. For a profinite $R$-sheaf $\Fil = (\Fil_{\lambda})_{\lambda \in \Lambda}$ on $S$, we set
\begin{eqnarray*}
  \Hil_{\t{et}}^*(S,\Fil) \coloneqq \varprojlim_{\lambda \in \Lambda} \t{H}_{\t{et}}^* \left( S, \Fil_{\lambda} \right),
\end{eqnarray*}
and endow it with the inverse limit topology of the discrete topologies. We call it {\it the prodiscrete cohomology of $\Fil$}.
\end{dfn}

\begin{rmk}
\label{derived limit 2}
Let $S$ be a proper algebraic variety over a separably closed field, and $\Fil$ a smooth $\Z_p$-sheaf on $S_{\t{\'et}}$. Then $\Fil$ is represented by a profinite $\Z_p$-sheaf $(\Fil_r)_{r \in \N}$ such that $\Fil_r$ is a finite Abelian sheaf of $(\Z/p^r \Z)$-modules over $S_{\t{\'et}}$ and the transition morphism induces an isomorphism $(\Z/p^r \Z) \otimes_{\Z/p^{r+1} \Z} \Fil_{r+1} \cong \Fil_r$ by definition. By the finiteness of the \'etale cohomology in \cite{Mil80} Corollary VI.2.8 and \cite{Mil13} Remark 17.9, there is a natural $\Z_p$-linear isomorphism
\begin{eqnarray*}
  \Hil_{\t{et}}^*(S,(\Fil_r)_{r \in \N}) \cong \t{H}_{\t{\'et}}^*(S,\Fil),
\end{eqnarray*}
and hence the prodiscrete cohomology can be computed as the ordinary \'etale cohomology. The same holds for the case where $S$ is not proper but the cohomology of $\Fil_r$ is a finite group for each $r \in \N$ by the vanishing of $\t{R}^1 \varprojlim$.
\end{rmk}

We give an explicit construction of a profinite $R$-sheaf with no use of a fundamental group. Suppose that $G$ is a discrete finite group. Let $Y_1$ be a Noetherian scheme with a $G$-torsor $Y \twoheadrightarrow Y_1$, where $G$ acts on $Y$ from the right. Let $(M,\rho)$ be a finite $R[G]$-module $M$. For a scheme $X$ and a set $I$, we denote by $X \times I$ the disjoint union of copies of $X$ indexed by $I$. We consider the right action of $G$ on $M$ given by setting $mg := \rho(g^{-1},m)$ for each $(m,g) \in M \times G$. We endow $Y \times M$ with the right diagonal action of $G$ over $Y_1$. We consider the right action of $G \times G$ on $Y$ given by the first projection, on $G$ given by setting $g(g_1,g_2) := g_1^{-1}gg_2$ for each $(g,(g_1,g_2)) \in G \times (G \times G)$, and on $G \times M \times M$ given by setting $(g,m_1,m_2)(g_1,g_2) := (g_1^{-1}gg_2,\rho(g_1^{-1},m_1),\rho(g_2^{-1},m))$ for each $((g,m_1,m_2),(g_1,g_2)) \in (G \times M \times M) \times (G \times G)$. We endow $Y \times G$ and $Y \times (G \times M \times M)$ with the right diagonal action of $G \times G$ over $Y_1$. Since $Y \twoheadrightarrow Y_1$ is a $G$-torsor, we have $(G \times G)$-equivariant isomorphisms
\begin{eqnarray*}
  Y \times_{Y_1} Y & \cong & Y \times G \\
  (Y \times M) \times_{Y_1} (Y \times M) & \cong & Y \times (G \times M \times M)
\end{eqnarray*}
over $Y_1$. We obtain
\begin{eqnarray*}
  (G \backslash (Y \times M)) \times_{Y_1} (G \backslash (Y \times M)) & \cong & (G \times G) \backslash ((Y \times M) \times_{Y_1} (Y \times M)) \\
  & \cong & (G \times G) \backslash (Y \times (G \times M \times M))
\end{eqnarray*}
The map $G \times M \times M \to M \colon (g,m_1,m_2) \mapsto m_1 + \rho(g,m_2)$ defines a morphism $(Y \times M) \times_{Y_1} (Y \times M) \cong Y \times (G \times M \times M) \to Y \times M$, and induces an addition on $G \backslash (Y \times M)$ over $Y_1$. The map $M \times R \to M \mapsto (m,r) \mapsto rm$ defines a $G$-equivariant morphism $Y \times (M \times R) \to Y \times M$, and induces a scalar multiplication $(G \backslash (Y \times M)) \times R \to G \backslash (Y \times M)$ compatible with the addition. We put $(\underline{(M,\rho)})_{Y_1} := G \backslash (Y \times M)$. Then $(\underline{(M,\rho)})_{Y_1}$ is finite \'etale over $Y_1$, and we regard it as a locally constant \'etale sheaf of finite $R$-modules on $Y_1$. In particular, when $G$ acts trivially on $M$, then $(\underline{M})_{Y_1} := (\underline{(M,\rho)})_{Y_1}$ is a constant sheaf independent of $(G,\rho)$ and $Y \twoheadrightarrow Y_1$.

\vspace{0.2in}
Suppose that $G$ is a profinite group. Let $\Nor$ be a coinitial subset of the set of open normal subgroups of $G$ containing $G \in \Nor$. Let $(Y_H)_{H \in \Nor}$ be a projective system of Noetherian schemes satisfying the following axiom:
\begin{itemize}
\item[(i)] For each $H \in \Nor$, a right action of $G/H$ on $Y_H$ is given.
\item[(ii)] Every transitive morphism is $G$-equivariant.
\item[(iii)] For any $H \in \Nor$, $Y_H$ forms a $(G/H)$-torsor over $Y_1 := Y_G$.
\end{itemize}
Let $(M,\rho)$ be a profinite $R[G]$-module $M$. We denote by $O_{(M,\rho)}$ the set of open $R[G]$-submodules of $(M,\rho)$. For each $L \in O_{(M,\rho)}$, we put $H_L := \set{g \in G}{\rho(g,m) - m \in L, {}^{\forall} m \in M} \in \Nor$. For each $L \in O_{(M,\rho)}$ and $H \in \Nor$ with $H \subset H_L$, we have a functorial construction of a locally constant \'etale sheaf $(\underline{(M,\rho)/L})_{Y_1}$ of finite $R$-modules over $Y_1$ given by the $(G/H)$-torsor $Y_H \twoheadrightarrow Y_1$, which is independent of the choice of $H$ because $Y_H \twoheadrightarrow Y_{H_L}$ is an $(H_L/H)$-torsor. We obtain a well-defined profinite $R$-sheaf $(\underline{(M,\rho)})_{Y_1} := ((\underline{(M,\rho)/L})_{Y_1})_{L \in O_{(M,\rho)}}$ on $Y_1$.

\vspace{0.2in}
We give an example of a tower $(Y_H)_{H \in \Nor}$ satisfying the axiom above. Henceforth, we fix an algebraic closure $\overline{\Q}$ of $\Q$ and an embedding $\iota_{0,\infty} \colon \overline{\Q} \hookrightarrow \C$. For each $N \in \N$, we put
\begin{eqnarray*}
  & & \Gamma_1(N) \coloneqq
  \left(
    \begin{array}{cc}
      1 + N \Z & \Z \\
      N \Z & 1 + N \Z
    \end{array}
  \right)
  \cap \t{SL}_2(\Z) \\
  & & \Gamma(N) \coloneqq
  \left(
    \begin{array}{cc}
      1 + N \Z & N \Z \\
      N \Z & 1 + N \Z
    \end{array}
  \right)
  \cap \t{SL}_2(\Z) = \ker \left( \Gamma_1(N) \twoheadrightarrow \t{SL}_2(\Z/N \Z) \right).
\end{eqnarray*}
For each $N \in \N$ with $N \geq 5$ (resp.\ $N \geq 3$), we denote by $Y_1(N)$ (resp.\ $Y(N)$) the modular curve of level $\Gamma_1(N) \subset \t{SL}_2(\Z)$ (resp.\ $\Gamma(N) \subset \t{SL}_2(\Z)$), i.e.\ a moduli of pairs $(E,\beta)$ of an elliptic curve $E$ and a projection $\beta \colon E[N] \twoheadrightarrow (\underline{\Z/N \Z})$ between group schemes (resp.\ a moduli space of pairs $(E,(\alpha_1,\alpha_2))$ of an elliptic curve $E$ and a $(\Z/N \Z)$-linear basis $(\alpha_1,\alpha_2)$ of $E[N]$). For each $N \in \N \cap [5,\infty)$, we remark that for an elliptic curve $E$, giving a projection $\beta \colon E[N] \twoheadrightarrow (\underline{\Z/N \Z})$ between group schemes is equivalent to giving a closed immersion $\iota \colon \G_{\t{m}}[N] \hookrightarrow E[N]$ between group schemes through the Weil pairing $\langle \cdot, \cdot \rangle_N \colon E[N] \times E[N] \to \G_{\t{m}}[N]$ (\cite{Sil86} III 8), and emphasis that this formulation of $Y_1(N)$ is the one dealt with in \cite{Gro90} Proposition 2.1 and \cite{KM85} 4.8, and is distinct from the usual one in \cite{KM85} 3.2 unless the base field is extended to $\Q[X_N]/(P_N(X_N))$, where $P_N(X_N) \in \Q[X_N]$ is the $N$-th cyclotomic polynomial. For each $N \in \N \cap [3,\infty)$, we have a right action
\begin{eqnarray*}
  (E,(\alpha_1,\alpha_2))
  \left(
    \begin{array}{cc}
      a & b \\
      c & d
    \end{array}
  \right)^{\t{op}}
  \coloneqq \left( E, \left( a \alpha_1 + c \alpha_2, b \alpha_1 + d \alpha_2 \right) \right)
\end{eqnarray*}
of $\t{GL}_2(\Z/N \Z)$ on $Y(N)$. We recall that $Y_1(N)$ (resp.\ $Y(N)$) is an algebraic curve defined over $\t{Spec}(\Q)$ (resp.\ $\t{Spec}(\Q[X_N]/(P_N(X_N)))$), and the analytification of the base change of $Y_1(N)$ (resp.\ $Y(N)$) by $\t{Spec}(\C) \to \t{Spec}(\Q)$ (resp.\ $\t{Spec}(\C) \to \t{Spec}(\Q[X_N]/(P_N(X_N)))$ given by $\Q[X_N]/(P_N(X_N)) \hookrightarrow \C \colon X_N \mapsto \exp(2N^{-1} \pi \ \sqrt[]{\mathstrut -1})$) is biholomorphic to the quotient $\Gamma_1(N) \backslash \Hlf$ (resp.\ the quotient $\Gamma(N) \backslash \Hlf$) of $\Hlf$ with respect to the action of $\Gamma_1(N)$ (resp.\ $\Gamma(N)$) given in Example \ref{linear fractional transformation}. The moduli interpretations ensure that there is a natural finite surjective \'etale morphism
\begin{eqnarray*}
  Y(N) & \to & Y_1(N) \\
  (E,(\alpha_1,\alpha_2)) & \mapsto & \left( E, \beta_{\alpha_1, \alpha_2} \colon a_1 \alpha_1 + a_2 \alpha_2 \mapsto a_2 \right)
\end{eqnarray*}
for each $N \in \N \cap [5,\infty)$, and that there are natural finite surjective \'etale morphisms
\begin{eqnarray*}
\begin{array}{crcl}
  & Y_1(nN) & \to & Y_1(N) \\
  & (E,\beta) & \mapsto & \left( E, \beta + N \Z \colon \alpha \mapsto \beta(\alpha) + N \Z \right) \\
  \t{resp.} & Y(nN) & \to & Y(N) \times_{\Q[X_N]/(P_N(X_N))} \Q[X_{nN}]/(P_{nN}(X_{nN})) \\
  & (E,\alpha_1,\alpha_2) & \mapsto & \left( E, n \alpha_1, n \alpha_2 \right)
\end{array}
\end{eqnarray*}
for each $n,N \in \N$ with $n \geq 1$ and $N \geq 5$ (resp.\ and $N \geq 3$) for which $(Y_1(N))_{N \in \N \cap [5,\infty)}$ (resp.\ $(Y(N))_{N \in \N \cap [3,\infty)}$) is a compatible system. The natural morphisms $Y(N) \to Y_1(N)$ for each $N \in \N \cap [5,\infty)$ gives a morphism $(Y(N))_{N \in \N \cap [5,\infty)} \to (Y_1(N))_{N \in \N \cap [5,\infty)}$ of compatible systems.

\vspace{0.2in}
Henceforth, we fix an $N \in \N \cap [5,\infty)$. We put
\begin{eqnarray*}
  & & \t{G} \hat{\Gamma}_e(N) \coloneqq
  \left(
    \begin{array}{cc}
      \hat{\Z} & \hat{\Z} \\
      N \hat{\Z} & 1 + N \hat{\Z}
    \end{array}
  \right)
  \cap \t{GL}_2 (\hat{\Z}) \\
  & & \hat{\Gamma}_1(N) \coloneqq
  \left(
    \begin{array}{cc}
      1 + N \hat{\Z} & \hat{\Z} \\
      N \hat{\Z} & 1 + N \hat{\Z}
    \end{array}
  \right)
  \cap \t{SL}_2 (\hat{\Z}).
\end{eqnarray*}
The Galois group of the finite \'etale covering $Y(nN) \to Y(N)$ is naturally isomorphic to the finite group
\begin{eqnarray*}
  & & \ker(\t{GL}_2(\Z/nN \Z) \twoheadrightarrow \t{GL}_2(\Z/N \Z)) \\
  & = &
  \left(
    \begin{array}{cc}
      1 + N (\Z/nN \Z) & N (\Z/nN \Z) \\
      N (\Z/nN \Z) & 1 + N (\Z/nN \Z)
    \end{array}
  \right)
  \cap \t{GL}_2(\Z/nN \Z),
\end{eqnarray*}
and the composite $Y(nN) \to Y(N) \to Y_1(N) \times_{\Q} \Q[X_N]/(P_N(X_N)) \to Y_1(N)$ corresponds to the group
\begin{eqnarray*}
  \t{G} \Gamma_e(N,n) \coloneqq
  \left(
    \begin{array}{cc}
      \Z/nN \Z & \Z/nN \Z \\
      N (\Z/nN \Z) & 1 + N (\Z/nN \Z)
    \end{array}
  \right)
  \cap \t{GL}_2(\Z/nN \Z)
\end{eqnarray*}
through the right action for each $n \in \N \backslash \ens{0}$. We have a natural homeomorphic group isomorphism $\t{G} \hat{\Gamma}_e(N) \to \varprojlim_{n \in \N} \t{G} \hat{\Gamma}_e(N,n!)$. Therefore $\t{G} \hat{\Gamma}_e(N)$ acts from right on the tower $(Y(Nn!))_{n \in \N}$ of torsors over $Y_1(N)$ in a desired way. Similarly, $\hat{\Gamma}_1(N)$ acts from right on the tower $(Y(Nn!)_{\overline{\Q}})_{n \in \N}$ of torsors over $Y_1(N)_{\overline{\Q}}$ in a desired way.

\begin{exm}
\label{symmetric product 2}
Let $p$ be a prime number dividing $N$, and $n \in \N$. Then $\t{Sym}^n(\Z_p^2, \rho_{\Z_p^2})$ (Example \ref{symmetric product}) is a profinite $\Z_p[\t{M}_2(\Z_p)]$-module by Proposition \ref{finitely generated - p-adic}, and in particular, we regard it as a profinite $\Z_p[\t{G} \hat{\Gamma}_e(N)]$-module through the composite $\t{G} \hat{\Gamma}_e(N) \hookrightarrow \t{M}_2(\hat{\Z}) \twoheadrightarrow \t{M}_2(\Z_p)$. It yields a profinite $\Z_p$-sheaf on $Y_1(N)_{\t{\'et}}$ naturally isomorphic to the profinite $\Z_p$-sheaf
\begin{eqnarray*}
  \t{Sym}^n \left( \t{R}^1 (\pi_N)_* (\underline{\Z}{}_p)_{E_1(N)} \right) \coloneqq \left( \t{Sym}^n \left( \t{R}^1 (\pi_N)_* \left( \underline{\Z/p^r \Z} \right)_{E_1(N)} \right) \right)_{r = 0}^{\infty}
\end{eqnarray*}
representing a smooth $p$-adic sheaf. Indeed, let $U$ be a scheme with an \'etale morphism $\iota_U \colon U \to Y_1(N)$. Then $\t{Hom}_{Y_1(N)}(S,Y(Np^r))$ is naturally identified with the set of $(\underline{\Z/Np^r \Z})_S$-linear bases $(\alpha_1,\alpha_2)$ of $\iota_U^* E_1(N)[Np^r]$ such that $(\iota_U^* E_1(N),\beta_{p^r \alpha_1,p^r \alpha_2})$ corresponds to $\iota_U \in \t{Hom}(S,Y_1(N))$ for each $r \in \N$. Putting
\begin{eqnarray*}
  M_r \coloneqq \t{Hom}_{Y_1(N)}(S,Y(Np^r)) \times \t{Sym}^{k-2} \left( (\Z/Np^r \Z)^2, \rho_{(\Z/Np^r \Z)^2} \right),
\end{eqnarray*}
we have a $\t{G} \Gamma_e(N,p^r)$-equivariant map
\begin{eqnarray*}
  M_r & \to & \t{Sym}^{k-2} \left( \t{Hom} \left( \t{H}_{\t{\'et}}^0 \left( S, \iota_U^* E_1(N)[Np^r] \right), \Z/Np^r \Z \right) \right) \\
  \left( (\alpha_1,\alpha_2), \sum_{i = 0}^{k-2} c_i T_1^i T_2^{k-2-i} \right) & \mapsto & \sum_{i = 0}^{k-2} c_i \beta_{\alpha_2,\alpha_1}^{\otimes i} \otimes \beta_{\alpha_1,\alpha_2}^{\otimes k-2-i}
\end{eqnarray*}
for each $r \in \N$. Through the duality between $\Hil_{\t{\'et}}^1(E_1(N),(\underline{\Z}{}_p)_{E_1(N)})$ and the Tate module of $E_1(N)$, it induces a natural identification
\begin{eqnarray*}
  \left( \underline{\t{Sym}^{k-2} \left( \Z_p^2, \rho_{\Z_p^2} \right)} \right) \cong \t{Sym}^n \left( \t{R}^1 (\pi_N)_* (\underline{\Z}{}_p)_{E_1(N)} \right)
\end{eqnarray*}
of profinite $\Z_p$-sheaves on $Y_1(N)_{\t{\'et}}$.
\end{exm}

For a sheaf $\Fil$ of finite Abelian groups on $Y_1(N)_{\t{\'et}}$, we denote by $\Fil_{\overline{\Q}}$ the \'etale sheaf of finite Abelian groups on $Y_1(N)_{\overline{\Q}} \coloneqq Y_1(N) \times_{\Q} \overline{\Q}$ obtained as the inverse image of $\Fil$, and put
\begin{eqnarray*}
  \t{H}_{\t{et}}^* \left( Y_1(N)_{\overline{\Q}}, \Fil \right) \coloneqq \t{H}_{\t{et}}^* \left( Y_1(N)_{\overline{\Q}}, \Fil_{\overline{\Q}} \right).
\end{eqnarray*}
For a profinite $R$-sheaf $\Fil = (\Fil_{\lambda})_{\lambda \in \Lambda}$ on $Y_1(N)$, we denote by $\Fil_{\overline{\Q}}$ the profinite $R$-sheaf on $Y_1(N)_{\overline{\Q}}$ obtained as the inverse system $((\Fil_{\lambda})_{\overline{\Q}})_{\lambda \in \Lambda}$, and we put
\begin{eqnarray*}
  \Hil_{\t{et}}^* \left( Y_1(N)_{\overline{\Q}}, \Fil \right) \coloneqq \Hil_{\t{et}}^* \left( Y_1(N)_{\overline{\Q}}, \Fil_{\overline{\Q}} \right).
\end{eqnarray*}
For a profinite $R[\t{G} \hat{\Gamma}_e(N)]$-module $(M,\rho)$, we have a natural identification
\begin{eqnarray*}
  \left( \left( \underline{(M,\rho)} \right)_{Y_1(N)} \right)_{\overline{\Q}} \cong \left( \underline{\t{Res}_{\hat{\Gamma}_1(N)}^{\t{G} \hat{\Gamma}_e(N)}(M,\rho)} \right)_{Y_1(N)_{\overline{\Q}}}
\end{eqnarray*}
as profinite $R$-sheaves on $Y_1(N)_{\overline{\Q}}$.

\begin{prp}
\label{comparison 2}
For any first countable profinite $R[\t{G} \hat{\Gamma}_e(N)]$-module $(M,\rho)$, there is a natural homeomorphic $R$-linear isomorphism
\begin{eqnarray*}
  \Hil_{\t{et}}^1 \left( Y_1(N)_{\overline{\Q}}, \left( \underline{(M,\rho)} \right)_{Y_1(N)} \right) \cong \t{H}^1 \left( \Gamma_1(N), \t{Res}_{\Gamma_1(N)}^{\t{G} \hat{\Gamma}_e(N)}(M,\rho) \right).
\end{eqnarray*}
\end{prp}

\begin{proof}
Since $\Gamma_1(N)$ is a fundamental group of the open complex manifold $\Gamma_1(N) \backslash \Hlf$ of dimension $1$, it is a finitely generated free group. By Lemma \ref{prodiscrete cohomology 2}, we have an isomorphism
\begin{eqnarray*}
  \t{H}^1 \left( \Gamma_1(N), \t{Res}_{\Gamma_1(N)}^{\t{G} \hat{\Gamma}_e(N)}(M,\rho) \right) \cong \Hil^1 \left( \Gamma_1(N), \t{Res}_{\Gamma_1(N)}^{\t{G} \hat{\Gamma}_e(N)}(M,\rho) \right),
\end{eqnarray*}
which is a homeomorphism by the definition of the topology of the left hand side introduced right after Lemma \ref{prodiscrete cohomology 2}. By the definition of $\Hil^1$, the assertion for the general case follows from the case where $(M,\rho)$ is a finite $R[\hat{\Gamma}_1(N)]$-module. In this case, the assertion is well-known by the interpretation as the set of isomorphism classes of torsors.
\end{proof}

\begin{crl}
For any finite (resp.\ first countable profinite) $R[\t{G} \hat{\Gamma}_e(N)]$-module $(M,\rho)$, the prodiscrete cohomology
\begin{eqnarray*}
  \Hil_{\t{et}}^1 \left( Y_1(N)_{\overline{\Q}}, \left( \underline{(M,\rho)} \right)_{Y_1(N)} \right)
\end{eqnarray*}
is a finite (resp.\ first countable profinite) $R$-module.
\end{crl}

\begin{proof}
Since $\Gamma_1(N)$ is a finitely generated, $\t{Z}^1(\Gamma_1(N),(M,\rho))$ is a finite (resp.\ first countable profinite) $R$-module, and hence so is $\t{H}^1(\Gamma_1(N), (M,\rho))$. The natural homeomorphic $R$-linear isomorphism
\begin{eqnarray*}
  \Hil_{\t{et}}^1 \left( Y_1(N)_{\overline{\Q}}, \left( \underline{(M,\rho)} \right)_{Y_1(N)} \right) \cong \t{H}^1 \left( \Gamma_1(N), \t{Res}_{\Gamma_1(N)}^{\t{G} \hat{\Gamma}_e(N)}(M,\rho) \right)
\end{eqnarray*}
in Proposition \ref{comparison 2} guarantees that the left hand side is a finite (resp.\ first countable profinite) $R$-module.
\end{proof}

\begin{crl}
\label{symmetric product 3}
Let $p$ be a prime number dividing $N$. For any $n \in \N$, there is a natural homeomorphic $\Z_p$-linear isomorphism
\begin{eqnarray*}
  \Hil^* \left( Y_1(N)_{\overline{\Q}}, \t{Sym}^n \left( \t{R}^1 (\pi_N)_* (\underline{\Z}{}_p)_{E_1(N)} \right) \right) \cong \Hil^* \left( \Gamma_1(N), \t{Res}_{\Gamma_1(N)}^{\t{M}_2(\Z_p)} \left( \t{Sym}^n \left( \Z_p^2, \rho_{\Z_p^2} \right) \right) \right).
\end{eqnarray*}
\end{crl}

\begin{proof}
The assertion follows from Lemma \ref{prodiscrete cohomology 2} and Proposition \ref{comparison 2} by the natural isomorphism
\begin{eqnarray*}
  \left( \underline{\t{Res}_{\t{G} \hat{\Gamma}_e(N)}^{\t{M}_2(\Z_p)} \left( \t{Sym}^n \left( \Z_p^2, \rho_{\Z_p^2} \right) \right)} \right)_{Y_1(N)} \cong \t{Sym}^n \left( \t{R}^1 (\pi_N)_* (\underline{\Z}{}_p)_{E_1(N)} \right)
\end{eqnarray*}
as profinite $\Z_p$-sheaf in Example \ref{symmetric product 2}.
\end{proof}

\subsection{Actions of the Absolute Galois Group and Hecke Operators}
\label{Actions of the Absolute Galois Group and Hecke Operators}

Henceforth, we fix a prime number $p$ dividing $N$. The natural projection $\hat{\Z} \twoheadrightarrow \Z_p$ gives a continuous monoid homomorphism $\t{G} \hat{\Gamma}_e(N) \to \Pi_0(p)$, and hence we regard a topological $R[\Pi_0(p)]$-module as a topological $R[\t{G} \hat{\Gamma}_e(N)]$-module. Let $(M,\rho)$ be a profinite $R[\Pi_0(p)]$-module, and $\Fil$ the profinite $R$-sheaf on $Y_1(N)$ associated to $(M,\rho)$. If $(M,\rho)$ is a finite $R[\t{G} \hat{\Gamma}_e(N)]$-module, then the action of $\t{Gal}(\overline{\Q}/\Q)$ on $Y_1(N)_{\overline{\Q}}$ gives a continuous action on the finite Abelian group $\t{H}_{\t{et}}^*(Y_1(N)_{\Q},\Fil)$ in a functorial way. In general, the action of $\t{Gal}(\overline{\Q}/\Q)$ on $Y_1(N)_{\overline{\Q}}$ gives a continuous action on $\Hil_{\t{et}}^*(Y_1(N)_{\Q},\Fil)$, because it is defined as the inverse limit of finite $R[\t{Gal}(\overline{\Q}/\Q)]$-modules by $R$-linear $\t{Gal}(\overline{\Q}/\Q)$-equivariant homomorphisms. The actions of $\t{Gal}(\overline{\Q}/\Q)$ on prodiscrete cohomologies of profinite $R$-sheaves associated to profinite $R[\Pi_0(p)]$-modules are functorial with respect to continuous $R$-linear $\Pi_0(p)$-equivariant homomorphisms.

\vspace{0.2in}
Let $A \in \t{GL}_2(\Q_p)$ satisfying $A^{\iota} \coloneqq \det(A) A^{-1} \in \Pi_0(p)$ with the orbit decomposition $\bigsqcup_{\theta = 1}^{m} \Gamma_1(N) A_{\theta}$ of the double coset $\Gamma_1(N) A \Gamma_1(N) \subset \t{GL}_2(\Q_p)$ with respect to the left action of $\Gamma_1(N)$. For each $(\gamma,\theta) \in \Gamma_1(N) \times (\N \cap [1,m])$, we put $A_{\theta} \gamma = \gamma_{\theta} A_{\delta(\gamma,\theta)}$ by a unique $(\gamma_{\theta},\delta(\gamma,\theta)) \in \Gamma_1(N) \times (\N \cap [1,m])$. For a $1$-cocycle $c \colon \Gamma_1(N) \to \v{(M,\rho)}$, we define $A c \colon \Gamma_1(N) \to \v{(M,\rho)}$ by setting
\begin{eqnarray*}
  (A c)(\gamma) \coloneqq \sum_{\theta = 1}^{m} \rho \left( A_{\theta}^{\iota}, c(\gamma_{\theta}) \right) \in \v{M}
\end{eqnarray*}
for each $\gamma \in \Gamma_1(N)$. Then $A c$ is a $1$-cocycle, and its cohomology class is independent of the presentation of the double coset decomposition. The action of $A$ induces an $R$-linear endomorphism on $\t{H}^1(\Gamma_1(N),(M,\rho))$, and we call it {\it the double coset operator associated to $A$}. For a prime number $\ell$, we denote by $T_{\ell}$ the $R$-linear endomorphism on $\t{H}^1(\Gamma_1(N),(M,\rho))$ given by the double coset operator associated to 
$
  \left(
    \begin{array}{cc}
      1 & 0 \\
      0 & \ell
    \end{array}
  \right)
$.
 For each $n \in \N \backslash \ens{0}$ with a prime factorisation $n = \prod_{j = 1}^{d} \ell_j^{s_j}$, we set $T_n \coloneqq \prod_{j = 1}^{d} T_{\ell_j}^{s_j}$. For each $\overline{n} \in (\Z/N \Z)^{\times}$, we denote by $\langle \overline{n} \rangle$ the $R$-linear endomorphism on $\t{H}^1(\Gamma_1(N),(M,\rho))$ given by the double coset operator associated to a $D(n) \in \t{SL}_2(\Z)$ of the form 
$
  \left(
    \begin{array}{cc}
      a & b \\
      c & n
    \end{array}
  \right)
$
 with $c \in N \Z$, where $n \in \Z$ is a representative of $\overline{n}$. The operator $\langle \overline{n} \rangle$ is independent of the choice of $D(n)$ for any $\overline{n} \in (\Z/N \Z)^{\times}$. As in \S \ref{p-adic Modular Forms and Hecke Algebras}, we also call these operators {\it Hecke operators}. The actions of Hecke operators on the first cohomology group (and the module of $1$-cocycles if we fix presentations of double coset decompositions) are functorial with respect to continuous $R$-linear $\Pi_0(p)$-equivariant homomorphisms.

\begin{prp}
\label{Galois - Hecke}
For any first countable profinite $R[\Pi_0(p)]$-module $(M,\rho)$, the natural homeomorphic $R$-linear isomorphism
\begin{eqnarray*}
  \Hil_{\t{et}}^1 \left( Y_1(N)_{\overline{\Q}}, \underline{(M,\rho)} \right)_{Y_1(N)} \cong \t{H}^1 \left( \Gamma_1(N), \t{Res}_{\Gamma_1(N)}^{\Pi_0(p)}(M,\rho) \right)
\end{eqnarray*}
in Proposition \ref{comparison 2} gives an $R$-linear $\t{Gal}(\overline{\Q}/\Q)$-equivariant action of Hecke operators on the left hand side.
\end{prp}

\begin{proof}
The prodiscrete cohomologies are defined as the inverse limits of cohomologies corresponding to a finite $R[\Pi_0(p)]$-module and a sheaf of finite Abelian $p$-groups on $Y_1(N)_{\t{\'et}}$ associated to it. The isomorphism in the assertion is given as the inverse limit of isomorphisms between group cohomologies of finite $R[\Pi_0(p)]$-modules and \'etale cohomologies of the associated sheaves of finite Abelian $p$-groups on $Y_1(N)_{\t{\'et}}$. Transition maps on the right hand is Hecke-equivariant, while those on the left hand side is $\t{Gal}(\overline{\Q}/\Q)$-equivariant. Therefore it suffices to verify the assertion in the case where $(M,\rho)$ is a finite $R[\Pi_0(p)]$-module. Imitating \cite{Del69} Proposition 3.18, we compare the action of Hecke operators on the right hand side with a $\t{Gal}(\overline{\Q}/\Q)$-equivariant endomorphisms on the left hand side induced by a Hecke correspondence. Let $\ell$ be a prime number. We deal only with $T_{\ell}$. We denote by $\Fil$ the sheaf of finite Abelian groups on $Y_1(N)_{\t{\'et}}$ associated to $(M,\rho)$. The construction of Hecke operators on the prodiscrete cohomology of $\Fil$ by a Hecke correspondence include the following three steps: First, we define a correspondence as the graph associated to two projections $\t{pr}_1, \t{pr}_2 \colon Y_1(N,\ell) \twoheadrightarrow Y_1(N)$ from a curve $Y_1(N,\ell)$ with a moduli interpretation. Secondly, we construct a natural morphism $\t{pr}_1^* \Fil \to \t{pr}_2^* \Fil$, and give a definition of an operator $T_{\ell}$ acting on the \'etale cohomology of $\Fil$. Finally, we verify that the isomorphism in the assertion is $T_{\ell}$-equivariant.

\vspace{0.2in}
Firstly, we set
\begin{eqnarray*}
  \t{G} \hat{\Gamma}_e(N,\ell) & \coloneqq &
  \Set
  {
    \left(
      \begin{array}{cc}
        a & b \\
        c & d
      \end{array}
    \right)
    \in \t{GL}_2(\hat{\Z})
  }
  {
    \begin{array}{l}
      b \in \ell \hat{\Z} \\
      c \in N \hat{\Z} \\
      d \in (1 + N) \hat{\Z}
    \end{array}
  }.
\end{eqnarray*}
Let $Y_1(N,\ell)$ denote the algebraic curve over $\Q$ obtained as the quotient of the finite \'etale covering $Y(N \ell) \twoheadrightarrow Y_1(N)$ corresponding to the subgroup $\t{G} \hat{\Gamma}_e(N,\ell) \subset \t{G} \hat{\Gamma}_e(N)$. Then $Y_1(N,\ell)$ is a moduli of triads $(E,\beta,C)$ of an elliptic curve $E$, a projection $\beta \colon E[N] \twoheadrightarrow (\underline{\Z/N \Z})$ between group schemes, and a cyclic subgroup $C \subset E[\ell]$ of order $\ell$ with $C \cap \ker(\beta) = \emptyset$. Let $\t{pr}_1$ denote the canonical projection $Y_1(N,\ell) \twoheadrightarrow Y_1(N)$, which corresponds to the natural transform $(E,\beta,C) \mapsto (E,\beta)$ between moduli. As in \S \ref{Profinite Zp Sheaves on Modular Curves}, we identify $Y_1(N)$ with the moduli of pairs $(E,\iota)$ of an elliptic curve $E$ and a closed immersion $\iota \colon \G_{\t{m}}[N] \hookrightarrow E[N]$ between group schemes. Similarly, $Y_1(N,\ell)$ is identified with the moduli of triads $(E,\iota,C)$ of an elliptic curve $E$, a closed immersion $\iota \colon \G_{\t{m}}[N] \hookrightarrow E[N]$ between group schemes, and a cyclic subgroup $C \subset E[\ell]$ of order $\ell$ with $C \cap \im(\iota) = \ens{0}$. Then the projection $\t{pr}_1$ is given by the natural transform $(E,\iota,C) \mapsto (E,\iota)$ between moduli. There is another projection $\t{pr}_2 \colon Y_1(N,\ell) \twoheadrightarrow Y_1(N)$ given by the natural transform $(E,\iota,C) \mapsto (E/C,\iota + C)$ between moduli. For each triad $(E,\beta,C)$ of an elliptic curve $E$, a projection $\beta \colon E[N] \twoheadrightarrow (\underline{\Z/N \Z})$ between group schemes, and a cyclic subgroup $C \subset E[\ell]$ of order $\ell$ with $C \cap \ker(\beta) = \ens{0}$, we denote by $\beta/C$ the projection $(E/C)[N] \twoheadrightarrow (\underline{\Z/N \Z})$ between group schemes determined by the condition that $(E/C,\beta/C)$ is the image of $(E,\beta,C)$. We consider the correspondence $\t{pr}_1 \times \t{pr}_2 \colon Y_1(N,\ell) \to Y_1(N) \times_{\Q} Y_1(N)$. Before that, we calculate the difference of the two embeddings $\eta_1, \eta_2 \colon \t{G} \hat{\Gamma}_e(N,\ell) \hookrightarrow \t{G} \hat{\Gamma}_e(N)$ induced by $\t{pr}_1$ and $\t{pr}_2$. Let $n \in \N \backslash \ens{0}$. We consider the finite groups
\begin{eqnarray*}
  \t{G} \Gamma_e(N,\ell,n) & \coloneqq &
  \Set
  {
    \left(
      \begin{array}{cc}
        \overline{a} & \overline{b} \\
        \overline{c} & \overline{d}
      \end{array}
    \right)
    \in \t{GL}_2(\Z/N \ell n \Z)
  }
  {
  \begin{array}{l}
    \overline{b} \in \ell \Z/N \ell n \Z \\
    \overline{c} \in N (\Z/N \ell n \Z) \\
    \overline{d} \in 1 + N (\Z/N \ell n \Z)
  \end{array}
  } \\
  \t{G} \Gamma_e(N,1,n) & \coloneqq &
  \Set
  {
    \left(
      \begin{array}{cc}
        \overline{a} & \overline{b} \\
        \overline{c} & \overline{d}
      \end{array}
    \right)
    \in \t{GL}_2(\Z/Nn \Z)
  }
  {
  \begin{array}{l}
    \overline{c} \in N (\Z/Nn \Z) \\
    \overline{d} \in 1 + N (\Z/Nn \Z)
  \end{array}
  }.
\end{eqnarray*}
We put 
$
  A_{\ell} \coloneqq
  \left(
    \begin{array}{cc}
      1 & 0 \\
      0 & \ell
    \end{array}
  \right)
  \in \t{G} \hat{\Gamma}_e(N)
$
. The injective continuous group homomorphism
\begin{eqnarray*}
  \t{ad}_{A_{\ell}} \colon \t{G} \hat{\Gamma}_e(N,\ell) & \to & \t{G} \hat{\Gamma}_e(N) \\
  \left(
    \begin{array}{cc}
      a & b \\
      c & d
    \end{array}
  \right)
  & \mapsto &
  \left(
    \begin{array}{cc}
      a & \ell^{-1} b \\
      \ell c & d
    \end{array}
  \right) =
  A_{\ell}
  \left(
    \begin{array}{cc}
      a & b \\
      c & d
    \end{array}
  \right)
  A_{\ell}^{-1}.
\end{eqnarray*}
induces a well-defined group homomorphism
\begin{eqnarray*}
  \t{ad}_{A_{\ell},n} \colon \t{G} \Gamma_e(N,\ell,n) & \to & \t{G} \Gamma_e(N,1,n) \\
  \left(
    \begin{array}{cc}
      a + N \ell n \Z & b + N \ell n \Z \\
      c + N \ell n \Z & d + N \ell n \Z
    \end{array}
  \right)
  & \mapsto &
  \left(
    \begin{array}{cc}
      a + Nn \Z & \ell^{-1} b + Nn \Z \\
      \ell c + Nn \Z & d + Nn \Z
    \end{array}
  \right).
\end{eqnarray*}
We recall that the right action of $\t{GL}_2(\Z/N \ell n \Z)$ (resp.\ $\t{GL}_2(\Z/Nn \Z)$) on $Y(N \ell n)$ (resp.\ $Y(Nn)$) is given by
\begin{eqnarray*}
  (E,(\alpha_1,\alpha_2))
  \left(
    \begin{array}{cc}
      a + N \ell n \Z & b + N \ell n \Z \\
      c + N \ell n \Z & d + N \ell n \Z
    \end{array}
  \right)^{\t{op}}
  = (E, a \alpha_1 + c \alpha_2, b \alpha_1 + d \alpha_2).
\end{eqnarray*}
The canonical projection $Y(N \ell n) \twoheadrightarrow Y_1(N,\ell)$ (resp.\ $Y(Nn) \twoheadrightarrow Y_1(N)$) is given by the natural transform
\begin{eqnarray*}
  & & (E,(\alpha_1,\alpha_2)) \mapsto \left( E, \beta_{n \ell \alpha_1, n \ell \alpha_2} \colon a_1 (n \ell \alpha_1) + a_2 (n \ell \alpha_2) \mapsto a_2, \langle Nn \alpha_2 \rangle \right) \\
  & \t{resp.} & (E,(\alpha_1,\alpha_2)) \mapsto \left( E, \beta_{n \ell \alpha_1, n \ell \alpha_2} \colon a_1 (n \alpha_1) + a_2 (n \alpha_2) \mapsto a_2 \right)
\end{eqnarray*}
between moduli, and hence the right action induces an isomorphism
\begin{eqnarray*}
  & & \t{Gal} \left( Y(N \ell n) \twoheadrightarrow Y_1(N,\ell) \right) \cong \t{G} \Gamma_e(N,\ell,n) \\
  & \t{resp.} & \t{Gal} \left( Y(Nn) \twoheadrightarrow Y_1(N) \right) \cong \t{G} \Gamma_e(N,1,n).
\end{eqnarray*}
The composite of $\t{pr}_2$ and the canonical projection $Y(N \ell n) \twoheadrightarrow Y_1(N,\ell)$ is given by the natural transform
\begin{eqnarray*}
  (E,(\alpha_1,\alpha_2)) \mapsto \left( E/\langle Nn \alpha_2 \rangle, \beta_{n \ell \alpha_1, n \ell \alpha_2}/\langle Nn \alpha_2 \rangle \right)
\end{eqnarray*}
between moduli, and hence the finite \'etale morphism $\sigma_n \colon Y(N \ell n) \twoheadrightarrow Y(Nn)$ given by the natural transform
\begin{eqnarray*}
  (E,(\alpha_1,\alpha_2)) \mapsto \left( E/\langle Nn \alpha_2 \rangle, (\ell \alpha_1 + \langle Nn \alpha_2 \rangle, \alpha_2 + \langle Nn \alpha_2 \rangle) \right)
\end{eqnarray*}
between moduli makes the diagram
\begin{eqnarray*}
\begin{CD}
  Y(N \ell n) @>{\sigma_n}>> Y(Nn) \\
  @VVV                       @VVV \\
  Y_1(N,\ell) @>{\t{pr}_2}>> Y_1(N)
\end{CD}
\end{eqnarray*}
commutes. For any 
$
  \left(
    \begin{array}{cc}
      a + N \ell n \Z & b + N \ell n \Z \\
      c + N \ell n \Z & d + N \ell n \Z
    \end{array}
  \right)
  \in \t{G} \Gamma_e(N,\ell,n)
$,
 noting $d + \ell \Z \in (\Z/\ell \Z)^{\times}$, we have an equality
\begin{eqnarray*}
  & & \sigma_n
  \left(
    (E,(\alpha_1,\alpha_2))
  \left(
    \begin{array}{cc}
      a + N \ell n \Z & b + N \ell n \Z \\
      c + N \ell n \Z & d + N \ell n \Z
    \end{array}
  \right)^{\t{op}}
  \right)
  = \sigma_n \left( (E, (a \alpha_1 + c \alpha_2, b \alpha_1 + d \alpha_2)) \right) \\
  & = & \left( E/\langle dNn \alpha_2 \rangle, \left( (\ell a \alpha_1 + \ell c \alpha_2) + \langle dNn \alpha_2 \rangle, (b \alpha_1 + d \alpha_2) + \langle dNn \alpha_2 \rangle \right) \right) \\
  & = & \left( E/\langle Nn \alpha_2 \rangle, \left( (a(\ell \alpha_1) + (\ell c) \alpha_2) + \langle dNn \alpha_2 \rangle, (\ell^{-1}b (\ell \alpha_1) + d \alpha_2) + \langle dNn \alpha_2 \rangle \right) \right) \\
  & = & \left( E/\langle Nn \alpha_2 \rangle, (\ell \alpha_1 + \langle Nn \alpha_2 \rangle, \alpha_2 +  \langle Nn \alpha_2 \rangle) \right)
  \left(
    \begin{array}{cc}
      a + Nn \Z & \ell^{-1} b + Nn \Z \\
      \ell c + Nn \Z & d + Nn \Z
    \end{array}
  \right)^{\t{op}} \\
  & = & \sigma_n((E,(\alpha_1,\alpha_2))) \ \t{ad}_{A_{\ell},n}
  \left(
      \left(
      \begin{array}{cc}
        a + N \ell n \Z & b + N \ell n \Z \\
        c + N \ell n \Z & d + N \ell n \Z
      \end{array}
    \right)
  \right)^{\t{op}},
\end{eqnarray*}
and hence the group homomorphism $\t{G} \Gamma_e(N,\ell,n) \to \t{G} \Gamma_e(N,1,n)$ obtained as the composite of the natural isomorphisms $\t{Gal}(\varpi) \cong \t{G} \Gamma_e(N,\ell,n)$ and $\t{Gal}(Y(Nn) \twoheadrightarrow Y_1(N)) \cong \t{G} \Gamma_e(N,1,n)$ and $(\sigma_n)_* \colon \t{Gal}(\varpi) \to \t{Gal}(Y(Nn) \twoheadrightarrow Y_1(N))$ coincides with $\t{ad}_{A_{\ell},n}$. As a consequence, the homeomorphic group isomorphism
\begin{eqnarray*}
  \sigma \colon \varprojlim_{n \in \N} \t{Gal} \left( Y(N \ell n!) \twoheadrightarrow Y_1(N,\ell) \right) \stackrel{\sim}{\to} \varprojlim_{n \in \N} \t{Gal}(Y(Nn!) \twoheadrightarrow Y_1(N))
\end{eqnarray*}
induced by the compatible system $((\sigma_{n!})_*)_{n \in \N}$ corresponds to the inverse limit of the compatible system $(\t{ad}_{A_{\ell},n!})_{n \in \N}$ through the natural isomorphisms above, and hence is compatible with $\t{ad}_{A_{\ell}}$ through the natural homeomorphic isomorphisms
\begin{eqnarray*}
  \t{G} \hat{\Gamma}_e(N,\ell) & \cong & \varprojlim_{n \in \N} \t{G} \hat{\Gamma}_e(N,\ell,n!) \\
  \t{G} \hat{\Gamma}_e(N) & \cong & \varprojlim_{n \in \N} \t{G} \hat{\Gamma}_e(N,1,n!).
\end{eqnarray*}
Secondly, we have natural isomorphisms
\begin{eqnarray*}
  \t{pr}_1^* \Fil & \cong & \underline{\t{Res}_{\left( \t{G} \hat{\Gamma}_e(N,\ell), \eta_1 \right)}^{\t{G} \hat{\Gamma}_e(N)} \left( \t{Res}_{\t{G} \hat{\Gamma}_e(N)}^{\Pi_0(p)}(M,\rho) \right)} \\
  & = & \underline{\t{Res}_{\left( \t{G} \hat{\Gamma}_e(N,\ell), \iota_N \circ \eta_1 \right)}^{\t{G} \hat{\Gamma}_e(N,\ell)} \left( \t{Res}_{\t{G} \hat{\Gamma}_e(N,\ell)}^{\Pi_0(p)}(M,\rho) \right)} \\
  \t{pr}_2^* \Fil & \cong & \underline{\t{Res}_{\left( \t{G} \hat{\Gamma}_e(N,\ell), \eta_2 \right)}^{\t{G} \hat{\Gamma}_e(N)} \left( \t{Res}_{\t{G} \hat{\Gamma}_e(N)}^{\Pi_0(p)}(M,\rho) \right)} \\
  & = & \underline{\t{Res}_{\left( \t{G} \hat{\Gamma}_e(N,\ell), \iota_N \circ \eta_2 \right)}^{\t{G} \hat{\Gamma}_e(N)} \left( \t{Res}_{\t{G} \hat{\Gamma}_e(N)}^{\Pi_0(p)}(M,\rho) \right)} \\
  & = & \underline{\t{Res}_{\left( \t{G} \hat{\Gamma}_e(N,\ell), \iota_N \circ \eta_1 \right)}^{\t{G} \hat{\Gamma}_e(N,\ell)} \left( \t{Res}_{\left( \t{G} \hat{\Gamma}_e(N,\ell), \t{ad}_{A_{\ell}} \right)}^{\t{G} \hat{\Gamma}_e(N)} \left( \t{Res}_{\t{G} \hat{\Gamma}_e(N)}^{\Pi_0(p)}(M,\rho) \right) \right)}.
\end{eqnarray*}
By the functoriality of the correspondence $(M,\rho) \rightsquigarrow (\underline{\t{Res}_{(\t{G} \hat{\Gamma}_e(N,\ell), \iota_N \circ \eta_1)}^{\t{G} \hat{\Gamma}_e(N,\ell)}(M,\rho)})_{Y_1(N,\ell)}$, the $R$-linear $\t{G} \hat{\Gamma}_e(N,\ell)$-equivariant homomorphism
\begin{eqnarray*}
  \varphi \colon \t{Res}_{\left( \t{G} \hat{\Gamma}_e(N,\ell), \t{ad}_{A_{\ell}} \right)}^{\t{G} \hat{\Gamma}_e(N)} \left( \t{Res}_{\t{G} \hat{\Gamma}_e(N)}^{\Pi_0(p)}(M,\rho) \right) & \to & \t{Res}_{\t{G} \hat{\Gamma}_e(N,\ell)}^{\Pi_0(p)}(M,\rho) \\
  m & \mapsto & \rho(A_{\ell}^{\iota},m)
\end{eqnarray*}
induces a morphism $\underline{\varphi} \colon \t{pr}_2^* \Fil \to \t{pr}_1^* \Fil$. We define an $R$-linear endomorphism $T_{\ell}$ on $\Hil_{\t{et}}^1(Y_1(N)_{\overline{\Q}}, \Fil)$ as the composite
\begin{eqnarray*}
  & & \Hil_{\t{et}}^1 \left( Y_1(N)_{\overline{\Q}}, \Fil \right) \xrightarrow[]{\t{pr}_2^*} \Hil_{\t{et}}^1 \left( Y_1(N,\ell)_{\overline{\Q}}, \t{pr}_2^* \Fil \right) \xrightarrow[]{\Hil_{\t{et}}^1 \left( \underline{\varphi} \right)} \Hil_{\t{et}}^1 \left( Y_1(N,\ell)_{\overline{\Q}}, \t{pr}_1^* \Fil \right) \\
  & \xrightarrow[]{(\t{pr}_1)_*} & \Hil_{\t{et}}^1 \left( Y_1(N)_{\overline{\Q}}, \Fil \right),
\end{eqnarray*}
where $(\t{pr}_1)_*$ is the trace map associated to the finite Galois covering $\t{pr}_1$.

\vspace{0.2in}
Finally, we verify that the isomorphism in the assertion is $T_{\ell}$-equivariant. Put
\begin{eqnarray*}
  \Gamma_1(N,\ell) \coloneqq
  \left(
    \begin{array}{cc}
      1 + N \Z & \ell \Z \\
      N \Z & 1 + N \Z
    \end{array}
  \right)
  \cap \t{SL}_2(\Z) = \Gamma_1(N) \cap A_{\ell}^{-1} \Gamma_1(N) A_{\ell} \subset \t{SL}_2(\Q).
\end{eqnarray*}
Take a presentation $\Gamma_1(N) A_{\ell} \Gamma_1(N) = \bigsqcup_{\theta = 1}^{m} \Gamma_1(N) g_{\theta}$ of the right coset decomposition. We have
\begin{eqnarray*}
  \Gamma_1(N)/\Gamma_1(N,\ell) = \bigsqcup_{\theta = 1}^{m} (g_{\theta}^{-1} A_{\ell}) \Gamma_1(N,\ell).
\end{eqnarray*}
For each $\gamma \in \Gamma_1(N)$ and $\theta \in \N \cap [1,m]$, let $\delta_{\gamma,\theta} \in \N \cap [1,m]$ denote a unique integer with $g_{\theta} \gamma \in \Gamma_1(N) g_{\delta(\gamma,\theta)}$, and put $\gamma_{\theta} \coloneqq g_{\theta} \gamma g_{\delta(\gamma,\theta)}^{-1} \in \Gamma_1(N)$. The trace map $(\t{pr}_1)_*$ corresponds to the $R$-linear homomorphism
\begin{eqnarray*}
  \t{Tr} \colon \t{H}^1 \left( \Gamma_1(N,\ell), \t{Res}_{\Gamma_1(N,\ell)}^{\Gamma_1(N)} \left( \t{Res}_{\Gamma_1(N)}^{\Pi_0(p)}(M,\rho) \right) \right) \to \t{H}^1 \left( \Gamma_1(N), \t{Res}_{\Gamma_1(N)}^{\Pi_0(p)}(M,\rho) \right)
\end{eqnarray*}
sending the cohomology class of a $1$-cocycle
\begin{eqnarray*}
  c \colon \Gamma_1(N,\ell) \to \t{Res}_{\Gamma_1(N,\ell)}^{\Gamma_1(N)}(\t{Res}_{\Gamma_1(N)}^{\Pi_0(p)}(M,\rho)))
\end{eqnarray*}
to the cohomology class of the $1$-cocycle
\begin{eqnarray*}
  \t{Tr}(c) \colon \Gamma_1(N) & \to & \t{Res}_{\Gamma_1(N)}^{\Pi_0(p)}(M,\rho) \\
  \gamma & \mapsto & \sum_{\theta = 1}^{m} \rho \left( g_{\theta}^{-1} A_{\ell}, c \left( \t{ad}_{A_{\ell}}^{-1}(\gamma_{\theta}) \right) \right).
\end{eqnarray*}
Therefore the endomorphism on $\t{H}^1(\Gamma_1(N), \t{Res}_{\Gamma_1(N)}^{\Pi_0(p)}(M,\rho))$ induced by the action of $T_{\ell}$ on $\Hil_{\t{et}}^1(Y_1(N)_{\overline{\Q}},\Fil)$ through the isomorphism in the assertion sends the cohomology class of a $1$-cocycle $c \colon \Gamma_1(N) \to \t{Res}_{\Gamma_1(N)}^{\Pi_0(p)}(M,\rho)$ to a $1$-cocycle $c' \colon \Gamma_1(N) \to \t{Res}_{\Gamma_1(N)}^{\Pi_0(p)}(M,\rho)$ given by setting
\begin{eqnarray*}
  & & c'(\gamma) \coloneqq \t{Tr} \left( \varphi \circ c \circ \t{ad}_{A_{\ell}} \right)(\gamma) = \sum_{\theta = 1}^{m} \rho \left( g_{\theta}^{-1} A_{\ell}, \left( \varphi \circ c \circ \t{ad}_{A_{\ell}} \right) \left( \t{ad}_{A_{\ell}}^{-1}(\gamma_{\theta}) \right) \right) \\
  & = & \sum_{\theta = 1}^{m} \rho \left( g_{\theta}^{-1} A_{\ell}, \rho \left( A_{\ell}^{\iota}, c \left( \t{ad}_{A_{\ell}} \left( \t{ad}_{A_{\ell}}^{-1}(\gamma_{\theta}) \right) \right) \right) \right) = \sum_{\theta = 1}^{m} \rho \left( g_{\theta}^{-1} A_{\ell} A_{\ell}^{\iota}, c(\gamma_{\theta}) \right) \\
  & = & \sum_{\theta = 1}^{m} \rho \left( \ell g_{\theta}^{-1}, c(\gamma_{\theta})) \right) = \sum_{\theta = 1}^{m} \rho \left( g_{\theta}^{\iota}, c(\gamma_{\theta})) \right) = T_{\ell}(c)(\gamma)
\end{eqnarray*}
for each $\gamma \in \Gamma_1(N)$. Thus the isomorphism in the assertion is $T_{\ell}$-equivariant.
\end{proof}

\section{Interpolation of \'Etale Cohomologies}
\label{Interpolation of Etale Cohomologies}

Henceforth, we only consider the case $p \neq 2$. In this section, we interpolate the family $(\t{Sym}^{k-2}(\Z_p^2,\rho_{\Z_p^2}))_{k = 2}^{\infty}$ along weights $k \geq 2$. Their scalar extensions by $\Q_p$ are irreducible $\Q_p$-linear representations of $\Pi_0(p)$ of pairwise distinct dimensions. In order to compare them with each other, we construct infinite dimensional extensions of them, which share the underlying topological $\Z_p$-module $\Z_p^{\N}$.

\subsection{Interpolation along the Weight Spaces}
\label{Interpolation along the Weight Spaces}

We construct a profinite $\Z_p[\Pi_0(p)]$-module interpolating profinite $\Z_p[\Pi_0(p)]$-modules $(\t{Sym}^n(\Z_p^2,\rho_{\Z_p^2}))_{n = 0}^{\infty}$ along weights $n + 2 \in \N$. As is dealt with in \S \ref{Actions of the Absolute Galois Group and Hecke Operators}, an action of $\Pi_0(p)$ plays an important role for a geometric construction of a family of Galois representations. To begin with, we extend several functions on $\N$ to the weight spaces $\Z_p$ and $W = \t{Hom}^{\t{cont}}(\Z_p^{\times},\Z_p^{\times})$. For each $(n,m) \in \Z_p \times \N$, we set
\begin{eqnarray*}
  \left(
    \begin{array}{c}
      n \\
      m
    \end{array}
  \right)
  \coloneqq \frac{1}{m!} \prod_{h = 0}^{m-1} (n-h).
\end{eqnarray*}
It gives a unique continuous function $\Z_p \times \N \to \Q_p$ extending the binomial coefficient function on the dense subset $\set{(n,m) \in \N \times \N}{n \geq m} \subset \Z_p \times \N$. Since the image of the dense subset $\set{(n,m) \in \N \times \N}{n \geq m}$ by the binomial coefficient function is $\N \subset \Z_p$, the extended binomial coefficient gives a continuous function $\Z_p \times \N \to \Z_p$. For any $(d,n) \in (1 + p \Z_p) \times \Z_p$, we set
\begin{eqnarray*}
  d^n \coloneqq \sum_{h = 0}^{\infty}
  \left(
    \begin{array}{c}
      n \\
      h
    \end{array}
  \right)
  (d-1)^h.
\end{eqnarray*}
The infinite sum converges in $1 + p \Z_p$, and it gives a unique continuous function $(1 + p \Z_p) \times \Z_p \to 1 + p \Z_p$ extending the restriction $(1 + p \Z) \times \N \to 1 + p \Z$ of the exponential function $\Q^{\times} \times \Z \to \Q^{\times} \colon (d,n) \mapsto d^n$. Every $n \in \Z_p$ associates a continuous character $\chi_{p,n} \colon \Z_p^{\times} \to \Z_p^{\times}$ in the following way. To begin with, we define $\chi_{p,n} |_{1 + p \Z_p} \colon 1 + p \Z_p \to 1 + p \Z_p$ by setting $\chi_{p,n} |_{1 + p \Z_p}(d) \coloneqq d^n$ for each $d \in 1 + p \Z_p$. The infinite sum
\begin{eqnarray*}
  d(p) \coloneqq \sum_{h = 0}^{\infty}
  \left(
    \begin{array}{c}
      (p-1)^{-1} \\
      h
    \end{array}
  \right)
  (d^{p-1}-1)^h
\end{eqnarray*}
converges in $1 + \Z_p$ for any $d \in \Z_p^{\times}$, and the map
\begin{eqnarray*}
  (\cdot)(p) \colon \Z_p^{\times} & \to & 1 + p \Z_p \\
  d & \mapsto & d(p)
\end{eqnarray*}
is a continuous group homomorphism whose restriction on the subgroup $1 + p \Z_p \subset \Z_p^{\times}$ is the identity map. We define a continuous character $\chi_{p,n} \colon \Z_p^{\times} \to \Z_p^{\times}$ as the composite of $(\cdot)(p)$, $\chi_{p,n} |_{1 + p \Z_p}$, and the inclusion $1 + p \Z_p \hookrightarrow \Z_p^{\times}$. On the other hand, the canonical isomorphism
\begin{eqnarray*}
  \Z_p^{\times} & \stackrel{\sim}{\to} & \F_p^{\times} \times (1 + p \Z_p) \\
  d & \mapsto & (d + p \Z_p, d(p))
\end{eqnarray*}
gives a well-defined decomposition
\begin{eqnarray*}
  (\Z/(p-1) \Z) \times \Z_p & \stackrel{\sim}{\to} & W \\
  (n_0 + (p-1) \Z, n_p) & \mapsto & \chi_{n_0} \chi_{p,n_p-n_0}
\end{eqnarray*}
as a group, where we denote by $\chi_n \in W$ the continuous character $\Z_p^{\times} \to \Z_p^{\times} \colon d \mapsto d^n$ for each $n \in \Z$. For each $\chi \in W$, we denote by $(n^p(\chi), n_p(\chi))$ its image in $(\Z/(p-1) \Z) \times \Z_p$. By definition, we have $(n^p(\chi_{p,n}),n_p(\chi_{p,n})) = (0,n)$ for any $n \in \Z_p$, and $(n^p(\chi_n),n_p(\chi_n)) = (n + (p-1) \Z, n)$ for any $n \in \Z$. For each $(d,\chi) \in \Z_p^{\times} \times W$, we put $d^{\chi} \coloneqq \chi(d) \in \Z_p^{\times}$. For each $(\chi,m) \in W \times \N$, we set
\begin{eqnarray*}
  \left(
    \begin{array}{c}
      \chi \\
      m
    \end{array}
  \right)
  \coloneqq
  \left(
    \begin{array}{c}
      n_p(\chi) \\
      m
    \end{array}
  \right) \in \Z_p.
\end{eqnarray*}
It gives a unique continuous function $W \times \N \to \Z_p$ extending the binomial coefficient function on the dense subset $\set{(n,m) \in \N \times \N}{n \geq m}$ with respect to the embedding $\N \hookrightarrow W \colon n \mapsto \chi_n$. Henceforth, we often abbreviate $\chi_n$ to $n$ for each $n \in \Z$.

\begin{prp}
Let $\chi \in W$. For any $(A,\alpha,i) \in \Pi_0(p) \times \Z_p^{\N} \times \N$ with
$
  A = 
  \left(
    \begin{array}{cc}
      a & b \\
      c & d
    \end{array}
  \right)
$
 and $\alpha = (\alpha_j)_{j = 0}^{\infty}$, the infinite sum
\begin{eqnarray*}
  \rho_{\chi}(A,\alpha)_i \coloneqq \sum_{j = 0}^{\infty} \alpha_j \sum_{h = 0}^{\min \Ens{i,j}}
  \left(
    \begin{array}{c}
      i \\
      h
    \end{array}
  \right)
  \left( \prod_{m = i}^{i+j-h-1} (n_p(\chi)-m) \right) a^h b^{i-h} \frac{c^{j-h}}{(j-h)!} d^{\chi-i-j+h}
\end{eqnarray*}
converges in $\Z_p$, and the map
\begin{eqnarray*}
  \rho_{\chi} \colon \Pi_0(p) \times \Z_p^{\N} & \to & \Z_p^{\N} \\
  (A,\alpha) & \mapsto & \left( \rho_{\chi}(A,\alpha)_i \right)_{i \in \N}
\end{eqnarray*}
is continuous.
\end{prp}

\begin{proof}
Let $(A,\alpha,i) \in \Pi_0(p) \times \Z_p^{\N} \times \N$ with
$
  A =
  \left(
    \begin{array}{cc}
      a & b \\
      c & d
    \end{array}
  \right)
$
 and $\alpha = (\alpha_j)_{j = 0}^{\infty}$. For any $j \in \N$, we have
\begin{eqnarray*}
  & & \max_{
  \t{\scriptsize $
    \begin{array}{c}
      h \in \N \\
      0 \leq h \leq \min \ens{i,j}
    \end{array}
  $}
  }
  \vv{\frac{c^{j-h}}{(j-h)!}} = \max_{
  \t{\scriptsize $
    \begin{array}{c}
      h \in \N \\
      0 \leq h \leq \min \ens{i,j}
    \end{array}
  $}
  }
  \frac{\vv{c}^{j-h}}{\vv{p}^{\sum_{r = 1}^{\infty} \left\lfloor \frac{j-h}{p^r} \right\rfloor}} \leq \max_{
  \t{\scriptsize $
    \begin{array}{c}
      h \in \N \\
      0 \leq h \leq \min \ens{i,j}
    \end{array}
  $}
  }
  \vv{p}^{(j-h) \left( 1 - \sum_{r = 1}^{\infty} \frac{1}{p^r} \right)} \\
  & = & \vv{p}^{(j - \min \Ens{i,j}) \left( 1 - \frac{1}{p-1} \right)} \stackrel{j \to \infty}{\longrightarrow} 0,
\end{eqnarray*}
where $\lfloor x \rfloor \in \Z$ denotes the largest integer which is not larger than $x$ for each $x \in \R$, and hence
\begin{eqnarray*}
  \vv{\alpha_j \sum_{h = 0}^{\min \Ens{i,j}}
    \left(
      \begin{array}{c}
        i \\
        h
      \end{array}
    \right)
    \left( \prod_{m = i}^{i+j-h-1} (n_p(\chi)-m) \right) a^h b^{i-h} \frac{c^{j-h}}{(j-h)!} d^{\chi-i-j+h}} \stackrel{j \to \infty}{\longrightarrow} 0.
\end{eqnarray*}
It implies that $\rho_{\chi}(A,\alpha)_i$ converges in $\Z_p$. The continuity of $\rho_{\chi}$ follows from the convergence of the infinite sum in the definition of $\rho_{\chi}(A,\alpha)_i$, because of the continuity of each term of the infinite sum.
\end{proof}

Following the abbreviation of $\chi_n$ to $n$, we put $\rho_n \coloneqq \rho_{\chi_n}$ for each $n \in \Z$. In order to verify that the topological space $(\Z_p^{\N},\rho_{\chi})$ with the action of the underlying topological space of $\Pi_0(p)$ is a profinite $\Z_p[\Pi_0(p)]$-module for any $\chi \in W$, we compare it with $\t{Sym}^n(\Z_p^2,\rho_{\Z_p^2})$ for infinitely many $n \in \N$.

\begin{lmm}
\label{specialisation}
Let $n \in \N$. For each $i \in \N \cap [0,n]$, put
\begin{eqnarray*}
  e_{n,i} \coloneqq
  \left(
    \begin{array}{c}
      n \\
      i 
    \end{array}
  \right)
  T_1^i T_2^{n-i}
  \in \t{Sym}^n \left( \Z_p^2 \right).
\end{eqnarray*}
Then the map
\begin{eqnarray*}
  \varpi_n \colon (\Z_p^{\N},\rho_n) & \to & \t{Sym}^n \left( \Z_p^2, \rho_{\Z_p^2} \right) \\
  (\alpha_i)_{i = 0}^{\infty} & \mapsto & \sum_{i = 0}^{n} \alpha_i e_{n,i}
\end{eqnarray*}
is a continuous $\Z_p$-linear $\Pi_0(p)$-equivariant homomorphism.
\end{lmm}

For the convention of the symmetric product, see Example \ref{symmetric product}.

\begin{proof}
Let $(A,\alpha) \in \Pi_0(p) \times \Z_p^{\N}$ with
$
  A =
  \left(
    \begin{array}{cc}
      a & b \\
      c & d
    \end{array}
  \right)
$
 and $\alpha = (\alpha_i)_{i = 0}^{\infty}$. We have
\begin{eqnarray*}
  & & \varpi_n(\rho_n(A,\alpha)) = \varpi_n((\rho_n(A,\alpha)_i)_{i = 0}^{\infty}) = \sum_{i = 0}^{n} \rho_n(A,\alpha)_i e_{n,i} \\
  & = & \sum_{i = 0}^{n} \sum_{j = 0}^{\infty} \alpha_j \sum_{h = 0}^{\min \Ens{i,j}}
  \left(
    \begin{array}{c}
      i \\
      h
    \end{array}
  \right)
  \left( \prod_{d = i}^{i+j-h-1} (n-d) \right) a^h b^{i-h} \frac{c^{j-h}}{(j-h)!} d^{n-i-j+h}
  \left(
    \begin{array}{c}
      n \\
      i
    \end{array}
  \right)
  T_1^i T_2^{n-i} \\
  & = & \sum_{i = 0}^{n} \sum_{j = 0}^{n} \alpha_j \sum_{h = \max \Ens{0,i+j-n}}^{\min \Ens{i,j}}
  \left(
    \begin{array}{c}
      i \\
      h
    \end{array}
  \right)
  \left( \prod_{d = i}^{i+j-h-1} (n-d) \right) a^h b^{i-h} \frac{c^{j-h}}{(j-h)!} d^{n-i-j+h}
  \left(
    \begin{array}{c}
      n \\
      i
    \end{array}
  \right)
  T_1^i T_2^{n-i} \\
  & = & \sum_{i = 0}^{n} \sum_{j = 0}^{n} \alpha_j \sum_{h = \max \Ens{0,i+j-n}}^{\min \Ens{i,j}} \frac{n!}{h! (i-h)! (j-h)! (n-i-j+h)!} a^h b^{i-h} c^{j-h} d^{n-i-j+h} T_1^i T_2^{n-i} \\
  & = & \sum_{j = 0}^{n} \alpha_j
  \left(
    \begin{array}{c}
      n \\
      j
    \end{array}
  \right)
  \sum_{i = 0}^{n}
  \sum_{h = \max \Ens{0,i+j-n}}^{\min \Ens{i,j}}
  \left(
    \begin{array}{c}
      j \\
      h
    \end{array}
  \right)
  (aT_1)^h (cT_2)^{j-h}
  \left(
    \begin{array}{c}
      n-j \\
      i-h
    \end{array}
  \right)
  (bT_1)^{i-h} (dT_2)^{n-i-j+h} \\
  & = & \sum_{j = 0}^{n} \alpha_j
  \left(
    \begin{array}{c}
      n \\
      j
    \end{array}
  \right)
  \sum_{h = 0}^{j}
  \left(
    \begin{array}{c}
      j \\
      h
    \end{array}
  \right)
  (aT_1)^h (cT_2)^{j-h}
  \sum_{h' = 0}^{n-j}
  \left(
    \begin{array}{c}
      n-j \\
      h'
    \end{array}
  \right)
  (bT_1)^{h'} (dT_2)^{n-j-h'} \\
  & = & \sum_{j = 0}^{n} \alpha_j
  \left(
    \begin{array}{c}
      n \\
      j
    \end{array}
  \right)
  (aT_1+cT_2)^j (bT_1+dT_2)^{n-j} \\
  & = & \t{Sym}^n \left( \rho_{\Z_p^2} \right)
  \left(
    \left(
      \begin{array}{cc}
        a & b \\
        c & d
      \end{array}
    \right),
    \sum_{j = 0}^{n} \alpha_j
    \left(
      \begin{array}{c}
        n \\
        j
      \end{array}
    \right)
    T_1^j T_2^{n-j}
  \right)
  = \t{Sym}^n \left( \rho_{\Z_p^2} \right) \left( A, \sum_{j = 0}^{n} \alpha_j e_{n,j} \right) \\
  & = & \t{Sym}^n \left( \rho_{\Z_p^2} \right)(A, \varpi_n(\alpha)).
\end{eqnarray*}
Thus $\varphi_n$ is a $\Z_p$-linear $\Pi_0(p)$-equivariant homomorphism.
\end{proof}

For each $n \in \N$, we denote by $\t{Sym}_0^n(\Z_p^2) \subset \t{Sym}^n(\Z_p^2)$ the image of $\varpi_n$. It is a lattice of $\t{Sym}^n(\Q_p^2)$ with a $\Z_p$-linear basis $(e_{n,i})_{i = 0}^{n}$, and is a $\Z_p[\Pi_0(p)]$-submodule of $\t{Sym}^n(\Z_p^2,\rho_{\Z_p^2})$. In fact, it is easily seen that $\t{Sym}_0^n(\Z_p^2)$ is a $\Z_p[\t{M}_2(\Z_p)]$-submodule of $\t{Sym}^n(\Z_p^2,\rho_{\Z_p^2})$, but we do not use this fact.

\begin{dfn}
For each $n \in \N$, we put
\begin{eqnarray*}
  \Leb_n \coloneqq \t{Sym}^n \left( \Z_p^2,\rho_{\Z_p^2} \middle) \right| \t{Sym}_0^n \left( \Z_p^2 \right).
\end{eqnarray*}
See Example \ref{induced} (i) for this convention.
\end{dfn}

The modified symmetric product $\Leb_n$ has good congruence relation with respect to $n \in \N$ as is shown in the following.

\begin{lmm}
\label{congruence}
For any $(r,n_0,n_1) \in (\N \backslash \ens{0}) \times \N \times \N$ with $n_0 \leq n_1$ and $n_1 - n_0 \in p^{r-1}(p-1) \Z$, the canonical projection
\begin{eqnarray*}
  \varpi^r_{n_1,n_0} \colon \Leb_{n_1}/p^r & \twoheadrightarrow & \Leb_{n_0}/p^r \\
  \sum_{i = 0}^{n_1} \overline{\alpha}_i e_{n_1,i} & \mapsto & \sum_{i = 0}^{n_0} \overline{\alpha}_i e_{n_0,i}
\end{eqnarray*}
is a $(\Z/p^r \Z)$-linear $\Pi_0(p)$-equivariant homomorphism.
\end{lmm}

\begin{proof}
Let $\tilde{\varpi}{}^r_{\chi,\chi'} \colon (\Z_p^{\N},\rho_{\chi})/p^r \to (\Z_p^{\N},\rho_{\chi'})/p^r$ denote $\t{id}_{\Z_p^{\N}/p^r \Z_p^{\N}}$ for each $(\chi,\chi') \in W \times W$, and $\varpi^r_n \colon (\Z_p^{\N},\rho_n)/p^r \twoheadrightarrow \Leb_n/p^r$ denote the surjective $(\Z/p^r \Z)$-linear $\Pi_0(p)$-equivariant homomorphism induced by $\varpi_n$ for each $n \in \N$. Let $(\chi,\chi') \in W \times W$ with $\chi - \chi' \in p^{r-1}(p-1)W$.  For any $d \in \Z_p^{\times}$, we have $d^{\chi} - d^{\chi'} \in p^r \Z_p$. By the definition of $\varpi^r_{n_1,n_0}$, we have $\varpi^r_{n_1,n_0} \circ \varpi^r_{n_1} = \varpi^r_{n_0} \circ \tilde{\varpi}{}^r_{n_1,n_0}$. Since the matrix representation of $\rho_{\chi}$ with respect to the canonical topological basis of $\Z_p^{\N}$ is given as a function on $\chi \in W$ belonging to the closed $\Z_p$-subalgebra of $\t{C}(W,\Z_p)$ generated by polynomials of the functions $\chi \mapsto n_p(\chi)$ and $\chi \mapsto \chi(d)$ for each $d \in \Z_p^{\times}$ with coefficients in $p \Z_p$, $\tilde{\varpi}{}^r_{n_1,n_0}$ is a $(\Z/p^r \Z)$-linear $\Pi_0(p)$-equivariant isomorphism. Thus the assertion follows from the surjectivity of $\varpi^r_{n_1}$.
\end{proof}

For any $(r,n_0) \in (\N \backslash \ens{0}) \times \N$, the family
\begin{eqnarray*}
  \left( \Leb_{n_0+p^{r-1}(p-1)m}/p^r \right)_{m = 0}^{\infty}
\end{eqnarray*}
forms an inverse system by the canonical projections $(\varpi^r_{n_2,n_1})_{n_2,n_1}$. Let $\chi \in W$ and $r \in \N$. Although we have not verified that the continuous action $\rho_{\chi}$ of the underlying topological space of $\Pi_0(p)$ is a continuous action of $\Pi_0(p)$ yet, it is a $\Z_p$-linear action, and hence the convention $(\Z_p^{\N},\rho_{\chi})/p^{r+1}$ naturally makes sense.

\begin{dfn}
For each $\chi \in W$, we denote by $\chi^{(r)} \in \N$ the smallest non-negative integer satisfying $\chi - \chi^{(r)} \in p^r(p-1)W$.
\end{dfn}

Let $\chi \in W$. Following the convention in the proof of Lemma \ref{congruence}, the family
\begin{eqnarray*}
  \left( \varpi^{r+1}_{\chi^{(r)}+p^r(p-1)m} \circ \tilde{\varpi}{}^{r+1}_{\chi,\chi^{(r)}_0+p^r(p-1)m} \right)_{m = 0}^{\infty}
\end{eqnarray*}
is a compatible system of $(\Z/p^{r+1} \Z)$-linear $\Pi_0(p)$-equivariant homomorphisms, and induces a continuous $(\Z/p^{r+1} \Z)$-linear $\Pi_0(p)$-equivariant homomorphism
\begin{eqnarray*}
\begin{array}{crcl}
  & (\Z_p^{\N},\rho_{\chi})/p^{r+1} & \to & \displaystyle\varprojlim_{m \in \N} \left( \Leb_{\chi^{(r)}+p^r(p-1)m}/p^{r+1} \right) \\
  & (\overline{\alpha}_i)_{i = 0}^{\infty} & \mapsto & \left( \sum_{i = 0}^{\chi^{(r)}+p^r(p-1)m} \overline{\alpha}_i e_{\chi^{(r)}+p^r(p-1)m,i} \right)_{m = 0}^{\infty}.
\end{array}
\end{eqnarray*}
It is a homeomorphic isomorphism because the canonical projections give natural identifications
\begin{eqnarray*}
  \Z_p^{\N} \cong \varprojlim_{m \in \N} \Z_p^{n+p^r(p-1)m} \cong \varprojlim_{m \in \N} \Leb_{n+p^r(p-1)m}
\end{eqnarray*}
as profinite $\Z_p$-modules for any $n \in \N$.

\begin{dfn}
Let $n \in \N$. We put
\begin{eqnarray*}
  \Fil_n/p^r \coloneqq \left( \underline{\Leb_n/p^r} \right)_{Y_1(N)}
\end{eqnarray*}
for each $r \in \N$, and set
\begin{eqnarray*}
  \Fil_n \coloneqq (\Fil_n/p^r)_{r = 0}^{\infty} \cong \left( \underline{\Leb_n} \right)_{Y_1(N)}.
\end{eqnarray*}
See \S \ref{Profinite Zp Sheaves on Modular Curves} for this convention.
\end{dfn}

\begin{rmk}
\label{symmetric product 4}
Let $k_0 \in \N \cap [2,\infty)$. We consider the action of Hecke operators on the profinite $\Z_p[\t{Gal}(\overline{\Q}/\Q)]$-module $\Hil_{\t{et}}^1(Y_1(N)_{\overline{\Q}}, \Fil_{k_0-2})$. For the definition of Hecke operators, see \S \ref{Actions of the Absolute Galois Group and Hecke Operators}. We denote by $\t{T}_{k_0,N}^{\t{\'et}} \subset \t{End}_{\Z_p}(\Hil_{\t{et}}^1(Y_1(N)_{\overline{\Q}}, \Fil_{k_0-2}))$ the commutative $\Z_p$-subalgebra generated by Hecke operators. For each $n \in \Z$ coprime to $N$, we put $S_n \coloneqq n^{k_0-2} \langle n + N \Z \rangle \in \t{T}_{k_0,N}^{\t{\'et}}$. Since $M$ is a finitely generated $\Z_p$-module, $\t{T}_{k_0,N}^{\t{\'et}}$ is finitely generated as a $\Z_p$-module. We endow $\t{T}_{k_0,N}^{\t{\'et}}$ with the $p$-adic topology, and regard it as a profinite $\Z_p$-algebra. The continuous action of $\t{T}_{k_0,N}^{\t{\'et}}$ on $\Hil_{\t{et}}^1(Y_1(N)_{\overline{\Q}}, \Fil_{k_0-2})$ induces a continuous action of $\t{T}_{k_0,N}^{\t{\'et}}$ on $\Hil_{\t{et}}^1(Y_1(N)_{\overline{\Q}}, \Fil_{k_0-2})_{\t{free}}$ (Definition \ref{free}, Example \ref{induced} (ii)). We put
\begin{eqnarray*}
  & & \t{H}_{\t{et}}^* \left( Y_1(N)_{\overline{\Q}}, \t{Sym}^{k_0-2} \left( \t{R}^1 (\pi_N)_* (\underline{\Q}{}_p)_{E_1(N)} \right) \right) \\
  & \coloneqq & \Q_p \otimes_{\Z_p} \Hil_{\t{et}}^* \left( Y_1(N)_{\overline{\Q}}, \t{Sym}^{k_0-2} \left( \t{R}^1 (\pi_N)_* (\underline{\Z}{}_p)_{E_1(N)} \right) \right).
\end{eqnarray*}
This coincides with the ordinary \'etale cohomology of the smooth $\Q_p$-sheaf represented by $\Q_p \otimes_{\Z_p} \t{Sym}^{k_0-2}(\t{R}^1 (\pi_N)_*(\underline{\Z}{}_p)_{E_1(N)})$ by Remark \ref{derived limit 2}. By the argument in Example \ref{symmetric product 2}, we have a natural identification
\begin{eqnarray*}
  & & \Q_p \otimes_{\Z_p} \Hil_{\t{et}}^1 \left( Y_1(N)_{\overline{\Q}}, \left( \underline{\t{Sym}^{k_0-2} \left( \Z_p,\rho_{\Z_p^2} \right)} \right)_{Y_1(N)} \right) \\
  & \cong & \t{H}_{\t{et}}^1 \left( Y_1(N)_{\overline{\Q}}, \t{Sym}^{k_0-2} \left( \t{R}^1 (\pi_N)_* (\underline{\Q}{}_p)_{E_1(N)} \right) \right),
\end{eqnarray*}
as linearly complete $\Z_p[\t{Gal}(\overline{\Q}/\Q)]$-modules with respect to the topologies induced by the $p$-adic norm of their natural integral structures for each $k \in \N \cap [2,\infty)$, and we equip the right hand side with the action of Hecke operators through the the identification. By a similar argument to that in Example \ref{symmetric product 2}, the Hecke action coincides with the usual one dealt with in \cite{Gro90} 3.1, 3.11, and 3.13 induced by the automorphism $Y_1(N) \to Y_1(N) \colon (E,\beta) \mapsto (E,\beta^{\overline{n}})$ for $\langle \overline{n} \rangle$ with $\overline{n} \in (\Z/N \Z)^{\times}$ and by the morphism $\t{pr}_2^* \t{R}^1 (\pi_N)_* (\underline{\Z}{}_p) \to \t{pr}_1^* \t{R}^1 (\pi_N)_* (\underline{\Z}{}_p)$ associated to the $\ell$-isogeny $\t{pr}_1^* E_1(N) \cong E_1(N,\ell) \twoheadrightarrow E_1(N,\ell)/C_1(N,\ell) \cong \t{pr}_2^* E_1(N)$ over $Y_1(N,\ell)$ (introduced in the proof of Proposition \ref{Galois - Hecke}) with the universal object $(E_1(N,\ell),\beta_1(N,\ell),C_1(M,\ell))$ for $T_{\ell}$ with a prime number $\ell$. Since the embedding $\Leb_{k_0-2} \hookrightarrow \t{Sym}^{k_0-2}(\Z_p^2,\rho_{\Z_p^2})$ induces a $\Q_p$-linear $\t{M}_2(\Z_p)$-equivariant isomorphism
\begin{eqnarray*}
  \Q_p \otimes_{\Z_p} \Leb_{k_0-2} \cong \Q_p \otimes_{\Z_p} \t{Sym}^{k_0-2} \left( \Z_p^2,\rho_{\Z_p^2} \right) \cong \t{Sym}^{k_0-2} \left( \Q_p^2,\rho_{\Q_p^2} \right),
\end{eqnarray*}
the natural $\Q_p$-linear $\t{Gal}(\overline{\Q}/\Q)$-equivariant homomorphism
\begin{eqnarray*}
  \Q_p \otimes_{\Z_p} \Hil_{\t{et}}^1 \left( Y_1(N)_{\overline{\Q}}, \Fil_{k_0-2} \right) \to \t{H}_{\t{et}}^* \left( Y_1(N)_{\overline{\Q}}, \t{Sym}^{k_0-2} \left( \t{R}^1 (\pi_N)_* (\underline{\Q}{}_p)_{E_1(N)} \right) \right)
\end{eqnarray*}
is an isomorphism by Proposition \ref{comparison 2}. Therefore the Eichler--Shimura isomorphism (\cite{Shi59} 5 Th\'eor\`eme 1, \cite{Hid93} 6.3 Theorem 4) and the comparison theorem of cohomologies (\cite{SGA4} Expos\'e XI Th\'eor\`eme 4.4 (iii), \cite{SGA4} Expos\'e XVI Corollaire 1.6) give a homeomorphic $\Z_p$-algebra isomorphism $(\t{T}_{k_0,N}^{\t{\'et}})_{\t{free}} \cong \t{T}_{k_0,N}$ preserving $T_{\ell}$ for each prime number $\ell$ and $S_n$ for each $n \in \N$ coprime to $N$. We regard $\Hil_{\t{et}}^1(Y_1(N)_{\overline{\Q}}, \Fil_{k_0-2})_{\t{free}}$ as a topological $\t{T}_{k_0,N}$-module through the isomorphism. It is finitely generated as a $\Z_p$-module, and hence is a profinite $\t{T}_{k_0,N}$-module by Proposition \ref{finitely generated - p-adic}.
\end{rmk}

\begin{thm}
\label{interpolation}
For any $\chi \in W$, $(\Z_p^{\N},\rho_{\chi})$ is a profinite $\Z_p[\Pi_0(p)]$-module.
\end{thm}

\begin{proof}
By the argument above, we have a homeomorphic $\Z_p$-linear $\Pi_0(p)$-equivariant isomorphism
\begin{eqnarray*}
  (\Z_p^{\N},\rho_{\chi}) \cong \varprojlim_{r \in \N} \left( (\Z_p^{\N},\rho_{\chi})/p^{r+1} \right) \to \varprojlim_{r \in \N} \varprojlim_{m \in \N} \left( \Leb_{\chi^{(r)}+p^r(p-1)m}/p^{r+1} \right)
\end{eqnarray*}
of topological spaces with actions of the underlying topological space of $\Pi_0(p)$. Since the target is a projective limit of finite $\Z_p[\Pi_0(p)]$-modules, the source is a profinite $\Z_p[\Pi_0(p)]$-module.
\end{proof}

\begin{crl}
\label{combinatorial}
The equality
\begin{eqnarray*}
  \sum_{m = h}^{\min \Ens{i,j}} (-1)^{m-h}
  \left(
    \begin{array}{c}
      n-m \\
      i-m
    \end{array}
  \right)
  \left(
    \begin{array}{c}
      j \\
      m
    \end{array}
  \right)
  \left(
    \begin{array}{c}
      m \\
      h
    \end{array}
  \right)
  =
  \left(
    \begin{array}{c}
      n-j \\
      i-h
    \end{array}
  \right)
  \left(
    \begin{array}{c}
      j \\
      h
    \end{array}
  \right)
\end{eqnarray*}
holds for any $(n,i,j,h) \in \Z_p \times \N \times \N \times \N$ with $h \leq \min \ens{i,j}$.
\end{crl}

Of course, this equality can be obtained in many ways with no use of $p$-adic representations, and hence must be well-known.

\begin{proof}
It is easily seen that the assertion is equivalent to the condition that $\rho_n(A_0 A_1,\alpha) = \rho_n(A_0, \rho_n(A_1,\alpha))$ for any $(A_0,A_1,\alpha) \in \Pi_0(p) \times \Pi_0(p) \times \Z_p^{\N}$ by $p$-adic Lie algebra theory, $p$-adic analysis to Baker--Campbell--Hausdorff formula, and Schneider--Teitelbaum theory. Thus the assertion follows from Theorem \ref{interpolation}.
\end{proof}

\begin{rmk}
Theorem \ref{interpolation} is deeply related to \cite{PS11} 3.3 and 7.1. Robert Pollack and Glenn Stevens defined a continuous right action of the topological monoid
\begin{eqnarray*}
  \Sigma_0(p) \coloneqq
  \left(
    \begin{array}{cc}
      \Z_p^{\times} & \Z_p \\
      p \Z_p & \Z_p
    \end{array}
  \right)
  \cap \t{GL}_2(\Q_p)
\end{eqnarray*}
of non-negative integral weight $n$ on the topological $\Z_p$-algebra of distributions on $\Z_p$ in \cite{PS11} 3.3, and proved that the closed $\Z_p$-subalgebra of distributions with integral moments, which is canonically homeomorphically isomorphic to $\Z_p[[w]]$, is stable under the action of $\Sigma_0(p)$ in Proposition 7.1. By the canonical topological basis $(w^h)_{h \in \N}$ of $\Z_p[[w]]$, we identify $\Z_p[[w]]$ with $\Z_p^{\N}$. The map $\Pi_0(p) \cap \t{GL}_2(\Q_p) \to \Sigma_0(p)^{\t{op}} \colon A \mapsto (A^{\iota})^{\t{op}} = \det(A) (A^{-1})^{\t{op}}$ associating the cofactor matrices is a homeomorphic group isomorphism, and hence the notion of a right action of $\Sigma_0(p)$ is equivalent to that of a left action of $\Pi_0(p) \cap \t{GL}_2(\Q_p)$. Therefore we obtain a continuous action $\rho'_n$ of $\Pi_0(p) \cap \t{GL}_2(\Q_p)$ of non-negative integral weight $n$ on $\Z_p^{\N}$. The homeomorphic $\Z_p$-linear $\Pi_0(p)$-equivariant isomorphism
\begin{eqnarray*}
  (\Z_p^{\N},\rho_{\chi}) \cong \varprojlim_{r \in \N} \varprojlim_{m \in \N} \left( \Leb_{\chi^{(r)}+p^r(p-1)m}/p^{r+1} \right)
\end{eqnarray*}
for a $\chi \in W$ in the proof of Theorem \ref{interpolation} ensures that if $\chi = \chi_n$ for an $n \in \N$, then the restriction of $\rho_{\chi}$ on the submonoid $\Pi_0(p) \cap \t{GL}_2(\Q_p) \subset \Pi_0(p)$ coincides with $\rho'_n$. Thus the construction of $(\Z_p^{\N},\rho_{\chi})$ is a generalisation of that of $(\Z_p^{\N},\rho'_n)$ in the sense that the former one deals with a general weight and $\Pi_0(p)$ while the latter one deals with a non-negative integral weight and $\Pi_0(p) \cap \t{GL}_2(\Q_p)$.
\end{rmk}

\begin{rmk}
We have a geometric construction of $(\Z_p^{\N},\rho_{\chi})$ in the case where $\chi = \chi_{p,n}$ for an $n \in \Z_p$. In Example \ref{modular form}, we constructed a linearly complete topological $\Z_p[\Pi_0(p)^{\t{op}}]$-module $(\t{C}(\Z_p,\Z_p),(m_p^{\vee})_{\kappa(\chi)})$ using $p$-adic linear fractional transformations. We note that $m_p$ extends to an action of $\Pi_0(p)$ in an obvious way because diagonal matrices in $\Pi_0(p)$ acts trivially on $\Z_p$ via $m_p$. Since $p$-adic linear fractional transformations and $\chi_{p,n}(cz+d)$ for any $(c,d) \in p \Z_p \times \Z_p^{\times}$ are rigid analytic functions on $\Z_p$, the $p$-adically complete $\Z_p$-subalgebra $\Z_p \ens{w} \subset \t{C}(\Z_p,\Z_p)$ consisting of rigid analytic functions of Gauss norm $\leq 1$ is stable under the action of $\Pi_0(p)^{\t{op}}$. The Iwasawa-type dual (\cite{ST02} Theorem 1.2) of the Banach $\Q_p$-algebra $\Q_p \ens{w} \cong \Q_p \otimes_{\Z_p} (\Z_p \ens{w})$ is the profinite $\Z_p$-algebra $\Z_p[[w]]$ of distributions with integral moments. Although the Iwasawa-type duality for Banach representations of a profinite group (\cite{ST02} Theorem 2.3) does not extend to duality of Banach unitary representations for a topological monoid in a direct way, it is easily seen that the continuous action of $\Pi_0(p)^{\t{op}}$ on $\Z_p \ens{w}$ induces a continuous action $\rho''_{\chi_{p,n}}$ of $\Pi_0(p)$ on $\Z_p[[w]]$, and $\rho''_{\chi_{p,n}}$ corresponds to $\rho_{\chi_{p,n}}$ through the identification $\Z_p[[w]] \cong \Z_p^{\N}$. This gives an alternative proof of Theorem \ref{interpolation 2} for the case where the weight $\chi$ is of the form $\chi_{p,n}$ for an $n \in \Z_p$.
\end{rmk}

\begin{rmk}
Let $\chi \in W$. We have a natural identification
\begin{eqnarray*}
  \left( \underline{\t{Res}_{\t{G} \hat{\Gamma}_e(N)}^{\Pi_0(p)} (\Z_p^{\N},\rho_{\chi})} \right)_{Y_1(N)} \cong \left( \Fil_{\chi^{(r)}+p^r(p-1)m}/p^{r+1} \right)_{m,r = 0}^{\infty}
\end{eqnarray*}
as profinite $\Z_p$-sheaves on $Y_1(N)$ by Theorem \ref{interpolation}. When $\chi = \chi_n$ for some $n \in \N$, then we have an identification
\begin{eqnarray*}
  \left( \underline{\t{Res}_{\t{G} \hat{\Gamma}_e(N)}^{\Pi_0(p)} (\Z_p^{\N},\rho_n)} \right)_{Y_1(N)} \cong \left( \Fil_{n+p^rm}/p^{r+1}\right)_{m,r = 0}^{\infty}
\end{eqnarray*}
as profinite $\Z_p$-sheaves on $Y_1(N)$ by a calculation of the image of $\t{G} \hat{\Gamma}_e(N) \to \Pi_0(p)$. We do not use this fact in this paper.
\end{rmk}

\begin{rmk}
\label{extension}
Let $n \in \N$. As is constructed in Lemma \ref{specialisation}, there is a canonical projection $\varpi_n \colon (\Z_p^{\N},\rho_n) \twoheadrightarrow \Leb_n$. Taking the Iwasawa-type dual (\cite{ST02} Theorem 2.3) in Schneider--Teitelbaum theory, we obtain an exact sequence
\begin{eqnarray*}
  0 \to \t{Sym}^n(\Q_p^2, \rho_{\Q_p^2}) \to (\t{C}(\Z_p,\Q_p), \rho_n^{\vee}) \to (\ker(\varpi_n)^{\vee}, (\rho_n | \ker(\varpi_n))^{\vee}) \to 0
\end{eqnarray*}
of unitary Banach $\Q_p$-linear representations, and $(\ker(\varpi_n)^{\vee}, (\rho_n | \ker(\varpi_n))^{\vee})$ is an infinite dimensional irreducible unitary Banach $\Q_p$-linear representation. Thus $(\Z_p^{\N},\rho_n)$ is an infinite dimensional extension of $\Leb_n$ by the Iwasawa-type dual of an infinite dimensional irreducible unitary Banach $\Q_p$-linear representation.
\end{rmk}

Now we interpolate the family $(\t{Sym}^{k-2}(\Z_p^2,\rho_{\Z_p^2}))_{k = 2}^{\infty}$ with respect to weights $k$ as elements of $W$. We put $\Lambda_0 \coloneqq \Z_p[[X]]^{p(p-1)}$. We regard $\Lambda_0$ as a $\Z_p$-submodule of $\t{C}(W,\Z_p)$ by the embedding
\begin{eqnarray*}
  \Z_p[[X]]^{p(p-1)} & \hookrightarrow & \t{C}(W,\Z_p) \\
  (F_{\zeta}(X))_{\zeta = 0}^{p(p-1)-1} & \mapsto & \left( \chi \mapsto F_{\chi^{(1)}}(n_p(\chi) - n^{(1)}) \right).
\end{eqnarray*}
This embedding is an injective continuous homomorphism from a compact module to a Hausdorff module, and hence is a homeomorphic isomorphism onto the closed image. For each $\chi \in W$, we denote by $\t{sp}_{\chi}$ the continuous surjective $\Z_p$-algebra homomorphism
\begin{eqnarray*}
  \Lambda_0 & \twoheadrightarrow & \Z_p \\
  f & \mapsto & f(\chi),
\end{eqnarray*}
and call it {\it a specialisation map}. For each $\chi \in W$, we regard $(\Z_p^{\N},\rho_{\chi-2})$ as a profinite $\Lambda_0[\Pi_0(p)]$-module through $\t{sp}_{\chi}$.

\vspace{0.2in}
Since $\N \cap [2,\infty)$ is dense in $W$, the evaluation map
\begin{eqnarray*}
  \t{sp} \coloneqq \prod_{k = 2}^{\infty} \t{sp}_k \colon \Lambda_0 & \hookrightarrow & \prod_{k = 2}^{\infty} \Z_p \\
  f & \mapsto & (\t{sp}_k(f))_{k = 2}^{\infty}
\end{eqnarray*}
is an injective continuous $\Z_p$-linear homomorphism between compact Hausdorff modules, and hence is a homeomorphic isomorphism onto the closed image. In particular, we regard $\Lambda_0^{\N}$ as a closed $\Z_p$-submodule of $\prod_{k = 2}^{\infty} \Z_p^{\N}$ by the embedding
\begin{eqnarray*}
  \t{sp}^{\N} \colon \Lambda_0^{\N} & \hookrightarrow & \prod_{k = 2}^{\infty} \Z_p^{\N} \\
  (f_i)_{i = 0}^{\infty} & \mapsto & ((\t{sp}_k(f_i))_{i = 0}^{\infty})_{k = 2}^{\infty}.
\end{eqnarray*}
Through the homeomorphic group isomorphism $W \cong (\Z/(p-1) \Z) \times \Z_p$, we identify $W$ as the analytic space given as the disjoint union of $p-1$ copies of $\Z_p$. As a closed $\Z_p$-subalgebra of $\t{C}(W,\Z_p)$, $\Lambda_0$ consists of locally analytic functions on $\Z_p$ whose restrictions on $\zeta + p(p-1)W \subset \zeta + (p-1)W \cong \Z_p$ are given by single power series in $\Z_p[[X - \zeta]]$ for any $\zeta \in \N \cap [0,p(p-1)-1]$. In particular, it contains $n_p$ and the characteristic functions $1_{\zeta + p(p-1)W}$ of $\zeta + p(p-1)W$ for each $\zeta \in \N \cap [0,p(p-1)-1]$. We denote by $z \in \Lambda_0$ the element corresponding to $n_p$, and by $e_{\zeta} \in \Lambda_0$ the idempotent corresponding to $1_{\zeta + p(p-1)W}$ for each $\zeta \in \N \cap [0,p(p-1)-1]$. We put $z_{\zeta} \coloneqq z e_{\zeta} \in \Lambda_0$ for each $\zeta \in \N \cap [0,p(p-1)-1]$. Identifying $e_{\zeta} \Lambda_0$ with $\Lambda_0/(1-e_{\zeta}) \Lambda_0 \cong \Z_p[[X]]$, we obtain a presentation $\Lambda_0 = \prod_{\zeta = 0}^{p(p-1)-1} \Z_p[[z_{\zeta} - \zeta]]$. Let $(h_0,f_0) \in \N \times \t{C}(W,\Z_p)$. We define a map
\begin{eqnarray*}
  \left(
    \begin{array}{c}
      f_0 \\
      h_0
    \end{array}
  \right)
  \colon W & \to & \Z_p \\
  \chi & \mapsto &
  \left(
    \begin{array}{c}
      f_0(\chi) \\
      h_0
    \end{array}
  \right).
\end{eqnarray*}
It is a polynomial function on $f_0$ with coefficients in $\Q_p$, and hence is continuous. Therefore we regard it as an element of $\t{C}(W,\Z_p)$. Let $(d,f) \in (1 + p \Z_p) \times \t{C}(W,\Z_p)$. The infinite sum
\begin{eqnarray*}
  d^f \coloneqq \sum_{h = 0}^{\infty}
  \left(
    \begin{array}{c}
      f \\
      h
    \end{array}
  \right)
  (d-1)^h
\end{eqnarray*}
converges in $1 + p \t{C}(W,\Z_p) \subset \t{C}(W,\Z_p)$ because $d-1 \in p \Z_p$. Suppose $f \in \Lambda_0$. We have
\begin{eqnarray*}
  \left(
    \begin{array}{c}
      f \\
      h
    \end{array}
  \right)
  (d-1)^h
  \left\{
    \begin{array}{ll}
      = 1 & (h = 0) \\
      \in p \Z_p[f] \subset \Lambda_0 & (h \geq 1)
    \end{array}
  \right.
\end{eqnarray*}
because we have
\begin{eqnarray*}
  & & \vv{\frac{(d-1)^h}{h!}} = \vv{d-1}^h \vv{p}^{- \sum_{r = 1}^{\infty} \left\lfloor \frac{h}{p^r} \right\rfloor} \leq \vv{p}^{h - \sum_{r = 1}^{\infty} \frac{h}{p^r}} = \vv{p}^{h - \frac{h}{p-1}} = \vv{p}^{\frac{h(p-2)}{p-1}}
  \left\{
    \begin{array}{ll}
      = 1 & (h = 0) \\
      < 1 & (h \geq 1)
    \end{array}
  \right.
\end{eqnarray*}
by $p \neq 2$. Since $\Lambda_0$ is closed in $\t{C}(W,\Z_p)$, $d^f$ also lies in $\Lambda_0$. Since the embedding $\Lambda_0 \hookrightarrow \t{C}(W,\Z_p)$ is a homeomorphism onto the image, the infinite sum in the definition of $d^f$ also converges to $d^f$ in $\Lambda_0$. More concretely, $d^f$ lies in the closure of $1 + p \Z_p[f] \subset \Lambda_0$.

\vspace{0.2in}
By the universality of Iwasawa algebra, the continuous group homomorphism
\begin{eqnarray*}
  1 + N \Z_p & \hookrightarrow & \t{C}(W,\Z_p)^{\times} \\
  \gamma & \mapsto & \gamma^z
\end{eqnarray*}
induces a continuous $\Z_p$-algebra homomorphism $\Z_p[[1 + N \Z_p]] \to \t{C}(W,\Z_p)$. We remark that since $1 + N \Z_p$ is contained in $1 + p \Z_p$, we have $\gamma^{z(\chi)} = \gamma^{n_p(\chi)} = \chi(\gamma)$ for any $(\gamma,\chi) \in (1 + N \Z_p) \times W$. Through the Amice transform
\begin{eqnarray*}
  \Z_p[[X]] & \stackrel{\sim}{\to} & \Z_p[[1 + N \Z_p]] \\
  X & \mapsto & [1+N] - 1,
\end{eqnarray*}
it corresponds to the $\Z_p$-algebra homomorphism
\begin{eqnarray*}
  \Z_p[[X]] & \to & \t{C}(W,\Z_p) \\
  X & \mapsto & (1+N)^z - 1,
\end{eqnarray*}
which is injective by the Weierstrass preparation theorem and the fact that the exponential function $(1+N)^z$ is a transcendental function. This embedding factors through $\Lambda_0 \hookrightarrow \t{C}(W,\Z_p)$, because $\Lambda_0$ is closed in $\t{C}(W,\Z_p)$ and the $\Z_p$-subalgebra of $\t{C}(W,\Z_p)$ generated by $(1+N)^z \in \Lambda_0$ is dense in the image of $\Z_p[[1 + N \Z_p]]$. We regard $\Lambda_0$ as a profinite $\Z_p[[1 + N \Z_p]]$-algebra through the embedding by Corollary \ref{profinite ring}.

\begin{prp}
For any $(A,F,i) \in \Pi_0(p) \times \Lambda_0^{\N} \times \N$ with 
$
  A = 
  \left(
    \begin{array}{cc}
      a & b \\
      c & d
    \end{array}
  \right)
$
 and $F = (F_j)_{j = 0}^{\infty}$, the infinite sum
\begin{eqnarray*}
  \rho_{\bullet - 2}(A,F)_i \coloneqq \sum_{j = 0}^{\infty} F_j \sum_{h = 0}^{\min \Ens{i,j}}
  \left(
    \begin{array}{c}
      i \\
      h
    \end{array}
  \right)
  \left( \prod_{m = i}^{i+j-h-1} (z-2-m) \right) a^h b^{i-h} \frac{c^{j-h}}{(j-h)!} d^{z-2-i-j+h}
\end{eqnarray*}
converges in $\Lambda_0$, and the map
\begin{eqnarray*}
  \rho_{\bullet - 2} \colon \Pi_0(p) \times \Lambda_0^{\N} & \to & \Lambda_0^{\N} \\
  (A,F) & \mapsto & \left( \rho_{\bullet - 2}(A,F)_i \right)_{i \in \N}
\end{eqnarray*}
is continuous.
\end{prp}

\begin{proof}
Let $(A,F,i) \in \Pi_0(p) \times \Lambda_0^{\N} \times \N$ with 
$
  A = 
  \left(
    \begin{array}{cc}
      a & b \\
      c & d
    \end{array}
  \right)
$
 and $F = (F_j)_{j = 0}^{\infty}$. Each term in the infinite sum in the definition of $\rho_{\bullet - 2}(A,F)_i$ lies in $\Lambda_0$ by the argument above. For any $\chi \in W$, we have $\prod_{m = 0}^{p-1} (z - n_p(\chi) - m) \in p \Lambda_0 + (z^p-z) \Lambda_0$. The family
\begin{eqnarray*}
  \Set{(p \Lambda_0 + (z^p-z) \Lambda_0)^h \Lambda_0}{h \in \N}
\end{eqnarray*}
forms a fundamental system of neighbourhoods of $0$, because of the presentation
\begin{eqnarray*}
  \Lambda_0 = \prod_{\zeta = 0}^{p(p-1)-1} \Z_p[[z_{\zeta} - \zeta]] \cong \varprojlim_{r,h \in \N} \prod_{\zeta = 0}^{p(p-1)-1} (\Z/p^r \Z)[z_{\zeta} - \zeta]/(z_{\zeta} - \zeta)^h.
\end{eqnarray*}
We have
\begin{eqnarray*}
  & & F_j \sum_{h = 0}^{\min \Ens{i,j}}
  \left(
    \begin{array}{c}
      i \\
      h
    \end{array}
  \right)
  \left( \prod_{m = i}^{i+j-h-1} (z-2-m) \right) a^h b^{i-h} \frac{c^{j-h}}{(j-h)!} d^{z-i-j+h} \\
  & \in & \left( \prod_{m = i}^{j-1} (z-2-m) \right) \Lambda_0 \subset \sum_{r = 0}^{\left\lfloor \frac{j-i}{p} \right\rfloor} p^r (z^p-z)^{\left\lfloor \frac{j-i}{p} \right\rfloor - r} \Lambda_0 \subset (p \Lambda_0 + (z^p-z) \Lambda_0)^{\left\lfloor \frac{j-i}{p} \right\rfloor}
\end{eqnarray*}
for any $j \in \N$ with $j > i$, and hence $\rho_{\bullet - 2}(A,F)_i$ converges in $\Lambda_0$ by the linear completeness of the profinite $\Z_p$-algebra $\Lambda_0$. The continuity of $\rho_{\bullet - 2}$ follows from that of
\begin{eqnarray*}
  \prod_{k = 2}^{\infty} \rho_{k-2} \colon \Pi_0(p) \times \prod_{k = 2}^{\infty} \Z_p^{\N} & \to & \prod_{k = 2}^{\infty} \Z_p^{\N} \\
  (A,((\alpha_{k,i})_{i = 0}^{\infty})_{k = 2}^{\infty}) & \mapsto & (\rho_{k-2}(A,(\alpha_{k,i})_{i = 0}^{\infty}))_{k = 2}^{\infty},
\end{eqnarray*}
because $\prod_{k = 2}^{\infty} \rho_{k-2} \circ (\t{id}_{\Pi_0(p)} \times \t{sp}^{\N}) = \t{sp}^{\N} \circ \rho_{\bullet - 2}$.
\end{proof}

\begin{thm}
\label{interpolation 2}
The pair $(\Lambda_0^{\N}, \rho_{\bullet - 2})$ is a profinite $\Lambda_0[\Pi_0(p)]$-module, and the map
\begin{eqnarray*}
  \left( \Lambda_0^{\N}, \rho_{\bullet - 2} \right) & \to & \prod_{k = 2}^{\infty} \Leb_{k-2} \\
  (f_i)_{i = 0}^{\infty} & \mapsto & \left( \sum_{i = 0}^{k-2} \t{sp}_k(f_i) e_{k-2,i} \right)_{k = 2}^{\infty},
\end{eqnarray*}
is an injective continuous $\Lambda_0$-linear $\Pi_0(p)$-equivariant homomorphism.
\end{thm}

\begin{proof}
The embedding $\t{sp}^{\N} \colon (\Lambda_0^{\N}, \rho_{\bullet - 2}) \hookrightarrow \prod_{k = 2}^{\infty} (\Z_p^{\N},\rho_{k-2})$ is an injective continuous $\Lambda_0$-linear $\Pi_0(p)$-equivariant homomorphism onto the closed image by the definition of $\rho_{\bullet - 2}$. Since its target is a profinite $\Lambda_0[\Pi_0(p)]$-module, so is the source. Let $\iota \colon (\Lambda_0^{\N}, \rho_{\bullet - 2}) \to \prod_{k = 2}^{\infty} \Leb_{k-2}$ denote the map in the assertion. Then $\iota$ is a continuous $\Lambda_0$-linear $\Pi_0(p)$-equivariant homomorphism because it is the composite of $\t{sp}^{\N}$ and the canonical projection
\begin{eqnarray*}
  \prod_{k = 2}^{\infty} \varpi_{k-2} \colon \prod_{k = 2}^{\infty} \Z_p^{\N} & \to & \prod_{k = 2}^{\infty} \Leb_{k-2} \\
  ((\alpha_{k,i})_{i = 0}^{\infty})_{k = 2}^{\infty} & \mapsto & \left( \varpi_{k-2} \left( (\alpha_{k,i})_{i = 0}^{\infty} \right) \right)_{k = 2}^{\infty},
\end{eqnarray*}
which is a continuous $\Lambda_0$-linear $\Pi_0(p)$-equivariant homomorphism by Lemma \ref{specialisation}. Let $f = (f_i)_{i = 0}^{\infty} \in \ker(\iota)$. For any $i \in \N$, $f_i \colon W \to \Z_p$ is zero on the subset $\N \cap [i+2,\infty)$ which shares infinitely many points with $\zeta + p(p-1)W$ for each $\zeta \in \N \cap [0,p(p-1)-1]$, and hence $f_i = 0$ by the identity theorem for rigid analytic functions on $\zeta + p(p-1)W \cong p \Z_p$ for each $\zeta \in \N \cap [0,p(p-1)-1]$. Thus $f = 0$. We conclude that $\iota$ is injective.
\end{proof}

\begin{rmk}
The profinite $\Lambda_0[\Pi_0(p)]$-module $(\Lambda_0^{\N}, \rho_{\bullet-2})$ also admits a geometric construction using $p$-adic linear fractional transformations and distributions. We define a continuous action $1 \times m_p$ of $\Pi_0(p)$ on $\Z_p \times \Z_p$ by setting $(1 \times m_p)(A,(\chi,z)) \coloneqq (\chi, m_p(A,z))$ for each $(A,\chi,z) \in \Pi_0(p) \times W \times \Z_p$. Then by Proposition \ref{geometric action}, we obtain a commutative linearly complete $\Z_p[\Pi_0(p)]$-algebra $(\t{C}(W \times \Z_p, \Z_p), (1 \times m_p)^{\vee})$. The map
\begin{eqnarray*}
  \kappa \colon \Pi_0(p)^{\t{op}} & \to & \t{C}(W \times \Z_p,\Z_p) \\
  \left(
    \begin{array}{cc}
      a & b \\
      c & d
    \end{array}
  \right)^{\t{op}}
  & \mapsto & \left( \chi(cz+d) \colon (\chi_0,z_0) \mapsto \chi_0(cz_0+d) \right)
\end{eqnarray*}
satisfies the condition in Corollary \ref{geometric action 2} with respect to $1 \times m_p$. Therefore we obtain a linearly complete $\Z_p[\Pi_0(p)^{\t{op}}]$-module $(\t{C}(W \times \Z_p,\Z_p), (1 \times m_p)^{\vee}_{\kappa})$. Since $1 \times m_p$ and $\chi(cz+d)$ for any $(c,d) \in p \Z_p \times \Z_p^{\times}$ are rigid analytic functions, the $p$-adically complete $\Z_p$-subalgebra $\Z_p \ens{z,w}^{p-1} \subset \t{C}(\Z_p \times \Z_p,\Z_p)^{p-1} \cong \t{C}(\Z_p^{\sqcup p-1} \times \Z_p,\Z_p) \cong \t{C}(W \times \Z_p,\Z_p)$ consisting of rigid analytic functions of Gauss norm $\leq 1$ is stable under the action of $\Pi_0(p)$. Similarly, the $p$-adically complete $\Z_p$-subalgebra $\prod_{\zeta = 0}^{p(p-1)-1} \Z_p \ens{z_{\zeta} - \zeta, w} \subset \t{C}(W \times \Z_p,\Z_p)$ consisting of locally analytic functions whose restriction on the subspace $(\zeta + p(p-1)W) \times \Z_p \subset W \times \Z_p$ is given as the restriction of a rigid analytic function on $W \times \Z_p$ of Gauss norm $\leq 1$ for any $\zeta \in \N \cap [0,p(p-1)-1]$ is stable under the action of $\Pi_0(p)$. The Iwasawa-type dual of $\prod_{\zeta = 0}^{p(p-1)-1} \Z_p \ens{z_{\zeta} - \zeta, w}$ is naturally identified with $\prod_{\zeta = 0}^{p(p-1)-1} \Z_p[[z_{\zeta} - \zeta, w]] \cong \Lambda_0[[w]] \cong \Lambda_0^{\N}$, and hence $(1 \times m_p)^{\vee}_{\kappa}$ induces a continuous action of $\Pi_0(p)$ on $\Lambda_0^{\N}$. The action coincides with $\rho_{\bullet-2}$.
\end{rmk}

For each $k \in \N \cap [2,\infty)$, we also denote by $\t{sp}_k$ the continuous $\Z_p$-algebra homomorphism obtained as the composite
\begin{eqnarray*}
  \Lambda_0 \xrightarrow[]{\t{sp}_k} \Z_p \hookrightarrow \t{T}_{k,N}.
\end{eqnarray*}
We regard $\t{T}_{k,N}$ as a profinite $\Lambda_0$-algebra through $\t{sp}_k$ by Corollary \ref{profinite ring}. It is easy to see that $\t{sp}_k \colon \Lambda_0 \to \t{T}_{k,N}$ is a $\Z_p[[1 + N \Z_p]]$-algebra homomorphism. We recall that we defined the structure of $\t{T}_{k,N}$ as a profinite $\Z_p[[1 + N \Z_p]]$-algebra in \S \ref{p-adic Modular Forms and Hecke Algebras}.

\vspace{0.2in}
We regard
\begin{eqnarray*}
  \prod_{k = 2}^{k_0} \Hil_{\t{et}}^1 \left( Y_1(N)_{\overline{\Q}}, \Fil_{k-2} \right)_{\t{free}}
\end{eqnarray*}
as a profinite $\t{T}_{\leq k_0,N}$-module through the embedding
\begin{eqnarray*}
  \t{T}_{\leq k_0,N} \hookrightarrow \prod_{k = 2}^{k_0} \t{T}_{k,N},
\end{eqnarray*}
and also as a profinite $\Lambda_0$-module through the evaluation
\begin{eqnarray*}
  \prod_{k = 2}^{k_0} \t{sp}_k \colon \Lambda_0 & \to & \Z_p^{k_0-1} \hookrightarrow \prod_{k = 2}^{k_0} \t{T}_{k,N} \\
  F & \mapsto & (\t{sp}_k(F))_{k = 2}^{k_0}.
\end{eqnarray*}
The actions of $\t{T}_{\leq k_0,N}$ and $\Lambda_0$ give two actions of $\Z_p[[1 + N \Z_p]]$, and they coincide with each other. Therefore we regard 
\begin{eqnarray*}
  \prod_{k = 2}^{k_0} \Hil_{\t{et}}^1 \left( Y_1(N)_{\overline{\Q}}, \Fil_{k-2} \right)_{\t{free}}
\end{eqnarray*}
as a profinite $(\t{T}_{\leq k_0,N} \hat{\otimes}_{\Z_p[[1 + N \Z_p]]} \Lambda_0)$-module in the way in Example \ref{tensor 2}. Taking the inverse limit, we regard
\begin{eqnarray*}
  \prod_{k = 2}^{\infty} \Hil_{\t{et}}^1 \left( Y_1(N)_{\overline{\Q}}, \Fil_{k-2} \right)_{\t{free}}
\end{eqnarray*}
as a profinite $(\Lambda_0 \T_N)$-module.

\vspace{0.2in}
We denote by
\begin{eqnarray*}
  \int_{\Z_p}^{\boxplus} \t{Sym}_0^{k-2} \left( \Z_p^2 \right) dk \subset \prod_{k = 2}^{\infty} \t{Sym}_0^{k-2}(\Z_p^2)
\end{eqnarray*}
the image of $(\Lambda_0^{\N}, \rho_{\bullet - 2})$ by the embedding in Theorem \ref{interpolation 2}, and put
\begin{eqnarray*}
  \int_{\Z_p}^{\boxplus} \Leb_{k-2} dk \coloneqq \left( \prod_{k = 2}^{\infty} \Leb_{k-2} \middle) \middle| \middle( \int_{\Z_p}^{\boxplus} \t{Sym}_0^{k-2} \left( \Z_p^2 \right) dk \right).
\end{eqnarray*}
It is a profinite $\Lambda_0[\Pi_0(p)]$-module admitting specialisation maps
\begin{eqnarray*}
  \t{sp}_{k_0} \colon \int_{\Z_p}^{\boxplus} \Leb_{k-2} dk \twoheadrightarrow \Leb_{k_0-2}
\end{eqnarray*}
given by the canonical projections for each $k_0 \in \N \cap [2,\infty)$. For a formal symbol $H \in \ens{\Hil^1, \t{H}^1}$, we denote by
\begin{eqnarray*}
  \int_{\Z_p}^{\boxplus} H \left( \Gamma_1(N), \Leb_{k-2} \right) dk
\end{eqnarray*}
the image of the continuous $\Lambda_0$-linear Hecke-equivariant homomorphism
\begin{eqnarray*}
  H \left( \Gamma_1(N), \int_{\Z_p}^{\boxplus} \Leb_{k-2} dk \right) \xrightarrow[]{\prod_{k = 2}^{\infty} \t{sp}_k} \prod_{k = 2}^{\infty} H \left( \Gamma_1(N), \Leb_{k-2} \right),
\end{eqnarray*}
which is a profinite $\Lambda_0$-module endowed with an action of $T_{\ell}$ for each prime number $\ell$ and $S_n$ for each $n \in \N$ coprime to $N$. By Lemma \ref{prodiscrete cohomology 2}, we have a natural homeomorphic $\Lambda_0$-linear isomorphism
\begin{eqnarray*}
  \int_{\Z_p}^{\boxplus} \t{H}^* \left( \Gamma_1(N), \Leb_{k-2} \right) dk \stackrel{\sim}{\to} \int_{\Z_p}^{\boxplus} \Hil^* \left( \Gamma_1(N), \Leb_{k-2} \right) dk
\end{eqnarray*}
because $(\Lambda_0^{\N}, \rho_{\bullet - 2})$ is a first countable profinite $\Lambda_0[\Pi_0(p)]$-module.

\vspace{0.2in}
We set
\begin{eqnarray*}
  \int_{\Z_p}^{\boxplus} \Fil_{k-2} dk \coloneqq \left( \underline{\int_{\Z_p}^{\boxplus} \Leb_{k-2} dk} \right)_{Y_1(N)}.
\end{eqnarray*}
See \S \ref{Profinite Zp Sheaves on Modular Curves} for this convention. We have specialisation maps
\begin{eqnarray*}
  \Hil_{\t{et}}^1 \left( Y_1(N)_{\overline{\Q}}, \int_{\Z_p}^{\boxplus} \Fil_{k-2} dk \right) \xrightarrow[]{\t{sp}_{k_0}} \Hil_{\t{et}}^1 \left( Y_1(N)_{\overline{\Q}}, \Fil_{k_0-2} \right)
\end{eqnarray*}
associated to the specialisation map $\t{sp}_{k_0}$ for the corresponding topological $\Lambda_0[\Pi_0(p)]$-modules for each $k_0 \in \N \cap [2,\infty)$. We denote by
\begin{eqnarray*}
  \int_{\Z_p}^{\boxplus} \Hil_{\t{et}}^1 \left( Y_1(N)_{\overline{\Q}}, \Fil_{k-2} \right) dk
\end{eqnarray*}
the image of the continuous $\Lambda_0$-linear $\t{Gal}(\overline{\Q}/\Q)$-equivariant homomorphism
\begin{eqnarray*}
  \Hil_{\t{et}}^1 \left( Y_1(N)_{\overline{\Q}}, \int_{\Z_p}^{\boxplus} \Fil_{k-2} dk \right) \xrightarrow[]{\prod_{k = 2}^{\infty} \t{sp}_k} \prod_{k = 2}^{\infty} \Hil_{\t{et}}^1 \left( Y_1(N)_{\overline{\Q}}, \Fil_{k-2} \right),
\end{eqnarray*}
and it is a profinite $\Lambda_0[\t{Gal}(\overline{\Q}/\Q)]$-module. By Proposition \ref{comparison 2}, we have a natural homeomorphic $\Lambda_0$-linear isomorphism
\begin{eqnarray*}
  \int_{\Z_p}^{\boxplus} \Hil_{\t{et}}^1 \left( Y_1(N)_{\overline{\Q}}, \Fil_{k-2} \right) dk \cong \int_{\Z_p}^{\boxplus} \t{H}^1 \left( \Gamma_1(N), \Leb_{k-2} \right) dk,
\end{eqnarray*}
on which the action of $T_{\ell}$ for each prime number $\ell$ and $S_n$ for each $n \in \N$ coprime to $N$ on the right hand side commutes with that of $\t{Gal}(\overline{\Q}/\Q)$ on the left hand side by Proposition \ref{Galois - Hecke}.

\begin{prp}
\label{specialisation 2}
The specialisation map
\begin{eqnarray*}
  \Hil_{\t{et}}^1 \left( Y_1(N)_{\overline{\Q}}, \int_{\Z_p}^{\boxplus} \Fil_{k-2} dk \right) \xrightarrow[]{\t{sp}_{k_0}} \Hil_{\t{et}}^1 \left( Y_1(N)_{\overline{\Q}}, \Fil_{k_0-2} \right)
\end{eqnarray*}
given by the canonical projection is surjective for any $k_0 \in \N \cap [2,\infty)$.
\end{prp}

\begin{proof}
By the definition of the prodiscrete cohomology and Proposition \ref{comparison 2}, we have natural homeomorphic $\Lambda_0$-linear isomorphisms
\begin{eqnarray*}
  \Hil_{\t{et}}^1 \left( Y_1(N)_{\overline{\Q}}, \int_{\Z_p}^{\boxplus} \Fil_{k-2} dk \right) & \cong & \Hil^1 \left( \Gamma_1(N), \int_{\Z_p}^{\boxplus} \Leb_{k-2} dk \right) \\
  \Hil_{\t{et}}^1 \left( Y_1(N)_{\overline{\Q}}, \Fil_{k_0-2} \right) & \cong & \Hil^1 \left( \Gamma_1(N), \Leb_{k_0-2} \right).
\end{eqnarray*}
Therefore the assertion follows from Proposition \ref{right exact 2}.
\end{proof}

For a formal symbol
\begin{eqnarray*}
  (H_k)_{k = 2}^{\infty} \in \Ens{\left( \Hil^1 \left( \Gamma_1(N), \Leb_{k-2} \right) \right)_{k = 2}^{\infty}, \left( \t{H}^1 \left( \Gamma_1(N), \Leb_{k-2} \right) \right)_{k = 2}^{\infty}, \left( \Hil_{\t{et}}^1 \left( Y_1(N)_{\overline{\Q}}, \Fil_{k-2} \right) \right)_{k = 2}^{\infty}}
\end{eqnarray*}
we denote by
\begin{eqnarray*}
  \int_{\Z_p}^{\boxplus} (H_k)_{\t{free}} dk
\end{eqnarray*}
the image of the composite
\begin{eqnarray*}
  \int_{\Z_p}^{\boxplus} H_k dk \hookrightarrow \prod_{k = 2}^{\infty} H_k \twoheadrightarrow \prod_{k = 2}^{\infty} (H_k)_{\t{free}},
\end{eqnarray*}
and regard it as a profinite $\Lambda_0$-module endowed with a continuous action of $T_{\ell}$ for each prime number $\ell$ and $S_n$ for each $n \in \N$ coprime to $N$ when $(H_k)_{k = 2}^{\infty}$ is a formal symbol corresponding to cohomologies of $(\Leb_{k-2})_{k = 2}^{\infty}$, and with a continuous action of $\t{Gal}(\overline{\Q}/\Q)$ when $(H_k)_{k = 2}^{\infty}$ is a formal symbol corresponding to cohomologies of $(\Fil_{k-2})_{k = 2}^{\infty}$. We have natural homeomorphic $\Lambda_0$-linear isomorphisms
\begin{eqnarray*}
  \int_{\Z_p}^{\boxplus} \t{H}^1 \left( \Gamma_1(N), \Leb_{k-2} \right)_{\t{free}} dk & \cong & \int_{\Z_p}^{\boxplus} \Hil^1 \left( \Gamma_1(N), \Leb_{k-2} \right)_{\t{free}} dk \\
  & \cong & \int_{\Z_p}^{\boxplus} \Hil_{\t{et}}^1 \left( Y_1(N)_{\overline{\Q}}, \Fil_{k-2} \right)_{\t{free}} dk,
\end{eqnarray*}
such the first isomorphism is Hecke-equivariant, and the action of $T_{\ell}$ for each prime number $\ell$ and $S_n$ for each $n \in \N$ coprime to $N$ commutes with that of $\t{Gal}(\overline{\Q}/\Q)$.

\begin{thm}
\label{universal Hecke module}
The action of $T_{\ell}$ for each prime number $\ell$ and $S_n$ for each $n \in \N$ coprime to $N$ induces a well-defined faithful continuous $\Lambda_0$-linear $\t{Gal}(\overline{\Q}/\Q)$-equivariant action
\begin{eqnarray*}
  \T_N \times \int_{\Z_p}^{\boxplus} \Hil_{\t{et}}^1 \left( Y_1(N)_{\overline{\Q}}, \Fil_{k-2} \right)_{\t{free}} dk \to \int_{\Z_p}^{\boxplus} \Hil_{\t{et}}^1 \left( Y_1(N)_{\overline{\Q}}, \Fil_{k-2} \right)_{\t{free}} dk
\end{eqnarray*}
of $\T_N$.
\end{thm}

\begin{proof}
We put
\begin{eqnarray*}
  L & \coloneqq & \int_{\Z_p}^{\boxplus} \Hil_{\t{et}}^1 \left( Y_1(N)_{\overline{\Q}}, \Fil_{k-2} \right)_{\t{free}} dk \\
  M & \coloneqq & \prod_{k = 2}^{\infty} \Hil_{\t{et}}^1 \left( Y_1(N)_{\overline{\Q}}, \Fil_{k-2} \right)_{\t{free}}
\end{eqnarray*}
It follows from the natural isomorphism $(\t{T}_{k_0,N}^{\t{\'et}})_{\t{free}} \to \t{T}_{k_0,N}$ that the action of Hecke operators gives a well-defined faithful action
\begin{eqnarray*}
 \t{T}_{k_0,N} \times \Hil_{\t{et}}^1 \left( Y_1(N)_{\overline{\Q}},\Fil_{k-2} \right)_{\t{free}} \to \Hil_{\t{et}}^1 \left( Y_1(N)_{\overline{\Q}}, \Fil_{k-2} \right)_{\t{free}}
\end{eqnarray*}
for any $k_0 \in \N \cap [2,\infty)$. Therefore the action of $T_{\ell}$ for each prime number $\ell$ and $S_n$ for each $n \in \N$ coprime to $N$ induces a well-defined faithful action $\T_N \times M \to M$, which is continuous by the universality of a direct product, because it is given as the inverse limit of the continuous actions
\begin{eqnarray*}
  \t{T}_{\leq k_1,N} \times \prod_{k_0 = 2}^{k_1} \Hil_{\t{et}}^1 \left( Y_1(N)_{\overline{\Q}}, \Fil_{k-2} \right)_{\t{free}} \to \prod_{k_0 = 2}^{k_1} \Hil_{\t{et}}^1 \left( Y_1(N)_{\overline{\Q}}, \Fil_{k-2} \right)_{\t{free}}
\end{eqnarray*}
for each $k_1 \in \N \cap [2,\infty)$. Since $\t{T}_{k_0,N}$ is generated by Hecke operators for any $k_0 \in \N \cap [2,\infty)$, the $\Z_p$-subalgebra $A \subset \T_N$ generated by $T_{\ell}$ for each prime number $\ell$ and $S_n$ for each $n \in \N$ coprime to $N$ is dense by the definition of the inverse limit topology. Therefore for any $(T,c) \in \T_N \times L$, $T(c)$ lies in the closure of the image of $L$ in $M$. Since $L$ is compact and $M$ is Hausdorff, $T(c)$ lies in the image of $L$. Therefore the action of $T_{\ell}$ for each prime number $\ell$ and $S_n$ for each $n \in \N$ coprime to $N$ induces a well-defined faithful $\Z_p$-linear $\t{Gal}(\overline{\Q}/\Q)$-equivariant action $\T_N \times L \to L$, which is continuous because the topology of $\int_{\Z_p}^{\boxplus} \Hil_{\t{\'et}}^1(Y_1(N)_{\overline{\Q}}, \Fil_{k-2} dk)_{\t{Free}}$ coincides with the relative topology of $\prod_{k_0 = 2}^{\infty} \Hil_{\t{et}}^1(Y_1(N)_{\overline{\Q}}, \Fil_{k-2})_{\t{free}}$.
\end{proof}

We put $\Lambda_0 \T_N \coloneqq \Lambda_0 \hat{\otimes}_{\Z_p[[1 + N \Z_p]]} \T_N$. By the action of $\T_N$ in Theorem \ref{universal Hecke module}, we regard
\begin{eqnarray*}
  \int_{\Z_p}^{\boxplus} \Hil_{\t{et}}^1 \left( Y_1(N)_{\overline{\Q}}, \Fil_{k-2} \right)_{\t{free}} dk
\end{eqnarray*}
as a profinite $\Lambda_0 \T_N[\t{Gal}(\overline{\Q}/\Q)]$-module. It is a huge module, and we cut it by a slope condition in the next section.

\subsection{Restriction to Families of Finite Slope}
\label{Restriction to Families of Finite Slope}

Let $s \in \N \backslash \ens{0}$. We extract the component of slope $< s$ from the huge cohomology dealt with in the end of \S \ref{Interpolation along the Weight Spaces}. For conventions of Hecke algebras of finite slope, see \S \ref{p-adic Modular Forms and Hecke Algebras}. For each $k_0 \in \N \cap [2,\infty)$, we set
\begin{eqnarray*}
  \Hil_{\t{et}}^1 \left( Y_1(N)_{\overline{\Q}}, \Fil_{k_0-2} \right)^{< s} \coloneqq \left( \t{T}_{k_0,N}^{< s} \hat{\otimes}_{\t{T}_{k_0,N}} \Hil_{\t{et}}^1 \left( Y_1(N)_{\overline{\Q}}, \Fil_{k_0-2} \right) \right)_{\t{free}},
\end{eqnarray*}
where $\t{T}_{k_0,N}^{< s}$ is regarded as a profinite $\t{T}_{k_0,N}[\t{Gal}(\overline{\Q}/\Q)]$-algebra by the trivial action of $\t{Gal}(\overline{\Q}/\Q)$. It is a profinite $\t{T}_{k_0,N}^{< s}[\t{Gal}(\overline{\Q}/\Q)]$-modules finitely generated as $\Z_p$-modules. We put $\Lambda_0 \T_N^{< s} \coloneqq \Lambda_0 \hat{\otimes}_{\Z_p[[1 + N \Z_p]]} \T_N^{< s}$ (resp.\ $\Lambda_0 \T_N^{[< s]} \coloneqq \Lambda_0 \hat{\otimes}_{\Z_p[[1 + N \Z_p]]} \T_N^{[< s]}$), and regard it a profinite $\Lambda_0 \T_N$-algebra. We denote by
\begin{eqnarray*}
  \int_{\Z_p}^{\boxplus} \Hil_{\t{et}}^1 \left( Y_1(N)_{\overline{\Q}}, \Fil_{k-2} \right)^{< s} dk
\end{eqnarray*}
the image of the natural continuous homomorphism
\begin{eqnarray*}
  \Lambda_0 \T_N^{< s} \hat{\otimes}_{\Lambda_0 \T_N} \int_{\Z_p}^{\boxplus} \Hil_{\t{et}}^1 \left( Y_1(N)_{\overline{\Q}}, \Fil_{k-2} \right)_{\t{free}} dk \to \prod_{k = 2}^{\infty} \Hil_{\t{et}}^1 \left( Y_1(N)_{\overline{\Q}}, \Fil_{k-2} \right)^{< s},
\end{eqnarray*}
and regard it as a profinite $\Lambda_0 \T_N^{< s}[\t{Gal}(\overline{\Q}/\Q)]$-module.

\vspace{0.2in}
Let $k_0 \in \N \cap [2,\infty)$. The truncation maps
\begin{eqnarray*}
  \tau_{k_0,-} \colon \Lambda_0^{\N} & \to & \Lambda_0^{\N} \\
  (F_i)_{i = 0}^{\infty} & \mapsto & (F_0, \ldots, F_{k_0-2}, 0 ,0 \ldots)
\end{eqnarray*}
and
\begin{eqnarray*}
  \tau_{k_0,+} \colon \Lambda_0^{\N} & \to & \Lambda_0^{\N} \\
  (F_i)_{i = 0}^{\infty} & \mapsto & (0, \ldots, 0, F_{k_0-1}, F_{k_0}, \ldots)
\end{eqnarray*}
are continuous $\Lambda_0$-linear idempotents. For a $1$-cocycle
\begin{eqnarray*}
  c \colon \Gamma_1(N) \to \int_{\Z_p}^{\boxplus} \Leb_{k-2} dk,
\end{eqnarray*}
we consider the composite $c' \colon \Gamma_1(N) \to (\Lambda_0^{\N}, \rho_{\bullet - 2})$ of $c$ and the inverse of the homeomorphic $\Lambda_0$-linear $\Pi_0(p)$-equivariant isomorphism
\begin{eqnarray*}
  \prod_{k = 2}^{\infty} \t{sp}_k \colon (\Lambda_0^{\N}, \rho_{\bullet - 2}) \stackrel{\sim}{\to} \int_{\Z_p}^{\boxplus} \Leb_{k-2} dk.
\end{eqnarray*}
By the proof of Proposition \ref{comparison 2}, $\Gamma_1(N)$ is a finitely generated free group. Henceforth, we fix a basis $(\gamma_h)_{h = 1}^{d}$ of the finitely generated free group $\Gamma_1(N)$. By the isomorphism in Lemma \ref{cocycle}, we identify $\t{Z}^1(\Gamma_1(N),(\Lambda_0^{\N},\rho_{\bullet - 2}))$ with $(\Lambda_0^{\N})^d$. Put $c' = (c'_h)_{h = 1}^{d}$ through the identification. Set $\tau_{k_0,\pm}(c') \coloneqq (\tau_{k_0,\pm}(c'_h))_{h = 1}^{d}$, and denote by
\begin{eqnarray*}
  \tau_{k_0,\pm}(c) \colon \Gamma_1(N) \to \int_{\Z_p}^{\boxplus} \Leb_{k-2} dk,
\end{eqnarray*}
the $1$-cocycle obtained as the composite of $\tau_{k_0,\pm}(c')$ and $\prod_{k = 2}^{\infty} \t{sp}_k$. We obtain continuous idempotents $\tau_{k_0,-}$ and $\tau_{k_0,+}$ on
\begin{eqnarray*}
  \t{Z}^1 \left( \Gamma_1(N), \int_{\Z_p}^{\boxplus} \Leb_{k-2} dk \right)
\end{eqnarray*}
with $\tau_{k_0,-} + \tau_{k_0,+} = \t{id}$.

\begin{lmm}
\label{truncate}
If $p^s \mid N$, then the image of
\begin{eqnarray*}
  \tau_{k_0,-} \left( \t{Z}^1 \left( \Gamma_1(N), \int_{\Z_p}^{\boxplus} \Leb_{k-2} dk \right) \right)
\end{eqnarray*}
generates a dense $\Lambda_0 \T_N^{< s}$-submodule of
\begin{eqnarray*}
  \int_{\Z_p}^{\boxplus} \Hil_{\t{et}}^1 \left( Y_1(N)_{\overline{\Q}}, \Fil_{k-2} \right)^{< s} dk
\end{eqnarray*}
for any $k_0 \geq s+1$.
\end{lmm}

\begin{proof}
We always identify the group cohomologies and the corresponding \'etale cohomology by Proposition \ref{comparison 2}. Put
\begin{eqnarray*}
  L & \coloneqq & \int_{\Z_p}^{\boxplus} \Hil_{\t{et}}^1 \left( Y_1(N)_{\overline{\Q}}, \Fil_{k-2} \right)^{< s} dk \\
  \widetilde{M} & \coloneqq & \t{Z}^1 \left( \Gamma_1(N), \int_{\Z_p}^{\boxplus} \Leb_{k-2} dk \right) \\
  M & \coloneqq & \Hil_{\t{et}}^1 \left( Y_1(N)_{\overline{\Q}}, \int_{\Z_p}^{\boxplus} \Fil_{k-2} dk \right)
\end{eqnarray*}
and
\begin{eqnarray*}
  M_{\leq k_1} \coloneqq \prod_{k = 2}^{k_1} \Hil_{\t{et}}^1 \left( Y_1(N)_{\overline{\Q}}, \Fil_{k-2} \right)^{< s}
\end{eqnarray*}
for each $k_1 \in \N \cap [2,\infty)$. Let $\overline{c} \in L$. By the definition of the inverse limit topology, it suffices to verify that the image of $\overline{c}$ in $M_{\leq k_1}$ is contained in the $(\Lambda_0 \otimes_{\Z_p[[1 + N\Z_p]]} \t{T}_{\leq k_1,N}^{< s})$-submodule generated by the image of $\tau_{k_0,-}(\widetilde{M})$ for any $k_1 \in \N \cap [k_0,\infty)$. Let $k_1 \in \N \cap [p^s,\infty)$. By Proposition \ref{p^s/T_p} and Proposition \ref{right exact 2}, there is a lift
\begin{eqnarray*}
  c_{k_1} \in \T_N[X] \otimes_{\T_N} M
\end{eqnarray*}
of the image $\overline{c}_{k_1}$ of $\overline{c}$ in $M_{\leq k_1}$ with respect to the $\T_N$-algebra homomorphism $\T_N[X] \twoheadrightarrow \t{T}_{\leq k_1,N}^{< s} \colon X \mapsto p^sT_p^{-1}$. Put $c_{k_1} = \sum_{n = 0}^{n_0} X^n \otimes c_{k_1,n}$ for a $(c_{k_1,n})_{n = 0}^{n_0} \in M^{n_0+1}$. Take a lift $\tilde{c}_{k_1,n} \in \widetilde{M}$ of $c_{k_1,n}$ for each $n \in \N \cap [0,n_0]$. We have $\tilde{c}_{k_1,n} = \tau_{k_0,-}(\tilde{c}_{k_1,n}) + \tau_{k_0,+}(\tilde{c}_{k_1,n})$. As a consequence, $\overline{c}_{k_1}$ is decomposed into the sum of the images of
\begin{eqnarray*}
  & & \sum_{n = 0}^{n_0} X^n \otimes \tau_{k_0,-}(\tilde{c}_{k_1,n}) \in \Z_p[X] \otimes_{\Z_p} \tau_{k_0,-}(\widetilde{M}) \\
  & & \sum_{n = 0}^{n_0} X^n \otimes \tau_{k_0,+}(\tilde{c}_{k_1,n}) \in \Z_p[X] \otimes_{\Z_p} \tau_{k_0,+}(\widetilde{M}),
\end{eqnarray*}
and hence it suffices to verify that the image of $\tau_{k_0,+}(\widetilde{M})$ in $M_{\leq k_1}$ is contained in the $(\Lambda_0 \otimes_{\Z_p[[1 + N \Z_p]]} \t{T}_{\leq k_1,N}^{< s})$-submodule generated by $\tau_{k_0,-}(\widetilde{M})$. Before that, we calculate the image of $\tau_{k_0,+}(\widetilde{M})$ by $T_p$.

\vspace{0.2in}
By the homeomorphic $\Lambda_0$-linear $\Pi_0(p)$-equivariant isomorphism
\begin{eqnarray*}
  \prod_{k = 2}^{\infty} \t{sp}(k) \colon (\Lambda_0^{\N}, \rho_{\bullet - 2}) \stackrel{\sim}{\to} \int_{\Z_p}^{\boxplus} \Leb_{k-2} dk.
\end{eqnarray*}
and Lemma \ref{cocycle}, we identify $\widetilde{M}$ with $(\Lambda_0^{\N})^d$. The double coset decomposition corresponding to $T_p$ is given as
\begin{eqnarray*}
  \Gamma_1(N)
  \left(
    \begin{array}{cc}
      1 & 0 \\
      0 & p
    \end{array}
  \right)
  \Gamma_1(N)
  =
  \bigsqcup_{\theta = 0}^{p-1} \Gamma_1(N)
  \left(
    \begin{array}{cc}
      1 & \theta \\
      0 & p
    \end{array}
  \right).
\end{eqnarray*}
Let $h \in \N \cap [1,d]$ and $\theta \in \N \cap [0,p-1]$. Since the right hand side is a double coset of the action of $\Gamma_1(N)$, there is a unique $\delta(h,\theta) \in \N \cap [1,p-1]$ satisfying
\begin{eqnarray*}
  A_{h,\theta} \coloneqq
  \left(
    \begin{array}{cc}
      1 & \theta \\
      0 & p
    \end{array}
  \right)
  \gamma_h
  \left(
    \begin{array}{cc}
      1 & \delta(h,\theta) \\
      0 & p
    \end{array}
  \right)^{-1}
  \in \Gamma_1(N).
\end{eqnarray*}
We denote by $\ell_{h,\theta} \in \N$ the word length of $A_{h,\theta}$ with respect to the basis $(\gamma_h)_{h = 1}^{d}$. Put
\begin{eqnarray*}
  A_{h,\theta} = \gamma_{h'_{h,\theta,1}}^{\sigma_{h,\theta,1}} \cdots \gamma_{h'_{h,\theta,\ell_{h,\theta}}}^{\sigma_{h,\theta,\ell_{h,\theta}}}
\end{eqnarray*}
for unique $(h'_{h,\theta,n})_{n = 1}^{\ell_{h,\theta}} \in (\N \cap [1,d])^{\ell_{h,\theta}}$ and $(\sigma_{h,\theta,n})_{n = 1}^{\ell_{h,\theta}} \in \ens{1, -1}^{\ell_{h,\theta}}$ with $\gamma_{h'_{h,\theta,n}}^{\sigma_{h,\theta,n}} \gamma_{h'_{h,\theta,n+1}}^{\sigma_{h,\theta,n+1}} \neq 1$ for any $n \in \N \cap [1,\ell_{h,\theta}-1]$. We define an $(A_{h,\theta,n})_{n = 1}^{\ell_{h,\theta}} \in \Gamma_1(N)^{\ell_{h,\theta}}$ by setting
\begin{eqnarray*}
  A_{h,\theta,n} \coloneqq
  \left\{
    \begin{array}{ll}
      \gamma_{h'_{h,\theta,1}}^{\sigma_{h,\theta,1}} \cdots \gamma_{h'_{h,\theta,n-1}}^{\sigma_{h,\theta,n-1}} & (\sigma_{h,\theta,n} = 1) \\
      & \\
      \gamma_{h'_{h,\theta,1}}^{\sigma_{h,\theta,1}} \cdots \gamma_{h'_{h,\theta,n-1}}^{\sigma_{h,\theta,n-1}} \gamma_{h'_{h,\theta,n}}^{-1} & (\sigma_{h,\theta,n} = -1)
    \end{array}
  \right.
\end{eqnarray*}
for each $n \in \N \cap [1,\ell_{h,\theta}]$. For any $1$-cocycle $\tilde{c} \colon \Gamma_1(N) \to (\Lambda_0^{\N},\rho_{\bullet - 2})$, we have
\begin{eqnarray*}
  & & \tilde{c}(A_{h,\theta}) = \tilde{c} \left( \gamma_{h'_{h,\theta,1}}^{\sigma_{h,\theta,1}} \cdots \gamma_{h'_{h,\theta,\ell_{h,\theta}}}^{\sigma_{h,\theta,\ell_{h,\theta}}} \right) = \sum_{n = 1}^{\ell_{h,\theta}-1} \rho_{\bullet - 2} \left( \gamma_{h'_{h,\theta,1}}^{\sigma_{h,\theta,1}} \cdots \gamma_{h'_{h,\theta,n-1}}^{\sigma_{h,\theta,n-1}}, \tilde{c} \left( \gamma_{h'_{h,\theta,n}}^{\sigma_{h,\theta,n}} \right) \right) \\
  & = & \sum_{n = 1}^{\ell_{h,\theta}-1} \sigma_{h,\theta,n} \rho_{\bullet - 2 } \left( A_{h,\theta,n}, \tilde{c}(\gamma_{h'_{h,\theta,n}}) \right)
\end{eqnarray*}
for any $h \in \N \cap [1,d]$ and $\theta \in \N \cap [0,p-1]$. Therefore the double coset operator $T_p$ with respect to the presentation of the decomposition above acts on $\widetilde{M}$ as
\begin{eqnarray*}
  T_p \colon (\Lambda_0^{\N})^d & \to & (\Lambda_0^{\N})^d \\
  \left( \left( F_{h,i} \right)_{i = 0}^{\infty} \right)_{h = 1}^{d} & \mapsto &
  \left(
     \sum_{\theta = 0}^{p-1} \rho_{\bullet - 2}
    \left(
      \left(
        \begin{array}{cc}
          1 & \theta \\
          0 & p
        \end{array}
      \right)^{\iota},
      \sum_{n = 1}^{\ell_{h,\theta}} \sigma_{h,\theta,n} \rho_{\bullet - 2} \left( A_{h,\theta,n}, \left( F_{h'_{h,\theta,n},i} \right)_{i = 0}^{\infty} \right)
    \right)
  \right)_{h = 1}^{d}.
\end{eqnarray*}
For any $(c,n) \in N \Z_p \times (\N \backslash \ens{0})$, we have
\begin{eqnarray*}
  \vv{\frac{c^n}{n!}} & = & \vv{c}^n \vv{p}^{- \sum_{h = 1}^{\infty} \left\lfloor \frac{n}{p^h} \right\rfloor} =
  \left\{
    \begin{array}{ll}
      \v{c}^n & (1 \leq n \leq p-1) \\
      \v{c}^{\frac{n(p-2)}{p-1}} \vv{p}^{\frac{n}{p-1} - \sum_{h = 1}^{\infty} \left\lfloor \frac{n}{p^h} \right\rfloor} & (n > p-1)
    \end{array}
  \right. \\
  & \leq & \v{c} \leq \v{N}
\end{eqnarray*}
by $p \neq 2$. For any $(A,F) \in \Gamma_1(N) \times \tau_{k_0,+}(\Lambda_0^{\N})$ with 
$
  A =
  \left(
    \begin{array}{cc}
      a & b \\
      c & d
    \end{array}
  \right)
$
 and $F = (F_i)_{i = 0}^{\infty}$, we have
\begin{eqnarray*}
  \rho_{\bullet - 2}(A,F) & = & \left( \sum_{j = k_0-1}^{\infty} F_j \sum_{h = 0}^{\min \Ens{i,j}}
  \left(
    \begin{array}{c}
      i \\
      h
    \end{array}
  \right)
  \left( \prod_{m = i}^{i+j-h-1} (z-2-m) \right) a^h b^{i-h} \frac{c^{j-h}}{(j-h)!} d^{z-2-i-j+h} \right)_{i = 0}^{\infty} \\
  & \in & N(z - k_0) \tau_{k_0,-}(\Lambda_0^{\N}) \oplus \tau_{k_0,+}(\Lambda_0^{\N}) \subset p^s(z - k_0) \tau_{k_0,-}(\Lambda_0^{\N}) \oplus \tau_{k_0,+}(\Lambda_0^{\N})
\end{eqnarray*}
by the condition $p^s \mid N$. For any $\theta \in \N \cap [0,p-1]$ and $F \in p^s (z-k_0) \tau_{k_0,-}(\Lambda_0^{\N}) \oplus \tau_{k_0,+}(\Lambda_0^{\N})$ with $F = (F_i)_{i = 0}^{\infty}$, we have
\begin{eqnarray*}
  & & \rho_{\bullet - 2} \left(
    \left(
      \begin{array}{cc}
        1 & \theta \\
        0 & p
      \end{array}
    \right)^{\iota},
    F  
  \right) \\
  & = & \left( \sum_{j = 0}^{\infty} F_j \sum_{h = 0}^{\min \Ens{i,j}}
  \left(
    \begin{array}{c}
      i \\
      h
    \end{array}
  \right)
  \left( \prod_{m = i}^{i+j-h-1} (z-2-m) \right) p^h (-\theta)^{i-h} \frac{0^{j-h}}{(j-h)!} 1^{z-i-j+h} \right)_{i = 0}^{\infty} \\
  & = & \left( \sum_{j = 0}^{i} F_j
  \left(
    \begin{array}{c}
      i \\
      j
    \end{array}
  \right)
  p^j (-\theta)^{i-j} \right)_{i = 0}^{\infty} \in p^s (z - k_0) \tau_{k_0,-}(\Lambda_0^{\N}) \oplus p^s \tau_{k_0,+}(\Lambda_0^{\N})
\end{eqnarray*}
by the condition $k_0 \geq s+1$. Therefore the image of $\tau_{k_0,+}(\widetilde{M})$ by $T_p$ is contained in $p^s (z-k_0) \tau_{k_0,-}(\widetilde{M}) \oplus p^s \tau_{k_0,+}(\widetilde{M}) \subset p^s \widetilde{M}$.

\vspace{0.2in}
Let $\tilde{c} \in \tau_{k_0,+}(\widetilde{M})$, and put $\tilde{c} = (\tilde{c}_1, \ldots, \tilde{c}_d) (\Lambda_0^{\N})^d$ by the identification. Set $\tilde{c}[0] \coloneqq \tilde{c}$, and put $T_p(\tilde{c}[0]) = p^s \tilde{d}$ by a $\tilde{d} \in (z - k_0) \tau_{k_0,-}(\widetilde{M}) \oplus \tau_{k_0,+}(\widetilde{M})$. As an equality of the images in $M_{\leq k_1}$, we have
\begin{eqnarray*}
  p^s \tilde{c}[0] = T_pX \otimes \tilde{c} = X \otimes T_p \tilde{c} = p^sX \otimes \tilde{d},
\end{eqnarray*}
and hence $\tilde{c}[0] = X \otimes \tilde{d}$ because $M_{\leq k_1}$ is torsionfree as a $\Z_p$-module. Put $\tilde{c}[1] \coloneqq \tau_{k_0,+}(\tilde{d})$. Then the image of $\tilde{c}[0]$ in $M_{\leq k_1}$ coincides with that of
\begin{eqnarray*}
  X \otimes \tau_{k_0,-}(\tilde{d}) + X \otimes \tilde{c}[1] \in \Z_p[X] \otimes_{\Z_p} \left( (z-k_0) \tau_{k_0,-}(\widetilde{M}) \oplus \tau_{k_0,+}(\widetilde{M}) \right)
\end{eqnarray*}
Repeating similar calculations, we obtain a $\tilde{c}[n] \in \tau_{k_0,+}(\widetilde{M})$ such that the image of $\tilde{c}[0]$ coincides with that of $X^n \otimes \tilde{c}[n]$ modulo the $(\Lambda_0 \otimes_{\Z_p[[1 + N \Z_p]]} \t{T}_{\leq k_1,N}^{[< s]})$-submodule generated by the image of $(z-k_0) \tau_{k_0,-}(\widetilde{M})$. Therefore the image of $\tilde{c}[0]$ in $M_{\leq k_1}$ lies in the $p$-adic closure of the $(\Lambda_0 \otimes_{\Z_p[[1 + N \Z_p]]} \t{T}_{\leq k_1,N}^{< s})$-submodule generated by the image of $(z-k_0) \tau_{k_0,-}(\widetilde{M})$, because the action of $p^sT_p^{-1}$ on the image of $\widetilde{M}$ in $M_{\leq k_1}$ is topologically nilpotent with respect to the $p$-adic topology by the proof of Proposition \ref{topologically nilpotent}. Since $M_{\leq k_1}$ is finitely generated as a $\Z_p$-module, every $\Z_p$-submodule is $p$-adically closed. Thus the assertion holds.
\end{proof}

\begin{thm}
\label{finiteness}
Suppose $p^s \mid N$. Then the profinite $\Lambda_0 \T_N^{< s}[\t{Gal}(\overline{\Q}/\Q)]$-module
\begin{eqnarray*}
  \int_{\Z_p}^{\boxplus} \Hil_{\t{et}}^1 \left( Y_1(N)_{\overline{\Q}}, \Fil_{k-2} \right)^{< s} dk
\end{eqnarray*}
is finitely generated as a $(\Lambda_0 \T_N^{< s})$-module.
\end{thm}

\begin{proof}
Put
\begin{eqnarray*}
  M & \coloneqq & \tau_{s+1,-} \left( \t{Z}^1 \left( \Gamma_1(N), \int_{\Z_p}^{\boxplus} \Leb_{k-2} dk \right) \right) \\
  L & \coloneqq & \int_{\Z_p}^{\boxplus} \Hil_{\t{et}}^1 \left( Y_1(N)_{\overline{\Q}}, \Fil_{k-2} \right)^{< s} dk.
\end{eqnarray*}
The continuous $\Lambda_0$-linear homomorphism $(\Lambda_0^{s+1})^d \to M$ obtained by the embeddings
\begin{eqnarray*}
  (\Lambda_0^{s+1})^d & \hookrightarrow & (\Lambda_0^{\N})^d \\
  ((F_{h,i})_{i = 0}^{s})_{h = 1}^{d} & \mapsto & ((F_{h,0}, \ldots, F_{f,s}, 0, \ldots))_{h = 1}^{d} \\
  M & \hookrightarrow & (\Lambda_0^{\N})^d \\
  c & \mapsto & \left( \left( \prod_{k = 2}^{\infty} \t{sp}_k \right)^{-1}(c(\gamma_h)) \right)_{h = 1}^{d}
\end{eqnarray*}
is a homeomorphic isomorphism by the definition of $M$, and hence $M$ is a finitely generated free $\Lambda_0$-module. Let $E \subset M$ be a $\Lambda_0$-linear basis. For each $c \in E$, we denote by $\overline{c} \in L$ the image of $c$. By Lemma \ref{truncate}, the image of the continuous $(\Lambda_0 \T_N^{< s})$-linear homomorphism
\begin{eqnarray*}
  \varpi \colon \left( \Lambda_0 \T_N^{< s} \right)^E & \to & L \\
  (F_c)_{c \in E} & \mapsto & \sum_{c \in E} F_c \overline{c}
\end{eqnarray*}
generates a dense $(\Lambda_0 \T_N^{< s})$-submodule of $L$. Since $\Lambda_0 \T_N^{< s}$ is compact and $L$ is Hausdorff, the image of $\varpi$ is closed. Thus $\varpi$ is surjective, and $L$ is generated by the image of the finite set $E$ as a $(\Lambda_0 \T_N^{< s})$-module.
\end{proof}

\begin{dfn}
We set
\begin{eqnarray*}
  & & \t{H}_{\t{et}}^1 \left( Y_1(N)_{\overline{\Q}}, \t{Sym}^{k_0-2} \left( \t{R}^1 (\pi_N)_* (\underline{\Q}{}_p)_{E_1(N)} \right) \right)^{< s} \\
  & \coloneqq & \Q_p \otimes_{\Z_p} \Hil_{\t{et}}^1 \left( Y_1(N)_{\overline{\Q}}, \Fil_{k-2} \right)^{< s} \\
  & & \int_{\Z_p}^{\boxplus} \t{H}_{\t{et}}^1 \left( Y_1(N)_{\overline{\Q}}, \t{Sym}^{k-2} \left( \t{R}^1 (\pi_N)_* (\underline{\Q}{}_p)_{E_1(N)} \right) \right)^{< s} dk \\
  & \coloneqq & \Q_p \otimes_{\Z_p} \left( \int_{\Z_p}^{\boxplus} \Hil_{\t{et}}^1 \left( Y_1(N)_{\overline{\Q}}, \Fil_{k-2} \right)^{< s} dk \right),
\end{eqnarray*}
and regard them as $(\Q_p \otimes_{\Z_p} \Lambda_0 \T_N^{< s})$-modules endowed with a $(\Q_p \otimes_{\Z_p} \Lambda_0 \T_N^{< s})$-linear action of $\t{Gal}(\overline{\Q}/\Q)$.
\end{dfn}

\begin{rmk}
\label{symmetric product 5}
By Remark \ref{symmetric product 4}, we have a natural projection
\begin{eqnarray*}
  & & \t{H}_{\t{et}}^1 \left( Y_1(N)_{\overline{\Q}}, \t{Sym}^{k_0-2} \left( \t{R}^1 (\pi_N)_* (\underline{\Q}{}_p)_{E_1(N)} \right) \right) \\
  & \twoheadrightarrow & \t{H}_{\t{et}}^1 \left( Y_1(N)_{\overline{\Q}}, \t{Sym}^{k_0-2} \left( \t{R}^1 (\pi_N)_* (\underline{\Q}{}_p)_{E_1(N)} \right) \right)^{< s}
\end{eqnarray*}
corresponding to the canonical projection $\Q_p \otimes_{\Z_p} \t{T}_{k_0,N}^{\t{\'et}} \cong \Q_p \otimes_{\Z_p} \t{T}_{k_0,N} \twoheadrightarrow \Q_p \otimes_{\Z_p} \t{T}_{k_0,N}^{[< s]} \cong \Q_p \otimes_{\Z_p} \t{T}_{k_0,N}^{< s}$.
\end{rmk}

We denote by $P_{k_0}^{< s}$ the kernel of the specialisation map $\Lambda_0 \T_N^{< s} \twoheadrightarrow \t{T}_{k_0,N}^{< s}$.

\begin{lmm}
\label{specialisation 3}
If $p^s \mid N$, then for any $k_0 \in \N \cap [2,\infty)$, the specialisation map
\begin{eqnarray*}
  \int_{\Z_p}^{\boxplus} \Hil_{\t{et}}^1 \left( Y_1(N)_{\overline{\Q}}, \Fil_{k-2} \right)^{< s} dk \hookrightarrow \prod_{k = 2}^{\infty} \Hil_{\t{et}}^1 \left( Y_1(N)_{\overline{\Q}}, \Fil_{k-2} \right)^{< s} \twoheadrightarrow \Hil_{\t{et}}^1 \left( Y_1(N)_{\overline{\Q}}, \Fil_{k_0-2} \right)^{< s}
\end{eqnarray*}
is a surjective continuous $\Lambda_0 \T_N^{< s}$-linear $\t{Gal}(\overline{\Q}/\Q)$-equivariant homomorphism, and if $k_0 \geq \max \ens{s+1,3}$, then the kernel of its localisation
\begin{eqnarray*}
  & & \int_{\Z_p}^{\boxplus} \t{H}_{\t{et}}^1 \left( Y_1(N)_{\overline{\Q}}, \t{Sym}^{k-2} \left( \t{R}^1 (\pi_N)_* (\underline{\Q}{}_p)_{E_1(N)} \right) \right)^{< s} dk \\
  & \to & \t{H}_{\t{et}}^1 \left( Y_1(N)_{\overline{\Q}}, \t{Sym}^{k_0-2} \left( \t{R}^1 (\pi_N)_* (\underline{\Q}{}_p)_{E_1(N)} \right) \right)^{< s}
\end{eqnarray*}
coincides with
\begin{eqnarray*}
  P_{k_0}^{< s} \left( \int_{\Z_p}^{\boxplus} \t{H}_{\t{et}}^1 \left( Y_1(N)_{\overline{\Q}}, \t{Sym}^{k-2} \left( \t{R}^1 (\pi_N)_* (\underline{\Q}{}_p)_{E_1(N)} \right) \right)^{< s} dk \right).
\end{eqnarray*}
\end{lmm}

\begin{proof}
Let $\varphi$ denote the homomorphism in the assertion. Put
\begin{eqnarray*}
  L & \coloneqq & \int_{\Z_p}^{\boxplus} \Hil_{\t{et}}^1 \left( Y_1(N)_{\overline{\Q}}, \Fil_{k-2} \right)^{< s} dk \\
  M & \coloneqq & \int_{\Z_p}^{\boxplus} \Hil_{\t{et}}^1 \left( Y_1(N)_{\overline{\Q}}, \Fil_{k-2} \right)_{\t{free}} dk
\end{eqnarray*}
and
\begin{eqnarray*}
  M_{k_1} & \coloneqq & \Hil_{\t{et}}^1 \left( Y_1(N)_{\overline{\Q}}, \Fil_{k_1-2} \right)_{\t{free}} \\
  M_{k_1}^{[< s]} & \coloneqq & \left( \t{T}_{k_1,N}^{[< s]} \otimes_{\t{T}_{k_1,N}} M_{k_1} \right)_{\t{free}} \\
  M_{k_1}^{< s} & \coloneqq & \Hil_{\t{et}}^1 \left( Y_1(N)_{\overline{\Q}}, \Fil_{k_1-2} \right)^{< s}
\end{eqnarray*}
for each $k_1 \in \N \cap [2,\infty)$. The natural continuous $\Lambda_0 \T_N$-linear homomorphisms
\begin{eqnarray*}
  \Lambda_0 \T_N^{< s} \hat{\otimes}_{\Lambda_0 \T_N} M & \to & \Lambda_0 \T_N^{< s} \hat{\otimes}_{\Lambda_0 \T_N} M_{k_1} \\
  \Lambda_0 \T_N^{< s} \hat{\otimes}_{\Lambda_0 \T_N} M_{k_1} & \to & \t{T}_{k_0,N}^{< s} \hat{\otimes}_{\Lambda_0 \T_N} M_{k_1} \cong M_{k_1}^{< s}
\end{eqnarray*}
are surjective by Proposition \ref{right exact}, Proposition \ref{specialisation 2}, and the surjectivity of the canonical projections $\Lambda_0 \T_N \twoheadrightarrow \t{T}_{k_0,N}$ and $\Lambda_0 \T_N^{< s} \twoheadrightarrow \t{T}_{k_0,N}^{< s}$. Therefore $\varphi$ is surjective. Before calculating the kernel of $\Q_p \otimes_{\Z_p} \varphi$, we verify that the natural $\t{T}_{k_1,N}^{[< s]}$-linear homomorphism $M_{k_1}^{[< s]} \to M_{k_1}^{< s}$ is injective for any $k_1 \in \N \cap [2,\infty)$.

\vspace{0.2in}
Let $\overline{c}_{k_1}^{[< s]} \in M_{k_1}^{[< s]}$ be an element whose image $\overline{c}_{k_1}^{< s}$ in $M_{k_1}^{< s}$ is $0$. By the definition of $M_{k_1}^{< s}$, the image of $\overline{c}_{k_1}^{[< s]}$ in $(\t{T}_{k_1,N}^{[< s]}[X]/(T_p X - p^s)) \otimes_{\t{T}_{k_1,N}^{[< s]}} M_{k_1}^{[< s]}$ is annihilated by $p^r$ for an $r \in \N$, and hence the image of $p^r \overline{c}_{k_1}^{[< s]}$ in $\t{T}_{k_1,N}^{[< s]}[X] \otimes_{\t{T}_{k_1,N}^{[< s]}} M_{k_1}^{[< s]}$ lies in $(T_p X - p^s) \t{T}_{k_1,N}^{[< s]}[X] \otimes_{\t{T}_{k_1,N}^{[< s]}} M_{k_1}^{[< s]}$. We have an identification $\t{T}_{k_1,N}^{[< s]}[X] \otimes_{\t{T}_{k_1,N}^{[< s]}} M_{k_1}^{[< s]} \cong (M_{k_1}^{[< s]})^{\oplus \N}$ given by the basis $(X^h)_{h = 0}^{\infty}$ of the free $\t{T}_{k_1,N}^{[< s]}$-module $\t{T}_{k_1,N}^{[< s]}[X]$. Since $T_p$ is integral over $\Z_p$ as an element of $\t{T}_{k_1,N}^{[< s]}$, there is an $(n,A) \in \N \times \t{T}_{k_1,N}^{[< s]}$ such that $AT_p = p^n$. Since $M_{k_1}^{[< s]}$ is torsionfree as a $\Z_p$-module, the equality $AT_p = p^n$ ensures that the endomorphism on $M_{k_1}^{[< s]}$ given by $T_p$ is injective. Therefore we obtain
\begin{eqnarray*}
  M_{k_1}^{[< s]} \cap \left( (T_p X - p^s)\t{T}_{k_1,N}^{[< s]}[X] \otimes_{\t{T}_{k_1,N}^{[< s]}} M_{k_1}^{[< s]} \right) = 0.
\end{eqnarray*}
It implies $p^r \overline{c}_{k_1}^{[< s]} = 0$. Since $M_{k_1}^{[< s]}$ is torsionfree as a $\Z_p$-module, we get $\overline{c}_{k_1}^{[< s]} = 0$.

\vspace{0.2in}
In the following, for each $\overline{c} \in M$ and $k_1 \in \N \cap [2,\infty)$, and for each formal symbol $(\kappa,\sigma) \in \ens{k_0, \leq k_1} \times \ens{\emptyset, [< s], < s}$, we denote by $\overline{c}_{\kappa}^{\sigma}$ the image of $\overline{c}$ in $\prod_{k = 2}^{k_1} M_k^{\sigma}$ when $\kappa$ is the formal symbol $\leq k_1$, and in $M_{k_0}^{\sigma}$ when $\kappa$ is the formal symbol $k_0$. Similarly, for each $\overline{c}^{< s} \in L$ and $k_1 \in \N \cap [2,\infty)$, and for each formal symbol $\kappa \in \ens{k_0, \leq k_1}$, we denote by $\overline{c}_{\kappa}^{< s}$ the image of $\overline{c}^{< s}$ in $\prod_{k = 2}^{k_1} M_k^{< s}$ when $\kappa$ is the formal symbol $\leq k_1$, and in $M_{k_0}^{< s}$ when $\kappa$ is the formal symbol $k_0$. Let $\overline{c}[0]^{< s} \in \ker(\Q_p \otimes_{\Z_p} \varphi)$. We prove $\overline{c}[0]^{< s} \in P_{k_0}^{< s} (\Q_p \otimes_{\Z_p} L)$. Multiplying $\overline{c}[0]^{< s}$ by $p$ sufficiently many times, we may assume that $\overline{c}[0]^{< s}$ lies in the image of $L$. Let $q_1 \in \N \backslash \ens{0}$ denote the order of the finite Abelian $p$-group $\t{tor}_p(\t{T}_{k_1,N}^{[< s]} \otimes_{\t{T}_{k_1,N}} M_{k_1})$, and $q_2 \in \N \backslash \ens{0}$ the order of the finite Abelian $p$-group $\t{tor}_p(\t{H}^1(\Gamma_1(N),\Leb^{k_0-2}))$. We prove $q_1 q_2 \overline{c}[0]^{< s} \in P_{k_0}^{< s} L$. Since $P_{k_0}^{< s}$ is the kernel of a continuous homomorphism, the profiniteness of $\Lambda_0 \T_N^{< s}$ ensures that of $P_{k_0}^{< s}$. Since $P_{k_0}^{< s}$ is profinite and the map
\begin{eqnarray*}
  (P_{k_0}^{< s})^l & \to & L \\
  (a_j)_{j = 1}^{l} & \mapsto & \sum_{j = 1}^{l} a_j \overline{d}_j
\end{eqnarray*}
is continuous, its image $P_{k_0}^{< s} L$ is closed. Therefore in order to prove $q_1 q_2 \overline{c}[0]^{< s} \in P_{k_0}^{< s} L$, it suffices to verify that $q_1 q_2 \overline{c}[0]^{< s}_{\leq k_1}$ lies in the image of $P_{k_0}^{< s} L$ for any $k_1 \in \N \cap [k_0+1,\infty)$. By the definition of $L$, there is an $(n,\overline{c}[1]) \in \N \times M$ such that $(p^sT_p^{-1})^n \overline{c}[1]_{\leq k_1}^{< s} = \overline{c}[0]_{\leq k_1}^{< s}$. In particular, $T_p^n \overline{c}[0]_{\leq k_1}^{< s} = p^{ns} \overline{c}[1]_{\leq k_1}^{< s}$ lies in the image of the natural $\t{T}_{\leq k_1}(N)^{[< s]}$-linear homomorphism $\prod_{k = 2}^{k_1} M_k^{[< s]} \to \prod_{k = 2}^{k_1} M_k^{< s}$, which is injective by the previous argument. We have
\begin{eqnarray*}
  p^{ns} \overline{c}[1]_{k_0}^{< s} = T_p^n (p^sT_p^{-1})^n \overline{c}[1]_{k_0}^{< s} = T_p^n \overline{c}[0]_{k_0}^{< s} = T_p^n \varphi(\overline{c}^{< s}) = 0,
\end{eqnarray*}
and hence $\overline{c}[1]_{k_0}^{< s} = 0$ because $\prod_{k = 2}^{k_1} M_k^{< s}$ is torsion free as a $\Z_p$-module. It implies $\overline{c}[1]_{k_0}^{[< s]} = 0$ by the injectivity of the natural $\t{T}_{k_0,N}$-linear homomorphism $M_{k_0}^{[< s]} \to M_{k_0}^{< s}$. Therefore the image of $\overline{c}[1]$ in $\t{T}_{k_0,N}^{[< s]} \otimes_{\t{T}_{k_0,N}} M_{k_0}$ lies in $\t{tor}_p(\t{T}_{k_1,N}^{[< s]} \otimes_{\t{T}_{k_1,N}} M_{k_1})$, and hence that of $q_1 \overline{c}[1]$ is $0$. Let $S \subset \t{T}_{k_0,N}$ be a finite subset of generators of the kernel of the canonical projection $\t{T}_{k_0,N} \twoheadrightarrow \t{T}_{k_0,N}^{[< s]}$. Then the image of $q_1 \overline{c}[1]$ in $M_{k_0}$ lies in $\sum_{A \in S} A M_{k_0}$. Take a lift $\tilde{S} \subset \T_N$ of $S$. By the definition of $S$, the image of $\tilde{S}$ in $\T_N^{< s}$ is contained in $P_{k_0}^{< s}$. Take a lift $\overline{c}[2] \in \sum_{\tilde{s} \in \tilde{S}} \tilde{s} M \subset M$ of the image of $q_1 \overline{c}[1]_{k_0}$. Since the image of $\tilde{S}$ in $\T_N^{< s}$ is contained in $P_{k_0}^{< s}$, $\overline{c}[2]_{\leq k_1}^{< s}$ lies in the image of $P_{k_0}^{< s} L$.

\vspace{0.2in}
We identify the group cohomologies and the corresponding \'etale cohomology by Proposition \ref{comparison 2}. In the following, for each $1$-cocycle $c \colon \Gamma_1(N) \to (\Lambda_0^{\N},\rho_{\bullet-2})$, we denote by $c_{k_0} \colon \Gamma_1(N) \to \Leb_{k_0-2}$ the specialisation of $c$ at $k_0$, and by $\overline{c} \in M$ the image of the cohomology class of $c$. Take a $1$-cocycle $c[3] \colon \Gamma_1(N) \to (\Lambda_0^{\N},\rho_{\bullet - 2})$ representing $q \overline{c}[1] - \overline{c}[2]$ through the homeomorphic $\T_N$-linear isomorphism
\begin{eqnarray*}
  \prod_{k = 2}^{\infty} \t{sp}_k \colon (\Lambda_0^{\N},\rho_{\bullet - 2}) \stackrel{\sim}{\to} \int_{\Z_p}^{\boxplus} \Leb_{k-2} dk.
\end{eqnarray*}
Since $\overline{c}[3]_{k_0} = q_1 \overline{c}[1]_{k_0} - \overline{c}[2]_{k_0} = 0$, the cohomology class of $c[3]_{k_0}$ is annihilated by $p^r$ for an $r \in \N$, and hence $q_2 c[3]_{k_0}$ is a $1$-coboundary. Take a $b \in \Leb_{k_0-2}$ such that the $1$-coboundary $\partial b$ associated to $b$ coincides with $q_2 c[3]_{k_0}$. Let $\tilde{b}$ denote the image of $b$ by the $\Z_p$-linear embedding
\begin{eqnarray*}
  \t{Sym}_0^{k_0-2}(\Z_p^2) & \hookrightarrow & \Lambda_0^{\N} \\
  \sum_{i = 0}^{k_0-2} b_i e_{k_0-2,i} & \mapsto & (b_0, \ldots, b_{k_0-2}, 0 ,0, \ldots).
\end{eqnarray*}
We denote by $\partial \tilde{b} \in \t{B}^1(\Gamma_1(N),(\Lambda_0^{\N},\rho_{\bullet-2}))$ the $1$-coboundary associated to $\tilde{b}$. Then we have $(\partial \tilde{b})_{k_0} = \partial b = q_2 c[3]_{k_0}$. Therefore every value of the $1$-cocycle $q_2 c[3] - \partial \tilde{b}$ is an element of $\Lambda_0^{\N}$ whose specialisation at $k_0$ vanishes. By the factor theorem for a rigid analytic function, there is a set-theoretical map $c[4] \colon \Gamma_1(N) \to \Lambda_0^{\N}$ such that $(z-k_0)c[4] = q_2 c[3] - \partial \tilde{b}$. Since $\Lambda_0^{\N}$ is a torsionfree $\Lambda_0$-module, the cocycle condition for $q_2 c[3] - \partial \tilde{b}$ ensures that $c[4] \colon \Gamma_1(N) \to (\Lambda_0^{\N},\rho_{\bullet - 2})$ is a $1$-cocycle. We conclude
\begin{eqnarray*}
  q_1 q_2 \overline{c}[0]_{\leq k_1}^{< s} = (p^sT_p^{-1})^n \left( q_1 q_2 \overline{c}[1]_{\leq k_1}^{< s} \right) = (p^sT_p^{-1})^n \left( q_2 \overline{c}[2]_{\leq k_1}^{< s} + (z-k_0) \overline{c}[4]_{\leq k_1}^{< s} \right),
\end{eqnarray*}
and the right hand side lies in the image of $P_{k_0}^{< s} L$.
\end{proof}

\subsection{$p$-adic Family of Modular Forms of Finite Slope}
\label{p-adic Family of Modular Forms of Finite Slope}

A prime ideal of $\Lambda_0$ is said to be {\it of weight $k$} for a $k \in \Z_p$ if its preimage in $\Z_p[[1 + N \Z_p]]$ coincides with the prime ideal of height $1$ obtained as the kernel of the continuous $\Z_p$-algebra homomorphism $\Z_p[[1 + N \Z_p]] \to \Z_p$ associated to the continuous character
\begin{eqnarray*}
  1 + N \Z_p & \to & \Z_p^{\times} \\
  \gamma & \mapsto & \gamma^k
\end{eqnarray*}
of weight $k$ by the universality of the Iwasawa algebra.

\begin{prp}
Let $k_0 \in \Z_p$. For any $u \in \N \cap [0,p-2]$, the principal ideal
\begin{eqnarray*}
  m_{u,k_0} & \coloneqq & \left( (z - k_0) e_{(u \chi_{p,k_0 - u})^{(1)}} + \sum_{
  \t{\scriptsize $
    \begin{array}{c}
      \zeta = 0 \\
      \zeta \neq (u \chi_{p,k_0 - u})^{(1)}
    \end{array}
  $}
  }^{p(p-1)-1} e_{\zeta} \right) \Lambda_0 \\
  & = & \left( z - k_0 \right) e_{(u \chi_{p,k_0 - u})^{(1)}} \Lambda_0 \oplus \bigoplus_{
  \t{\scriptsize $
    \begin{array}{c}
      \zeta = 0 \\
      \zeta \neq (u \chi_{p,k_0 - u})^{(1)}
    \end{array}
  $}
  }^{p(p-1)-1} e_{\zeta} \Lambda_0 \subset \Lambda_0
\end{eqnarray*}
is a closed prime ideal of height $1$. Moreover, a prime ideal $m \subset \Lambda_0$ is of weight $k_0$ if and only if $m$ coincides with $m_{u,k_0}$ for a $u \in \N \cap [0,p-2]$.
\end{prp}

See the beginning of \S \ref{Interpolation along the Weight Spaces} for the convention of $\chi_{p,n}$.

\begin{proof}
Since $\Z_p[[X]]$ is Noetherian (resp.\ compact, resp.\ Hausdorff), so is $\Lambda_0$. Therefore every ideal of $\Lambda_0$ is closed. Let $u \in \N \cap [0,p-2]$. We have
\begin{eqnarray*}
  & & \Lambda_0/m_{u,k_0} \\
  & \cong & \left( \Z_p[[z_{(u \chi_{p,k_0 - u})^{(1)}} - (u \chi_{p,k_0 - u})^{(1)}]]/(z_{(u \chi_{p,k_0 - u})^{(1)}} - k_0) \right) \times \prod_{
  \t{\scriptsize $
    \begin{array}{c}
      \zeta = 0 \\
      \zeta \neq (u \chi_{p,k_0 - u})^{(1)}
    \end{array}
  $}
  }^{p(p-1)-1} \Z_p[[z_{\zeta} - \zeta]]/e_{\zeta} \Lambda_0 \\
  & \cong & \Z_p \times 0 \cong \Z_p,
\end{eqnarray*}
and hence $m_{u,k_0}$ is a closed prime ideal of height $1$. The composite $\varphi$ of the embedding $\Z_p[[1 + N \Z_p]] \hookrightarrow \Lambda_0$ and the canonical projection $\Lambda_0 \twoheadrightarrow \Lambda_0/m_{u,k_0} \cong \Z_p$ coincides with the continuous $\Z_p$-algebra homomorphism $\varphi_{k_0}$ associated to the continuous character $1 + N \Z_p \to \Z_p^{\times} \colon \gamma \mapsto \gamma^{k_0}$ by the universality of the Iwasawa algebra. It implies that $m_{u,k_0}$ is of weight $k_0$.

\vspace{0.2in}
On the other hand, let $m \subset \Lambda_0$ be a closed prime ideal of height $1$ of weight $k_0$. Since $\ker(\varphi_{k_0})$ does not contain $p$, we have $p \notin m$. Therefore $m \cap p \Lambda_0 = pm$ because $m$ is a prime ideal. By definition, we have $(1+N)^z - (1+N)^{k_0} \in m$ because $[1+N] - (1+N)^{k_0} \in \ker(\varphi_{k_0})$. Let $f \in 1 + p \t{C}(W,\Z_p)$. We have 
\begin{eqnarray*}
  \frac{1}{h}(f - 1)^h = \frac{p^h}{h} \left( \frac{f - 1}{p} \right)^h \in \Z_p \left\lbrack \frac{f-1}{p} \right\rbrack \cap p^{h - \lfloor \log_p h \rfloor} \t{C}(W,\Z_p) \subset \t{C}(W,\Q_p)
\end{eqnarray*}
for any $h \in \N \backslash \ens{0}$, and hence the infinite sum
\begin{eqnarray*}
  \log f \coloneqq \sum_{h = 1}^{\infty} \frac{1}{h}(f - 1)^h
\end{eqnarray*}
converges in $\t{C}(W,\Z_p)$. If $f \in 1 + p \Lambda_0$, then we have
\begin{eqnarray*}
  \frac{1}{h}(f - 1)^h \in \Z_p \left\lbrack \frac{f-1}{p} \right\rbrack \subset \Lambda_0,
\end{eqnarray*}
for any $h \in \N \backslash \ens{0}$, and hence $\log f \in \Lambda_0$ because $\Lambda_0$ is closed in $\t{C}(W,\Z_p)$. Since the embedding $\Lambda_0 \hookrightarrow \t{C}(W,\Z_p)$ is a homeomorphism onto the image, the infinite sum in the definition of $\log f$ converges to $\log f$ in $\Lambda_0$. For any $f_0 \in 1 + p \Lambda_0$ with $f - f_0 \in m$, we have
\begin{eqnarray*}
  \frac{1}{h}(f - 1)^h = \frac{p^h}{h} \left( \frac{f_0 - 1}{p} + \frac{f - f_0}{p} \right)^h \in \frac{1}{h}(f_0 - 1)^h + m
\end{eqnarray*}
because $f - f_0 \in m \cap p \Lambda_0 = pm$. Since $m$ is closed, we obtain
\begin{eqnarray*}
  \log f = \sum_{h = 1}^{\infty} \frac{1}{h}(f - 1)^h \in \sum_{h = 1}^{\infty} \left( \frac{1}{h}(f_0 - 1)^h + m \right) = \left( \sum_{h = 1}^{\infty} \frac{1}{h}(f_0 - 1)^h \right) + m = (\log f_0) + m.
\end{eqnarray*}
In particular, we obtain
\begin{eqnarray*}
  (z - k_0) \log (1+N) = \log (1+N)^z - \log (1+N)^{k_0} \in m
\end{eqnarray*}
by a usual calculation. Since $1 + N$ is not a root of unity, we have $\log (1 + N) \neq 0$. Therefore we obtain $\log (1 + N) \in \Z_p \backslash \ens{0} = \bigsqcup_{r \in \N} p^r \Z_p^{\times}$, and hence $z - k_0 \in m$ by $p \notin m$. It implies $m_{u,k_0} \subset m$ for some $u \in \N \cap [0,p-2]$, because we have
\begin{eqnarray*}
  (z - k_0) \Lambda_0 = \left( \prod_{u_0 = 0}^{p-2} \left( (z - k_0) e_{(u_0 \chi_{p,k_0 - u_0})^{(1)}} + \sum_{
  \t{\scriptsize $
    \begin{array}{c}
      \zeta = 0 \\
      \zeta \neq (u_0 \chi_{p,k_0 - u_0})^{(1)}
    \end{array}
  $}
  }^{p(p-1)-1} e_{\zeta} \right) \right) \Lambda_0.
\end{eqnarray*}
Since $m$ shares height with $m_{u,k_0}$, we conclude $m = m_{u,k_0}$.
\end{proof}

For a topological $\Lambda_0$-algebra $\Lambda_1$, we denote by $\Omega(\Lambda_1)$ the set of continuous $\Z_p$-algebra homomorphisms $\Lambda_1 \to \overline{\Z}_p$.

\begin{prp}
\label{Baire}
Let $\Lambda_1$ be a compact topological $\Lambda_0$-algebra. For any $\varphi \in \Omega(\Lambda_1)$, $\varphi(\Lambda_1)$ is a $\Z_p$-subalgebra of $\overline{\Z}_p$ finitely generated as a $\Z_p$-module.
\end{prp}

\begin{proof}
Let $\varphi \in \Omega(\Lambda_1)$. Since $\varphi$ is continuous, $\varphi(\Lambda_1)$ is a compact $\Z_p$-subalgebra of the Hausdorff topological $\Z_p$-algebra $\overline{\Z}_p$. Therefore $\varphi(\Lambda_1)$ is a compact Hausdorff topological $\Z_p$-algebra with respect to the relative topology. Let $\Fil$ denote the set of $\Z_p$-subalgebras $R$ of $\overline{\Z}_p$ integrally closed in $\Q_p \otimes_{\Z_p} R \subset \overline{\Q}_p$ and finitely generated as $\Z_p$-modules. The set $\Fil$ is directed by inclusions. We have $\bigcup_{R \in \Fil} R = \overline{\Z}_p$, and hence $\bigcup_{R \in \Fil} (R \cap \varphi(\Lambda_1)) = \varphi(\Lambda_1)$. For any $R \in \Fil$, $R$ is compact topological $\Z_p$-algebra, and hence $R \cap \varphi(\Lambda_1)$ is closed in $\varphi(\Lambda_1)$. By Krasner's lemma, every finite subextension of $\overline{\Q}_p/\Q_p$ can be obtained as the $p$-adic closure of a finite subextension of $\overline{\Q}_p/\Q$, and hence $\Fil$ is a countable set. By Baire category theorem for a $\check{\m{C}}$ech-complete (e.g.\ locally compact Hausdorff) topological space (\cite{Bai99} 59, \cite{Eng77} 3.9.3 Theorem), there exists some $R_0 \in \Fil$ such that $R_0 \cap \varphi(\Lambda_1)$ admits non-empty interior in $\varphi(\Lambda_1)$. It ensures that $R_0 \cap \varphi(\Lambda_1)$ is an open $\Z_p$-subalgebra of $\varphi(\Lambda_1)$. Since $\varphi(\Lambda_1)$ is compact, the quotient $\varphi(\Lambda_1)/(R_0 \cap \varphi(\Lambda_1))$ as additive groups is a finite group. Let $a_1, \ldots, a_d \in \varphi(\Lambda_1)$ be a complete representative of the canonical projection $\varphi(\Lambda_1) \twoheadrightarrow \varphi(\Lambda_1)/(R_0 \cap \varphi(\Lambda_1))$. Since $\bigcup_{R \in \Fil} (R \cap \varphi(\Lambda_1)) = \varphi(\Lambda_1)$, there exists some $(R_j)_{j = 1}^{d} \in \Fil^d$ such that $a_i \in R_i \cap \varphi(\Lambda_1)$ for any $i \in \N \cap [1,d]$. The integral closure $R \in \Fil$ of the $\Z_p$-subalgebra of $\overline{\Z}_p$ generated by $\bigcup_{i = 0}^{d} R_i$ satisfies $\varphi(\Lambda_1) = R \cap \varphi(\Lambda_1)$, and hence $\varphi(\Lambda_1)$ is a $\Z_p$-subalgebra of $R$. Therefore $\varphi(\Lambda_1)$ is finitely generated as a $\Z_p$-module because $\Z_p$ is Noetherian.
\end{proof}

Let $\Lambda_1$ be a compact topological $\Lambda_0$ algebra. For each $\varphi \in \Omega(\Lambda_1)$, we put $\Z_p[\varphi] \coloneqq \Lambda_1/\ker(\varphi)$, and endow it with the quotient topology. For any $\varphi \in \Omega(\Lambda_1)$, $\Z_p[\varphi]$ is $p$-adically complete by Proposition \ref{Baire}, and the $p$-adic topology coincides with the original topology by Proposition \ref{finitely generated - p-adic}. In particular, the continuous $\Z_p$-algebra isomorphism $\Z_p[\varphi] \stackrel{\sim}{\to} \varphi(\Lambda_1)$ is a homeomorphism, and hence we identify $\Z_p[\varphi]$ with $\varphi(\Lambda_1)$ for any $\varphi \in \Omega(\Lambda_1)$. For a $k_0 \in \Z_p$, a $\varphi \in \Omega(\Lambda_1)$ is said to be a {\it $\overline{\Z}_p$-valued point of $\Lambda_1$ of weight $k_0$} if the preimage of $\ker(\varphi)$ in $\Lambda_0$ is of weight $k_0$, and we denote by $\Omega(\Lambda_1)_{k_0} \subset \Omega(\Lambda_1)$ the subset of $\overline{\Z}_p$-valued points of $\Lambda_1$ of weight $k_0$. We set $\t{supp}(\Lambda_1) \coloneqq \set{k_0 \in \Z_p}{\Omega(\Lambda_1)_{k_0} \neq \emptyset}$. We put $\Omega(\Lambda_1)_S \coloneqq \bigsqcup_{k \in S} \Omega(\Lambda_1)_k$ for each $S \subset \Z_p$, and denote by $\t{wt} \colon \Omega(\Lambda_1)_{\Z_p} \twoheadrightarrow \t{supp}(\Lambda_1)$ the canonical projection. 

\begin{dfn}
\label{Lambda-adic domain}
A {\it $\Lambda$-adic domain} is a compact Hausdorff topological $\Lambda_0$-algebra $\Lambda_1$ satisfying the following conditions:
\begin{itemize}
\item[(i)] The intersection $\t{supp}(\Lambda_1) \cap (\N \cap [2,\infty))$ is an infinite set.
\item[(ii)] For any infinite subset $\Sigma \subset \Omega(\Lambda_1)_{\N \cap [2,\infty)}$, the equality $\bigcap_{\varphi \in \Sigma} \ker(\varphi) = \ens{0}$ holds.
\end{itemize}
\end{dfn}

\begin{prp}
\label{integral}
Every $\Lambda$-adic domain is an integral domain.
\end{prp}

\begin{proof}
Let $\Lambda_1$ be a $\Lambda$-adic domain. Assume $f_1 f_2 = 0$ for some $(f_1,f_2) \in \Lambda_1^2$. Then for each $\varphi \in \Omega(\Lambda_1)_{\N \cap [2,\infty)}$, we have $\varphi(f_1) \varphi(f_2) = \varphi(f_1 f_2) = \varphi(0) = 0 \in \overline{\Z}_p$, and hence either $\varphi(f_1) = 0$ or $\varphi(f_2) = 0$ holds. Therefore by the pigeonhole principle, one of the subsets $\set{\varphi \in \Omega(\Lambda_1)_{\N \cap [2,\infty)}}{\varphi(f_1) = 0}$ and $\set{\varphi \in \Omega(\Lambda_1)_{\N \cap [2,\infty)}}{\varphi(f_2) = 0}$ is an infinite set, because $\Omega(\Lambda_1)_{\N \cap [2,\infty)}$ is an infinite set by the condition (i). It implies that either $f_1 = 0$ or $f_2 = 0$ holds by the condition (ii). Thus $\Lambda_1$ is an integral domain.
\end{proof}

\begin{prp}
\label{profinite}
Every $\Lambda$-adic domain is a commutative profinite $\Lambda_0$-algebra.
\end{prp}

\begin{proof}
Let $\Lambda_1$ be a $\Lambda$-adic domain. The conditions (i) and (ii) ensure that the continuous $\Z_p$-algebra homomorphism
\begin{eqnarray*}
  \Lambda_1 & \to & \prod_{\varphi \in \Omega(\Lambda_1)} \Z_p[\varphi] \\
  f & \mapsto & (f + \ker(\varphi))_{\varphi \in \Omega(\Lambda_1)}
\end{eqnarray*}
is injective, and hence is a homeomorphic isomorphism onto the closed image, because $\Lambda_1$ is compact and $\Z_p[\varphi]$ is Hausdorff for any $\varphi \in \Omega(\Lambda_1)$. Thus the assertion holds because the target is a commutative profinite $\Lambda_0$-algebra.
\end{proof}

We show an explicit way to construct a $\Lambda$-adic domain. For this sake, we introduce a notion of the analytic space associated to a $\Lambda$-adic domain. We identify $W$ with the disjoint union of $p(p-1)$ copies of open unit balls. For each open disc $D \subset W$ of radius $\v{p}^r$ for an $r \in \N$, we denote by $\Lambda_0(D)$ the profinite $\Lambda_0$-algebra of $\Z_p$-valued rigid analytic functions on $D$.

\begin{dfn}
A $\Lambda$-adic domain $\Lambda_1$ is said to be {\it affinoid} if the structure morphism $\Lambda_0 \to \Lambda_1$ factors through the Weierstrass localisation $\Lambda_0 \to \Lambda_0(D)$ for an open disc $D \subset W$ of radius $\v{p}^r$ for an $r \in \N$, and if $\Lambda_1$ is finitely generated as a $\Lambda_0(D)$-module.
\end{dfn}

\begin{prp}
\label{affinoid}
Let $\Lambda_1$ be a $\Lambda$-adic domain (resp.\ an affinoid $\Lambda$-adic domain), and $\Lambda_2$ a commutative \'etale $\Lambda_1$-algebra which is an integral domain and is finitely generated as a $\Lambda_1$-module. Then $\Lambda_2$ is a $\Lambda$-adic domain (resp.\ an affinoid $\Lambda$-adic domain) with respect to the canonical topology on $\Lambda_2$ as a $\Lambda_1$-module.
\end{prp}

\begin{proof}
For the detail of the canonical topology, see Proposition \ref{canonical topology} and Proposition \ref{canonical topology 2}. Let $k_0 \in \N \cap [2,\infty)$, and $\varphi \in \Omega(\Lambda_1)_{k_0}$. Since $\Lambda_1$ is an integral domain and $\Lambda_2$ is a commutative \'etale $\Lambda_1$-algebra finitely generated as a $\Lambda_1$-module, there is a $\Lambda_1$-algebra homomorphism $\tilde{\varphi} \colon \Lambda_2 \to \overline{\Q}_p$ extending $\varphi$ by going up theorem. By the definition of the canonical topology, $\tilde{\varphi}$ is continuous and lies in $\Omega(\Lambda_2)_{k_0}$. Therefore $\t{supp}(\Lambda_2) \cap (\N \cap [2,\infty))$ is an infinite set. Let $\Sigma \subset \Omega(\Lambda_2)_{\N \cap [2,\infty)}$ be an infinite subset. For each $\varphi \in \Sigma$, we denote by $\overline{\varphi} \in \Omega(\Lambda_1)_{\N \cap [2,\infty)}$ the composite of $\varphi$ and the structure morphism $\Lambda_1 \to \Lambda_2$. Since $\Lambda_2$ is finitely generated as a $\Lambda_1$-module, $\overline{\Sigma} \coloneqq \set{\overline{\varphi}}{\varphi \in \Sigma} \subset \Omega(\Lambda_1)_{\N \cap [2,\infty)}$ is an infinite set. Let $F \in \bigcap_{\varphi \in \Sigma} \ker(\varphi)$. Let $P(X) = P_n X^n + \sum_{h = 0}^{n-1} P_hX^h \in \Lambda_1[X]$ denote the minimal polynomial of $F$ over $\Lambda_1$. For any $\varphi \in \Sigma$, we have
\begin{eqnarray*}
  \overline{\varphi}(P_0) = \varphi \left( - P_n F^n - \sum_{h = 1}^{n-1} P_h F^h \right) = - \overline{\varphi}(P_n) \varphi(F)^n - \sum_{h = 1}^{n-1} \overline{\varphi}(P_h) \varphi(F)^h = 0.
\end{eqnarray*}
Since $\Lambda_1$ is a $\Lambda$-adic domain and $\overline{\Sigma}$ is an infinite subset of $\Omega(\Lambda_1)_{\N \cap [2,\infty)}$, we obtain $P_0 = 0$. Therefore $P(X) \in \Lambda_1[X]X$. Since $\Lambda_2$ is an integral domain, we get $n = 1$ and $P(X) = X$. Thus $F = 0$. We conclude that $\Lambda_2$ is a $\Lambda$-adic domain (resp.\ an affinoid $\Lambda$-adic domain).
\end{proof}

We denote by $\M(\A)$ the Berkovich spectrum of $\A$ for each affinoid $\Q_p$-algebra $\A$. For details of analytic spaces, see \cite{Ber90} and \cite{Ber93}. Let $\Lambda_1$ be an affinoid $\Lambda$-algebra, and $D \subset W$ be an open disc of radius $\v{p}^r$ for an $r \in \N$ such that the structure morphism $\Lambda_0 \to \Lambda_0$ factors through the Weierstrass localisation $\Lambda_0 \to \Lambda_0(D)$ and $\Lambda_1$ is finitely generated as a $\Lambda_0(D)$-module. Taking a finite subset $\set{F_h}{h \in \N \cap [1,d]} \subset \Lambda_1$ of generators as a $\Lambda_0(D)$-module, we obtain a surjective $\Lambda_0(D)$-algebra homomorphism
\begin{eqnarray*}
  \varpi \colon \Lambda_0(D)[X_1, \ldots, X_d] & \twoheadrightarrow & \Lambda_1 \\
  X_h & \mapsto & F_f
\end{eqnarray*}
Let $\mathbb{D} \subset \overline{\Z}_p$ denote the open unit disc containing $D$ sharing the radius with $D$, i.e.\ $\mathbb{D} \coloneqq \set{z \in \overline{\Z}_p}{{}^{\forall} w_0 \in D, {}^{\exists} w_1 \in D, \v{z-w_0} \leq \v{w_1-w_0}}$. Then $\varpi$ yields a $1$-dimensional analytic subset
\begin{eqnarray*}
  \M_{\eta}(\Lambda_1)(\overline{\Q}_p) \coloneqq \Set{(z_i)_{i = 0}^{d} \in \overline{\Z}_p^{d+1}}{z_0 \in \mathbb{D}, F(z_0, z_1, \ldots, z_d) = 0, {}^{\forall} F \in \ker(\varpi)}.
\end{eqnarray*}
More precisely, $\Lambda_1$ corresponds to a $\Q_p$-analytic space in the following way: For each $r \in \N \backslash \ens{0}$, we regard $\prod_{\zeta = 0}^{p(p-1)-1} \Q_p \ens{\v{p}^{- \frac{1}{r}}(z_{\zeta} - \zeta)}$ as a topological $\Lambda_0$-algebra by the continuous embedding
\begin{eqnarray*}
  \Lambda_0 = \prod_{\zeta = 0}^{p(p-1)-1} \Z_p[[z_{\zeta} - \zeta]] & \hookrightarrow & \prod_{\zeta = 0}^{p(p-1)-1} \Q_p \ens{\v{p}^{- \frac{1}{r}}(z_{\zeta} - \zeta)} \\
  (F_{\zeta}(z_{\zeta} - \zeta))_{\zeta = 0}^{p(p-1)-1} & \mapsto & (F_{\zeta}(z_{\zeta} - \zeta))_{\zeta = 0}^{p(p-1)-1}.
\end{eqnarray*}
Since $\Lambda_1$ is finitely generated as a $\Lambda_0(D)$-module, $(\prod_{\zeta = 0}^{p(p-1)-1} \Q_p \ens{\v{p}^{- \frac{1}{r}}(z_{\zeta} - \zeta)}) \otimes_{\Lambda_0} \Lambda_1$ is a $1$-dimensional affinoid $\Q_p$-algebra over $\prod_{\zeta = 0}^{p(p-1)-1} \Q_p \ens{\v{p}^{- \frac{1}{r}}(z_{\zeta} - \zeta)}$ with respect to a complete non-Archimedean norm unique up to equivalence for each $r \in \N \backslash \ens{0}$. We obtain a locally compact $\sigma$-compact Hausdorff $\Q_p$-analytic space as the colimit
\begin{eqnarray*}
  \M_{\eta}(\Lambda_1) \coloneqq \bigcup_{r \in \N \backslash \ens{0}} \M \left( \left( \prod_{\zeta = 0}^{p(p-1)-1} \Q_p \ens{\v{p}^{- \frac{1}{r}}(z_{\zeta} - \zeta)} \right) \otimes_{\Lambda_0} \Lambda_1 \right).
\end{eqnarray*}
We remark that there is another way to construct $\M_{\eta}(\Lambda_1)$ independent of the presentation $\varphi$, whose underlying set is naturally identified with the set of continuous multiplicative seminorms on $\Lambda_1$. We call $\M_{\eta}(\Lambda_1)$ {\it the formal affinoid space associated to $\Lambda_1$}. It is not an affinoid space unless $\Lambda_1 = 0$, and is a countable union of affinoid spaces. For each $\zeta \in \N \cap [0,p(p-1)-1]$, $\Lambda_0/e_{\zeta} \Lambda_0$ is an affinoid $\Lambda$-adic domain. We put $\W \coloneqq \bigsqcup_{\zeta \in \N \cap [0,p(p-1)-1]} \M_{\eta}(\Lambda_0/e_{\zeta} \Lambda_0)$. We have a natural \'etale morphism from $\M_{\eta}(\Lambda_1) \to \W$ by the construction. We remark that $\Lambda_1$ can not be reconstructed from $\M_{\eta}(\Lambda_1)$. Indeed, every $\Lambda_1$-subalgebra $\Lambda_1'$ of the integral closure of the image of $\Lambda_1$ in $\Lambda_1 \otimes_{\Z_p} \Q_p$ finitely generated as a $\Lambda_0$-module is an affinoid $\Lambda$-adic domain with a natural isomorphism $\M_{\eta}(\Lambda_1') \stackrel{\sim}{\to} \M_{\eta}(\Lambda_1)$. Now we verify that analytic curves with suitable conditions admits an \'etale covering by the formal affinoid spaces associated to affinoid $\Lambda$-adic domains.

\begin{thm}
\label{open disc}
Every closed good $\Q_p$-analytic space \'etale over $\W$ admits an open covering by formal affinoid spaces associated to affinoid $\Lambda$-adic domains.
\end{thm}

See \cite{Ber90} 3.1.2, \cite{Ber93} 1.2.15, and \cite{Ber93} Definition 3.3.4, for the notion of a closed analytic space, a good analytic space, and an etale morphism between analytic spaces respectively.

\begin{proof}
Let $C$ be a closed good $\Q_p$-analytic space with an \'etale morphism $\psi \colon C \to \W$, and $x \in C$. Since $\W$ is a smooth $\Q_p$-analytic space, so is $C$. Put $y \coloneqq \psi(x)$. Since $C$ is good and $\psi$ is \'etale, there are affinoid neighbourhoods $V \subset C$ and $U \subset \W$ of $x$ and $y$ respectively such that $V \subset \psi^{-1}(U)$ and $\psi |_V \colon V \to U$ is a finite \'etale morphism with $V \cap \psi^{-1}(x) = \ens{y}$. In particular, $\t{H}^0(V,\Opn_C)$ is a commutative $\t{H}^0(U,\Opn_{\W})$-algebra finitely generated as a $\t{H}^0(U,\Opn_{\W})$-module. Since $C$ is closed, replacing $U$ by a larger one, we may assume that $y$ does not lie in the relative boundary of $U$ in $\W$. Therefore there is an open disc $\mathbb{D} \subset \W$ of radius $\v{p}^r$ for an $r \in \N$ such that $x \in \mathbb{D} \subset U$ and $V \cap \psi^{-1}(\mathbb{D})$ is connected. Since $\mathbb{D}$ is the increasing union of closed discs, $V \cap \psi^{-1}(\mathbb{D})$ is the increasing union of connected affinoid subdomains. It implies that $\t{H}^0(V \cap \psi^{-1}(\mathbb{D}),\Opn_C)$ has no non-trivial idempotent by Shilov idempotent theorem (\cite{Ber90} Theorem 7.4.1). Since $C$ is smooth, $\t{H}^0(V \cap \psi^{-1}(\mathbb{D}),\Opn_C)$ is an integral domain. It is easy to see that $\Lambda_0(D)$ itself is an affinoid $\Lambda$-adic domain. The commutative \'etale $\Lambda_0(D)$-algebra $\t{H}^0(V \cap \psi^{-1}(\mathbb{D}),\Opn_C) \cong \Lambda_0(D) \otimes_{\t{H}^0(U,\Opn_{\W})} \t{H}^0(V,\Opn_C)$ is finitely generated as a $\Lambda_0(D)$-module. We endow $\t{H}^0(V \cap \psi^{-1}(\mathbb{D}),\Opn_C)$ with the canonical topology as a $\Lambda_0(D)$-module. By Proposition \ref{affinoid}, $\t{H}^0(V \cap \psi^{-1}(\mathbb{D}),\Opn_C)$ is an affinoid $\Lambda$-adic domain. Thus the assertion follows from the fact that the affinoid localisation $\t{H}^0(V,\Opn_C) \to \t{H}^0(V \cap \psi^{-1}(\mathbb{D}),\Opn_C)$ gives an identification $\M_{\eta}(\t{H}^0(V \cap \psi^{-1}(\mathbb{D}),\Opn_C)) \cong V \cap \psi^{-1}(\mathbb{D})$.
\end{proof}

\begin{rmk}
\label{eigencurve}
Theorem \ref{open disc} ensures that there are many explicit examples of $\Lambda$-adic domains. One of the most important example of a closed good $\Q_p$-analytic space over $\W$ is the reduced eigencurve introduced in \cite{Eme} Theorem 2.23 obtained as the closed subspace of $\t{Spf}(\T_N^{(p)}) \times \Aff^1_{\Q_p}$ interpolating classical Hecke eigenforms, where $\T_N^{(p)}$ is the universal Hecke algebra of level $N$ generated by Hecke operators $T_{\ell}$ for each prime number $\ell \neq p$ and $S_{\ell}$ for each prime number $\ell$ coprime to $N$. Unlike the original reduced eigencurve introduced in \cite{CM98} 6.1 Definition 1 and \cite{CM98} 7.5, the reduced eigencurve forms a family of modular forms of the fixed level $N$. Every closed good reduced $\Q_p$-analytic space of dimension $1$ admits a smooth alteration given by the normalisation. Therefore the reduced eigencurve admits a smooth alteration with an open covering of the complement of a discrete subspace by formal affinoid spaces associated to affinoid $\Lambda$-adic domains.
\end{rmk}

Henceforth, let $\Lambda_1$ denote a $\Lambda$-adic domain. Imitating the definition of a Berkovich spectrum (\cite{Ber90} 1.2), we endow $\Omega(\Lambda_1)$ with the weakest topology for which the map
\begin{eqnarray*}
  \v{f} \colon \Omega(\Lambda_1) & \to & [0,\infty) \\
  \varphi & \mapsto & \v{\varphi(f)}
\end{eqnarray*}
is continuous for any $f \in \Lambda_1$. Let $\Sigma \subset \Omega(\Lambda_1)$ be an infinite subset endowed with the relative topology. The evaluation map
\begin{eqnarray*}
  \Lambda_1 & \to & \overline{\Z}{}_p^{\Sigma} \\
  f & \mapsto & (\varphi(f))_{\varphi \in \Sigma}
\end{eqnarray*}
is injective because $\Lambda_1$ is a $\Lambda$-adic domain. The image of $\Lambda_1$ is contained in $\t{C}(\Sigma,\overline{\Z}_p)$ by the definition of the topology of $\Omega(\Lambda_1)$. The induced map $\Lambda_1 \hookrightarrow \t{C}(\Sigma,\overline{\Z}_p)$ is not necessarily continuous with respect to the $p$-adic topology of $\t{C}(\Sigma,\overline{\Z}_p)$, while the inverse of its restriction onto the image is continuous by Corollary \ref{p-adically open}. We remark that the relative topology on $\t{C}(\Sigma,\overline{\Z}_p) \subset \overline{\Z}{}_p^{\Sigma}$ is the topology of pointwise convergence, and hence $\Lambda_1$ can be naturally identified with a $\Z_p$-subalgebra of $\t{C}(\Sigma,\overline{\Z}_p)$ compact with respect to the topology of pointwise convergence. Now we introduce a notion of a $\Lambda_1$-adic form. For the convention of slope, see \S \ref{p-adic Modular Forms and Hecke Algebras}.

\begin{dfn}
A {\it $\Lambda_1$-adic form of level $N$} is an $f(q) \in \Lambda_1[[q]]$ such that $f(\varphi)(q) \coloneqq \sum_{h = 0}^{\infty} \varphi(a_h(f)) q^h \in \Z_p[\varphi][[q]]$ lies in $\t{M}_{\t{wt}(\varphi)}(\Gamma_1(N),\Z_p[\varphi])$ for all but finitely many $\varphi \in \Omega(\Lambda_1)_{\N \cap [2,\infty)}$. We denote by $\Mod(\Gamma_1(N),\Lambda_1) \subset \Lambda_1[[q]]$ the $\Lambda_1$-submodule of $\Lambda_1$-adic forms of level $N$.
\end{dfn}

When $\Lambda_1$ is an affinoid $\Lambda$-adic domain appearing in an open subspace of a smooth alteration of the cuspidal locus of the reduced eigencurve as in Remark \ref{eigencurve}, then we obtain a $\Lambda_1$-adic form of level $N$ by pulling back the family of systems of cuspidal Hecke eigenvalues on the reduced eigencurve. Therefore a reader does not have to mind the existence of a non-trivial $\Lambda_1$-adic form. On the other hand, we have no evidence of the finiteness of $\Mod(\Gamma_1(N),\Lambda_1)$ as a $\Lambda_1$-module. Restricting it to the case where a family is allowed to be of finite slope in the following sense, we verify the finiteness.

\begin{dfn}
Let $s \in \N \backslash \ens{0}$. A $\Lambda_1$-adic form $f(q)$ of level $N$ is said to be {\it locally of slope $< s$} if $f(\varphi)(q) \in \Z_p[\varphi][[q]]$ lies in $\t{M}_{\t{wt}(\varphi)}(\Gamma_1(N),\Z_p[\varphi])^{< s}$ for all but finitely many $\varphi \in \Omega(\Lambda_1)_{\N \cap [2,\infty)}$. We denote by $\Mod(\Gamma_1(N),\Lambda_1)^{[< s]} \subset \Mod(\Gamma_1(N),\Lambda_1)$ the $\Lambda_1$-submodule of $\Lambda_1$-adic forms of level $N$ locally of slope $< s$.
\end{dfn}

Henceforth, we fix an $s \in \N \backslash \ens{0}$. Let $R \subset \overline{\Q}_p$ be a subring, and $\epsilon \colon (\Z/N \Z)^{\times} \to \overline{\Q}{}_p^{\times}$ a Dirichlet character. We put
\begin{eqnarray*}
  \t{M}_{k_0}(\Gamma_1(N),\epsilon,R)^{< s} & \coloneqq & \t{M}_{k_0}(\Gamma_1(N),R)^{< s} \cap \t{M}_{k_0}(\Gamma_1(N),\epsilon,R) \\
  & = & \t{M}_{k_0}(\Gamma_1(N),\overline{\Q}_p)^{< s} \cap \t{M}_{k_0}(\Gamma_1(N),\epsilon,R).
\end{eqnarray*}
If $R$ contains the image of $\epsilon$, then $\t{M}_{k_0}(\Gamma_1(N),\epsilon,R)^{< s}$ is an intersection of $R$-submodules $\t{M}_{k_0}(\Gamma_1(N),\overline{\Q}_p)^{< s}, \t{M}_{k_0}(\Gamma_1(N),\epsilon,R) \subset \t{M}_{k_0}(\Gamma_1(N),\overline{\Q}_p)$ stable under the action of Hecke operators, and hence is an $R$-submodule of $\t{M}_{k_0}(\Gamma_1(N),\overline{\Q}_p)$ stable under the action of Hecke operators. In fact, the assumption that $R$ contains the image of $\epsilon$ can be removed by \cite{DI95} Theorem 12.3.4/2, \cite{DI95} Proposition 12.3.11, and \cite{DI95} Proposition 12.4.1, but we do not use the result.

\vspace{0.2in}
Let $\chi \colon (\Z/N \Z)^{\times} \to \Lambda_1^{\times}$ be a group homomorphism. A $\Lambda_1$-adic form $f(q)$ of level $N$ is said to be {\it with character $\chi$} if $f(\varphi)(q) \in \Z_p[\varphi][[q]]$ lies in $\t{M}_{\t{wt}(\varphi)}(\Gamma_1(N), \varphi \circ \chi, \Z_p[\varphi])$ for all but finitely many $\varphi \in \Omega(\Lambda_1)_{\N \cap [2,\infty)}$. We denote by $\Mod(\Gamma_1(N),\chi,\Lambda_1) \subset \Mod(\Gamma_1(N),\Lambda_1)$ the $\Lambda_1$-submodule of $\Lambda_1$-adic form of level $N$ with character $\chi$. We put
\begin{eqnarray*}
  \Mod(\Gamma_1(N),\chi,\Lambda_1)^{[< s]} \coloneqq \Mod(\Gamma_1(N),\chi,\Lambda_1) \cap \Mod(\Gamma_1(N),\Lambda_1)^{[< s]}.
\end{eqnarray*}
There is a unique $\Lambda_1$-linear action of $T_{\ell}$ for each prime number $\ell$ and $S_n$ for each $n \in \N$ coprime to $N$ on $\Mod(\Gamma_1(N),\chi,\Lambda_1)^{[< s]}$ compatible with the specialisation maps. The action is given explicitly in the following way:
\begin{eqnarray*}
  T_{\ell} \colon \Mod(\Gamma_1(N),\chi,\Lambda_1)^{[< s]} & \to & \Mod(\Gamma_1(N),\chi,\Lambda_1)^{[< s]} \\
  f(q) & \mapsto &
  \left\{
    \begin{array}{ll}
      \sum_{h = 0}^{\infty} a_{\ell h}(f) q^h + \sum_{h = 0}^{\infty} a_h(f) \chi(\ell + N \Z) \ell^{k-2} q^{\ell h} & (\ell \mid \hspace{-.62em}/ N) \\
      \sum_{h = 0}^{\infty} a_{\ell h}(f) q^h & (\ell \mid N)
    \end{array}
  \right. \\
  S_n \colon \Mod(\Gamma_1(N),\chi,\Lambda_1)^{[< s]} & \to & \Mod(\Gamma_1(N),\chi,\Lambda_1)^{[< s]} \\
  f(q) & \mapsto & \chi(n + N \Z) n^{k-2} f(q).
\end{eqnarray*}
Henceforth, we fix a group homomorphism $\chi \colon (\Z/N \Z)^{\times} \to \Lambda_1^{\times}$. We show a certain finiteness of $\Mod(\Gamma_1(N),\Lambda_1)^{[< s]}$. For a ring $R$ and a left $R$-module $M$ of finite length, we denote by $\ell_R(M) \in \N$ the length of $M$.

\begin{lmm}
\label{finiteness 2}
The set $\set{\dim_{\overline{\Q}_p} \t{M}_k(\Gamma_1(N),\overline{\Q}_p)^{[< s]}}{k \in \N \cap [2,\infty)}$ is uniformly bounded by the constant
\begin{eqnarray*}
  \max_{
  \t{\scriptsize $
    \begin{array}{c}
      k \in \N \\
      2 \leq k \leq s+p^s
    \end{array}
  $}
  }
  \ell_{\Z_p} \left( \t{H}^1 \left( \Gamma_1(N), \Leb_{k-2}/p^{s+1} \right) \right).
\end{eqnarray*}
\end{lmm}

\begin{proof}
We denote by $C \in \N$ the constant in the assertion. Let $k_1 \in [2,\infty)$. For each commutative topological $\Z_p$-algebra $R$, put
\begin{eqnarray*}
  M_R \coloneqq R \otimes_{\Z_p} \t{H}^1 \left( \Gamma_1(N), \Leb_{k_1-2} \right).
\end{eqnarray*}
If $k_1 \in \N \cap [2,s+p^s]$, then we have
\begin{eqnarray*}
  & & \dim_{\overline{\Q}_p} \t{M}_{k_1}(\Gamma_1(N),\overline{\Q}_p)^{[< s]} \leq \dim_{\overline{\Q}_p} \t{M}_{k_1}(\Gamma_1(N),\overline{\Q}_p) \\
  & \leq & \dim_{\overline{\Q}_p} M_{\overline{\Q}_p} = \dim_{\Q_p} M_{\Q_p} =  \t{rank}_{\Z_p} (M_{\Z_p})_{\t{free}} \leq \ell_{\Z_p}(M_{\Z/p^{s+1} \Z}) \leq C
\end{eqnarray*}
by the Eichler--Shimura isomorphism (\cite{Shi59} 5 Th\'eor\`eme 1, \cite{Hid93} 6.3 Theorem 4). Therefore we may assume $k_1 \geq s+p^s+1$. Let $K/\Q_p$ denote the finite Galois subextension of $\overline{\Q}_p/\Q_p$ generated by eigenvalues of $T_p$ acting on $\t{M}_{k_1}(\Gamma_1(N),\overline{\Q}_p)$. Put
\begin{eqnarray*}
  d_0 & \coloneqq & \dim_{\overline{\Q}_p} \t{M}_{k_1}(\Gamma_1(N),\overline{\Q}_p)^{< s} \\
  d_1 & \coloneqq & \dim_{\overline{\Q}_p} M_{\overline{\Q}_p} = \dim_K M_K.
\end{eqnarray*}
Every eigenvalue of $T_p$ acting on $M_K$ is contained in $K$, and the sum of the dimensions of the generalised eigenspaces of $T_p$ acting on $M_K$ with eigenvalues $\alpha$ satisfying $\v{\alpha} > \v{p}^s$ is greater than or equal to $d_0$ by the Eichler--Shimura isomorphism again. In particular, we have $d_0 \leq d_1$. Take a basis $(c_i)_{i = 1}^{d_1}$ of $M_{\overline{\Q}_p}$ such that the matrix representation of $T_p$ with respect to $(c_i)_{i = 1}^{d_1}$ is a Jordan normal form with diagonal $(\alpha_i)_{i = 1}^{d_1}$. Put $V \coloneqq M_K$ and $\Fil^{i_0} V \coloneqq \bigoplus_{i = 1}^{i_0} K c_i \subset V$ for each $i_0 \in \N \cap [0,d_1]$. The increasing filtration $(\Fil^i V)_{i = 0}^{d_1}$ is stable under the action of $T_p$ by the choice of $(c_i)_{i = 1}^{d_1}$. We denote by $O_K$ the valuation ring of $K$, and by $W \subset V$ the image of $M_{O_K}$. By the functoriality of the action of $T_p$, $W$ is a $T_p$-stable lattice of $V$. Put $\Fil^i W \coloneqq W \cap \Fil^i V$ for each $i \in \N \cap [0,d_1]$. We verify $e_K d_0 < \ell_{O_K}(T_p(W)/(T_p(W) \cap p^{s+1} W))$, where $e_K \in \N \backslash \ens{0}$ is the ramification index of $K/\Q_p$.

\vspace{0.2in}
For any $i \in \N \cap [1,d_1]$, we have
\begin{eqnarray*}
  & & \ell_{O_K} \left( T_p(\t{gr}_{\Fil}^i(W))/(T_p(\t{gr}_{\Fil}^i(W)) \cap p^{s+1} \t{gr}_{\Fil}^i(W)) \right) \\
  & = & \ell_{O_K} \left( \alpha_i \t{gr}_{\Fil}^i(W)/(\alpha_i \t{gr}_{\Fil}^i(W) \cap p^{s+1} \t{gr}_{\Fil}^i(W)) \right) \\
  & = &
  \left\{
    \begin{array}{ll}
      - \log_{\vv{\pi_K}} \vv{\frac{\alpha_i}{p^{s+1}}} & (\vv{\alpha_i} > \vv{p}^{s+1}) \\
      0 & (\vv{\alpha_i} \leq \vv{p}^{s+1})
    \end{array}
  \right.,
\end{eqnarray*}
where $\pi_K$ is a uniformiser of $K$, and in particular, the inequality
\begin{eqnarray*}
  \ell_{O_K}(T_p(\t{gr}_{\Fil}^i(W))/(T_p(\t{gr}_{\Fil}^i(W)) \cap p^{s+1} \t{gr}_{\Fil}^i(W))) > \log_{\v{\pi_K}} \v{p} = e_K
\end{eqnarray*}
holds for any $i \in \N \cap [1,d_1]$ with $\v{\alpha_i} > \v{p}^s$. Since the sum of the dimensions of the generalised eigenspaces of $T_p$ acting on $\t{H}^1(\Gamma_1(N), \t{Sym}^{k_1-2}(K^2, \rho_{K^2}))$ with eigenvalues $\alpha_i$ satisfying $\v{\alpha_i} > \v{p}^s$ is greater than or equal to $d_0$, we obtain
\begin{eqnarray*}
  \sum_{i = 1}^{d_1} \ell_{O_K} \left( T_p(\t{gr}_{\Fil}^i(W))/(T_p(\t{gr}_{\Fil}^i(W)) \cap p^{s+1} \t{gr}_{\Fil}^i(W)) \right) > e_K d_0.
\end{eqnarray*}
Therefore it suffices to show
\begin{eqnarray*}
  & & \sum_{i = 1}^{d_1} \ell_{O_K} \left( T_p(\t{gr}_{\Fil}^i(W))/(T_p(\t{gr}_{\Fil}^i(W)) \cap p^{s+1} \t{gr}_{\Fil}^i(W)) \right) \\
  & \leq & \ell_{O_K}(T_p(W)/(T_p(W) \cap p^{s+1} W)).
\end{eqnarray*}
Since $W$ is an $O_K$-submodule of a $K$-vector space $V$, $W$ is a torsionfree $O_K$-module. Since $(\Fil^i W)_{i = 0}^{d_1}$ is induced by the increasing filtration $(\Fil^i V)_{i = 0}^{d_1}$ of $K$-vectors spaces, we have $\Fil^i W \cap p^{s+1} W = p^{s+1} \Fil^i W$ for any $i \in \N \cap [0,d_1]$. Since $(\Fil^i W)_{i = 0}^{d_1}$ is induced by the increasing filtration $(\Fil^i V)_{i = 0}^{d_1}$ of $K$-vectors spaces again, $\t{gr}_{\Fil}^i W$ is naturally regarded as an $O_K$-submodule of a $K$-vector space $\t{gr}_{\Fil}^i V$, and hence is a torsionfree $O_K$-module for any $i \in \N \cap [1,d_1]$. Therefore the exact sequence
\begin{eqnarray*}
  0 \to \Fil^i W \to \Fil^{i+1} W \to \t{gr}_{\Fil}^{i+1} W \to 0
\end{eqnarray*}
induces an exact sequence
\begin{eqnarray*}
  0 \to (\Fil^i W)/p^{s+1} \to (\Fil^{i+1} W)/p^{s+1} \to (\t{gr}_{\Fil}^{i+1} W)/p^{s+1} \to 0,
\end{eqnarray*}
and the commutative diagram
\begin{eqnarray*}
\begin{CD}
  0 @>>> \Fil^i W           @>>> \Fil^{i+1} W           @>>> \t{gr}_{\Fil}^{i+1} W           @>>> 0 \\
  @.     @V{T_p}VV               @V{T_p}VV                   @V{T_p}VV                            @. \\
  0 @>>> (\Fil^i W)/p^{s+1} @>>> (\Fil^{i+1} W)/p^{s+1} @>>> (\t{gr}_{\Fil}^{i+1} W)/p^{s+1} @>>> 0,
\end{CD}
\end{eqnarray*}
induces a complex
\begin{eqnarray*}
  & & T_p(\Fil^i W)/(T_p(\Fil^i W) \cap p^{s+1} W) \\
  & \hookrightarrow & T_p(\Fil^{i+1} W)/(T_p(\Fil^{i+1} W) \cap p^{s+1} W) \\
  & \twoheadrightarrow & T_p(\t{gr}_{\Fil}^{i+1} W)/(T_p(\t{gr}_{\Fil}^{i+1} W) \cap p^{s+1} \t{gr}_{\Fil}^{i+1} W).
\end{eqnarray*}
Therefore the inequality
\begin{eqnarray*}
  & & \ell_{O_K} \left( T_p(\Fil^i W)/(T_p(\Fil^i W) \cap p^{s+1} W) \right) \\
  & & + \ell_{O_K} \left( T_p(\t{gr}_{\Fil}^{i+1} W)/(T_p(\t{gr}_{\Fil}^{i+1} W) \cap p^{s+1} \t{gr}_{\Fil}^{i+1} W) \right) \\
  & \leq & \ell_{O_K} \left( T_p(\Fil^{i+1} W)/(T_p(\Fil^{i+1} W) \cap p^{s+1} W) \right)
\end{eqnarray*}
holds for any $i \in \N \cap [0,d_1-1]$. As a consequence, we obtain the inequality
\begin{eqnarray*}
  & & \ell_{O_K} \left( T_p(W)/(T_p(W) \cap p^{s+1} W) \right) = \ell_{O_K} \left( T_p(\Fil^{d_1} W)/(T_p(\Fil^{d_1} W) \cap p^{s+1} W) \right) \\
  & \geq & \ell_{O_K} \left( T_p(\Fil^{d_1-1} W)/(T_p(\Fil^{d_1-1} W) \cap p^{s+1} W) \right) \\
  & & + \ell_{O_K} \left( T_p(\t{gr}_{\Fil}^{d_1} W)/(T_p(\t{gr}_{\Fil}^{d_1} W) \cap p^{s+1} \t{gr}_{\Fil}^{d_1} W) \right) \\
  & \geq & \ell_{O_K} \left( T_p(\Fil^{d_1-2} W)/(T_p(\Fil^{d_1-2} W) \cap p^{s+1} W) \right) \\
  & & + \sum_{i = d_1-1}^{d_1} \ell_{O_K} \left( T_p(\t{gr}_{\Fil}^i W)/(T_p(\t{gr}_{\Fil}^i W) \cap p^{s+1} \t{gr}_{\Fil}^i W) \right) \\
  & \geq & \ell_{O_K} \left( T_p(\Fil^1 W)/(T_p(\Fil^1 W) \cap p^{s+1} W) \right) \\
  & & + \sum_{i = 2}^{d_1} \ell_{O_K} \left( T_p(\t{gr}_{\Fil}^i W)/(T_p(\t{gr}_{\Fil}^i W) \cap p^{s+1} \t{gr}_{\Fil}^i W) \right) \\
  & \geq & \sum_{i = 1}^{d_1} \ell_{O_K} \left( T_p(\t{gr}_{\Fil}^i W)/(T_p(\t{gr}_{\Fil}^i W) \cap p^{s+1} \t{gr}_{\Fil}^i W) \right),
\end{eqnarray*}
which was what we wanted.

\vspace{0.2in}
Put $W_0 \coloneqq (M_{\Z_p})_{\t{free}}$. By the flatness of $O_K$ as a $\Z_p$-module, the natural $O_K$-linear homomorphism
\begin{eqnarray*}
  O_K \otimes_{\Z_p} \left( T_p(W_0)/(T_p(W_0) \cap p^{s+1} W_0) \right) \to T_p(W)/(T_p(W) \cap p^{s+1} W) 
\end{eqnarray*}
is an isomorphism, and hence we have
\begin{eqnarray*}
  \ell_{\Z_p} \left( T_p(W_0)/(T_p(W_0) \cap p^{s+1} W_0) \right) = \frac{1}{e_K} \ell_{O_K} \left( T_p(W)/(T_p(W) \cap p^{s+1} W) \right) > d_0.
\end{eqnarray*}
The canonical projection $M_{\Z_p} \twoheadrightarrow W_0$ induces a surjective $\Z_p$-linear homomorphism $T_p(M_{\Z_p})/(T_p(M_{\Z_p}) \cap p^{s+1} M_{\Z_p}) \twoheadrightarrow T_p(W_0)/(T_p(W_0) \cap p^{s+1} W_0)$, and hence we obtain
\begin{eqnarray*}
  \ell_{\Z_p} \left( T_p(M_{\Z_p})/(T_p(M_{\Z_p}) \cap p^{s+1} M_{\Z_p}) \right) \geq \ell_{\Z_p} \left( T_p(W_0)/(T_p(W_0) \cap p^{s+1} W_0) \right) > d_0.
\end{eqnarray*}
Therefore in order to verify the assertion, it suffices to show
\begin{eqnarray*}
  \ell_{\Z_p} \left( T_p(M_{\Z_p})/(T_p(M_{\Z_p}) \cap p^{s+1} M_{\Z_p}) \right) \leq \ell_{\Z_p} \left( \t{H}^1 \left( \Gamma_1(N), \Leb_{k_0-2}/p^{s+1} \right) \right),
\end{eqnarray*}
where $k_0$ is the unique integer with $k_0 \in \N \cap [s+1,s+p^s]$ and $k_1-k_0 \in p^s \Z$. Since $M_{\Z_p}$ is finitely generated as a $\Z_p$-module, the natural $(\Z/p^{s+1} \Z)$-linear Hecke-equivariant homomorphism
\begin{eqnarray*}
  \t{H}^1 \left( \Gamma_1(N), \Leb_{k_1-2}/p^{s+1} \right) \to M_{\Z/p^{s+1} \Z}
\end{eqnarray*}
is an isomorphism by the proof of Proposition \ref{right exact 2}. Since $k_1 \geq s+p^s+1$, we have $k_1 > k_0$. By Lemma \ref{congruence}, we have a surjective $(\Z/p^{s+1} \Z)$-linear Hecke-equivariant homomorphism
\begin{eqnarray*}
  \varpi^{s+1}_{k_1-2,k_0-2} \colon M_{\Z_p}/p^{s+1} \twoheadrightarrow \t{H}^1 \left( \Gamma_1(N), \Leb_{k_0-2}/p^{s+1} \right),
\end{eqnarray*}
and hence
\begin{eqnarray*}
  \ell_{\Z_p} \left( \t{H}^1 \left( \Gamma_1(N), \Leb_{k_0-2}/p^{s+1} \right) \right) + \ell_{\Z_p}(\ker(\varpi^{s+1}_{k_1-2,k_0-2})) = \ell_{\Z_p}(M_{\Z_p}/p^{s+1}).
\end{eqnarray*}
Moreover, the action of $T_p$ on $\ker(\varpi^{s+1}_{k_1-2,k_0-2})$ is $0$ by the proof of Lemma \ref{truncate}, we obtain
\begin{eqnarray*}
  \ell_{\Z_p}(\ker(\varpi^{s+1}_{k_1-2,k_0-2})) + \ell_{\Z_p} \left( T_p(M_{\Z_p})/(T_p(M_{\Z_p}) \cap p^{s+1} M_{\Z_p}) \right) \leq \ell_{\Z_p}(M_{\Z_p}/p^{s+1}).
\end{eqnarray*}
It ensures the inequality
\begin{eqnarray*}
  \ell_{\Z_p} \left( T_p(M_{\Z_p})/(T_p(M_{\Z_p}) \cap p^{s+1} M_{\Z_p}) \right) \leq \ell_{\Z_p} \left( \t{H}^1 \left( \Gamma_1(N), \Leb_{k_0-2}/p^{s+1} \right) \right).
\end{eqnarray*}
We conclude
\begin{eqnarray*}
  d_0 \leq \ell_{\Z_p} \left( \t{H}^1 \left( \Gamma_1(N), \Leb_{k_0-2}/p^{s+1} \right) \right) \leq C.
\end{eqnarray*}
\end{proof}

\begin{dfn}
Let $R$ be a ring. A left $R$-module $M$ is said to be {\it adically finite} if there is an $r \in R$ such that $r$ is not a zero divisor and $rM$ is contained in a finitely generated left $R$-submodule of $M$.
\end{dfn}

\begin{thm}
\label{finiteness 3}
The $\Lambda_1$-modules $\Mod(\Gamma_1(N),\Lambda_1)^{[< s]}$ and $\Mod(\Gamma_1(N),\chi,\Lambda_1)^{[< s]}$ are adically finite.
\end{thm}

The proof is quite similar to that of \cite{Hid93} 7.3 Theorem 1.

\begin{proof}
We deal only with $\Mod(\Gamma_1(N),\Lambda_1)^{[< s]}$, because a similar proof with the following works for $\Mod(\Gamma_1(N),\chi,\Lambda_1)^{[< s]}$. Since $\Lambda_1$ is an integral domain by Proposition \ref{integral}, $\Mod(\Gamma_1(N),\Lambda_1)^{[< s]} \subset \Lambda_1[[q]]$ is a torsionfree $\Lambda_1$-module. We regard $\Mod(\Gamma_1(N),\Lambda_1)^{[< s]}$ as a $\Lambda_1$-submodule of $\t{Frac}(\Lambda_1) \otimes_{\Lambda_1} \Mod(\Gamma_1(N),\Lambda_1)^{[< s]}$. Let $d \in \N$, and assume
\begin{eqnarray*}
  \dim_{\t{Frac}(\Lambda_1)} \left( \t{Frac}(\Lambda_1) \otimes_{\Lambda_1} \Mod(\Gamma_1(N),\Lambda_1)^{[< s]} \right) \geq d.
\end{eqnarray*}
Take a system $(f_i)_{i \in \N \cap[1,d]} \in (\Mod(\Gamma_1(N),\Lambda_1)^{[< s]})^d$ of $\t{Frac}(\Lambda_1)$-linearly independent elements. We define a decreasing sequence $(V_h)_{h = 0}^{\infty}$ of $\t{Frac}(\Lambda_1)$-vector spaces by setting
\begin{eqnarray*}
  V_{h_0} \coloneqq \Set{(F_i)_{i = 1}^{d} \in \t{Frac}(\Lambda_1)^d}{\sum_{i = 1}^{d} F_i a_h(f_i) = 0, {}^{\forall} h \in \N \cap [1,h_0]}
\end{eqnarray*}
for each $h_0 \in \N$. Since $\t{Frac}(\Lambda_1)^d$ is of finite dimension, we have $\bigcap_{h = 0}^{\infty} V_h = V_{h_0}$ for some $h_0 \in \N \cap [1,\infty]$. On the other hand, $\bigcap_{h = 0}^{\infty} V_h$ coincides with $\ens{0}$ because $(f_i)_{i \in \N \cap [1,d]}$ is a system of $\t{Frac}(\Lambda_1)$-linearly independent elements. It implies that the system $((a_h(f_i))_{h = 0}^{h_0})_{i = 1}^{d} \in (\Lambda_1^{h_0+1})^d$ of $d$ vectors of length $h_0+1$ is $\t{Frac}(\Lambda_1)$-linearly independent, and hence there is a strictly increasing sequence $(h_j)_{j = 1}^{d} \in (\N \cap [0,h_0])^d$ such that $A \coloneqq (a_{h_j}(f_i))_{i,j = 1}^{d} \in \t{M}_d(\Lambda_1)$ lies in $\t{GL}_d(\t{Frac}(\Lambda_1))$.

\vspace{0.2in}
Put $D \coloneqq \det(A) \in \Lambda_1 \backslash \ens{0}$. Since $D \neq 0$, there is a $\varphi \in \Omega(\Lambda_1)_{\N \cap [2,\infty)}$ such that $f_i(\varphi) \in \t{M}_{\t{wt}(\varphi)}(\Gamma_1(N),\overline{\Q}_p)^{< s}$ for any $i \in \N \cap [1,d]$ and $\varphi(D) \neq 0$ because $\Lambda_1$ is a $\Lambda$-adic domain. In particular, $\varphi(A) \coloneqq (\varphi(a_{h_j}(f_i)))_{i,j = 1}^{d} = (a_{h_j}(f_i(\varphi)))_{i,j = 1}^{d} \in \t{M}_d(\overline{\Z}_p)$ satisfies $\det(\varphi(A)) = \varphi(D) \neq 0$, and hence lies in $\t{GL}_d(\overline{\Q}_p)$. It implies that the system $((a_{h_j}(f_i(\varphi)))_{j = 1}^{d})_{i = 1}^{d} \in (\overline{\Z}{}_p^d)^d$ of $d$ vectors of length $d$ is $\overline{\Q}_p$-linearly independent. In particular, the system $(f_i(\varphi))_{i = 1}^{d}$ is $\overline{\Q}_p$-linearly independent. It ensures that $d$ is bounded by the constant in Lemma \ref{finiteness 2}, and hence $\t{Frac}(\Lambda_1) \otimes_{\Lambda_1} \Mod(\Gamma_1(N),\Lambda_1)^{[< s]}$ is a finite dimensional $\overline{\Q}_p$-vector space.

\vspace{0.2in}
Let $d \in \N$ denote $\dim_{\t{Frac}(\Lambda_1)}(\t{Frac}(\Lambda_1) \otimes_{\Lambda_1} \Mod(\Gamma_1(N),\Lambda_1)^{[< s]}) \in \N$. By the argument above, there is an $((f_i)_{i = 1}^{d},(h_j)_{j = 1}^{d},\varphi) \in (\Mod(\Gamma_1(N),\Lambda_1)^{[< s]})^d \times \N^d \times \Omega(\Lambda_1)_{\N \cap [2,\infty)}$ such that the system $(f_i)_{i = 1}^{d}$ is $\t{Frac}(\Lambda_1)$-linearly independent, $(h_j)_{j = 1}^{d}$ is a strictly increasing sequence, $f_i(\varphi) \in \t{M}_{\t{wt}(\varphi)}(\Gamma_1(N),\overline{\Q}_p)^{< s}$ for any $i \in \N \cap [1,d]$, and $A \coloneqq (a_{h_j}(f_i))_{i,j = 1}^{d} \in \t{M}_d(\Lambda_1)$ lies in $\t{GL}_d(\t{Frac}(\Lambda_1))$. Then $\t{Frac}(\Lambda_1) \otimes_{\Lambda_1} \Mod(\Gamma_1(N),\Lambda_1)^{[< s]} = \bigoplus_{i = 1}^{d} \t{Frac}(\Lambda_1) f_i$ by the definition of $d$. Put $D \coloneqq \det(A) \in \Lambda_1 \backslash \ens{0}$. We verify $D \Mod(\Gamma_1(N),\Lambda_1)^{[< s]} \subset \bigoplus_{i = 1}^{d} \Lambda_1 f_i$. Let $f \in \Mod(\Gamma_1(N),\Lambda_1)^{[< s]}$. Since $\t{Frac}(\Lambda_1) \otimes_{\Lambda_1} \Mod(\Gamma_1(N),\Lambda_1)^{[< s]} = \bigoplus_{i = 1}^{d} \t{Frac}(\Lambda_1) f_i$, there is an $(F_i)_{i = 1}^{d} \in \t{Frac}(\Lambda_1)^d$ such that $\sum_{i = 1}^{d} F_i f_i = f$. We obtain a linear equation
\begin{eqnarray*}
  A (F_i)_{i = 1}^{d} = (a_{h_j}(f))_{j = 1}^{d},
\end{eqnarray*}
and hence
\begin{eqnarray*}
  D (F_i)_{i = 1}^{d} = D (A^{-1} (a_{h_j}(f))_{j = 1}^{d}) = (\det(A) A^{-1}) (a_{h_j}(f))_{j = 1}^{d} \in \Lambda_1^d.
\end{eqnarray*}
Thus $Df = \sum_{i = 1}^{d} (D F_i) f_i \in \bigoplus_{i = 1}^{d} \Lambda_1 f_i$. We conclude $D \Mod(\Gamma_1(N),\Lambda_1)^{[< s]} \subset \bigoplus_{i = 1}^{d} \Lambda_1 f_i$.
\end{proof}

\begin{crl}
\label{finiteness 4}
The following hold:
\begin{itemize}
\item[(i)] The $\Lambda_1$-modules $\Mod(\Gamma_1(N),\Lambda_1)^{[< s]}$ and $\Mod(\Gamma_1(N),\chi,\Lambda_1)^{[< s]}$ are generically finite.
\item[(ii)] If $\Lambda_1$ is Noetherian, then $\Mod(\Gamma_1(N),\Lambda_1)^{[< s]}$ and $\Mod(\Gamma_1(N),\chi,\Lambda_1)^{[< s]}$ are finitely generated.
\item[(iii)] There is a finite subset $S \subset \Omega(\Lambda_1)_{\N \cap [2,\infty)}$ such that $f(\varphi)(q) \in \overline{\Z}_p[\varphi][[q]]$ lies in $\t{M}_{\t{wt}(\varphi)}(\Gamma_1(N),\Z_p[\varphi])^{< s}$ for any $f(q) \in \Mod(\Gamma_1(N),\Lambda_1)^{[< s]}$ and $\varphi \in \Omega(\Lambda_1)_{\N \cap [2,\infty)} \backslash S$.
\item[(iv)] If $\Lambda_1$ is Noetherian, then the $\Lambda_1$-modules $\Mod(\Gamma_1(N),\Lambda_1)^{[< s]}$ and $\Mod(\Gamma_1(N),\chi,\Lambda_1)^{[< s]}$ are closed in $\Lambda_1[[q]]$, and hence are compact Hausdorff topological $\Lambda_1$-modules with respect to the relative topologies.
\end{itemize}
\end{crl}

\begin{proof}
The assertions (i) and (ii) immediately follow from Theorem \ref{finiteness 3}. For the assertion (iii), take a $\t{Frac}(\Lambda_1)$-basis $E \subset \Mod(\Gamma_1(N),\Lambda_1)^{[< s]}$ of $\t{Frac}(\Lambda_1) \otimes_{\Lambda_1} \Mod(\Gamma_1(N),\Lambda_1)^{[< s]}$ and a $D \in \Lambda_1 \backslash \ens{0}$ such that $D \Mod(\Gamma_1(N),\Lambda_1)^{[< s]} \subset \bigoplus_{f \in E} \Lambda_1 f$. For each $f(q) \in E$, let $S_f \subset \Omega(\Lambda_1)_{\N \cap [2,\infty)}$ denote a finite subset such that $f(\varphi)(q)$ lies in $\t{M}_{\t{wt}(\varphi)}(\Gamma_1(N),\Z_p[\varphi])^{< s}$ for any $\varphi \in \Omega(\Lambda_1)_{\N \cap [2,\infty)} \backslash S_f$. Let $S_D$ denote the support $\set{\varphi \in \Omega(\Lambda_1)_{\N \cap [2,\infty)}}{\varphi(D) = 0}$ of $D$. Since $D \neq 0$, $S_D$ is a finite set because $\Lambda_1$ is a $\Lambda$-adic domain. Set $S \coloneqq S_D \cup \bigcup_{f \in E} S_f \subset \Omega(\Lambda_1)_{\N \cap [2,\infty)}$. Let $f(q) \in \Mod(\Gamma_1(N),\Lambda_1)^{[< s]}$ and $\varphi \in \Omega(\Lambda_1)_{\N \cap [2,\infty)} \backslash S$. Since $S$ contains $\bigcup_{f \in E} S_f$, $\varphi(D) f(\varphi) = (Df)(\varphi)$ lies in $\t{M}_{\t{wt}(\varphi)}(\Gamma_1(N),\Z_p[\varphi])^{< s}$. Since $S$ contains $S_D$, we have $\varphi(D) \in \overline{\Q}{}_p^{\times}$ and hence $f(\varphi)$ lies in $\t{M}_{\t{wt}(\varphi)}(\Gamma_1(N),\Z_p[\varphi])^{< s}$. For the assertion (iv), we deal only with $\Mod(\Gamma_1(N),\Lambda_1)^{[< s]}$. Take a finite subset $E \subset \Mod(\Gamma_1(N),\Lambda_1)^{[< s]}$ of generators as a $\Lambda_1$-module. The $\Lambda_1$-linear homomorphism
\begin{eqnarray*}
  \Lambda_1^E & \to & \Lambda_1[[q]] \\
  (F_f)_{f \in E} & \mapsto & \sum_{f \in E} F_f f
\end{eqnarray*}
is continuous. Since $\Lambda_1^E$ is compact and $\Lambda_1[[q]]$ is Hausdorff, its image $\Mod(\Gamma_1(N),\Lambda_1)^{[< s]}$ is closed.
\end{proof}

We denote by $\t{Reg}^{< s}(\Lambda_1)$ the set of all $\varphi \in \Omega(\Lambda_1)_{\N \cap [2,\infty)}$ such that for any $f(q) \in \Mod(\Gamma_1(N),\Lambda_1)^{[< s]}$, $f(\varphi)(q) \in \Z_p[\varphi][[q]]$ lies in $\t{M}_{\t{wt}(\varphi)}(\Gamma_1(N), \Z_p[\varphi])^{< s}$. By Corollary \ref{finiteness 4} (iii), $\Omega(\Lambda_1)_{\N \cap [2,\infty)} \backslash \t{Reg}^{< s}(\Lambda_1)$ is a finite set, and hence $\t{Reg}^{< s}(\Lambda_1)$ is an infinite set. For each $\varphi \in \t{Reg}^{< s}(\Lambda_1)$, we denote by
\begin{eqnarray*}
  \t{M}_{\varphi}(\Gamma_1(N),\chi,\Z_p[\varphi])^{[< s]} \subset \t{M}_{\t{wt}(\varphi)}(\Gamma_1(N), \varphi \circ \chi, \Z_p[\varphi])^{< s}
\end{eqnarray*}
the image of $\Mod(\Gamma_1(N),\chi,\Lambda_1)^{[< s]}$ by the specialisation at $\varphi$, and by
\begin{eqnarray*}
  \Z_p[\varphi] \t{T}_{\varphi,N,\chi}^{[< s]} \subset \t{End}_{\Z_p[\varphi]} \left( \t{M}_{\varphi}(\Gamma_1(N),\chi,\Z_p[\varphi])^{[< s]} \right)
\end{eqnarray*}
the commutative $\Z_p[\varphi]$-subalgebra generated by Hecke operators. Then $\t{T}_{\varphi,N,\chi}^{[< s]}$ is a $\Z_p[\varphi]$-algebra finitely generated as a $\Z_p$-module, because $\t{M}_{\t{wt}(\varphi)}(\Gamma_1(N), \varphi \circ \chi, \Z_p[\varphi])^{< s}$ is finitely generated as a $\Z_p[\varphi]$-module and $\Z_p$ is Noetherian. Let $S \subset \t{Reg}^{< s}(\Lambda_1)$ be a finite subset. We denote by
\begin{eqnarray*}
  \left( \Lambda_1/\bigcap_{\varphi \in S} \ker(\varphi) \right) \t{T}_{\in S,N,\chi}^{[< s]} \subset \t{End}_{\Lambda_1/\bigcap_{\varphi \in S} \ker(\varphi)} \left( \bigoplus_{\varphi \in S} \t{M}_{\varphi}(\Gamma_1(N),\chi,\Z_p[\varphi]) \right)
\end{eqnarray*}
the commutative $(\Lambda_1/\bigcap_{\varphi \in S} \ker(\varphi))$-subalgebra generated by $T_{\ell}$ for each prime number $\ell$ and $S_n$ for each $n \in \N$ coprime to $N$. We have a natural embedding $\Lambda_1/\bigcap_{\varphi \in S} \ker(\varphi) \hookrightarrow \prod_{\varphi \in S} \Z_p[\varphi]$, and hence $\Lambda_1/\bigcap_{\varphi \in S} \ker(\varphi)$ is finitely generated as a $\Z_p$-module. Since $\bigoplus_{\varphi \in S} \t{M}_{\t{wt}(\varphi)}(\Gamma_1(N), \varphi \circ \chi, \Z_p[\varphi])$ is finitely generated as a $(\Lambda_1/\bigcap_{\varphi \in S} \ker(\varphi))$-module, $(\Lambda_1/\bigcap_{\varphi \in S} \ker(\varphi)) \t{T}_{\in S,N,\chi}^{[< s]}$ is finitely generated as a $\Z_p$-module. We set
\begin{eqnarray*}
  \Lambda_1 \T_{N,\chi}^{[< s]} \coloneqq \varprojlim_{S \subset \t{Reg}^{< s}(\Lambda_1)} \left( \Lambda_1/\bigcap_{\varphi \in S} \ker(\varphi) \right) \t{T}_{\in S,N,\chi}^{[< s]},
\end{eqnarray*}
where $S$ in the limit runs through all finite subsets of $\t{Reg}^{< s}(\Lambda_1)$. We regard it as a profinite $(\Lambda_1 \hat{\otimes}_{\Lambda_0} \Lambda_0 \T_N^{[< s]})$-algebra by Corollary \ref{profinite ring}.

\begin{prp}
\label{universal Hecke module 2}
If $\Lambda_1$ is Noetherian, then the action of $T_{\ell}$ for each prime number $\ell$ and $S_n$ for each $n \in \N$ coprime to $N$ extends to an action
\begin{eqnarray*}
  \Lambda_1 \T_{N,\chi}^{[< s]} \times \Mod(\Gamma_1(N),\chi,\Lambda_1)^{[< s]} \to \Mod(\Gamma_1(N),\chi,\Lambda_1)^{[< s]}
\end{eqnarray*}
continuous with respect to the relative topology of $\Mod(\Gamma_1(N),\chi,\Lambda_1)^{[< s]} \subset \Lambda_1[[q]]$.
\end{prp}

\begin{proof}
By Corollary \ref{finiteness 4} (iv), $\Mod(\Gamma_1(N),\chi,\Lambda_1)^{[< s]}$ is a closed $\Lambda_1$-submodule of $\Lambda_1[[q]]$. Therefore the assertion holds by a similar argument with that in the proof of Theorem \ref{universal Hecke module}.
\end{proof}

\begin{prp}
\label{Noetherian}
If $\Lambda_1$ is Noetherian, then $\Lambda_1 \T_{N,\chi}^{[< s]}$ is a commutative $\Lambda_1$-algebra finitely generated as a $\Lambda_1$-module, and hence is Noetherian.
\end{prp}

\begin{proof}
By Corollary \ref{finiteness 4} (ii), it suffices to show the injectivity of the $\Lambda_1$-algebra homomorphism
\begin{eqnarray*}
  \iota \colon \Lambda_1 \T_{N,\chi}^{[< s]} \to \t{End}_{\Lambda_1} \left( \Mod(\Gamma_1(N),\chi,\Lambda_1)^{[< s]} \right)
\end{eqnarray*}
induced by the action in Proposition \ref{universal Hecke module 2}. Let $A \in \ker(\iota)$. Let $\varphi \in \t{Reg}^{< s}(\Lambda_1)$ and $f_{\varphi}(q) \in \t{M}_{\varphi}(\Gamma_1(N),\chi,\Z_p[\varphi])^{[< s]}$, and take a lift $f(q) \in \Mod(\Gamma_1(N),\chi,\Lambda_1)^{[< s]}$ of $f_{\varphi}(q)$. Since the specialisation $\Mod(\Gamma_1(N),\chi,\Lambda_1)^{[< s]} \to \t{M}_{\varphi}(\Gamma_1(N),\chi,\Z_p[\varphi])^{[< s]}$ is continuous and compatible with the action of $T_{\ell}$ for each prime number $\ell$ and $S_n$ for each $n \in \N$ coprime to $N$, the equality $\iota(A)(f(q)) = 0$ implies that the image of $A$ in $\Z_p[\varphi] \t{T}_{\varphi,N,\chi}^{[< s]}$ sends $f_{\varphi}(q)$ to $0$. Therefore the image of $A$ in $\Z_p[\varphi] \t{T}_{\varphi,N,\chi}^{[< s]}$ is $0$. Since the natural homomorphism
\begin{eqnarray*}
  \Lambda_1 \T_{N,\chi}^{[< s]} \to \prod_{\varphi \in \t{Reg}^{< s}(\Lambda_1)} \Z_p[\varphi] \t{T}_{\varphi,N,\chi}^{[< s]}
\end{eqnarray*}
is injective by the definition of $\Lambda_1 \T_{N,\chi}^{[< s]}$, we conclude $A = 0$. Thus $\iota$ is injective.
\end{proof}

\begin{crl}
\label{pairing}
If $\Lambda_1$ is Noetherian, then the continuous $\Lambda_1$-bilinear pairing
\begin{eqnarray*}
  \Lambda_1 \T_{N,\chi}^{[< s]} \times \Mod(\Gamma_1(N),\chi,\Lambda_1)^{[< s]} & \to & \Lambda_1 \\
  (A,f(q)) & \mapsto & a_1(Af)
\end{eqnarray*}
is non-degenerate, and it gives $\t{Frac}(\Lambda_1)$-linear isomorphisms
\begin{eqnarray*}
  \t{Frac}(\Lambda_1) \otimes_{\Lambda_1} \Lambda_1 \T_{N,\chi}^{[< s]} & \cong & \left( \t{Frac}(\Lambda_1) \otimes_{\Lambda_1} \Mod(\Gamma_1(N),\chi,\Lambda_1)^{[< s]} \right)^{\vee} \\
  \t{Frac}(\Lambda_1) \otimes_{\Lambda_1} \Mod(\Gamma_1(N),\chi,\Lambda_1) & \cong & \left( \t{Frac}(\Lambda_1) \otimes_{\Lambda_1} \Lambda_1 \T_{N,\chi}^{[< s]} \right)^{\vee},
\end{eqnarray*}
where $V^{\vee}$ denotes the $\t{Frac}(\Lambda_1)$-linear dual $\t{Hom}_{\t{Frac}(\Lambda_1)}(V)$ for a $\t{Frac}(\Lambda_1)$-vector space $V$.
\end{crl}

\begin{proof}
The first assertion implies the second assertion by Proposition \ref{integral}, Corollary \ref{finiteness 4} (ii), and Proposition \ref{Noetherian}, because $\Mod(\Gamma_1(N),\chi,\Lambda_1)^{[< s]}$ is a torsionfree $\Lambda_1$-module. Let $A \in \Lambda_1 \T_{N,\chi}^{[< s]}$ with $a_1(Af) = 0$ for any $f(q) \in \Mod(\Gamma_1(N),\chi,\Lambda_1)^{[< s]}$. Let $f(q) \in \Mod(\Gamma_1(N),\chi,\Lambda_1)^{[< s]}$. For any $h \in \N \backslash \ens{0}$, we have $a_h(Af) = a_1(T_h Af) = a_1(A(T_hf)) = 0$. It implies that $Af$ is a constant. Since there is no nontrivial modular form which is a constant, $Af = 0$. Therefore the proof of Proposition \ref{Noetherian} ensures $A = 0$. Thus the pairing is right non-degenerate. Let $f(q) \in \Mod(\Gamma_1(N),\chi,\Lambda_1)^{[< s]}$ with $a_1(Af) = 0$ for any $A \in \Lambda_1 \T_{N,\chi}^{[< s]}$. For any $h \in \N \backslash \ens{0}$, we have $a_h(f) = a_1(T_h f) = 0$. It implies that $f$ is a constant. Since there is no nontrivial modular form which is a constant, $f = 0$. Thus the pairing is left non-degenerate.
\end{proof}

A $\Lambda_1$-adic form $f(q)$ of level $N$ is said to be a {\it $\Lambda_1$-adic eigenform of level $N$} if $f(q) \neq 0$ and for any Hecke operator $T$, there is a $\lambda_f(T) \in \Lambda_1$ such that $(T - \lambda_f(T))f(q) = 0$. Such a system $(\lambda_f(T))_T$ is unique because $\Lambda_1$ is an integral domain by Proposition \ref{integral}. It is obvious that the specialisations of a $\Lambda_1$-adic eigenform of level $N$ at all but finitely many $\varphi \in \Omega(\Lambda_1)_{\N \cap [2,\infty)}$ are eigenforms over $\overline{\Z}_p$ of level $N$, but we do not know when the converse holds. A $\Lambda_1$-adic form $f(q)$ of level $N$ is said to be a {\it $\Lambda_1$-adic cusp form of level $N$} if $f(\varphi)(q)$ is a cuspidal eigenform over $\overline{\Q}_p$ of level $N$ for all but finitely many $\varphi \in \Omega(\Lambda_1)_{\N \cap [2,\infty)}$. A $\Lambda_1$-adic eigenform $f(q)$ of level $N$ is said to be {\it normalised} if $a_1(f) = 1$, and is said to be a {\it $\Lambda_1$-adic cuspidal eigenform of level $N$} if $f$ is a $\Lambda_1$-adic cusp form of level $N$. Let $f(q)$ be a $\Lambda_1$-adic cuspidal eigenform of level $N$. Suppose that $f$ is normalised. We have $\lambda_f(T_h) = a_h(f) = a_1(T_hf)$ for any $h \in \N \backslash \ens{0}$, and hence $f(q) = \sum_{h = 1}^{\infty} \lambda_f(T_h)q^h$. Thus $f$ is explicitly determined by the system $(\lambda_f(T))_T$. Suppose that $f$ is not necessarily normalised. We have $a_0(f) = 0$ and $a_1(f) T_h f = a_h(f) f$ for any $h \in \N \backslash \ens{0}$ by definition. It implies that $a_1(f) \neq 0$, $a_h(f) \in a_1(f) \Lambda_1$ for any $h \in \N$, and $\lambda_f(T_h) = a_1(f)^{-1}a_h(f)$ for any $h \in \N \backslash \ens{0}$. Therefore $a_1(f)^{-1}f$ is a normalised $\Lambda_1$-adic cuspidal eigenform of level $N$. Thus every $\Lambda_1$-adic cuspidal eigenform of level $N$ is given as $af(q)$ for a unique pair $(a,f)$ of a constant $a \in \Lambda_1$ and a normalised $\Lambda_1$-adic cuspidal eigenform $f$ of level $N$. Now we show a relation between normalised $\Lambda_1$ cuspidal eigenforms of level $N$ and a continuous $\Lambda_1$-algebra homomorphisms $\Lambda_1 \T_{N,\chi}^{[< s]} \to \Lambda_1$.

\begin{prp}
\label{form - family}
If $\Lambda_1$ is Noetherian, then for any $\Lambda_1$-adic eigenform $f(q)$ of level $N$ with character $\chi$ locally of slope $< s$, there is a unique continuous $\Lambda_1$-algebra homomorphism $\lambda_f \colon \Lambda_1 \T_{N,\chi}^{[< s]} \to \Lambda_1$ extending the system $(\lambda_f(T))_T$.
\end{prp}

\begin{proof}
Let $\A \subset \Lambda_1 \T_{N,\chi}^{[< s]}$ be a $\Lambda_1$-subalgebra generated by $T_{\ell}$ for each prime number $\ell$ and $S_n$ for each $n \in \N$ coprime to $N$. Then $\A$ is a dense $\Lambda_1$-subalgebra of $\Lambda_1 \T_{N,\chi}^{[< s]}$ by the definition of the inverse limit topology. It implies the uniqueness of $\lambda_f$ because $\Lambda_1$ is Hausdorff. We define $\lambda_f$ by setting $\lambda_f(A) \coloneqq a_1(Af)$ for each $A \in \Lambda_1 \T_{N,\chi}^{[< s]}$, where $Af$ is given by the action defined in Proposition \ref{universal Hecke module 2}. Then $\lambda_f$ is continuous, and since $f$ is a $\Lambda_1$-adic eigenform of level $N$, $\lambda_f |_{\A}$ is a $\Lambda_1$-algebra homomorphism. It implies that $\lambda_f$ is a $\Lambda_1$-algebra homomorphism extending $(\lambda_f(T))_T$ by the continuity of the addition and the multiplication of $\Lambda_1 \T_{N,\chi}^{[< s]}$.
\end{proof}

Thus a normalised $\Lambda_1$-adic eigenform of level $N$ with character $\chi$ locally of slope $< s$ is regarded as a continuous $\Lambda_1$-algebra homomorphism $\Lambda_1 \T_{N,\chi}^{[< s]} \to \Lambda_1$ in the case where $\Lambda_1$ is Noetherian. The converse correspondence is a little more complicated.

\begin{prp}
\label{family - form}
If $\Lambda_1$ is Noetherian, then for any continuous $\Lambda_1$-algebra homomorphism $\lambda \colon \Lambda_1 \T_{N,\chi}^{[< s]} \to \Lambda_1$, there are some $a \in \Lambda_1 \backslash \ens{0}$ and $f(q) \in \Mod(\Gamma_1(N),\chi,\Lambda_1)^{[< s]}$ such that $a \lambda(A) = a_1(Af)$ for any $A \in \Lambda_1 \T_{N,\chi}^{[< s]}$. In addition, if $f$ can be taken as a cusp form, then there uniquely exists a normalised $\Lambda_1$-adic cuspidal eigenform $f_{\lambda}(q)$ of level $N$ with character $\chi$ locally of slope $< s$ with $\lambda = \lambda_{f_{\lambda}}$.
\end{prp}

\begin{proof}
By Corollary \ref{pairing}, there uniquely exists an $f_{\lambda} \in \t{Frac}(\Lambda_1) \otimes_{\Lambda_1} \Mod(\Gamma_1(N),\chi,\Lambda_1)^{[< s]}$ such that for any $a \in \Lambda_1 \backslash \ens{0}$ and $f \in \Mod(\Gamma_1(N),\chi,\Lambda_1)^{[< s]}$ with $a^{-1} \otimes f = f_{\lambda}$, the equality $a^{-1} a_1(Af) = \lambda(A)$ holds for any $A \in \Lambda_1 \T_{N,\chi}^{[< s]}$.

\vspace{0.2in}
Suppose that $f$ can be taken as a cusp form. We have $a_0(f) =0$ and $a_h(f) = a_1(T_h f) = a \lambda(T_h) \in a \Lambda_1$ for any $h \in \N \backslash \ens{0}$, and in particular, the equality $a_1(f) = a \lambda(1) = a$ holds. Therefore $f_{\lambda}$ lies in the image of $\Mod(\Gamma_1(N),\chi,\Lambda_1)^{[< s]}$. Namely, $f_{\lambda} = \sum_{h = 1}^{\infty} \lambda(T_h) q^h \in \Mod(\Gamma_1(N),\chi,\Lambda_1)^{[< s]}$. Let $A \in \Lambda_1 \T_{N,\chi}^{[< s]}$. We have
\begin{eqnarray*}
  a_h(A f_{\lambda}) = a_1(T_h A f_{\lambda}) = a \lambda(T_h A) = a \lambda(T_h) \lambda(A) = \lambda(A) a_h(f_{\lambda})
\end{eqnarray*}
for any $h \in \N \backslash \ens{0}$, and hence $A f_{\lambda} - \lambda(A) f_{\lambda}$ is a constant. Since there is no non-trivial modular form which is a constant, we obtain $A f_{\lambda} = \lambda(A) f_{\lambda}$. In particular, the equality $S_n f_{\lambda} = \lambda(S_n) f = \lambda(\chi(n + N \Z)) f = \chi(n + N \Z) f$ holds for any $n \in \N$ coprime to $N$. Thus $f_{\lambda}$ is a normalised $\Lambda_1$-adic cuspidal eigenform of level $N$ with character $\chi$ locally of slope $< s$ with $\lambda_{f_{\lambda}} = \lambda$.

\vspace{0.2in}
Let $f(q)$ be a normalised $\Lambda_1$-adic eigenform $f_{\lambda}(q)$ of level $N$ with character $\chi$ locally of slope $< s$ with $\lambda = \lambda_f$. Then we have $a_h(f) = \lambda_f(T_h) = \lambda(T_f) = a_h(f_{\lambda})$ for any $h \in \N \backslash \ens{0}$. It implies that $f - f_{\lambda}$ is a constant. Since there is no non-trivial modular form which is a constant, we obtain $f = f_{\lambda}$. We conclude that $f_{\lambda}$ is a unique normalised $\Lambda_1$-adic eigenform of level $N$ with character $\chi$ locally of slope $< s$ with $\lambda = \lambda_{f_{\lambda}}$.
\end{proof}

We would like to add the element $p^sT_p^{-1}$ to $\Lambda_1 \T_{N,\chi}^{[< s]}$. However, the endomorphism on $\Q_p \otimes_{\Z_p} \Mod(\Gamma_1(N),\chi,\Lambda_1)^{[< s]}$ given by $T_p$ seems not to be invertible, because an endomorphism on an infinite dimensional compactly generated topological vector space with infinitely many points on the resolvent is never diagonalisable. Therefore for a $\varphi \in \t{Reg}^{< s}(\Lambda_1)$, we do not know whether or not the operator $T_p$ is invertible in $\Q_p \otimes_{\Z_p} \Mod_{\varphi}(\Gamma_1(N),\chi,\Z_p[\varphi])^{[< s]}$, and hence we can not regard $p^sT_p^{-1}$ as an element of $\Q_p \otimes_{\Z_p} \t{End}_{\Q_p}(\Mod_{\varphi}(\Gamma_1(N),\chi,\Z_p[\varphi])^{[< s]})$. Imitating the result of Proposition \ref{p^s/T_p}, we set
\begin{eqnarray*}
  \Z_p[\varphi] \t{T}_{\varphi,N,\chi}^{< s} \coloneqq \left( \Z_p[\varphi] \t{T}_{\varphi,N,\chi}^{[< s]}[X]/(T_pX-p^s) \right)_{\t{free}}
\end{eqnarray*}
for each $\varphi \in \t{Reg}^{< s}(\Lambda_1)$. Since $\Z_p[\varphi] \t{T}_{\varphi,N,\chi}^{[< s]}$ is a $\Z_p$-algebra finitely generated as a $\Z_p$-module, so is $\Z_p[\varphi] \t{T}_{\varphi,N,\chi}^{< s}$ by a similar argument with that in \S \ref{p-adic Modular Forms and Hecke Algebras}. For each finite subset $S \subset \t{Reg}^{< s}(\Lambda_1)$, we denote by 
\begin{eqnarray*}
  \left( \Lambda_1/\bigcap_{\varphi \in S} \ker(\varphi) \right) \t{T}_{\in S,N,\chi}^{< s} \subset \prod_{\varphi \in S} \Z_p[\varphi] \t{T}_{\varphi,N,\chi}^{< s}
\end{eqnarray*}
the commutative $(\Lambda_1/\bigcap_{\varphi \in S} \ker(\varphi))$-subalgebra generated by $(T_{\ell})_{\varphi \in S}$ for each prime number $\ell$, $(S_n)_{\varphi \in S}$ for each $n \in \N$ coprime to $N$, and $(p^sT_p^{-1})_{\varphi \in S}$. We set
\begin{eqnarray*}
  \Lambda_1 \T_{N,\chi}^{< s} \coloneqq \varprojlim_{S \subset \t{Reg}^{< s}(\Lambda_1)} \left( \Lambda_1/\bigcap_{\varphi \in S} \ker(\varphi) \right) \t{T}_{\in S,N,\chi}^{< s},
\end{eqnarray*}
where $S$ in the limit runs through all finite subsets of $\t{Reg}^{< s}(\Lambda_1)$. We regard it as a profinite $\Lambda_1 \hat{\otimes}_{\Lambda_0} \Lambda_0 \T_N^{< s}$-algebra by Corollary \ref{profinite ring}.

\begin{prp}
\label{topologically nilpotent 2}
The continuous $\Lambda_1 \T_{N,\chi}^{[< s]}$-algebra homomorphism
\begin{eqnarray*}
  \Lambda_1 \T_{N,\chi}^{[< s]}[[X]] & \to & \Lambda_1 \T_{N,\chi}^{< s} \\
  X & \mapsto & p^sT_p^{-1}
\end{eqnarray*}
is surjective.
\end{prp}

\begin{proof}
The assertion can be easily verified by a similar argument in the proof of Proposition \ref{topologically nilpotent}.
\end{proof}

\begin{crl}
If $\Lambda_1$ is Noetherian, then so is $\Lambda_1 \T_{N,\chi}^{< s}$.
\end{crl}

We note that every continuous $\Lambda_1$-algebra homomorphism $\lambda \colon \Lambda_1 \T_{N,\chi}^{[< s]} \to \Lambda_1$ uniquely extends to a $\Lambda_1$-algebra homomorphism
\begin{eqnarray*}
  \Lambda_1 \T_{N,\chi}^{[< s]}[X]/(T_pX-p^s) & \to & \t{Frac}(\Lambda_1) \\
  X & \mapsto & p^s \lambda(T_p)^{-1}
\end{eqnarray*}
but does not necessarily extend to a $\Lambda_1$-algebra homomorphism $\Lambda_1 \T_{N,\chi}^{< s} \to \t{Frac}(\Lambda_1)$ because the image of $p^sT_p^{-1}$ in $\t{Frac}(\Lambda_1)$ does not necessarily topologically nilpotent even if one equips $\t{Frac}(\Lambda_1)$ with suitable topologies.

\begin{dfn}
A {\it $\Lambda_1$-adic family of systems of Hecke eigenvalues of level $N$ with character $\chi$ of slope $< s$} is a continuous $\Lambda_1$-algebra homomorphism $\Lambda_1 \T_{N,\chi}^{< s} \to \Lambda_1$.
\end{dfn}

\begin{dfn}
A $\Lambda_1$-adic eigenform $f(q)$ of level $N$ with character $\chi$ is said to be {\it of slope $< s$} if $p^s a_1(f) \in a_p(f) \Lambda_1$.
\end{dfn}

Suppose that $\Lambda_1$ is Noetherian. For a normalised $\Lambda_1$-adic eigenform $f(q)$ of level $N$ with character $\chi$ of slope $< s$, the homomorphism
\begin{eqnarray*}
  \Lambda_1 \T_{N,\chi}^{[< s]}[[X]] & \to & \Lambda_1 \\
  \sum_{m = 0}^{\infty} A_m X^m & \mapsto & a_1 \left( \sum_{m = 0}^{\infty} (p^s a_p(f)^{-1})^m A_m f \right)
\end{eqnarray*}
induces a $\Lambda_1$-adic family $\lambda_f \colon \Lambda_1 \T_{N,\chi}^{< s} \to \Lambda_1$ of systems of Hecke eigenvalues of level $N$ with character $\chi$ of slope $< s$ by a similar argument with that in the proof of Proposition \ref{form - family}. Thus a normalised $\Lambda_1$-adic eigenform of level $N$ of slope $< s$ is naturally identified with a $\Lambda_1$-adic family of systems of Hecke eigenvalues of level $N$ of slope $< s$.

\vspace{0.2in}
A $\Lambda_1$-adic family $\lambda \colon \Lambda_1 \T_{N,\chi}^{< s} \to \Lambda_1$ of systems of Hecke eigenvalues of level $N$ with character $\chi$ of slope $< s$ is said to be {\it a $\Lambda_1$-adic cuspidal family of systems of Hecke eigenvalues of level $N$ with character $\chi$ of slope $< s$} if the formal power series
\begin{eqnarray*}
  f_{\lambda} \coloneqq \sum_{h = 1}^{\infty} \lambda(T_h) q^h
\end{eqnarray*}
is a normalised $\Lambda_1$-adic eigenform of level $N$ with character $\chi$ of slope $< s$ such that $f_{\lambda}(\varphi)(q)$ is a cusp form over $\Z_p[\varphi]$ of weight $\t{wt}(\varphi)$ and level $N$ for all but finitely many $\varphi \in \Omega(\Lambda_1)_{\N \cap [2,\infty)}$. A $\Lambda_1$-adic form $f(q)$ of level $N$ with character $\chi$ of slope $< s$ is said to be a {\it normalised $\Lambda_1$-adic cuspidal eigenform of level $N$ with character $\chi$ of slope $< s$} if there is a cuspidal $\Lambda_1$-adic family $\lambda_f$ of systems of Hecke eigenvalues of level $N$ with character $\chi$ of slope $< s$ such that $f = f_{\lambda_f}$. The family $\lambda_f$ is unique by a similar argument in the proof of Proposition \ref{form - family}. By definition, the notion of a normalised $\Lambda_1$-adic cuspidal eigenform of level $N$ with character $\chi$ of slope $< s$ is equivalent to that of a cuspidal $\Lambda_1$-adic family of systems of Hecke eigenvalues of level $N$ with character $\chi$ of slope $< s$ even if $\Lambda_1$ is not Noetherian.

\subsection{$p$-adic Family of Galois Representations of Finite Slope}
\label{p-adic Family of Galois Representations of Finite Slope}

Henceforth, we fix an algebraic closure $\overline{\Q}_{\ell}$ of $\Q_{\ell}$ and an isomorphism $\iota_{\ell,\infty} \colon \overline{\Q}_{\ell} \stackrel{\sim}{\to} \C$, and regard $\t{Gal}(\overline{\Q}_{\ell}/\Q_{\ell})$ as a closed subgroup of $\t{Gal}(\overline{\Q}/\Q)$ through the embedding $\iota_{\ell,\infty}^{-1} \circ \iota_{0,\infty} \colon \overline{\Q} \hookrightarrow \overline{\Q}_{\ell}$ for each prime number $\ell$ coprime to $N$. We also regard $\t{Gal}(\C/\R)$ as a closed subgroup of $\t{Gal}(\overline{\Q}/\Q)$ through the embedding $\iota_{0,\infty} \colon \overline{\Q} \hookrightarrow \C$.

\vspace{0.2in}
We denote by $\Q(\G_{\t{m}}[N]) \subset \overline{\Q}$ the subfield generated by a primitive $N$-th root of unity. We regard each Dirichlet character $\epsilon \colon (\Z/N \Z)^{\times} \to \overline{\Q}{}_p^{\times}$ as the continuous character on $\t{Gal}(\overline{\Q}/\Q)$ given as the composite
\begin{eqnarray*}
  & & \t{Gal}(\overline{\Q}/\Q) \twoheadrightarrow \t{Gal}(\overline{\Q}/\Q)/\t{Gal}(\overline{\Q}/\Q(\G_{\t{m}}[N])) \cong \t{Gal}(\Q(\G_{\t{m}}[N](\overline{\Q}))/\Q) \\
  & \stackrel{\sim}{\to} & (\Z/N \Z)^{\times} \stackrel{\epsilon}{\to} \overline{\Q}{}_p^{\times}.
\end{eqnarray*}
Let $k_0 \in \N \cap [2,\infty)$. For any normalised cuspidal eigenform $f(q)$ over $\overline{\Q}_p$ of weight $k_0$ and level $N$ with character $\epsilon$, by \cite{Shi71} Theorem 7.24 for weight $2$ and \cite{Del69} $\t{N}^{\circ}$ 3-4 for general weights $\geq 2$, there is a $2$-dimensional irreducible continuous representation $V_f$ of $\t{Gal}(\overline{\Q}/\Q)$ over $\Q_p(f)$ called {\it the Galois representation associated to $f$} satisfying the following:
\begin{itemize}
\item[(i)] The restriction $V_f |_{\t{Gal}(\C/\R)}$ is odd, i.e.\ the complex conjugate $\varphi_{\infty} \in \t{Gal}(\C/\R)$ satisfies $\det(\varphi_{\infty} \mid V_f) = -1$.
\item[(ii)] The restriction $V_f |_{\t{Gal}(\overline{\Q_{\ell}}/\Q_{\ell})}$ is unramified, i.e.\ the inertia subgroup of $\t{Gal}(\overline{\Q_{\ell}}/\Q_{\ell})$ acts trivially on $V_f$ for any prime number $\ell$ coprime to $N$.
\item[(iii)] Every lift $\varphi_{\ell} \in \t{Gal}(\overline{\Q}_{\ell}/\Q_{\ell})$ of the $\ell$-th Frobenius satisfies $\det(X - \varphi_{\ell}^{-1} \mid V_f) = X^2 - a_p(f) X + \ell^{k_0-1} \epsilon(\ell + N \Z)$ for any prime number $\ell$ coprime to $N$.
\end{itemize}
We remark that such a representation is unique up to isomorphism by Chebotarev's density theorem (\cite{Tsc26} Hauptsatz p.\ 195) and \cite{Car94} Th\'eor\`eme 1. B.\ H.\ Gross proved that for any normalised cuspidal eigenform $f(q)$ over $\overline{\Q}_p$ of weight $2$ and level $N$, the quotient of the rational Tate module of the Jacobian of $Y_1(N)$ by the $(\Q_p \otimes_{\Z_p} \t{T}_{k_0,N})$-submodule generated by $\ker(\lambda_f) = \t{Ann}_{\t{T}_{k_0,N}}(f)$ is naturally isomorphic to $V_f$ twisted by $\epsilon^{-1}$ in the proof of \cite{Gro90} Theorem 11.4. To begin with, we verify an immediate generalisation of this construction.

\begin{thm}
\label{Galois representation}
Let $k_0 \in \N \cap [2,\infty)$. For any normalised cuspidal eigenform $f(q)$ over $\overline{\Q}_p$ of weight $k_0$ and level $N$ with character $\epsilon$, the quotient $\Hil_f$ of
\begin{eqnarray*}
  \t{H}_{\t{\'et}}^1 \left( Y_1(N)_{\overline{\Q}}, \t{Sym}^{k_0-2} \left( \t{R}^1 (\pi_N)_*(\underline{\Q}{}_p)_{E_1(N)} \right) \right)
\end{eqnarray*}
by the $(\Q_p \otimes_{\Z_p} \t{T}_{k_0,N})$-submodule generated by $\ker(\lambda_f)$ is naturally isomorphic to $V_f$ twisted by $\epsilon^{-1}$.
\end{thm}

\begin{proof}
For a commutative ring $R$, a commutative $R$-algebra $A$, and a scheme $Y$ over $\t{Spec}(R)$, we put $Y_A \coloneqq Y \times_{\t{Spec}(R)} \t{Spec}(A)$. We denote by $Y_1(N)'$ the moduli of pairs $(E,\alpha)$ of an elliptic curve $E$ and a primitive $N$-torsion $\alpha \in E[N]$, and by $\pi'_N \colon E_1(N)' \to Y_1(N)'$ the universal elliptic curve. We put
\begin{eqnarray*}
  & & \Hil \coloneqq \Hil_{\t{\'et}}^1 \left( Y_1(N)_{\overline{\Q}}, \t{Sym}^{k_0-2} \left( \t{R}^1 (\pi_N)_*(\underline{\Z}{}_p)_{E_1(N)} \right) \right) \\
  & & \Hil' \coloneqq \Hil_{\t{\'et}}^1 \left( Y_1(N)'_{\overline{\Q}}, \t{Sym}^{k_0-2} \left( \t{R}^1 (\pi'_N)_*(\underline{\Z}{}_p)_{E_1(N)'} \right) \right).
\end{eqnarray*}
For each primitive $N$-th root $\zeta \in \Q(\G_{\t{m}}[N])^{\times}$ of unity, we consider the isomorphism $\iota_{\zeta} \colon (\underline{\Z/N \Z})_{\Q(\G_{\t{m}}[N])} \to \G_{\t{m}}[N]_{\Q(\G_{\t{m}}[N])}$ of group schemes over $\Q(\G_{\t{m}}[N])$ given by the $\Q(\G_{\t{m}}[N])$-algebra isomorphism
\begin{eqnarray*}
  \Q(\G_{\t{m}}[N])[X_N]/(P_N(X_N)) & \to & \Q(\G_{\t{m}}[N])^{\Z/N \Z} \\
  X_N & \mapsto & (\zeta^i)_{i + N \Z \in \Z/N \Z}
\end{eqnarray*}
and the natural equivalence
\begin{eqnarray*}
  \psi_{\zeta} \colon Y_1(N)_{\Q(\G_{\t{m}}[N])} & \to & Y_1(N)'_{\Q(\G_{\t{m}}[N])} \\
  (E, \beta) & \mapsto & (E, \beta \circ \iota_{\zeta})
\end{eqnarray*}
between moduli, where $P_N(X_N) \in \Q[X_N]$ is the $N$-th cyclotomic polynomial as in \S \ref{Profinite Zp Sheaves on Modular Curves}. For each primitive $N$-th root $\zeta \in \Q(\G_{\t{m}}[N])^{\times}$ of unity, $\psi_{\zeta}$ yields an identification of $E_1(N)_{\Q(\G_{\t{m}}[N])}$ and $E_1(N)'_{\Q(\G_{\t{m}}[N])}$, and hence we obtain a $\Q_p$-linear $\t{Gal}(\overline{\Q}/\Q(\G_{\t{m}}[N]))$-equivariant isomorphism
\begin{eqnarray*}
  \Psi_{\zeta} \colon \t{Res}_{\t{Gal}(\overline{\Q}/\Q(\G_{\t{m}}[N]))}^{\t{Gal}(\overline{\Q}/\Q)} \left( \Q_p \otimes_{\Z_p} \Hil \right) \cong \t{Res}_{\t{Gal}(\overline{\Q}/\Q(\G_{\t{m}}[N]))}^{\t{Gal}(\overline{\Q}/\Q)} \left( \Q_p \otimes_{\Z_p} \Hil' \right),
\end{eqnarray*}
and the isomorphism is Hecke-equivariant by the Eichler--Shimura isomorphism (\cite{Shi59} 5 Th\'eor\`eme 1, \cite{Hid93} 6.3 Theorem 4) because both of the analytifications of $Y_1(N)_{\C}$ and $Y_1(N)'_{\C}$ are biholomorphic to $\Gamma_1(N) \backslash \Hlf$. Therefore $\t{Res}_{\t{Gal}(\overline{\Q}/\Q(\G_{\t{m}}[N]))}^{\t{Gal}(\overline{\Q}/\Q)}(\Hil_f)$ is isomorphic to $\t{Res}_{\t{Gal}(\overline{\Q}/\Q(\G_{\t{m}}[N]))}^{\t{Gal}(\overline{\Q}/\Q)}(V_f)$ by \cite{Del69} $\t{N}^{\circ}$ 3-4.

\vspace{0.2in}
Let $\zeta_N \in \Q(\G_{\t{m}}[N])^{\times}$ be a primitive $N$-th root of unity. For each $\varphi \in \t{Gal}(\overline{\Q}/\Q)$, we denote by $\varphi^* \overline{\Q}$ the $\overline{\Q}$-algebra which shares the underlying ring with $\overline{\Q}$ and whose $\overline{\Q}$-algebra structure is given by $\varphi$, and by $\langle \varphi \rangle$ the diamond operator $\langle \overline{n}_{\varphi} \rangle$ regarded as a correspondence on $Y_1(N)$, where $\overline{n}_{\varphi} \in (\Z/N \Z)^{\times}$ is the image of $\varphi$ by the natural homomorphism
\begin{eqnarray*}
  \t{Gal}(\overline{\Q}/\Q) \twoheadrightarrow \t{Gal}(\overline{\Q}/\Q)/\t{Gal}(\overline{\Q}/\Q(\G_{\t{m}}[N])) \cong \t{Gal}(\Q(\G_{\t{m}}[N])/\Q) \cong (\Z/N \Z)^{\times}
\end{eqnarray*}
Now $\Q_p \otimes_{\Z_p} \Hil$ and $\Q_p \otimes_{\Z_p} \Hil'$ are finitely generated $(\Q_p \otimes_{\Z_p} \t{T}_{k_0,N})$-module with $(\Q_p \otimes_{\Z_p} \t{T}_{k_0,N})$-linear actions of $\t{Gal}(\overline{\Q}/\Q)$. For any $\varphi \in \t{Gal}(\overline{\Q}/\Q)$, the composite
\begin{eqnarray*}
  Y_1(N)_{\overline{\Q}} \xrightarrow[]{\psi_{\zeta}} Y_1(N)'_{\overline{\Q}} \xrightarrow[]{\t{id}_{Y_1(N)'} \times \varphi} Y_1(N)'_{(\varphi^{-1})^* \overline{\Q}} \xrightarrow[]{\psi_{\zeta}^{-1}} Y_1(N)_{(\varphi^{-1})^* \overline{\Q}} \xrightarrow[]{\t{id}_{Y_1(N)} \times \varphi^{-1}} Y_1(N)_{\overline{\Q}}
\end{eqnarray*}
is given by the natural transform
\begin{eqnarray*}
  & & (E,\beta) \mapsto (E, \beta \circ \iota_{\zeta}) \stackrel{\t{id}}{\mapsto} (E, \beta \circ \iota_{\zeta}) \mapsto \left( E, \beta \circ \iota_{\zeta} \circ \iota_{\varphi(\zeta)}^{-1} \right) \stackrel{\t{id}}{\mapsto} \left( E, \beta \circ \iota_{\zeta} \circ \iota_{\varphi(\zeta)}^{-1} \right) \\
  & = & \left( E, \beta \circ \iota_{\zeta} \circ \iota_{\zeta^{\overline{n}_{\varphi}}}^{-1} \right) = \left( E, \beta^{\overline{n}{}_{\varphi}^{-1}} \right),
\end{eqnarray*}
between moduli over $\overline{\Q}$, and hence coincides with the correspondence $\langle \varphi \rangle^{-1}$. It implies that $\Psi_{\zeta_N}$ gives a $(\Q_p \otimes_{\Z_p} \t{T}_{k_0,N})$-linear $\t{Gal}(\overline{\Q}/\Q)$-equivariant isomorphism between $\Hil$ and $\Hil'$ twisted by the character
\begin{eqnarray*}
  \t{Gal}(\overline{\Q}/\Q) & \to & \Q_p \otimes \t{T}_{k_0,N}^{\times} \\
  \varphi & \mapsto & 1 \otimes \langle \varphi \rangle^{-1}.
\end{eqnarray*}
Thus $\Hil_f$ is isomorphic to $V_f$ twisted by $\epsilon^{-1}$ as a $\Q_p(f)$-linear representation of $\t{Gal}(\overline{\Q}/\Q)$.
\end{proof}

As in \S \ref{p-adic Family of Modular Forms of Finite Slope}, let $\Lambda_1$ be a $\Lambda$-adic domain, and $\chi$ a group homomorphism $(\Z/N \Z)^{\times} \to \Lambda_1^{\times}$. Henceforth, we regard $\chi$ as the continuous character on $\t{Gal}(\overline{\Q}/\Q)$ given as the composite
\begin{eqnarray*}
  & & \t{Gal}(\overline{\Q}/\Q) \twoheadrightarrow \t{Gal}(\overline{\Q}/\Q)/\t{Gal}(\overline{\Q}/\Q(\G_{\t{m}}[N])) \cong \t{Gal}(\Q(\G_{\t{m}}[N])/\Q) \\
  & \stackrel{\sim}{\to} & (\Z/N \Z)^{\times} \stackrel{\chi}{\to} \Lambda_1^{\times}.
\end{eqnarray*}
Let $\varpi_{N,\chi,\Lambda_1} \colon \Lambda_0 \T_N^{< s} \to \Lambda_1 \T_{N,\chi}^{< s}$ denote the natural homomorphism. For a normalised $\Lambda_1$-adic cuspidal eigenform $f(q)$ of level $N$ with character $\chi$ of slope $< s$, following a similar convention to that in Example \ref{tensor}, we put
\begin{eqnarray*}
  \Hil_f^{< s} \coloneqq \left( \Lambda_1, \lambda_f \right) \otimes_{\Lambda_1 \T_{N,\chi}^{< s}} \left( \Lambda_1 \T_{N,\chi}^{< s}, \varpi_{N,\chi,\Lambda_1} \right) \otimes_{\Lambda_0 \T_N^{< s}} \left( \int_{\Z_p}^{\boxplus} \Hil_{\t{et}}^1 \left( Y_1(N)_{\overline{\Q}}, \Fil_{k-2} \right)^{< s} dk \right)
\end{eqnarray*}
and regard it as a profinite $\Lambda_1[\t{Gal}(\overline{\Q}/\Q)]$-module, which is finitely generated as a $\Lambda_1$-module by Theorem \ref{finiteness}.

\begin{lmm}
\label{generic}
Suppose $p^s \mid N$. For any normalised $\Lambda_1$-adic cuspidal eigenform $f(q)$ of level $N$ of slope $< s$, $\t{Frac}(\Lambda_1) \otimes_{\Lambda_1} \Hil_f^{< s}$ is a $2$-dimensional representation of $\t{Gal}(\overline{\Q}/\Q)$ over $\t{Frac}(\Lambda_1)$.
\end{lmm}

\begin{proof}
By Theorem \ref{finiteness}, $\t{Frac}(\Lambda_1) \otimes_{\Lambda_1} \Hil_f^{< s}$ is a finite dimensional $\t{Frac}(\Lambda_1)$-vector space. Put $d \coloneqq \dim_{\t{Frac}(\Lambda_1)} (\t{Frac}(\Lambda_1) \otimes_{\Lambda_1} \Hil_f^{< s}) < \infty$. It suffices to verify $d = 2$. Since $f$ is a normalised $\Lambda_1$-adic cuspidal eigenform, there is a finite subset $S \subset \Omega(\Lambda_1)_{\N \cap [2,\infty)}$ such that $f(\varphi)(q)$ is a normalised cuspidal eigenform over $\overline{\Q}_p$ of weight $\t{wt}(\varphi)$ and level $N$ for any $\varphi \in \Omega(\Lambda_1)_{\N \cap [2,\infty)} \backslash S$. Let $\varphi \in \Omega(\Lambda_1)_{\N \cap [2,\infty)} \backslash S$. We have
\begin{eqnarray*}
  f(\varphi)(q) = \sum_{h = 1}^{\infty} \varphi(a_h(f)) q^h = \sum_{h = 1}^{\infty} \varphi(\lambda_f(T_h)) q^h = \sum_{h = 1}^{\infty} (\varphi \circ \lambda_f)(T_h) q^h.
\end{eqnarray*}
Since $f(\varphi)$ is a normalised cuspidal eigenform over $\overline{\Q}_p$ of weight $\t{wt}(\varphi)$ and level $N$, there is a $\Z_p$-algebra homomorphism $\lambda_{f(\varphi)} \colon \t{T}_{\t{wt}(\varphi),N} \to \overline{\Z}_p$ such that $f(\varphi) = \sum_{h = 1}^{\infty} \lambda_{f(\varphi)}(T_h) q^h$ by the duality (\cite{Hid93} 5.3 Theorem 1). Since $f(\varphi)$ is of slope $< s$, we have $\v{a_p(f(\varphi))} > \v{p}^s$, and hence $\lambda_{f(\varphi)}$ uniquely extends to a $\Z_p[\varphi]$-algebra homomorphism $\lambda_{f(\varphi)} \colon \Z_p[\varphi] \t{T}_{\varphi,N,\chi}^{< s} \to \overline{\Z}_p$ by Proposition \ref{p^s/T_p} and the duality again. By the definition of $\lambda_{f(\varphi)}$, $\lambda_{f(\varphi)}(\t{T}_{\t{wt}(\varphi),N}^{< s})$ is a $\Z_p$-subalgebra of $\lambda_{f(\varphi)}(\Z_p[\varphi] \t{T}_{\varphi,N,\chi}^{< s}) = (\varphi \circ \lambda_f)(\Lambda_1 \T_{N,\chi}^{< s}) = \varphi(\Lambda_1)$. Let $\varpi_{\varphi} \colon \Lambda_1 \T_{N,\chi}^{< s} \twoheadrightarrow \Z_p[\varphi] \t{T}_{\varphi,N,\chi}^{< s}$ denote the canonical projection. We have
\begin{eqnarray*}
  (\varphi \circ \lambda_f)(T_h) = a_h(f(\varphi)) = \lambda_{f(\varphi)}(T_h) = \lambda_{f(\varphi)}(\varpi_{\varphi}(T_h)) = (\lambda_{f(\varphi)} \circ \varpi_{\varphi})(T_h)
\end{eqnarray*}
for any $h \in \N \backslash \ens{0}$,
\begin{eqnarray*}
  & & (\varphi \circ \lambda_f)(S_n) = \varphi(\lambda_f(\chi(n + N \Z))) = (\varphi \circ \chi)(n + N \Z) = \lambda_{f(\varphi)}(S_n) = \lambda_{f(\varphi)}(\varpi_{\varphi}(S_n)) \\
  & = & (\lambda_{f(\varphi)} \circ \varpi_{\varphi})(S_n)
\end{eqnarray*}
for any $n \in N$ coprime to $N$, and
\begin{eqnarray*}
  & & (\varphi \circ \lambda_f)(p^sT_p^{-1}) = p^s (\varphi \circ \lambda_f)(T_p)^{-1} = p^s \lambda_{f(\varphi)}(T_p)^{-1} \\
  & = & \lambda_{f(\varphi)}(p^sT_p^{-1}) = \lambda_{f(\varphi)}(\varpi_{\varphi}(p^sT_p^{-1})) = (\lambda_{f(\varphi)} \circ \varpi_{\varphi})(p^sT_p^{-1}).
\end{eqnarray*}
Since these operators generate a dense $\Lambda_1$-subalgebra of $\Lambda_1 \T_{N,\chi}^{< s}$ by the definition of the inverse limit topology, we obtain $\varphi \circ \lambda_f = \lambda_{f(\varphi)} \circ \varpi_{\varphi}$. In particular, we get an inclusion
\begin{eqnarray*}
  \Lambda_1 \T_{N,\chi}^{< s} P_{\t{wt}(\varphi)}^{< s} \subset \ker(\varpi_{\varphi}) \subset \ker(\varphi \circ \lambda_f)
\end{eqnarray*}
of ideals of $\Lambda_1 \T_{N,\chi}^{< s}$. By Lemma \ref{specialisation 3}, we obtain
\begin{eqnarray*}
  & & \left( \Lambda_1/\ker(\varphi) \right) \otimes_{\Lambda_1} \Hil_f^{< s} \cong \left( \varphi(\Lambda_1), \varphi \right) \otimes_{\Lambda_1} \Hil_f^{< s} \\
  & \cong & \left( \varphi(\Lambda_1), \varphi \right) \otimes_{\Lambda_1} \left( \Lambda_1, \lambda_f \right) \otimes_{\Lambda_1 \T_{N,\chi}^{< s}} \left( \Lambda_1 \T_{N,\chi}^{< s}, \varpi_{N,\chi,\Lambda_1} \right) \\
  & & \otimes_{\Lambda_0 \T_N^{< s}} \left( \int_{\Z_p}^{\boxplus} \Hil_{\t{et}}^1 \left( Y_1(N)_{\overline{\Q}}, \Fil_{k-2} \right)^{< s} dk \right) \\
  & \cong & \left( \varphi(\Lambda_1), \varphi \circ \lambda_f \right) \otimes_{\Lambda_1 \T_{N,\chi}^{< s}} \left( \Lambda_1 \T_{N,\chi}^{< s}, \varpi_{N,\chi,\Lambda_1} \right) \otimes_{\Lambda_0 \T_N^{< s}} \left( \int_{\Z_p}^{\boxplus} \Hil_{\t{et}}^1 \left( Y_1(N)_{\overline{\Q}}, \Fil_{k-2} \right)^{< s} dk \right) \\
  & \cong & \left( \varphi(\Lambda_1), \lambda_{f(\varphi)} \circ \varpi_{\varphi} \right) \otimes_{\Lambda_1 \T_{N,\chi}^{< s}} \left( \Lambda_1 \T_{N,\chi}^{< s}, \varpi_{N,\chi,\Lambda_1} \right) \\
  & & \otimes_{\Lambda_0 \T_N^{< s}} \left( \int_{\Z_p}^{\boxplus} \Hil_{\t{et}}^1 \left( Y_1(N)_{\overline{\Q}}, \Fil_{k-2} \right)^{< s} dk \right) \\
  & \cong & \left( \varphi(\Lambda_1), \lambda_{f(\varphi)} \circ \varpi_{\varphi} \right) \otimes_{\Lambda_1 \T_{N,\chi}^{< s}/\Lambda_1 \T_{N,\chi}^{< s} P_{\t{wt}(\varphi)}^{< s}} \left( \Lambda_1 \T_{N,\chi}^{< s}/\Lambda_1 \T_{N,\chi}^{< s} P_{\t{wt}(\varphi)}^{< s}, \varpi_{N,\chi,\Lambda_1} \right) \\
  & & \otimes_{\Lambda_0 \T_N^{< s}/P_{\t{wt}(\varphi)}^{< s}} \left( \Lambda_0 \T_N^{< s}/P_{\t{wt}(\varphi)}^{< s} \right) \otimes_{\Lambda_0 \T_N^{< s}} \left( \int_{\Z_p}^{\boxplus} \Hil_{\t{et}}^1 \left( Y_1(N)_{\overline{\Q}}, \Fil_{k-2} \right)^{< s} dk \right) \\
  & \cong & \left( \varphi(\Lambda_1), \lambda_{f(\varphi)} \circ \varpi_{\varphi} \right) \otimes_{\Lambda_1 \T_{N,\chi}^{< s}/\Lambda_1 \T_{N,\chi}^{< s} P_{\t{wt}(\varphi)}^{< s}} \left( \Lambda_1 \T_{N,\chi}^{< s}/\Lambda_1 \T_{N,\chi}^{< s} P_{\t{wt}(\varphi)}^{< s}, \varpi_{N,\chi,\Lambda_1} \right) \\
  & & \otimes_{\Lambda_0 \T_N^{< s}/P_{\t{wt}(\varphi)}^{< s}} \Hil_{\t{et}}^1 \left( Y_1(N)_{\overline{\Q}}, \Fil_{\t{wt}(\varphi)-2} \right)^{< s} \\
  & \cong & \left( \varphi(\Lambda_1), \lambda_{f(\varphi)} \right) \otimes_{\t{T}_{\t{wt}(\varphi),N}^{< s}/\ker(\lambda_{f(\varphi)})} \left( \t{T}_{\t{wt}(\varphi),N}^{< s}/\ker(\lambda_{f(\varphi)}) \right) \otimes_{\t{T}_{\t{wt}(\varphi),N}^{< s}} \Hil_{\t{et}}^1 \left( Y_1(N)_{\overline{\Q}}, \Fil_{\t{wt}(\varphi)-2} \right)^{< s} \\
  & \cong & \varphi(\Lambda_1) \otimes_{\lambda_{f(\varphi)}(\t{T}_{\t{wt}(\varphi),N}^{< s})} \left( \lambda_{f(\varphi)}(\t{T}_{\t{wt}(\varphi),N}^{< s}), \lambda_{f(\varphi)} \right) \otimes_{\t{T}_{\t{wt}(\varphi),N}^{< s}} \Hil_{\t{et}}^1 \left( Y_1(N)_{\overline{\Q}}, \Fil_{\t{wt}(\varphi)-2} \right)^{< s},
\end{eqnarray*}
and hence putting $V \coloneqq \Q_p \otimes_{\Z_p} \Hil_f^{< s}$ and $\Q_p[\varphi] \coloneqq \Q_p \otimes_{\Z_p} \Z_p[\varphi]$, we get
\begin{eqnarray*}
  & & \left( \Q_p[\varphi], \t{id}_{\Q_p} \otimes \varphi \right) \otimes_{\Q_p \otimes_{\Z_p} \Lambda_1} V \cong \left( \left( \Q_p \otimes_{\Z_p} \Lambda_1 \right)/\ker(\t{id}_{\Q_p} \otimes \varphi) \right) \otimes_{\Q_p \otimes_{\Z_p} \Lambda_1} V \\
  & \cong & \Q_p \otimes_{\Z_p} \left( \Lambda_1/\ker(\varphi) \right) \otimes_{\Lambda_1} \Hil_f^{< s} \\
  & \cong & \Q_p \otimes_{\Z_p} \varphi(\Lambda_1) \otimes_{\lambda_{f(\varphi)}(\t{T}_{\t{wt}(\varphi),N}^{< s})} \left( \lambda_{f(\varphi)}(\t{T}_{\t{wt}(\varphi),N}^{< s}), \lambda_{f(\varphi)} \right) \otimes_{\t{T}_{\t{wt}(\varphi),N}^{< s}} \Hil_{\t{et}}^1 \left( Y_1(N)_{\overline{\Q}}, \Fil_{\t{wt}(\varphi)-2} \right)^{< s} \\
  & \cong & \Q_p[\varphi] \otimes_{\Q_p(f(\varphi))} \left( \Q_p(f(\varphi)), \t{id}_{\Q_p} \otimes \lambda_{f(\varphi)} \right) \\
  & & \otimes_{\Q_p \otimes_{\Z_p} \t{T}_{\t{wt}(\varphi),N}^{< s}} \t{H}_{\t{et}}^1 \left( Y_1(N)_{\overline{\Q}}, \t{Sym}^{\t{wt}(\varphi)-2} \left( \t{R}^1 (\pi_N)_* (\underline{\Q}{}_p)_{E_1(N)} \right) \right)^{< s}
\end{eqnarray*}
by Remark \ref{symmetric product 5}, where $\Q_p(f(\varphi))$ is the $p$-adic Hecke field $\Q_p \otimes_{\Z_p} \lambda_{f(\varphi)}(\t{T}_{\varphi,N,\chi}^{< s})$ associated to $f(\varphi)$ introduced in \S \ref{p-adic Modular Forms and Hecke Algebras}. The last term is the base change by the finite extension $\Q_p[\varphi]/\Q_p(f(\varphi))$ of the Galois representation associated to $f(\varphi)$ twisted by $\chi^{-1}$ by Theorem \ref{Galois representation}.

\vspace{0.2in}
We regard $\Q_p \otimes_{\Z_p} \Lambda_1$ as a $\Lambda_1$-subalgebra of $\t{Frac}(\Lambda_1)$, and identify $V$ with the $(\Q_p \otimes_{\Z_p} \Lambda_1)$-submodule $(\Q_p \otimes_{\Z_p} \Lambda_1) \otimes_{\Lambda_1} \Hil_f^{< s}$ of $\t{Frac}(\Lambda_1) \otimes_{\Lambda_1} \Hil_f^{< s}$. Take a $\t{Frac}(\Lambda_1)$-linear basis $E = \set{c_i}{i \in \N \cap [1,d]} \subset V$ of $\t{Frac}(\Lambda_1) \otimes_{\Lambda_1} \Hil_f^{< s}$. We consider a $(\Q_p \otimes_{\Z_p} \Lambda_1)$-linear homomorphism
\begin{eqnarray*}
  \iota \colon (\Q_p \otimes_{\Z_p} \Lambda_1)^d & \to & V \\
  (F_i)_{i = 1}^{d} & \mapsto & \sum_{i = 1}^{d} F_i c_i.
\end{eqnarray*}
Since $E$ is a set of $\t{Frac}(\Lambda_1)$-linearly independent elements, $\iota$ is injective. Since $V$ is finitely generated as a $(\Q_p \otimes_{\Z_p} \Lambda_1)$-module and $E$ generates $\t{Frac}(\Lambda_1) \otimes_{\Lambda_1} \Hil_f^{< s}$ as a $\t{Frac}(\Lambda_1)$-vector space, $\coim(\iota)$ is a torsion $(\Q_p \otimes_{\Z_p} \Lambda_1)$-module with non-trivial annihilators. Let $D \in (\Q_p \otimes_{\Z_p} \Lambda_1) \backslash \ens{0}$ be an annihilator of $\coim(\iota)$. Put $C \coloneqq \t{Spec}(\Q_p \otimes_{\Z_p} \Lambda_1)$. We denote by $\A \subset \t{Frac}(\Lambda_1)$ the localisation $(\Q_p \otimes_{\Z_p} \Lambda_1)[D^{-1}]$, and by $U \subset C$ the image of the open immersion $\t{Spec}(\A) \hookrightarrow C$. Then $\iota$ induces an $\A$-linear homomorphism
\begin{eqnarray*}
  \iota_U \colon \A^d & \to & \A \otimes_{\Q_p \otimes_{\Z_p} \Lambda_1} V \\
  (F_i)_{i = 1}^{d} & \mapsto & \sum_{i = 1}^{d} F_i c_i.
\end{eqnarray*}
The right exactness of the functor $\A \otimes_{\Q_p \otimes_{\Z_p} \Lambda_1} (\cdot)$ ensures $\coim(\iota_U) \cong \A \otimes_{\Q_p \otimes_{\Z_p} \Lambda_1} \coim(\iota) \cong 0$. Since $V \subset \t{Frac}(\Lambda_1) \otimes_{\Lambda_1} \Hil_f^{< s}$ is a torsionfree $(\Q_p \otimes_{\Z_p} \Lambda_1)$-module, the natural $\A$-linear homomorphism $\A \otimes_{\Q_p \otimes_{\Z_p} \Lambda_1} V \to \t{Frac}(\Lambda_1) \otimes_{\Q_p \otimes_{\Z_p} \Lambda_1} V \cong \t{Frac}(\Lambda_1) \otimes_{\Lambda_1} \Hil_f^{< s}$ is injective, and hence $\iota_U$ is injective. It implies that $\iota_U$ is an $\A$-linear isomorphism. Since $D \neq 0$, there is a $\varphi \in \t{Reg}^{< s}(\Lambda_1)$ with $\varphi(D) \neq 0$. The specialisation
\begin{eqnarray*}
  \t{id}_{\Q_p} \otimes \varphi \colon \Q_p \otimes_{\Z_p} \Lambda_1 & \to & \Q_p \otimes_{\Z_p} \Z_p[\varphi] = \Q_p[\varphi] \\
  c \otimes F & \mapsto & c \varphi(F)
\end{eqnarray*}
induces a $\Q_p[\varphi]$-linear homomorphism
\begin{eqnarray*}
  \iota_{\varphi} \colon \Q_p[\varphi]^d & \to & \left( \Q_p[\varphi], \t{id}_{\Q_p} \otimes \varphi \right) \otimes_{\Q_p \otimes_{\Z_p} \Lambda_1} V \\
  (\alpha_i)_{i = 1}^{d} & \mapsto & \sum_{i = 1}^{d} \alpha_i c_i.
\end{eqnarray*}
Since $\varphi(D) \neq 0$, $\t{id}_{\Q_p} \otimes \varphi$ factors through $\A$ by the universality of the localisation. Therefore the bijectivity of $\iota_U$ ensures that of $\iota_{\varphi}$, because $\A^d$ is a free $\A$-module. The target of $\iota_{\varphi}$ is isomorphic to the Galois representation associated to $f(\varphi)(q)$ over $\Q_p[\varphi]$ twisted by $\epsilon^{-1}$, and hence we obtain $d = 2$.
\end{proof}

Let $X$ be a topological space. For each $x \in X$, we denote by $\overline{\Q}_p(x)$ the $\t{C}(X,\overline{\Q}_p)$-algebra of dimension $1$ as a $\overline{\Q}_p$-vector space given as the quotient of $\t{C}(X,\overline{\Q}_p)$ by the maximal ideal $\set{F \in \t{C}(X,\overline{\Q}_p)}{F(x) = 0}$. Let $G$ be a monoid. For a $\t{C}(X,\overline{\Q}_p)$-module $M$ endowed with a $\t{C}(X,\overline{\Q}_p)$-linear action of $G$, we call $\overline{\Q}_p(x) \otimes_{\t{C}(X,\overline{\Q}_p)} M$ {\it the specialisation of $M$ at $x$}, and regard it as a $\overline{\Q}_p$-linear representation of $G$. In the case where $X$ is a subspace of $\t{Reg}^{< s}(\Lambda_1)$, for a $\varphi \in X$, we hope that a reader does not confuse $\overline{\Q}_p(\varphi)$ with $\Q_p[\varphi]$. In this case, we regard $\t{C}(X,\overline{\Q}_p)$ as a $\Lambda_1$-algebra in a similar way with that introduced after Proposition \ref{profinite}.

\begin{dfn}
We denote by $V_f^{< s}$ the profinite $\Lambda_1[\t{Gal}(\overline{\Q}/\Q)]$-module which shares the underlying topological $\Lambda_1$-module with $\Hil_f^{< s}$ and on which the action of $\t{Gal}(\overline{\Q}/\Q)]$ is given by that on $\Hil_f^{< s}$ twisted by $\chi$.
\end{dfn}

\begin{thm}
\label{geometric construction 2}
Suppose $p^s \mid N$. Then there is a finite subset $\Sigma_s \subset \Omega(\Lambda_1)_{\N \cap [2,\infty)}$ satisfying the following for any normalised $\Lambda_1$-adic cuspidal eigenform $f(q)$ of level $N$ with character $\chi$ of slope $< s$:
\begin{itemize}
\item[(i)] For any $\varphi \in \Omega(\Lambda_1)_{\N \cap [2,\infty)} \backslash \Sigma_s$, $f(\varphi)(q)$ is a normalised cuspidal eigenform over $\Z_p[\varphi]$ of weight $\t{wt}(\varphi)$ and level $N$ with character $\chi$ of slope $< s$.
\item[(ii)] The $\t{C}(\Omega(\Lambda_1)_{\N \cap [2,\infty)} \backslash \Sigma_s, \overline{\Q}_p)$-module
\begin{eqnarray*}
  V_f^{< s} |_{\Omega(\Lambda_1)_{\N \cap [2,\infty)} \backslash \Sigma_s} \coloneqq \t{C} \left( \Omega(\Lambda_1)_{\N \cap [2,\infty)} \backslash \Sigma_s, \overline{\Q}_p \right) \otimes_{\Lambda_1} V_f^{< s}
\end{eqnarray*}
endowed with a $\t{C}(\Omega(\Lambda_1)_{\N \cap [2,\infty)} \backslash \Sigma_s, \overline{\Q}_p)$-linear action of $\t{Gal}(\overline{\Q}/\Q)$ is free of rank $2$ as a $\t{C}(\Omega(\Lambda_1)_{\N \cap [2,\infty)} \backslash \Sigma_s, \overline{\Q}_p)$-module.
\item[(iii)] For any $\varphi \in \Omega(\Lambda_1)_{\N \cap [2,\infty)} \backslash \Sigma_s$, the specialisation
\begin{eqnarray*}
  \overline{\Q}_p(\varphi) \otimes_{\t{C} \left( \Omega(\Lambda_1)_{\N \cap [2,\infty)} \backslash \Sigma_s, \overline{\Q}_p \right)} V_f^{< s} |_{\Omega(\Lambda_1)_{\N \cap [2,\infty)} \backslash \Sigma_s}
\end{eqnarray*}
of $V_f^{< s} |_{\Omega(\Lambda_1)_{\N \cap [2,\infty)} \backslash \Sigma_s}$ at $\varphi$ is naturally isomorphic to the Galois representation associated to $f(\varphi)(q)$ over $\overline{\Q}_p$.
\end{itemize}
\end{thm}

\begin{proof}
Normalised $\Lambda_1$-adic cuspidal eigenforms of level $N$ with character $\chi$ of slope $< s$ are $\t{Frac}(\Lambda_1)$-linearly independent because they are simultaneous eigenvectors of $\t{Frac}(\Lambda_1)$-linear operators $\set{T_h}{h \in \N \backslash \ens{0}}$ with pairwise distinct systems of eigenvalues. By Corollary \ref{finiteness 4} (i), $\t{Frac}(\Lambda_1) \otimes_{\Lambda_1} \Mod(\Gamma_1(N),\chi,\Lambda_1)^{[< s]}$ is a finite dimensional $\t{Frac}(\Lambda_1)$-vector space, and hence there are at most finitely many normalised $\Lambda_1$-adic cuspidal eigenforms of level $N$ with character $\chi$ of slope $< s$. Therefore it suffices to verify that for any normalised $\Lambda_1$-adic cuspidal eigenform $f(q)$ of level $N$ with character $\chi$ of slope $< s$, there is a finite subset $S_0 \subset \Omega(\Lambda_1)_{\N \cap [2,\infty)}$ such that  $f(\varphi)(q)$ is a normalised cuspidal eigenform over $\Z_p[\varphi]$ of weight $\t{wt}(\varphi)$ and level $N$ with character $\chi$ of slope $< s$ for any $\varphi \in \Omega(\Lambda_1)_{\N \cap [2,\infty)} \backslash S_0$, and for any finite subset $S \subset \Omega(\Lambda_1)_{\N \cap [2,\infty)}$ containing $S_0$, the $\t{C}(\Omega(\Lambda_1)_{\N \cap [2,\infty)} \backslash S, \overline{\Q}_p)$-module
\begin{eqnarray*}
  \Hil_f^{< s} |_{\Omega(\Lambda_1)_{\N \cap [2,\infty)} \backslash S} \coloneqq \t{C} \left( \Omega(\Lambda_1)_{\N \cap [2,\infty)} \backslash S, \overline{\Q}_p \right) \otimes_{\Lambda_1} \Hil_f^{< s}
\end{eqnarray*}
is free of rank $2$ as a $\t{C}(\Omega(\Lambda_1)_{\N \cap [2,\infty)} \backslash S, \overline{\Q}_p)$-module, and its specialisation
\begin{eqnarray*}
  \overline{\Q}_p(\varphi) \otimes_{\t{C} \left( \Omega(\Lambda_1)_{\N \cap [2,\infty)} \backslash S, \overline{\Q}_p \right)} \Hil_f^{< s} |_{\Omega(\Lambda_1)_{\N \cap [2,\infty)} \backslash S}
\end{eqnarray*}
is naturally isomorphic to the Galois representation associated to $f(\varphi)(q)$ over $\overline{\Q}_p$ twisted by the character on $\t{Gal}(\overline{\Q}/\Q)$ induced by $\varphi \circ \chi$ for any $\varphi \in \Omega(\Lambda_1)_{\N \cap [2,\infty)} \backslash S$.

\vspace{0.2in}
Let $S_1 \subset \Omega(\Lambda_1)_{\N \cap [2,\infty)}$ denote the finite subset consisting of elements $\varphi$ such that $f(\varphi)(q)$ is not a normalised cuspidal eigenform over $\Z_p[\varphi]$ of weight $\t{wt}(\varphi)$ and level $N$ with character $\chi$ of slope $< s$. By the proof of Lemma \ref{generic}, there is a $D \in \Lambda_1 \backslash \ens{0}$ such that 
\begin{eqnarray*}
  \left( \Q_p[\varphi], \t{id}_{\Q_p} \otimes \varphi \right) \otimes_{\Q_p \otimes_{\Z_p} \Lambda_1} \left( \Q_p \otimes_{\Z_p} \Hil_f^{< s} \right)
\end{eqnarray*}
is naturally isomorphic to the Galois representation associated to $f(\varphi)(q)$ over $\overline{\Q}_p$ twisted by $\varphi \circ \chi$ for any $\varphi \in \t{Reg}^{< s}(\Lambda_1)$ with $\varphi(D) \neq 0$. Let $S_2 \subset \Omega(\Lambda_1)_{\N \cap [2,\infty)}$ denote the support $\set{\varphi \in \Omega(\Lambda_1)_{\N \cap [2,\infty)}}{\varphi(D) = 0}$. Since $\Lambda_1$ is a $\Lambda$-adic domain, $S_2$ is a finite subset. Put $S_0 \coloneqq S_1 \cup S_2 \cup (\Omega(\Lambda_1)_{\N \cap [2,\infty)} \backslash \t{Reg}^{< s}(\Lambda_1))$. Let $S \subset \Omega(\Lambda_1)_{\N \cap [2,\infty)}$ be a finite subset containing $S_0$. Then for any $\varphi \in \Omega(\Lambda_1)_{\N \cap [2,\infty)} \backslash S$, we have
\begin{eqnarray*}
  & & \overline{\Q}_p(\varphi) \otimes_{\t{C} \left( \Omega(\Lambda_1)_{\N \cap [2,\infty)} \backslash S, \overline{\Q}_p \right)} \Hil_f^{< s} |_{\Omega(\Lambda_1)_{\N \cap [2,\infty)} \backslash S} \\
  & \cong & \left( \overline{\Q}_p(\varphi), \t{id}_{\Q_p} \otimes \varphi \right) \otimes_{\Q_p \otimes_{\Z_p} \Lambda_1} \left( \Q_p \otimes_{\Z_p} \Lambda_1 \right) \otimes_{\Lambda_1} \Hil_f^{< s} \\
  & \cong & \left( \overline{\Q}_p(\varphi), \t{id}_{\Q_p} \otimes \varphi \right) \otimes_{\Q_p \otimes_{\Z_p} \Lambda_1} \left( \Q_p \otimes_{\Z_p} \Hil_f^{< s} \right),
\end{eqnarray*}
and hence $\overline{\Q}_p(\varphi) \otimes_{\t{C}(\Omega(\Lambda_1)_{\N \cap [2,\infty)} \backslash S, \overline{\Q}_p)} \Hil_f^{< s} |_{\Omega(\Lambda_1)_{\N \cap [2,\infty)} \backslash S}$ is naturally isomorphic to the Galois representation associated to $f(\varphi)(q)$ over $\overline{\Q}_p$ twisted by $\varphi \circ \chi^{-1}$ because $\varphi$ is contained in $\t{Reg}^{< s}(\Lambda_1) = \Omega(\Lambda_1) \backslash S_1$.

\vspace{0.2in}
We verify that $\Hil_f^{< s} |_{\Omega(\Lambda_1)_{\N \cap [2,\infty)} \backslash S}$ is a free $\t{C}(\Omega(\Lambda_1)_{\N \cap [2,\infty)} \backslash S, \overline{\Q}_p)$-module of rank $2$. By the definition of $S$, the image of $D$ in $\t{C}(\Omega(\Lambda_1)_{\N \cap [2,\infty)} \backslash S, \overline{\Q}_p)$ has no zero, and hence is invertible. We note that $\t{C}(\Omega(\Lambda_1)_{\N \cap [2,\infty)} \backslash S, \overline{\Q}_p)$ is not the $\overline{\Q}_p$-algebra of bounded continuous functions, and hence we need not to argue the lower bound of the absolute values of the image of $D$. By the universality of a localisation, the evaluation map $\Lambda_1 \hookrightarrow \t{C}(\Omega(\Lambda_1)_{\N \cap [2,\infty)} \backslash S, \overline{\Q}_p)$ factors through $\Lambda_1 \hookrightarrow \A \coloneqq (\Q_p \otimes_{\Z_p} \Lambda_1)[D^{-1}]$. By the proof of Lemma \ref{generic}, the $\A$-module $\A \otimes_{\Lambda_1} \Hil_f^{< s} \cong \A \otimes_{\Q_p \otimes_{\Z_p} \Lambda_1} (\Q_p \otimes_{\Z_p} \Hil_f^{< s})$ is free of rank $2$, and hence so is the $\t{C}(\Omega(\Lambda_1)_{\N \cap [2,\infty)} \backslash S, \overline{\Q}_p)$-module $\Hil_f^{< s} |_{\Omega(\Lambda_1)_{\N \cap [2,\infty)} \backslash S} \cong \t{C}(\Omega(\Lambda_1)_{\N \cap [2,\infty)} \backslash S, \overline{\Q}_p) \otimes_{\A} \A \otimes_{\Lambda_1} \Hil_f^{< s}$.
\end{proof}

\vspace{0.4in}
\addcontentsline{toc}{section}{Acknowledgements}
\noindent {\Large \bf Acknowledgements}
\vspace{0.1in}

I am extremely grateful to Takeshi Tsuji for helpful discussions and constructive advices in seminars, and also for careful reading of my paper. The proof of Theorem \ref{interpolation} was greatly shortened by him, and the proof of Lemma \ref{finiteness 2} was completed by him. I asked him plenty of questions about Galois representations. I profoundly thank to Masataka Chida for spending tens of hours on listening to my idea, Kentaro Nakamura for pointing out to me some relation between my work and others, and Yoichi Mieda for recommending me to apply a result of my calculation of the spectrum of a $p$-adic quantum group to interpolation of Galois representations. I deeply appreciate advices, suggestions, and opinions given by Tetsushi Ito, Shin-ichi Kobayashi, Masato Kurihara, Tadashi Ochiai, Atsushi Yamagami, and Seidai Yasuda. I am thankful to my colleagues for sharing so much time with me to discuss. In particular, Yuya Matsumoto told me an elementary proof of Corollary \ref{combinatorial} with no use of $p$-adic representations. I express my gratitude to family for their deep affection.

\vspace{0.2in}
I am a research fellow of Japan Society for the Promotion of Science. This work was supported by the Program for Leading Graduate Schools, MEXT, Japan.

\end{document}